\documentclass[a4paper]{article}
\usepackage[pdftex]{graphicx}
\usepackage{calrsfs}
\usepackage{epstopdf}
\usepackage{amsfonts}
\usepackage{amsthm}
\oddsidemargin=-\evensidemargin
\newtheorem{theorem}{Theorem}[section]
\newtheorem{lemma}[theorem]{Lemma}

\newtheorem{corollary}[theorem]{Corollary}
\newtheorem{remark}[theorem]{Remark}
\newtheorem{proposition}[theorem]{Proposition}

\newcommand{\uline}[1]{{\hbox to 0pt{\underline{\hphantom{{#1}}}\hss}{#1}}}

\newcommand\E\Sigma
\newcommand\ol\overline
\newcommand\T\Theta
\newcommand\Om\Omega

\newcommand\modphi{{\left|\varphi\right|}}
\newcommand\al\alpha
\newcommand\be\beta

\newcommand{\NN}{\mathbb{N}}
\newcommand{\ZZ}{\mathbb{Z}}
\newcommand{\R}{\mathbb{R}}

\newcommand{\Cyc}{\mathop{\mathrm{Cyc}}}
\newcommand{\Su}{\mathop{\mathrm{Dc}}}
\newcommand{\Pd}{\Su^{-1}}
\newcommand{\kSu}{\mathop{\mathrm{kDc}}}
\newcommand{\LL}{\mathop{\mathrm{LL}}}
\newcommand{\RL}{\mathop{\mathrm{RL}}}

\newcommand{\Fg}{\mathop{\mathrm{Fg}}}
\newcommand{\Ker}{\mathop{\mathrm{Ker}}}
\newcommand{\nker}{\mathop{\mathrm{nker}}}
\newcommand{\cKer}{\mathop{\mathrm{cKer}}\nolimits}
\newcommand{\ncker}{\mathop{\mathrm{ncker}}}
\newcommand{\RA}{\mathop{\mathrm{RA}}}
\newcommand{\LA}{\mathop{\mathrm{LA}}}
\newcommand{\CA}{\mathop{\mathrm{A}}}
\newcommand{\LB}{\mathop{\mathrm{LB}}}
\newcommand{\RB}{\mathop{\mathrm{RB}}}

\newcommand{\Cr}{\mathop{\mathrm{C\hspace{0em}}}}
\newcommand{\LpreP}{\mathop{\mathrm{LpreP}}\nolimits}
\newcommand{\RpreP}{\mathop{\mathrm{RpreP}}\nolimits}
\newcommand{\LR}{\mathop{\mathrm{LR}}}
\newcommand{\RR}{\mathop{\mathrm{RR}}}
\newcommand{\IpR}{\mathop{\mathrm{IpR}}}
\newcommand{\LpR}{\mathop{\mathrm{LpR}}\nolimits}
\newcommand{\RpR}{\mathop{\mathrm{RpR}}\nolimits}
\newcommand{\LP}{\mathop{\mathrm{LP}}}
\newcommand{\RP}{\mathop{\mathrm{RP}}}

\newcommand{\LBS}{\mathop{\mathrm{LBS}}}
\newcommand{\RBS}{\mathop{\mathrm{RBS}}}
\newcommand{\finmax}{\mathbf{L}}

\newcommand{\swa}[2]{{\al_{#1\ldots #2}}}

\begin{document}

\title{On Subword Complexity of Morphic Sequences}
\author{Rostislav Devyatov,\\
\small Moscow State University  and Independent University of Moscow,\\
\texttt{\small deviatov1@rambler.ru}}

\maketitle

\begin{abstract}
We study structure of pure morphic and morphic sequences and prove the following result: the subword complexity
of arbitrary morphic sequence is either $\T(n^{1+1/k})$ for some $k\in\NN$, 
or is $O(n \log n)$.
\end{abstract}

\section{Introduction}

Morphisms and morphic sequences are well known and well studied in
combinatorics on words (e.~g., see \cite{AlShall03}). We study their
subword complexity.

Let $\E$ be a finite alphabet. A mapping $\varphi\colon \E^*\to
\E^*$ is called a \emph{morphism} if $\varphi(\gamma\delta) =
\varphi(\gamma)\varphi(\delta)$ for all $\gamma,\delta\in \E^*$. A morphism is 
determined by its values on single-letter words. A morphism is
called \emph{nonerasing} if $|\varphi(a)| \ge 1$ for each $a \in \E$, and
is called \emph{coding} if $|\varphi(a)| = 1$ for each $a \in \E$.
Let $\modphi$ denote $\max_{a \in \Sigma} |\varphi(a)|$.

Let $\varphi(a) = a\gamma$ for some $a\in \E$, $\gamma\in \E^*$, and suppose
$\forall n\ \varphi^n(\gamma)$ is not empty. Then an infinite sequence
$\varphi^\infty(a) = \lim_{n \to \infty} \varphi^n(a)$
is well-defined and is
called \emph{pure morphic}. Sequences of the form
$\psi(\varphi^\infty(a))$ with coding $\psi$ are called \emph{morphic}.

In this paper we study a natural combinatorial characteristics of
sequences, namely subword complexity. The \emph{subword complexity} of a
sequence $\be$ is a function $p_\be\colon \NN \to \NN$ where
$p_\be(n)$ is the number of all different $n$-length subwords
occurring in~$\be$.
For a survey on subword complexity, see, e.~g., \cite{Fer}.
Pansiot showed~\cite{Pans84} that the subword
complexity of an arbitrary pure morphic sequence adopts one of the five
following asymptotic behaviors: $O(1)$,  $\Theta(n)$, $\Theta(n \log
\log n)$, $\Theta(n \log n)$, or $\Theta(n^2)$. Since codings can
only decrease subword complexity, the subword
complexity of every morphic sequence is $O(n^2)$.
We formulate the following main result.

\begin{theorem}\label{maintheorem}
The subword complexity $p_\be$ of a morphic sequence $\be$ is either
$p_\be(n) = \T(n^{1 + 1/k})$ for some $k \in \NN$, or $p_\be(n) =
O(n\log n)$.
\end{theorem}

Note that for each $k$ the complexity class 
$\T(n^{1 + 1/k})$ 
is non-empty~\cite{NicMasters}.

We give an example of a morphic sequence $\be$ with $p_\be =
\T(n^{3/2})$ in Section~\ref{An_Example}.

Let $\E$ be a finite alphabet, $\varphi\colon\E^*\to\E^*$ be a
morphism, $\psi\colon\E^*\to\E^*$ be a coding, $a\in\E$ be 
a letter such that $\varphi(a)$ starts with $a$, $\al=\varphi^\infty(a)$ be the pure
morphic sequence generated by $\varphi$ from $a$, and $\be = \psi(\al)$ be a
morphic sequence. By Theorem~7.7.1 from \cite{AlShall03} every
morphic sequence can be generated by a  nonerasing morphism, so
further we assume that $\varphi$ is nonerasing. To prove Theorem~\ref{maintheorem},
we will first replace $\varphi$ by its power so that it will have 
better properties, see Section \ref{morphperiodicity}.
It is already clear from the definition of a pure morphic sequence
that if we replace $\varphi$ by its power, then $\al$ and $\be$ will not change.

Possibly, we will also add several (at most two) "new" letters to $\E$
so that $\varphi$ and $\psi$ will be defined on the "old" letters as previously, 
and will map the "new" letters to the "new" letters only. This will not modify 
$\al$ and $\be$, and the only reason why we do that is that this simplifies formulations 
of some statements. For example, we may want to say that a (finite) subword $\gamma$ of 
$\be$ can be written as a finite word $\lambda\in\E^*$ repeated several times, where $\lambda$ 
belongs to a prefixed finite set.
In a particular case it can turn out that $\gamma$ is the empty word, and then it can 
be written as any word $\lambda\in\E^*$ repeated zero times. However, to 
ease the formulation of this statement, it is convenient to know that the set where 
we are allowed to take $\lambda$ from is nonempty, even if all letters of all possible words 
$\lambda$ are not present in $\be$ at all.


To prove Theorem~\ref{maintheorem}, we will have to develop some "structure theory"
of pure morphic and morphic sequences (see Sections \ref{sectionblocks}--\ref{sectionevolutions}).
We will introduce and study the notions of a letter of order $k$, of a $k$-block, of a $k$-multiblock, 
of a stable $k$-(multi)block, 
of an evolution, and of a continuously periodic evolution.
Actually, these notions 
will be defined correctly only after we replace $\varphi$ with $\varphi^n$ for an appropriate 
$n\in\NN$ and possibly add several letters to $\E$ as explained in Section \ref{morphperiodicity}
(more precisely, if $\varphi$ is a strongly 1-periodic morphism with long images, and if $\E$ contains 
at least one periodic letter of order 1 and at least one periodic letter of order 2).
These studies of the structure of pure morphic and morphic sequences may be of independent interest.

Using these notions, we can formulate the following two propositions, which the proof of 
Theorem~\ref{maintheorem} is based on:

\begin{proposition}\label{largecompl}
Let $k\in\NN$.
If $a\in\E$ is a letter such that $\varphi(a)=a\gamma$ for some $\gamma\in\E^*$, 
and there are evolutions of $k$-blocks arising in $\al=\varphi^\infty(a)$ that are not
continuously periodic, then the subword complexity of $\be=\psi(\al)$ is
$\Om(n^{1+1/(k-1)})$.
\end{proposition}

\begin{proposition}\label{smallcompl}
Let $k\in\NN$.
If $a\in\E$ is a letter of order at least $k+2$ such that $\varphi(a)=a\gamma$ for some $\gamma\in\E^*$, 
and all evolutions of $k$-blocks arising in $\al=\varphi^\infty(a)$ are continuously periodic,
then the subword complexity of $\be=\psi(\al)$ is $O(n^{1+1/k})$.
\end{proposition}


However, these two propositions do not cover all cases needed to prove Theorem \ref{maintheorem}.
This is not clear right now, before we give the definitions, but, for example, if $a$ is a letter of order $k+2$, where 
$k\in\NN$, such that $\varphi(a)=a\gamma$ for some $\gamma\in\E^*$, and evolutions of $(k+1)$-blocks that 
are not continuously periodic do not exist (as we will see later, in this case evolutions of $(k+1)$-blocks do not exist 
at all), then Proposition \ref{largecompl} does not give us any upper estimate, 
and we cannot use Proposition \ref{smallcompl} either, because if we want to use it for $(k+1)$-blocks, 
$a$ has to be a letter of order at least $k+3$. Also, Propositions
\ref{largecompl} and \ref{smallcompl} do not say anything about complexity $O(n\log n)$. The following three propositions 
will help us to prove Theorem \ref{maintheorem} in these cases:

\begin{proposition}\label{finordercomplstop}
Let $k\in\NN$.
Suppose that $\varphi$ is a strongly 1-periodic morphism with long images
and $a\in\E$ is a letter of order $k+2$ such that $\varphi(a)=a\gamma$ for some $\gamma\in\E^*$.
Suppose that all evolutions of $k$-blocks arising in $\al=\varphi^\infty(a)$ are continuously periodic.

Let $\al_i$ be the rightmost letter of order $k+1$ in $\varphi^{3k+1}(a)$, and let $\al_j$
be the rightmost letter of order $k+1$ in $\varphi^{3k+2}(a)$.

If there exists a final period $\lambda$ such that $\psi(\swa{i+1}j)$ is a completely $|\lambda|$-periodic
word with period $\lambda$, then the subword complexity of $\be=\psi(\al)$ is $O(1)$, otherwise it is $\Theta(n^{1+1/k})$.
\end{proposition}

\begin{proposition}\label{ordertwocompl}
If $a\in\E$ is a letter of order 2 such that $\varphi(a)=a\gamma$ for some $\gamma\in\E^*$,
then the subword complexity of $\be=\psi(\varphi^\infty(a))$ is $O(1)$.
\end{proposition}

\begin{proposition}\label{infordercomplstop}
Let $k\in\NN$.
Let $a\in\E$ be a letter of order $\infty$ such that $\varphi(a)=a\gamma$ for some $\gamma\in\E^*$, and let $\al=\varphi^\infty(a)$.
Suppose that if $b\in\E$ is a letter of finite order $k'$ and $b$ occurs in $\al$, then $k'<k$.
Suppose that all evolutions of $k$-blocks arising in $\al$ are continuously periodic.

Then
the subword complexity of $\be=\psi(\al)$ is
$O(n\log n)$.
\end{proposition}

\section{Preliminaries}\label{prelim}

When we speak about finite words or about words infinite to the right, 
their letters are enumerated by nonnegative integer indices
(starting from 0). The length of a finite word $\gamma$ is denoted by $|\gamma|$.

We will speak about occurrences in $\al=\al_0\al_1\al_2\ldots\al_i\ldots$. 
Strictly speaking, we call a pair of a word $\gamma$ and a location
$i$ in $\al$ \textit{an occurrence} if the subword of $\al$ that starts from position $i$
in $\al$ and is of length $|\gamma|$ is $\gamma$. This occurrence is denoted by $\swa ij$ if
$j$ is the index of the last letter that belongs to the occurrence. In particular, $\swa ii$ denotes
a single-letter occurrence, and $\swa i{i-1}$ denotes an occurrence of the empty word between
the $(i-1)$-th and the $i$-th letters. Since
$\al=\al_0\al_1\al_2\ldots=\varphi(\al)=\varphi(\al_0)\varphi(\al_1)\varphi(\al_2)\ldots$,
$\varphi$ might be considered either as a morphism on words (which we
call abstract words sometimes), or as
a mapping on the set of occurrences in $\al$. 
Usually we speak of
the latter, unless stated otherwise. Sometimes we write $\varphi^0$ for the identity 
morphism.

A finite word $\delta$ is called a \textit{prefix} of a (finite or infinite to the right) 
word $\gamma$ if $\gamma_{0\ldots|\delta|-1}=\delta$. A finite word $\delta$ is called
a \textit{suffix} of a finite word $\gamma$ if $\gamma_{|\gamma|-|\delta|\ldots|\gamma|-1}=\delta$.

We call a finite word $\gamma$ \emph{weakly $p$-periodic}
with a \textit{left} (resp.\ \textit{right}) period $\delta$ (where $p\in\NN$)
if $|\delta|=p$ and $\gamma = \delta\delta\ldots \delta\delta_{0\ldots r-1}$ (resp.\ 
$\gamma = \delta_{p-r\ldots p-1}\delta\ldots\delta$), 
where $r$ is the remainder of $|\gamma|$ modulo $p$, $r=0$ is allowed here.
We shortly say "a weakly left (resp.\ right) $\delta$-periodic word"
instead of "a weakly $|\delta|$-periodic word with left (resp.\ right)
period $\delta$".
$\delta$ will be always considered as an abstract word.
The subword $\gamma_{|\gamma|-r\ldots |\gamma|-1}$ 
(resp.\ $\gamma_{0\ldots r-1}$)
is called the \textit{incomplete occurrence}. All the
same is with sequences of symbols or numbers.
If $r=0$, then 
$\gamma$ is called \textit{completely $p$-periodic} 
with period $\delta$ (which is both left period and right period in this case, 
so we sometimes call it a \textit{complete} period).
Again, we shortly say "a completely $\delta$-periodic word"
instead of
"a completely $|\delta|$-periodic word with period $\delta$".


Clearly, a weakly $p$-periodic word with some left period
always is also weakly $p$-periodic with some right period, and these periods 
are cyclic shifts of each other. So, we introduce some notation for 
cyclic shifts. If $\delta$ is a finite word and $0\le r<|\delta|$, 
we denote the cyclic shift of $\delta$ that begins with the 
last $|\delta|-r$ letters of $\delta$ and 
ends with the first $r$ letters of $\delta$ by $\Cyc_r(\delta)$.
In other words, $\Cyc_r(\delta)=\delta_{r\ldots |\delta|-1}\delta_{0\ldots r-1}$.
If $n\in\ZZ$ and $r$ is the residue of $n$ modulo $|\delta|$, 
we denote $\Cyc_n(\delta)=\Cyc_r(\delta)$. In particular, if $0<n<|\delta|$, 
then $\Cyc_{-n}(\delta)=\Cyc_{|\delta|-n}(\delta)=\delta_{|\delta|-n\ldots|\delta|-1}\delta_{0\ldots|\delta|-n-1}$, 
in other words, $\Cyc_{-n}(\delta)$ is the cyclic shift of $\delta$ that begins with 
the last $n$ letters of $\delta$ and ends with the first $|\delta|-n$ letters of $\delta$.

We widely use the following easy properties of periods and cyclic shifts:
\begin{remark}
\begin{enumerate}
\item If $n,m\in\ZZ$, then $\Cyc_{n+m}(\delta)=\Cyc_n(\Cyc_m(\delta))$.
\item If a finite word $\gamma$ is weakly $p$-periodic with left period $\delta$, 
where $\delta$ is a word of length $p$, then $\gamma$ is also weakly $p$-periodic
with right period $\delta'=\Cyc_{|\gamma|}(\delta)$.
\item If a finite word $\gamma$ is weakly $p$-periodic with right period $\delta$, 
where $\delta$ is a word of length $p$, then $\gamma$ is also weakly $p$-periodic
with left period $\delta'=\Cyc_{-|\gamma|}(\delta)$.
\item If $\delta$ is a word of length $p$, two finite words $\gamma$ and $\gamma'$ 
are weakly $p$-periodic, and 
$\gamma$ (resp.\ $\gamma'$) is weakly $p$-periodic with right (resp.\ left) period $\delta$, 
then the concatenation $\gamma\gamma'$ is weakly $p$-periodic
with left period $\delta'=\Cyc_{-|\gamma|}(\delta)$
and is also weakly $p$-periodic with right period $\delta''=\Cyc_{|\gamma'|}(\delta)$.
\end{enumerate}
\end{remark}

The following lemma, informally speaking, shows that if we know a finite word 
is "long enough" and is weakly $p$-periodic for some $p$, 
which we maybe don't know itself, but we know that $p$ is "small enough", 
then these data determine $p$ and the left, right or complete period uniquely.

\begin{lemma}\label{finwordperiods}
Let $\gamma$ be a finite word. Suppose that $\gamma$ is weakly $p_1$-periodic 
with a left period $\delta$ and is weakly $p_2$-periodic with a left period $\sigma$ at the same time.
Suppose also that $|\gamma|\ge 2p_1$ and $|\gamma|\ge 2p_2$.
Then there
exists a finite word $\lambda$ such that $\delta$ is $\lambda$ repeated $k$ times and
$\sigma$ is $\lambda$ repeated $l$ times for some $k,l\in\mathbb N$.
\end{lemma}

\begin{proof}
If $p_1=p_2$, then the statement is obvious. Otherwise,
without loss of generality we may suppose that $p_1>p_2$. Then $\sigma=\gamma_{0\ldots p_2-1}=\delta_{0\ldots p_2-1}$. 

Note that the fact that
$\gamma$ is weakly $p$-periodic can be written as follows: 
for all $0\le i<|\gamma|-p$ one has $\gamma_i=\gamma_{i+p}$.
Let us prove that $\gamma$ is weakly $(p_1-p_2)$-periodic with a left period $\delta_{0\ldots p_1-p_2-1}$.
Choose an index $i$, $0\le i<|\gamma|-(p_1-p_2)$.
If $i<|\gamma|-p_1$, then $\gamma_i=\gamma_{i+p_1}$ (since $\gamma$ is weakly $p_1$-periodic) and 
$\gamma_{i+p_1}=\gamma_{i+p_1-p_2}$ (since $\gamma$ is weakly $p_2$-periodic). If $i\ge |\gamma|-p_1$, 
then $i\ge p_2$ since $2p_1\le|\gamma|$, so $p_1+p_2\le|\gamma|$. So $\gamma_i=\gamma_{i-p_2}=\gamma_{i+p_1-p_2}$.

Note that if $p_1$ is divisible by $p_2$, then the claim is also clear.
Otherwise set $q_2=\lfloor p_1/p_2\rfloor$, and write $p_1=q_2p_2+p_3$, where $0<p_3<p_2$. 
If we repeat the argument above $q_2$ times, we will see that $\gamma$ is weakly $p_3$-periodic with a left period 
$\delta_{0\ldots p_3-1}$.

Finally, we write Euclid algorithm for $p_1$ and $p_2$:\\
$p_1=p_2q_2+p_3$\\
$p_2=p_3q_3+q_4$\\
$\cdots$\\
$p_{m-2}=p_{m-1}q_{m-1}+p_m$\\
$p_{m-1}=p_mq_m$.\\
If we repeat all arguments above for each of the pairs $(p_1,p_2), (p_2,p_3),\ldots,(p_{m-2},p_{m-1})$, 
we will finally see that $\gamma$ is weakly $p_m$-periodic with a left period $\lambda=\delta_{0\ldots p_m-1}$, 
where $p_m$ is the g. c. d. of $p_1$ and $p_2$. In particular, since $\gamma$ is also 
weakly $p_1$-periodic with left period $\delta$ and $|\gamma|>p_1$, this also means that $\delta$
is $\lambda$ repeated $p_1/p_m$ times. Similarly, $\sigma$ is $\lambda$ repeated $p_2/p_m$ times.
\end{proof}

The same lemma for right periods instead of left ones can be proved in completely the same way. 
After we have this lemma, it is reasonable to give the following definition. A finite word $\lambda$
is called the \textit{minimal left (resp.\ right) period} of a finite word $\gamma$ if $2|\lambda|\le|\gamma|$,
$\gamma$ is weakly 
left (resp.\ right) $\lambda$-periodic
and 
$\gamma$ is not weakly $p$-periodic if $p<|\lambda|$. The following corollary provides more properties 
of the minimal periods if they exist.

\begin{corollary}
Let $\gamma$ be a finite word. If there exists $p\in\NN$ such that $\gamma$ is weakly $p$-periodic 
and $2p\le|\gamma|$, then there exist minimal left and right periods of $\gamma$. 

If $\lambda$ is 
the minimal left (resp.\ right) period of $\gamma$, and $\gamma$ is weakly $p$-periodic with left (resp.\ right) 
period $\delta$, where $2p\le|\gamma|$, then $p$ is divisible by $|\lambda|$ and 
$\delta$ is $\lambda$ repeated $p/|\lambda|$ times.\qed
\end{corollary}

A similar statement in the case of complete $p$-periodicity follows directly since a word is 
completely $p$-periodic exactly if it is weakly $p$-periodic and its length is divisible by $p$. 
A finite word $\lambda$ 
is called the \textit{minimal complete period} of a finite word $\lambda$ if $2|\lambda|\le|\gamma|$,
$\gamma$ is 
completely $\lambda$-periodic,
and $\gamma$ is not 
weakly $p$-periodic if $p<|\lambda|$.

\begin{corollary}\label{finwordmincompleteperiod}
Let $\gamma$ be a finite word. If there exists $p\in\NN$ such that $\gamma$ is completely $p$-periodic 
and $2p\le|\gamma|$, then there exist a minimal complete period of $\gamma$. 

If $\lambda$ is 
the complete period of $\gamma$, and $\gamma$ is weakly $p$-periodic with left (resp.\ right) 
period $\delta$, where $2p\le|\gamma|$, then $p$ is divisible by $|\lambda|$ and 
$\delta$ is $\lambda$ repeated $p/|\lambda|$ times.\qed
\end{corollary}

\begin{corollary}\label{overlapperiod}
Let $\gamma$ be a finite word, let $\gamma_{i\ldots j}$ and $\gamma_{i'\ldots j'}$ be two 
occurrences in $\gamma$. Suppose that $\gamma_{i\ldots j}$ is weakly $p$-periodic, 
and $\gamma_{i'\ldots j'}$ is weakly $p'$-periodic. Suppose also that these two occurrences overlap, 
and their intersection (denote it by $\gamma_{s\ldots t}$) has length at least $2\max(p,p')$. 
In other words, $s=\max(i,i')$, $t=\min(j,j')$, and $t-s+1\ge 2\max(p,p')$. 

Then the union of these two occurrences (i.~e.\ the occurrence $\gamma_{s'\ldots t'}$, where 
$s'=\min(i,i')$ and $t'=\max(j,j')$) is a weakly $\gcd(p,p')$-periodic word.
\end{corollary}

\begin{proof}
Without loss of generality, $i\le i'$. Then $s=i'$ and $s'=i$. Let $\delta$ be the left period of 
$\gamma_{i\ldots j}$ (so that $|\delta|=p$), and let $\delta'$ be the left period of $\gamma_{i'\ldots j'}$
(so that $|\delta'|=p'$). Denote the residue of $i'-i$ modulo $p$ by $r$. Then, if we write 
$\gamma_{i\ldots j}$ as $\delta$ repeated several times, $\gamma_{i'}$ will be $\delta_r$. Moreover, 
$\gamma_{s\ldots t}$ becomes a weakly $p$-periodic word with left period 
$\delta''=\delta_{r\ldots|\delta|-1}\delta_{0\ldots r-1}=\Cyc_r(\delta)$.
Since $\gamma_{s\ldots t}$ is a prefix of $\gamma_{i'\ldots j'}$, $\gamma_{s\ldots t}$ is also a weakly 
$p'$-periodic word with left period $\delta'$. Now, by Lemma \ref{finwordperiods}, 
there exists a word $\lambda$ of length $\gcd(p,p')$ such that $\delta'$ is $\lambda$ repeated $p'/\gcd(p,p')$
times and $\delta''$ is $\lambda$ repeated $p/\gcd(p,p')$ times. But then $\delta$ can also be written as 
a cyclic shift of $\lambda$ repeated $p/\gcd(p,p')$ times.

Now, since $\gamma_{i\ldots j}$ is weakly $p$-periodic with left period $\delta$, it is also 
weakly $\gcd(p,p')$-periodic. Since $\gamma_{i'\ldots j'}$ is weakly $p'$-periodic with 
left period $\delta'$, it is also 
weakly $\gcd(p,p')$-periodic. In other words, if $k$ and $k+\gcd(p,p')$ are two indices such that 
$i\le k$ and $k+\gcd(p,p')\le j$, then $\gamma_k=\gamma_{k+\gcd(p,p')}$ as an abstract letter.
Also, if $k$ and $k+\gcd(p,p')$ are two indices such that 
$i'\le k$ and $k+\gcd(p,p')\le j'$, then again $\gamma_k=\gamma_{k+\gcd(p,p')}$ as an abstract letter.

If $j\ge j'$, we are done. Otherwise $t=j$ and $t'=j'$, and we have $j-i'+1\ge 2\max(p,p')\ge 2\gcd(p,p')$.
So, if $k+\gcd(p,p')\le t'=j'$, but $k+\gcd(p,p')>j$, then $k+\gcd(p,p')\ge j+1$, $k\ge j+1-\gcd(p,p')\ge i'$, 
and $\gamma_k=\gamma_{k+\gcd(p,p')}$ anyway. Hence, $\gamma_{s'\ldots t'}=\gamma_{i\ldots j'}$ is 
weakly $\gcd(p,p')$-periodic.
\end{proof}

Note that in the last computation an inequality $t-s+1\ge \gcd(p,p')$ instead of $t-s+1\ge 2\max(p,p')$
would be enough, but we cannot replace $t-s+1\ge 2\max(p,p')$ with $t-s+1\ge \gcd(p,p')$
in the statement of the corollary, because we also need the inequality $t-s+1\ge 2\max(p,p')$
in Lemma \ref{finwordperiods}, and there it cannot be a priori replaced by $t-s+1\ge \gcd(p,p')$.

An infinite 
word 
$\gamma=\gamma_0\gamma_1\ldots\gamma_i\ldots$ 
(where $\gamma_i\in\E$) is called 
periodic with a period $\delta$ (where 
$\delta=\delta_0\ldots\delta_{p-1}$, 
$\delta_i\in\E$) 
if 
$\gamma=\delta\delta\delta\ldots$, 
in other words, if $\gamma_{ip+j}=\delta_j$ for all $i=0,1,2,\ldots$, $j=0,1,\ldots,p-1$.
An infinite 
word
$\gamma=\gamma_0\gamma_1\ldots\gamma_i\ldots$ 
(where $\gamma_i\in\E$) is called 
eventually periodic with a period $\delta$ 
(where 
$\delta=\delta_0\ldots\delta_{p-1}$, 
$\delta_i\in\E$) 
and a preperiod $\delta'$ 
(where 
$\delta'=\delta'_0\ldots\delta'_{p'-1}$, 
$\delta_i\in\E$) 
if 
$\gamma=\delta'\delta\delta\delta\ldots$, 
in other words, 
if $\gamma_i=\delta'_i$ for $i=0,1,\ldots,p'-1$ and 
$\gamma_{p'+ip+j}=\delta_j$ for all $i=0,1,2,\ldots$, $j=0,1,\ldots,p-1$.

Sometimes we will also speak about words infinite to the left. We enumerate indices in
such words by nonpositive indices, i.~e.\ such a word can be written as 
$\gamma=\ldots\gamma_{-i}\ldots\gamma_{-1}\gamma_0$ (where $\gamma_{-i}\in\E$, $i\in\ZZ_{\ge 0}$).
Such a word
is called 
periodic with a period $\delta$ (where 
$\delta=\delta_0\ldots\delta_{p-1}$, 
$\delta_i\in\E$) 
if 
$\gamma=\ldots\delta\delta\delta$, 
in other words, if $\gamma_{-ip+j+1}=\delta_j$ for all $i=1,2,\ldots$, $j=0,1,\ldots,p-1$.


\section{Periodicity properties of morphisms}\label{morphperiodicity}

For each letter $b\in\E$, the function $r_b\colon\mathbb N\to\mathbb N$, $r_b(n) =
|\varphi^n(b)|$ is called \emph{the growth rate} of~$b$. Let us define
\textit{orders of letters} with respect to $\varphi$. We say that
$b \in \E$ has \emph{order $k$} if $r_b(n) = \T(n^{k-1})$, and has
\emph{order $\infty$} if $r_b(n) = \Om(q^n)$ for some $q >
1$ ($q\in\mathbb R$).

Consider a directed graph $G$ defined as follows. Vertices of $G$ are
letters of~$\E$. For every $b,c \in \E$, for each occurrence of $c$
in $\varphi(b)$, construct an edge $b \to c$. For instance, if
$\varphi(b) = bccbc$, we construct two edges $b \to b$ and three
edges~$b \to c$.
Fig.~\ref{Gexample} shows an example of graph $G$.

\begin{figure}[!h]
\centering
\includegraphics{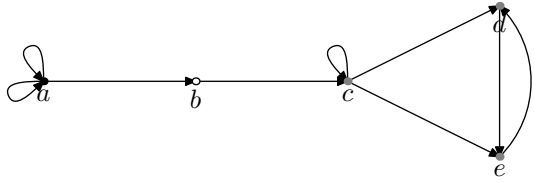}
\caption{An example of graph $G$ for the following morphism $\varphi$:
$\varphi(a)=aab$, $\varphi(b)=c$, $\varphi(c)=cde$, $\varphi (d)=e$, $\varphi (e)=d$.
Here $a$ is a letter of order $\infty$, $b$ is a preperiodic letter of order 2,
$c$ is a periodic letter of order 2, $d$ and $e$ are periodic letters of order 1.}
\label{Gexample}
\end{figure}

Using the graph $G$, let us prove the following lemma.

\begin{lemma}\label{rates}
For every $b\in\E$, either $b$ has some order $k < \infty$, or has
order $\infty$. If $b$ is a letter of order $k$, then $\varphi(b)$ contains 
at least one letter of order $k$. For every $b$ of order $k < \infty$, either $b$
never appears in $\varphi^n(b)$ (and then $b$ is called
\emph{preperiodic}), or for each $n$ a unique letter $c_n$ of order
$k$ occurs in $\varphi^n(b)$, and the sequence $(c_n)_{n \in \mathbb Z_{{}\ge0}}$ is
periodic (then $b$ is called \emph{periodic}).

If $b$ is a letter of order $\infty$, then $\varphi(b)$ contains at least one letter of order $\infty$, 
and $\varphi^n(b)$ contains at least two letters of order $\infty$ if $n$ is large enough.

If $b$ is a periodic
letter of order $k > 1$ and $b$ occurs in $\varphi^n(b)$, then at least one letter of order $k - 1$
occurs in $\varphi^n(b)$.
\end{lemma}

\begin{proof}
Consider also the following graph $\ol G$. Vertices of $\ol G$ are
strongly connected components of
$G$. There is an edge from $v\in\ol G$ to $u\in\ol G$ iff there is
an edge from some of $v$ vertices (in $G$) to some of $u$ vertices.
Fig.~\ref{Gpexample} shows an example of the corresponding graph $\ol G$.

\begin{figure}[!h]
\centering
\includegraphics{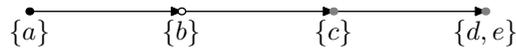}
\caption{An example of the graph $\ol G$ corresponding to the graph $G$ from the previous 
example.}
\label{Gpexample}
\end{figure}

Let $\ol G{}'$ be the subgraph of $\ol G$ induced by vertices $v\in\ol G$ such that for all
vertices $a\in v\subset G$
there is at most one edge
outgoing from $b$ to a vertex $c\in v$. Let $\ol G{}''$ be the subgraph of
$\ol G{}'$ induced by vertices $v\in\ol G{}'$ such that for all vertices $b\in v\subset G$
there are no edges outgoing from $b$ to a
vertex $c\in v$. In Fig.~\ref{Gpexample} and \ref{Gexample} 
the vertices of $\ol G{}\setminus \ol G{}'$ (resp.\ the 
corresponding vertices of $G$) are black,
the vertices of $\ol G{}'\setminus \ol G{}''$ (resp.\ the 
corresponding vertices of $G$) are gray, and the vertices of $\ol G{}''$ (resp.\ the 
corresponding vertices of $G$) are white.
We will now assign orders (natural numbers or infinity) to the vertices
of $\ol G$ (hence, to the vertices of $G$ too).

A vertex $v\in\ol G{}'$ is called a vertex of order one if it does not have outgoing edges 
(in $\ol G{}'$, not in $G$). Then assign order one to the vertices (if any) of graph
$\ol G{}''$ that have outgoing edges to the vertices that are already of order one only.
Repeat this operation until there are no new vertices of order one.

Suppose some vertices already are of order $k$ (and we don't want to assign order $k$ to
any other vertex of $\ol G$). Then a vertex $v\in\ol G{}'$ is called vertex of order $k+1$
if all the edges outgoing from it lead to vertices of order $k$ or less. Then, 
consider a vertex $w\in\ol G{}''$ that has not been currently assigned to be of some order.
If all its outgoing edges lead to vertices of orders $\le k+1$, assign
$w$ to be of order $k+1$. Repeat this operation
until there are no new vertices of order $k+1$.

All vertices that currently have no order assigned (after completing the above procedure
for each $k$), are called vertices of order $\infty$.

We have assigned orders to the vertices of $\ol G$, hence also to the vertices of $G$ (that are 
the letters of $\Sigma$). It follows directly from the definition of the order of a vertex
that if $b\in G$ is a vertex of order $k$, then there is an edge going from $a$ to (possibly another) vertex
of order $k$.
One can prove by induction on $k$ that

\textit{Any letter of finite order $k$ has the rate of growth $\T(n^{k-1})$. Any letter of
infinite order has the rate of growth $\Om(\gamma^n)$ for some $\gamma>1$.}

Thus, two definitions of the order of a letter are equivalent.

Vertices $v$ of $\ol G$ of order $\infty$ are exactly 
the vertices of $\ol G$ such that there exists a path from $v$ to a vertex 
$w\in\ol G\setminus \ol G'$. It is already clear that if $b$ is a letter of 
order $\infty$, then $\varphi(b)$ contains a letter of order $\infty$. To prove that 
if $n$ is large enough, then $\varphi^n(b)$ contains at least two letters of order $\infty$, 
we may assume without loss of generality that $b$ already belongs to a strongly connected component 
$v$ of $G$ such that $v\notin\ol G'$. Then there exists a vertex $c\in v$ such that there are 
at least two edges leading from $c$ to vertices of $G$ in $v\in\ol G$. This means that 
$\varphi(c)$ contains at least two letters of order $\infty$, and $\varphi^{n_0}(b)$ contains $c$ 
for some $n_0$. Then $\varphi^n(b)$ contains at least two letters of order $\infty$ if 
$n>n_0$.

A vertex of $\ol G{}'$ of finite order is called \emph{preperiodic} if it actually 
belongs to $\ol G{}''$, otherwise it is called \emph{periodic}. A vertex of $G$ (i. e.\ a letter) is 
called \emph{periodic} (resp.\ \emph{preperiodic}) iff the corresponding vertex of $\ol G$ is periodic 
(resp.\ preperiodic). If $b\in G$ is a periodic vertex of order $k<\infty$, it has exactly 
one outgoing edge to a vertex of order $k$. These two vertices correspond to the same
vertex $v\in\ol G$, and all vertices of $G$ that correspond to $v$ (i. e.\ that belong 
to the strongly connected component $v$) actually form a directed loop. Unlike that, 
any edge that starts in a preperiodic vertex $b\in G$ of order $k$, leads to a vertex 
that had been assigned to be of some order $\le k$ before $a$.
Hence, this definition 
of a periodic letter and the definition from the lemma statement are equivalent.

To prove the last claim, observe that if $v\in\ol G{}'$ is a periodic letter of order $k$, 
then there must be an edge going from $v$ to a vertex of order $k-1$, otherwise we would have 
assigned $v$ to be a vertex of order $k-1$ or less. Therefore, there is a vertex $b\in G$ corresponding 
to $v$ such such that there is an edge going from $b$ to a vertex of $G$ of order $k-1$. 
In other words, $\varphi (b)$ contains a letter of order $k-1$. Let $c$ be (possibly another) vertex of $G$ 
corresponding to $v\in\ol G{}'$. Then we already know that $b$ and $c$ are 
contained in a directed loop in $G$. If $c$ occurs in $\varphi^n(c)$, then $n$ is divisible 
by the length of this loop, hence $n$ is greater than or equal to the length of this loop, 
and there exists $m$ ($0\le m<n$) such that $\varphi^m(c)$ contains $b$. Then $\varphi(\varphi^m(c))$ contains 
a letter $d$ of order $k-1$, and $\varphi^n(c)$ contains $\varphi^{n-m-1}(d)$. The image of a letter of 
order $k-1$ always contains a letter of order $k-1$, so a letter of order $k-1$ occurs in 
$\varphi^{n-m-1}(d)$ and hence in $\varphi^n(c)$.
\end{proof}

In the example of a graph $G$ in Fig.~\ref{Gexample}, 
$d$ and $e$ are vertices of order one. We cannot assign any other vertex to be 
of order one, so we assign then $c$ to be of order two. It is a periodic vertex. Then we can see 
that $b$ has a single outgoing edge, and it leads to $c$. Thus, $b$ should be a 
preperiodic vertex of order two. The remaining vertex $a$ cannot be of finite order since it does 
not belong to $\ol G{}'$. It is a vertex of order $\infty$.

In general, it is possible that all letters in $\E$ have order $\infty$. However, it will be 
convenient for us if at least one periodic letter of order 1 and at least one periodic letter of order 2 exists.
So, first, if periodic letters of order 1 do not exist in $\E$ (then it follows from the construction above that 
all letters in $\E$ have order $\infty$), we add one more letter (that we temporarily denote by $b$) to $\E$ and 
set $\varphi(b)=b$, $\psi(b)=b$ (without varying $\varphi$ and $\psi$ on other letters). Then $b$ is a periodic letter 
of order 1. From now on, we suppose that periodic letters of order 1 exist in $\E$.

Second, suppose that periodic letters of order 1 exist in $\E$, but periodic letters of order 2 do not exist (it follows from the
above construction that in this case all letters in $\E$ have either order 1, or order $\infty$).
Let $b\in\E$ be a periodic letter of order 1. We add one more letter to $\E$ (denote it temporarily by $c$)
and set $\varphi(c)=bc$, $\psi(c)=c$ (again, we do not change $\varphi$ and $\psi$ on other letters).
Then $c$ is a periodic letter of order 2. From now on, we suppose that periodic letters of order 2 exist in $\E$.

Now we are going to replace $\varphi$ by $\varphi^n$ for some $n\in\NN$ to get a morphism satisfying better 
properties. 
Namely, first let us call a nonerasing morphism $\varphi'$ \textit{weakly 1-periodic} if:
\begin{enumerate}
\item If $b$ is a preperiodic letter of order $k$, then all letters of order $k$ in $\varphi'(b)$ are periodic.
\item If $b$ is a periodic letter of order $k$, then the letter of order $k$ contained in $\varphi'(b)$
is $b$.
\end{enumerate}

We would like to choose $n$ so that $\varphi^n$ is a weakly 1-periodic morphism.
Note first that the orders of letters with respect to $\varphi^n$ are the same as their orders
with respect to $\varphi$. Periodic and preperiodic letters with respect to $\varphi$ remain periodic 
and preperiodic (respectively) with respect to $\varphi^n$. If the first letter in $\varphi(a)$ is $a$ for some $a\in\E$, then 
$\varphi^n(a)$ begins with $a$ as well, and $(\varphi^n)^\infty(a)=\varphi^\infty (a)$.

\begin{lemma}
There exists $n\in\NN$ such that $\varphi^n$ is a weakly 1-periodic morphism.
\end{lemma}
\begin{proof}
If $b$ is a preperiodic letter of order $k$, then $\varphi^n(b)$ does not contain $b$ for any $n\ge 1$.
Therefore, there exists $n_0\in\NN$ such that if $n\ge n_0$, then all letters of order $k$ in $\varphi^n(b)$ 
are periodic. Take any $n\in\NN$ such that $n\ge n_0$ for all these numbers $n_0$ for all preperiodic letters $b\in\E$ of 
finite order. (Clearly, $n=|\E|$ is sufficient, in the example above we can take $n=1$.) Set $\varphi''=\varphi^n$. 
If $b$ is a preperiodic letter of order $k$, all letters of order $k$ in $\varphi''(b)$ are periodic, and 
all letters of order $k$ in $\varphi''^m(b)$ are also periodic. So, now it is sufficient 
to choose $m$ so that if $b$ is a periodic letter of order $k$, then the letter of order $k$ occurring 
in $\varphi''^m(b)$ is $b$ again. By the definition of a periodic letter, for each individual 
periodic letter $b$ there exists $m_0\in \NN$ such that if $m$ is divisible by $m_0$, then 
the letter of order $k$ contained in $\varphi''^m(b)$ is $b$. Now let us take $m\in \NN$ divisible
by all numbers $m_0$ for all periodic letters $b$. (E. g., we can always take $m=|\E|!$, and 
in the example above we can take $m=2$.) Then $\varphi'=\varphi''^m$ is a weakly 1-periodic morphism.
\end{proof}

From now on, we replace $\varphi$ by $\varphi'$ from the proof 
and assume that $\varphi$ is a weakly 1-periodic morphism.

Actually, we want to improve $\varphi$ more. For each $k\in\NN$ and for each letter $b\in\E$ of order $>k$,
the leftmost and rightmost letters of order $>k$ in $\varphi(b)$ will be important for us.
If $k\in \NN$ and $\gamma$ is a finite word in $\E$ containing at least one 
letter of order $>k$, denote the leftmost (resp.\ rightmost) letter of order $>k$ 
in $\gamma$ by $\LL_k(\gamma)$ (resp.\ by $\RL_k(\gamma)$). Observe that if $b\in\E$, then
$\LL_k(\varphi^n(b))=\LL_k(\varphi(\LL_k(\varphi^{n-1}(b))))$ since if $c$ is a letter of order $k$ or less, 
then $\varphi(c)$ consists of letters of order $k$ or less only. Hence, 
$b,\LL_k(\varphi(b)),\LL_k(\varphi^2(b)),\ldots,\LL_k(\varphi^n(b)),\ldots$ is an 
eventually periodic sequence. Similarly, 
$b,\RL_k(\varphi(b)),\RL_k(\varphi^2(b)),\ldots,\RL_k(\varphi^n(b)),\ldots$ 
is also an eventually periodic sequence. We want to make these sequence as simple as possible, 
so we call a morphism $\varphi$ \textit{strongly 1-periodic} if
for each $k\in\NN$ and for each letter $b\in\E$ of order 
$>k$, one has $\LL_k(\varphi(b))=\LL_k(\varphi^2(b))=\ldots=\LL_k(\varphi^n(b))=\LL_k(\varphi^{n+1}(b))=\ldots$
and $\RL_k(\varphi(b))=\RL_k(\varphi^2(b))=\ldots=\RL_k(\varphi^n(b))=\RL_k(\varphi^{n+1}(b))=\ldots$, 
in other words, the sequences
$b,\LL_k(\varphi(b)),\LL_k(\varphi^2(b)),\ldots,\LL_k(\varphi^n(b)),\ldots$ and 
$b,\RL_k(\varphi(b)),\RL_k(\varphi^2(b)),\ldots,\RL_k(\varphi^n(b)),\ldots$ 
are both eventually periodic with periods of length \textbf{\textit{one}} and preperiods of 
length 1.

Observe that the definition of a weakly 1-periodic morphism guarantees that if $b$ is a letter 
of \textbf{\textit{finite}} order, then these sequences have periods of length 1, but we cannot 
say anything about the length of the preperiods. Also, we cannot say anything
about the length of the period if all letters 
$b,\LL_k(\varphi(b)),\LL_k(\varphi^2(b)),\ldots,\LL_k(\varphi^n(b)),\ldots$ have order $\infty$.

\begin{lemma}
There exists $n\in\NN$ such that $\varphi^n$ is a strongly 1-periodic morphism.
\end{lemma}
\begin{proof}
The proof is similar to the proof of the previous lemma. Namely, if $n\in\NN$ is large enough, then 
for $\varphi''=\varphi^n$ sequences $b,\LL_k(\varphi''(b)),\LL_k(\varphi''^2(b)),\ldots,\LL_k(\varphi''^l(b)),\ldots$ 
and $b,\RL_k(\varphi''(b)),{}$\linebreak$\RL_k(\varphi''^2(b)),\ldots,\RL_k(\varphi''^l(b)),\ldots$ 
are eventually periodic with preperiods of length 1
for all $k\in\NN$ and for all letters $a\in\E$ of order $>k$. 
Again, $n=|\E|$ is sufficient for this purpose.

Now, if we take a large enough $m\in\NN$ and set $\varphi'=\varphi''^m$, then 
the sequences
$b,\LL_k(\varphi'(b)),{}$\linebreak$\LL_k(\varphi'^2(b)),\ldots$ and 
$b,\RL_k(\varphi'(b)),\RL_k(\varphi'^2(b)),\ldots$ 
will become eventually periodic with periods of length 1
for all $k\in\NN$ and for all letters $a\in\E$ of order $>k$.
This time, $m=|\E|!$ is sufficient. Clearly, the preperiods of length 1 will 
remain the same.
\end{proof}

From now on, we replace $\varphi$ by $\varphi'$ from the proof 
and assume that $\varphi$ is strongly 1-periodic.

Our final improvement of the morphism $\varphi$ will guarantee that 
the image of each letter is "sufficiently long". 
%
Namely, first we are going to define the set of \textit{final periods}.
Let $b\in\E$ be a letter such that $\LL_1(\varphi(b))=b$. Then the prefix of $\varphi(b)$
to the left of the leftmost occurrence of $b$ in $\varphi(b)$ consists of letters of order 1 only, 
denote it by $\gamma$. That is, if $\varphi(b)_i=b$ and $\varphi(b)_j\ne b$ for $0\le j<i$, 
then $\gamma=\varphi(b)_{0\ldots i-1}$. Suppose that $\gamma$ is nonempty. 
Then $\varphi(\gamma)$ consists of periodic letters of order 1 only. 
Recall that to construct a morphic sequence, we use $\varphi$ and also a coding $\psi$.
Consider the word $\psi(\varphi(\gamma))\psi(\varphi(\gamma))$. Since we have repeated
$\psi(\varphi(\gamma))$ twice, we can apply Corollary \ref{finwordmincompleteperiod}
and conclude that there exists the minimal complete period of $\psi(\varphi(\gamma))\psi(\varphi(\gamma))$, 
denote it by $\lambda$. We call $\lambda$, as well as all its cyclic shifts, \textit{final periods}.
Similarly, we can define a final period using a letter $b$ such that 
$\RL_1(\varphi(b))=b$ and considering the suffix of $\varphi(b)$ to the right of the rightmost occurrence of $b$.
These are all words we call final periods, i. e.\ a \textit{final period} is 
a word obtained from a letter $b\in\E$ such that $\LL_1(\varphi(b))=b$ and $\varphi(b)_0\ne b$
by the procedure described above or a word obtained from a letter $b\in\E$ such that $\RL_1(\varphi(b))=b$ 
and $\varphi(b)$ does not end with $b$ by a similar procedure.

%
%

\begin{lemma}\label{finalnomorethanonce}
If $\lambda$ is a final period, then $\lambda$ cannot be written as a finite word repeated more than once.
\end{lemma}
\begin{proof}
Since $\lambda$ is a final period, there exists a finite word $\lambda'$ (which is also a final period) and a 
finite word $\gamma$ such that $\lambda'$ is the minimal complete period of 
$\psi(\varphi(\gamma))\psi(\varphi(\gamma))$
and $\lambda=\Cyc_r(\lambda')$ for some $r$ ($0\le r < |\lambda'|$). 

Suppose that $\lambda$ can be written as a finite word $\mu$ repeated more than once, in other words, 
that $\lambda$ is a completely $\mu$-periodic word and $|\mu|<|\lambda|$. But then 
$\lambda'=\Cyc_{-r}(\lambda)$ is a completely $\mu'$-periodic word, where $\mu'=\Cyc_{-r}(\mu)$.
Then the word $\psi(\varphi(\gamma))\psi(\varphi(\gamma))$, which is $\lambda'$ repeated several times, 
is also a completely $\mu'$-periodic word. But $|\mu'|=|\mu|<|\lambda|=|\lambda'|$, and this is a contradiction 
with the fact that $\lambda'$ is the minimal complete period of $\psi(\varphi(\gamma))\psi(\varphi(\gamma))$.
\end{proof}

Note that final periods always exist if $\varphi$ is a strongly 1-periodic morphism and 
there is a periodic letter of order 2 in $\E$ (we have already assumed that this is true).
Indeed, if $b$ is a periodic letter of order 2, then $\varphi(b)$ contains exactly 
one occurrence of order 2, which is $b$, and at least one letter of order 1. In other 
words, $\varphi(b)$ can be written as $\gamma b\gamma'$, where the words $\gamma$ and $\gamma'$ consist of 
letters of order 1 only, and at least one of these words is nonempty. We can use this nonempty word 
to construct a final period.

Clearly, the amount of final periods is finite and their lengths are bounded.
Denote the maximal length of a final period by $\finmax$.

\begin{lemma}\label{finalperiodsstaythesame}
Let $k\in\NN$. Then the sets of final periods for $\varphi$ and for $\varphi'=\varphi^k$ are the same.

If $\gamma b$ (resp.\ $b \gamma$) is a prefix (resp.\ suffix) of $\varphi(b)$, where $b\in\E$ and 
$\gamma$ is a finite word consisting 
of letters of order 1 only, then 
$\varphi(\gamma)\ldots\varphi(\gamma)\gamma b$ (resp.\ $b\gamma\varphi(\gamma)\ldots\varphi(\gamma)$), 
where $\varphi(\gamma)$ is repeated $k-1$ times, is a prefix (resp.\ suffix) of $\varphi^k(b)$.
\end{lemma}
\begin{proof}
Choose a letter $b\in\E$ such that $\LL_1(\varphi(b))=b$. (The case $\RL_1(\varphi(b))=b$ is completely symmetric.)
Suppose that $\varphi(b)_0\ne b$ and denote the prefix of $\varphi(b)$ to the left of the leftmost 
occurrence of $b$ by $\gamma$. Then $\gamma b$ is a prefix of $\varphi(b)$ and $\gamma$ consists of 
letters of order 1 only. Let us prove by induction on $k\in\NN$ that $\varphi(\gamma)\ldots\varphi(\gamma)\gamma b$, 
where $\varphi(\gamma)$ is repeated $k-1$ times, is a prefix of $\varphi^k(b)$. For $k=1$ we already
know this. If $\varphi(\gamma)\ldots\varphi(\gamma)\gamma b$, 
where $\varphi(\gamma)$ is repeated $k-1$ times, is a prefix of $\varphi^k(b)$, then 
$\varphi^2(\gamma)\ldots\varphi^2(\gamma)\varphi(\gamma)\varphi(b)$, where $\varphi^2(\gamma)$ is repeated $k-1$ times, 
is a prefix of $\varphi^{k+1}(b)$. Recall that $\gamma b$ is a prefix of $\varphi(b)$, 
so $\varphi^2(\gamma)\ldots\varphi^2(\gamma)\varphi(\gamma)\gamma b$, where $\varphi^2(\gamma)$ is repeated $k-1$ times, 
is also a prefix of $\varphi^{k+1}(b)$. Finally, recall that $\varphi$ is (in particular) weakly 1-periodic, 
so $\varphi(\gamma)$ consists of periodic letters of order 1 only, and 
$\varphi^2(\gamma)=\varphi(\gamma)$. Therefore, 
$\varphi(\gamma)\ldots\varphi(\gamma)\gamma b$, where $\varphi(\gamma)$ is repeated $k$ times, 
is a prefix of $\varphi^{k+1}(b)$.

So, the prefix of $\varphi'(b)=\varphi^k(b)$ to the left of the leftmost occurrence 
of $b$ is $\gamma'=\varphi(\gamma)\ldots\varphi(\gamma)\gamma$, where $\varphi(\gamma)$ is repeated $k-1$ times.
If we apply $\varphi'$ to this prefix, we will get $\varphi(\gamma)$ repeated $k$ times 
(here we again use the fact that $\varphi^2(\gamma)=\varphi(\gamma)$).
Finally, $\psi(\varphi'(\gamma'))\psi(\varphi'(\gamma'))$ is $\psi(\varphi(\gamma))$ repeated $2k$ times, 
and, by Corollary \ref{finwordmincompleteperiod}, $\psi(\varphi'(\gamma'))\psi(\varphi'(\gamma'))$
has a minimal complete period, and it coincides with the minimal complete period of 
$\psi(\varphi(\gamma))\psi(\varphi(\gamma))$.
\end{proof}

After we have this lemma, it is reasonable to give the following definition: A strongly 1-periodic morphism 
is called \textit{a strongly 1-periodic morphism with long images} if the following holds:
\begin{enumerate}
\item For each letter $b\in\E$ such that $\LL_1(\varphi(b))=b$ and the prefix $\gamma$ of $\varphi(b)$ to the 
left of the leftmost occurrence of $b$ is nonempty, we have $|\gamma|\ge 2\finmax$.
\item For each letter $b\in\E$ such that $\RL_1(\varphi(b))=b$ and the suffix $\gamma$ of $\varphi(b)$ to the 
right of the rightmost occurrence of $b$ is nonempty, we have $|\gamma|\ge 2\finmax$.
\end{enumerate}
\begin{lemma}\label{finalperiodsarecomplete}
Let $\varphi$ be a strongly 1-periodic morphism with long images.
Then for each letter $b\in\E$ such that $\LL_1(\varphi(b))=b$ (resp.\ $\RL_1(\varphi(b))=b$) and the prefix (resp.\ suffix) 
$\gamma$ of $\varphi(b)$ to the 
left (resp.\ to the right) of the leftmost (resp.\ rightmost) occurrence of $b$ is nonempty, there exists a minimal complete period of 
$\psi(\varphi(\gamma))$, and it is a final period.
\end{lemma}
\begin{proof}
The claim for $\psi(\varphi(\gamma))\psi(\varphi(\gamma))$ instead of $\psi(\varphi(\gamma))$ is just the definition of a final 
period. Let $\lambda$ be the minimal complete period of $\psi(\varphi(\gamma))\psi(\varphi(\gamma))$. By Corollary 
\ref{finwordmincompleteperiod}, $\psi(\varphi(\gamma))$ is $\lambda$ repeated $|\psi(\varphi(\gamma))|/|\lambda|$
times, i.~e.\ $\psi(\varphi(\gamma))$ is 
completely $\lambda$-periodic.
Since $|\gamma|\ge 2\finmax$, we also have $|\psi(\varphi(\gamma))|\ge 2\finmax\ge 2|\lambda|$, and 
by Corollary \ref{finwordmincompleteperiod} again, there exists a minimal complete period $\lambda'$ of $\psi(\varphi(\gamma))$
and $\lambda$ is $\lambda'$ repeated several times. But then $\psi(\varphi(\gamma))\psi(\varphi(\gamma))$ 
is also completely 
$\lambda'$-periodic,
but $\lambda$ was the 
minimal complete period of $\psi(\varphi(\gamma))\psi(\varphi(\gamma))$, so $\lambda'=\lambda$.
\end{proof}

Again, let us prove that we can make a strongly 1-periodic morphism with long images
out of $\varphi$ by replacing it with $\varphi^n$.

\begin{lemma}
There exists $n\in\NN$ such that $\varphi^n$ is a strongly 1-periodic morphism with long images.
\end{lemma}
\begin{proof}
Observe first that if $\varphi$ is strongly 1-periodic, then $\varphi^n$ is also strongly 1-periodic.

Choose a letter $b\in\E$ such that $\LL_1(\varphi(b))=b$. (The case $\RL_1(\varphi(b))=b$ is completely symmetric.)
Suppose that $\varphi(b)_0\ne b$. Then, by the second statement of Lemma \ref{finalperiodsstaythesame}, 
if $n$ is large enough, then the length of the prefix of $\varphi^n(b)$ to the left of 
the leftmost occurrence of $b$ is at least $2\finmax$ (here we also use the fact that $\varphi$ is 
nonerasing, so in the statement of Lemma \ref{finalperiodsstaythesame} we have $|\varphi(\gamma)|\ge|\gamma|$).

Let $n_0$ be the maximum of all these numbers $n$ for all letters $a\in\E$ and for the left and the right side. 
($n_0=2\finmax$ is sufficient for this purpose, but a smaller $n_0$ can also work.) Then, by Lemma \ref{finalperiodsstaythesame}, 
$\varphi'=\varphi^n$ is a strongly 1-periodic morphism with long images.
\end{proof}

From now on, we replace $\varphi$ by $\varphi'$ from the proof and assume that $\varphi$ is a 
strongly 1-periodic morphism with long images.

\section{Blocks}\label{sectionblocks}

A (possibly empty) finite occurrence $\al_{i\ldots j}$ is \emph{a $k$-block} if it consists of
letters of order $\le k$, $i>0$, and the letters $\al_{i-1}$ and
$\al_{j+1}$ both have order $>k$. The occurrence of a single letter $\al_{i-1}$ is called
the \emph{left border} of this block and is denoted by $\LB(\swa ij)$.
The occurrence of a single letter $\al_{j+1}$ is called the \emph{right border} of this block
and is denoted by $\RB(\swa ij)$. Observe that if we have constructed 
the whole morphic sequence starting with a letter 
$a\in\E$ (i.~e.\ $\al=\varphi^\infty(a)$),
and $a$ is a letter of a finite order $k$, then all letters in 
$\al$ are of order $\le k$. So it makes no sense to define "$k$-blocks" 
of the form $\al_{0\ldots j}$ since it is not possible that all letters 
$\al_0=a,\al_1,\ldots,\al_j$ have orders $\le k$, and $\al_{j+1}$ has order $>k$.

Note that even if there are no letters of order $k$ in $\E$, $k$-blocks 
still may exist, then all letters in $k$-blocks will be of order $<k$ (or a $k$-block
can also be empty), and the borders of such a $k$-block will be letters of order $>k$ 
(in fact, as one can deduce from the assignment of orders to letters in the previous section, these 
letters must have order $\infty$). A problem that can arise is that letters 
of order $<k$ (or $\le k$) may form an infinite sequence, then they do not form a 
$k$-block by definition. Later we will see that this can really happen if all letters in $\al$ 
have finite orders, we will discuss this in Lemma \ref{splitconcatenation}.

The image under $\varphi$ of a letter of order $\le k$ cannot
contain letters of order~$>k$. Let $\al_{i\ldots j}$ be a $k$-block.
Then $\varphi(\al_{i\ldots j})$
is a suboccurrence of some $k$-block which is called
the \emph{descendant} of $\al_{i\ldots j}$ and is denoted by
$\Su_k(\al_{i\ldots j})$.
(We use the subscript $k$ here to underline that the same occurrence $\al_{i\ldots j}$ 
can be a $k$-block and an $m$-block for some $m\ne k$ at the same time, for example, 
if $\LB(\swa ij)$ and $\RB(\swa ij)$ are both of order $>k+1$, then 
$\swa ij$ is a $(k+1)$-block as well. In this case, 
$\Su_k(\al_{i\ldots j})$ and $\Su_{k+1}(\al_{i\ldots j})$ could be different occurrences in $\al$.)
The $l$-th superdescendant (denoted by $\Su^l_k(\al_{i\ldots j})$) is
the descendant of \textellipsis of the descendant of $\al_{i\ldots j}$ ($l$
times).

Let $\al_{s\ldots t}$ be a $k$-block in $\al$. Then 
if there exists a $k$-block $\al_{i\ldots j}$ such that $\Su_k(\al_{i\ldots j}) =
\al_{s\ldots t}$, it is unique. Indeed, otherwise there would be a letter of order 
$>k$ between those two $k$-blocks, and its image would contain a letter of order $>k$ again. 
But this letter would belong to $\swa st$.
If the $k$-block $\swa ij$ exists, is called the \emph{ancestor} of $\al_{s\ldots t}$
and is denoted $\Su^{-1}_k(\al_{s\ldots t})$. The $l$-th
superancestor (denoted by $\Su^{-l}_k(\al_{s\ldots t})$) is
the ancestor of \textellipsis of the ancestor of
$\al_{s\ldots t}$ ($l$ times). If $\Su^{-1}_k(\al_{s\ldots t})$ does
not exist (this can happen only if $\al_{s-1}$ and $\al_{t+1}$ belong to the
image of the same letter), then $\al_{s\ldots t}$ is called an \emph{origin}.
A sequence $\mathcal E$ of $k$-blocks, $\mathcal E_0 = \al_{i\ldots
j}, \mathcal E_1 = \Su_k(\al_{i\ldots j}),$ $\mathcal
E_2=\Su^2(\al_{i\ldots j}), \ldots, \mathcal E_l =
\Su^l_k(\al_{i\ldots j}), \ldots$, where $\al_{i\ldots j}$ is an
origin, is called an \emph{evolution}. The number $l$ is called 
the \textit{evolutional sequence number} of a $k$-block $\mathcal E_l$.

Let $\mathcal E$ be an evolution of $k$-blocks. The letter $\LB(\mathcal E_{l+1})$ is
the rightmost letter of order $>k$ in $\varphi(\LB(\mathcal E_l))$, 
i.~e.\ $\LB(\mathcal E_{l+1})=\RL_k(\varphi(\LB(\mathcal E_l)))$.
Since $\varphi$ is a strongly 1-periodic morphism, this means that 
$\LB(\mathcal E_l)$ does not depend on $l$ if $l\ge 1$.
Similarly, $\RB(\mathcal E_l)$ does not depend on $l$ if $l\ge 1$.
We call the abstract letter $\LB(\mathcal E_l)$ (resp.\ $\RB(\mathcal E_l)$)
for any $l\ge 1$ the \textit{left (resp right.) border of $\mathcal E$}
and denote it by $\LB(\mathcal E)$ (resp.\ by $\RB(\mathcal E)$).

\begin{lemma}
The set of all abstract words that can be origins in $\al$, is finite.
\end{lemma}

\begin{proof}
Each origin is a subword of $\varphi(b)$ where $b$ is a single letter. Moreover, this occurrence inside 
$\varphi(b)$ cannot be a prefix or a suffix.
\end{proof}

\begin{corollary}\label{finite-number-of-evolutions}
The set of all possible evolutions in $\al$ (considered as sequences
of abstract words rather than sequences of occurrences in $\al$), is
finite.
\end{corollary}

\begin{proof}
Let $\mathcal E_0$ be an origin, which is a suboccurrence of $\varphi(\al_i)$. Here 
$\al_i$ is a letter of order $>k$. Then $\varphi(\al_i)$ also contains $\LB(\mathcal E_0)$ 
and $\RB(\mathcal E_0)$. $\LB(\mathcal E_{l+1})$, $\RB(\mathcal E_{l+1})$ and $\mathcal E_{l+1}$ 
itself depend on \textit{abstract words} $\LB(\mathcal E_l)$, $\RB(\mathcal E_l)$ and 
$\mathcal E_l$ only. Thus, all these words became known after we had selected an abstract 
letter $b=\al_i$ and a suboccurrence inside $\varphi(b)$.
\end{proof}

\begin{lemma}\label{splitconcatenation}
Let $\al=\varphi^\infty(a)$, where $a\in\E$.
Then:
\begin{enumerate}
\item If $a$ is a letter of a finite order $K$, then $a$ is the only letter of 
order $\ge K$ in $\al$, and it only occurs once, as $\al_0$. For each $k<K-1$, $k\in\NN$, 
$\al$ splits into a concatenation of $k$-blocks and letters of order $>k$.
\item If $a$ is a letter of order $\infty$, then for each $k\in\NN$, $\al$
splits into a concatenation of $k$-blocks and letters of order $>k$.
\end{enumerate}
\end{lemma}
\begin{proof}
First assume that $a$ is a letter of finite order $K$. Then $a$ is a periodic letter of order $K$
since $\varphi(a)$ begins with $a$. Then each word $\varphi^l(a)$ ($l\in\NN$) contains only 
one letter of order $\ge K$ by a property of periodic letters. To prove the claim in this case, 
it suffices to prove that $\al$ contains infinitely many letters of order $K-1$. Let $\gamma$ be 
the finite word such that $\varphi(a)=a\gamma$. Then $\al=a\gamma\varphi(\gamma)\varphi^2(\gamma)\ldots\varphi^l(\gamma)\ldots$.
Since $a$ is a letter of order $k$, $\varphi(a)$ contains at least one letter of order $k-1$. But then 
$\varphi^l(\gamma)$ contains at least one letter of order $k-1$ for each $l$.

Now let us consider the case when $a$ is a letter of order $\infty$. Then it is sufficient to prove that 
$\al$ contains infinitely many letters of order $\infty$. By Lemma \ref{rates}, 
$\varphi^l(a)$ contains at least two letters of order $\infty$ if $l$ is large enough.
Again write $\varphi(a)=a\gamma$, then $\varphi^l(a)=a\gamma\varphi(\gamma)\varphi^2(\gamma)\ldots\varphi^{l-1}(\gamma)$
and $\al=a\gamma\varphi(\gamma)\varphi^2(\gamma)\ldots\varphi^l(\gamma)\ldots$.
If $\varphi^{l_0}(a)$ contains at least two letters of order $\infty$, then at least one of the words 
$\gamma, \varphi(\gamma), \ldots, \varphi^{l_0-1}(\gamma)$ contains a letter of order $\infty$.
But then, by Lemma \ref{rates} again, all words $\varphi^l(\gamma)$ for $l\ge l_0$ also contain a letter of 
order $\infty$, and $\al$ contains infinitely many letters of order $\infty$.
\end{proof}

Now, when we know that $\al$ can be split into a concatenation of alternating 
letters of order $\ge k$ and $k$-blocks
(at least for some $k\in\NN$), it is convenient to consider concatenations of finitely many 
$k$-blocks and letters of order $>k$ between them. However, it is not very convenient 
to consider them as just occurrences in $\al$, because $k$-blocks can be empty occurrences, 
and we want to distinguish clearly whether we include a $k$-block of the form $\swa {i+1}i$ 
(as it was pointed out above, this notation denotes the occurrence of the empty word between $\al_i$ and
$\al_{i+1}$) into a concatenation of the form $\swa {i+1}j$ or no. Also, we will need to consider possibly 
empty concatenations of $k$-blocks, and their exact locations will be important for us, 
in particular, if $\swa {i+1}i$ is an empty $k$-block, we want to distinguish
"the empty concatenation located directly to the left of $\swa {i+1}i$" from 
"the empty concatenation located directly to the right of $\swa {i+1}i$".
So we start with the following 
definition:

A pair of occurrences $(\swa ij,\swa {j+1}s)$ is called a \textit{$k$-delimiter} ($k\in\NN$) in $\al$
in one of the two cases:
\begin{enumerate}
\item if exactly one of these two occurrences is a (possibly empty) $k$-block, and the other 
one is a single letter of order $>k$, or
\item if $\swa ij=\swa 0{-1}$, the occurrence of the empty word before the actual beginning of 
$\al$, and $\swa {j+1}s=\swa 00$, a letter of order $>k$ (it follows from Lemma
\ref{splitconcatenation} that $\al_0$ cannot be contained in a $k$-block since $k$-blocks are finite 
by definition).
\end{enumerate}
Here $\swa ij$ is called \textit{the left part}
of the $k$-delimiter and $\swa {j+1}s$ is called \textit{the right part} of the $k$-delimiter.
Split $\al$ into a concatenation of $k$-blocks and letters of order $>k$. 
Write all these occurrences in $\al$ in an infinite sequence, mentioning each empty 
$k$-block explicitly. For example, if $\E=\{a,b,c\}$, 
$\varphi(a)=abb$, $\varphi(b)=bcc$, $\varphi(c)=c$, then 
the orders of letters $a,b,c$ are $3,2,1$, respectively, 
$\al=abbbccbccbccccbcccc\ldots$, and this sequence 
of occurrences is: 
$\swa 00,\swa 10,\swa 11,\swa 21,\swa 22,\swa 32, \swa 33, \swa 45, \swa 66, \swa 78,\ldots$.
As abstract words, the nonempty words in this sequence are: 
$\swa 00=a, \swa 11=b,\swa 22=b,\swa 33=b,\swa 45=cc, \swa 66=b, \swa 78=cc,\ldots$. The occurrences $\swa 10$, 
$\swa 21$, and $\swa 32$ here are empty 1-blocks.
Informally speaking, a $k$-delimiter is the "empty space" between two members of this sequence (the left and the right 
parts of the $k$-delimiter) or the "empty space" to the left of the whole sequence.
We say that a $k$-block or a single letter of order $>k$ $\swa ij$ is located \textit{strictly to the left}
from a $k$-block or a single letter of order $>k$ $\swa st$ if 
$\swa ij$ is written in this sequence before $\swa st$. In terms of indices 
this means that either $i<s$ ("the position where $\swa ij$ starts is before the 
position where $\swa st$ starts") or $i=s$ and $j<t$ ("the positions where $\swa ij$ and $\swa st$ start
coincide, but $\swa ij$ ends before $\swa st$ ends"), this is possible only if $\swa ij$ is an occurrence of 
the empty word ($j=i-1$) since $k$-blocks and letters of order $>k$ do not overlap.
We also say that a $k$-block or a single letter of order $>k$ $\swa ij$ is located \textit{strictly to the right}
from a $k$-block or a single letter of order $>k$ $\swa st$ if $\swa st$ is located strictly 
to the left from $\swa ij$. Clearly, if $\swa ij$ is a $k$-block or a letter of order $>k$ and $\swa st$
is also a $k$-block or a letter of order $>k$, then either $\swa ij$ is located strictly to the 
left from $\swa st$, or $\swa ij=\swa st$, or $\swa ij$ is located strictly to the 
right from $\swa st$.

Now let $(\swa ij,\swa {j+1}s)$ be a $k$-delimiter, and let $\swa {i'}{j'}$ be a $k$-block or a single 
letter of order $>k$. Then we want to define when 
$\swa {i'}{j'}$ is \textit{located at the right-hand side}
of $(\swa ij,\swa {j+1}s)$. If $(\swa ij,\swa {j+1}s)=(\swa 0{-1},\swa 00)$, then 
we always say that $\swa {i'}{j'}$ is located at the right-hand side
of $(\swa 0{-1},\swa 00)$. Otherwise we say that 
$\swa {i'}{j'}$ is located at the right-hand side
of $(\swa ij,\swa {j+1}s)$ if either 
$\swa {i'}{j'}=\swa {j+1}s$ as occurrences in $\al$,
or $\swa {i'}{j'}$ is located strictly to the right
from $\swa {j+1}s$.
In terms of indices this means that either $j+1=i'$ and $s=j'$, or $j+1<i'$, 
or $j+1=i'$ and $s<j'$. This can be rewritten shorter as follows: either 
$j+1=i'$ and $s\le j'$, or $s<j'$.
Similarly, if 
$(\swa ij,\swa {j+1}s)=(\swa 0{-1},\swa 00)$, then 
we never say that $\swa {i'}{j'}$ is \textit{located at the left-hand side of} 
$(\swa 0{-1},\swa 00)$. If $(\swa ij,\swa {j+1}s)\ne (\swa 0{-1},\swa 00)$, 
then we say that 
$\swa {i'}{j'}$ is located at the left-hand side
of $(\swa ij,\swa {j+1}s)$ if either 
$\swa {i'}{j'}=\swa ij$ as occurrences in $\al$, 
or $\swa {i'}{j'}$ is located strictly to the left
from $\swa ij$.
In terms of indices this means that either $i'=i$ and $j'=j$, or $i'<i$, 
or $i'=i$ and $j'<j$. This can be rewritten shorter as follows: either 
$i'=i$ and $j'\le j$, or $i'<i$.
Again, if $(\swa ij,\swa {j+1}s)$ is a $k$-delimiter, and $\swa {i'}{j'}$ is a $k$-block or a single 
letter of order $>k$, then either 
$\swa {i'}{j'}$ is located at the left-hand side of $(\swa ij,\swa {j+1}s)$,
or $\swa {i'}{j'}$ is located at the right-hand side of $(\swa ij,\swa {j+1}s)$.

If $(\swa ij,\swa{j+1}s)$ and $(\swa{i'}{j'},\swa{j'+1}{s'})$ are $k$-delimiters, 
we say that $(\swa{i'}{j'},\swa{j'+1}{s'})$ is 
\textit{located at the right-hand side of} 
$(\swa ij,\swa{j+1}s)$ if $\swa{i'}{j'}$ is located at the 
right-hand side of $(\swa ij,\swa{j+1}s)$. 
And $(\swa ij,\swa{j+1}s)$ is said to be 
\textit{located at the left-hand side of}
$(\swa{i'}{j'},\swa{j'+1}{s'})$ if 
$(\swa{i'}{j'},\swa{j'+1}{s'})$ is 
located at the right-hand side of
$(\swa ij,\swa{j+1}s)$. And again, if we have two $k$-delimiters, then either they coincide, 
or one of them is located at the left-hand side of the other one,
or one of them is located at the right-hand side of the other one.
Finally, we say that 
a $k$-block or a letter of order $>k$ is \textit{located between} one $k$-delimiter 
and another $k$-delimiter if this $k$-block or this letter of order $>k$ 
is located at the right-hand side of the first $k$-delimiter and 
at the left-hand side of the second $k$-delimiter.

Now we are ready to define $k$-multiblocks. We say that a \textit{$k$-multiblock}
is defined by the following data:
\begin{enumerate}
\item Two $k$-delimiters $(\swa ij, \swa{j+1}s)$ and $(\swa{i'}{j'}, \swa{j'+1}{s'})$, where 
$(\swa ij, \swa{j+1}s)$ either coincides with $(\swa{i'}{j'}, \swa{j'+1}{s'})$, 
or is located at the left-hand side of $(\swa{i'}{j'}, \swa{j'+1}{s'})$. Here 
$(\swa ij, \swa{j+1}s)$ (resp.\ $(\swa{i'}{j'}, \swa{j'+1}{s'})$) is 
called the \textit{left} (resp.\ \textit{right}) $k$-delimiter of the $k$-multiblock, 
\item The set of all $k$-blocks and letters of order $>k$ located between 
$(\swa ij, \swa{j+1}s)$ and $(\swa{i'}{j'}, \swa{j'+1}{s'})$.
\end{enumerate}

Two $k$-multiblocks are called \textit{consecutive} if the right $k$-delimiter of first 
$k$-multiblock coincides with the left $k$-delimiter of the second $k$-multiblock.
The $k$-multiblock whose left (resp.\ right) $k$-delimiter is the left (resp.\ right) 
$k$-delimiter of the first (resp.\ second) $k$-multiblock is called their \textit{concatenation}.
A $k$-multiblock is called \textit{empty} if the left and the right $k$-delimiters coincide, in other 
words, if the set of $k$-blocks and letters of order $k$ is empty. 
A $k$-multiblock consisting of a single empty $k$-block is not called an empty 
$k$-multiblock.

We need to introduce some convenient notation for $k$-multiblocks. First, 
a $k$-multiblock is determined by two $k$-delimiters
$(\swa ij, \swa{j+1}s)$ and $(\swa{i'}{j'}, \swa{j'+1}{s'})$, so we can denote it 
by $[(\swa ij, \swa{j+1}s), (\swa{i'}{j'}, \swa{j'+1}{s'})]$. Second, each $k$-delimiter
is determined by its left or right part, so we can denote the same 
$k$-multiblock by $[\swa{j+1}s, \swa{i'}{j'}]$ (and this notation agrees with 
the fact that if $(\swa ij, \swa{j+1}s)$ and $(\swa{i'}{j'}, \swa{j'+1}{s'})$ are two 
different $k$-delimiters, then the set of $k$-blocks and letters of order $>k$ 
in this $k$-multiblock 
is the subsequence of the sequence of all $k$-blocks and letters of order $>k$ 
in $\al$ that starts at $\swa{j+1}s$ and ends at $\swa{i'}{j'}$, inclusively).
Moreover, if $\swa{j+1}j$ is not a $k$-block, then the occurrence 
of the form $\swa{j+1}s$, which is a $k$-block or a letter of order $>k$, is 
determined uniquely by the index $j+1$. However, if $\swa{j+1}j$ is a $k$-block, 
then there are two occurrences of the form $\swa{j+1}s$ that we can use as a right part of
a $k$-delimiter: the empty $k$-block $\swa{j+1}j$ and also $\swa{j+1}{j+1}$, 
which must be a letter of order $>k$ in this case. In this case we denote the $k$-delimiter
whose right part is $\swa{j+1}j$ (i.~e.\ the \textbf{\textit{leftmost}} $k$-delimiter 
whose right part is of the form $\swa{j+1}s$) by $<,j+1$, and 
the $k$-delimiter
whose right part is $\swa{j+1}{j+1}$ 
(i.~e.\ the \textbf{\textit{rightmost}} $k$-delimiter 
whose right part is of the form $\swa{j+1}s$) by $>,j+1$. If $\swa{j+1}j$ is not 
a $k$-block, we say that $<,j+1$ and $>,j+1$ denote the same $k$-delimiter, 
namely, the unique $k$-delimiter whose right part is of the form $\swa {j+1}s$.
Similarly, if $\swa{j'+1}{j'}$ is a $k$-block, then there are two 
occurrences of the form $\swa{i'}{j'}$ that are $k$-blocks or letters of order $>k$:
the empty $k$-block $\swa{j'+1}{j'}$ and a letter $\swa{j'}{j'}$ of order $>k$.
And we denote the $k$-delimiter whose left part is $\swa{j'+1}{j'}$
(i.~e.\ the rightmost $k$-delimiter whose left part is of the form $\swa{i'}{j'}$) 
by $j',>$, and the $k$-delimiter whose left part is $\swa{j'}{j'}$
(i.~e.\ the leftmost $k$-delimiter whose left part is of the form $\swa{i'}{j'}$) 
by $j',<$. 
If $\swa{j'+1}{j'}$ is not a $k$-block, then there exists at most one occurrence of the 
form $\swa{i'}{j'}$ that can be the left part of a $k$-delimiter, and if it exists, 
we denote the $k$-delimiter with this left part by both $j',<$ and $j',>$.
Now we denote the same $k$-multiblock as before 
by $\al[\mathfrak x\ldots\mathfrak y]_k$, where $\mathfrak x$ (resp.\ $\mathfrak y$)
is a notation for a $k$-delimiter of the form $<,j+1$ or $>,j+1$ 
(resp.\ $j',<$ or $j',>$).

For example, if $\swa ij$ is a non-empty $k$-block, then the $k$-multiblock whose set of $k$-blocks 
and letters of order $>k$ between the $k$-delimiters consists of $\swa ij$ only, 
is denoted by $\al[<,i\ldots j,>]_k$ or by ${\al[>,i\ldots j,<]_k}$ (and two more possibilities).
If $\al_i$ is a letter of order $>k$, then the $k$-multiblock that consists of this letter itself
if denoted by $\al[>,i\ldots i,<]_k$ (and here the signs $>$ and $<$ are important if 
$\al_{i-1}$ or $\al_{i+1}$ is also a letter of order $>k$). Let us consider the 
example of an empty $k$-block $\swa{i+1}i$. In this case, 
$\al[<,i+1\ldots i,>]_k$ is the $k$-multiblock that consists of the empty $k$-block 
$\swa{i+1}i$ only (the $k$-block is located between the two $k$-delimiters), 
$\al[<,i+1\ldots i,<]_k$ is the empty $k$-multiblock "located directly at the left" 
of the $k$-block $\swa{i+1}i$ (the two $k$-delimiters coincide and are located at the left-hand side of 
$\swa{i+1}i$), $\al[>,i+1\ldots i,>]_k$ is the empty $k$-multiblock "located directly at the right" 
of the empty $k$-block, and $\al[>,i+1\ldots i,<]_k$ denotes nothing 
since the two $k$-delimiters do not coincide and are not in the correct order.
The $k$-multiblocks $\al[<,i+1\ldots i,<]_k$
and $\al[<,i+1\ldots i,>]_k$ are consecutive (and their concatenation is $\al[<,i+1\ldots i,>]_k$
again), and 
$\al[<,i+1\ldots i,<]_k$ and $\al[>,i+1\ldots i,>]_k$ are not.

More generally, if we know that there exists an (empty or nonempty) $k$-block 
of the form $\swa ij$ and $\swa st$ is a $k$-block or letter of order $>k$ that coincides with $\swa ij$ or located strictly to 
the right from $\swa ij$, then 
$\al[<,i\ldots t,?]_k$, where $?$ is one of the signs $<$ and $>$, always denotes 
a $k$-multiblock that includes $\swa ij$. And if we know that $\al_i$ is a letter of order $>k$, 
and and $\swa st$ is a $k$-block or letter of order $>k$ that coincides with $\al_i$ or located strictly to 
the right from $\al_i$, then 
then $\al[>,i\ldots t,?]_k$, where $?$ is one of the signs $<$ and $>$
always denotes a $k$-multiblock that begins with $\al_i$ as a set of consecutive $k$-blocks and letters of order $>k$.

For each $k$-multiblock one can consider the concatenation of all $k$-blocks and letters of order $>k$
between the two $k$-delimiters, this is an occurrence in $\al$. As we noted before, if the right part of the first 
$k$-delimiter is of the form $\swa {j+1}s$, and the left part of the second $k$-delimiter 
is of the form $\swa{i'}{j'}$, then this concatenation is $\swa {j+1}{j'}$. 
Therefore, if $\al[?,i\ldots j,?]_k$ is a $k$-multiblock, where each question mark denotes 
one of the signs $<$ or $>$, then this occurrence in $\al$ is $\swa ij$.
We call it the \textit{forgetful occurrence} of the $k$-multiblock and denote it by 
$\Fg(\al[?,i\ldots j,?]_k)$.

We did not define (and we are not going to define) any 0-blocks and 0-delimiters, however, it is convenient 
to have uniform notation and terminology for 0-multiblocks. We say that a \textit{$0$-multiblock}
is just a (possibly empty) finite occurrence in $\al$. We denote an occurrence $\swa ij$
by $\al[?,i\ldots j,?]_0$, where each question mark is one of the signs $<$ or $>$ (these signs do not play any role here).
The notions of consecutiveness and concatenation here are the usual notions of consecutiveness and concatenation for 
occurrences in $\al$. A 0-multiblock is called empty if it is an occurrence of the empty word.

Now we are ready to define descendants of $k$-multiblocks. 
First, let $\al_i$ be a letter of order $>k$ ($k\in\NN\cup 0$). 
Then the occurrence $\varphi(\al_i)$ contains at least one letter of order $>k$. 
Let $\al_j$ (resp.\ $\al_{j'}$) be the leftmost (resp.\ the rightmost) occurrence of a letter of order $>k$ in 
$\varphi(\al_i)$.
Then $\al[>,j\ldots j',<]_k$ is a $k$-multiblock that begins with $\al_j$ and ends with 
$\al_{j'}$ (and does not contain $k$-blocks of the form $\swa j{j-1}$ or $\swa {j'+1}{j'}$ even if these empty occurrences 
are $k$-blocks). We call $\al[>,j\ldots j',<]_k$ the descendant of the $k$-multiblock $\al[>,i\ldots i,<]_k$
(which consists of a single letter $\al_i$) and denote $\al[>,j\ldots j',<]_k$
by 
$\Su_k(\al[>,i\ldots i,<]_k)$. 

\begin{remark}\label{descendantofperiodic}
If $\al_i$ is a periodic letter of order $k+1$, then $\Su_k(\al[>,i\ldots i,<]_k)$
consists of a single letter of order $k+1$, namely, the unique
letter of order $k+1$ in $\varphi(\al_i)$. Moreover, since $\varphi$ is
(in particular) weakly 1-periodic, this letter coincides with $\al_i$ 
as an abstract letter.
\end{remark}



Second, as we have already noted, if $\swa ij$ is a $k$-block ($k\in \NN$), then 
$\al[<,i\ldots j,>]_k$ is always the $k$-multiblock that consists of $\swa ij$ only, 
independently of whether $\swa ij$ is empty or no. If $\Su_k(\swa ij)=\swa st$, 
then we say that $\Su_k(\al[<,i\ldots j,>]_k)=\al[<,s\ldots t,>]_k$
(the $k$-block that consists of $\swa st$ only).

Third, let us define the descendants of empty $k$-multiblocks ($k>0$). An empty 
$k$-multiblock is determined by a delimiter $(\swa ij, \swa st)$ repeated twice, both as the 
left and as the right delimiter of the $k$-multiblock. If $\swa ij=\swa0{-1}$, we say that the 
descendant of this $k$-multiblock is this $k$-multiblock itself. Otherwise either 
$\swa ij$ or $\swa st$ is a $k$-block. If $\swa ij$ is a $k$-block, then 
$\swa st$ is the right border of $\swa ij$, and we say that the descendant of
$[(\swa ij, \swa st),(\swa ij, \swa st)]$ is 
$[(\Su_k(\swa ij), \RB(\Su_k(\swa ij))),(\Su_k(\swa ij), \RB(\Su_k(\swa ij)))]$.
Similarly, if 
$\swa st$ is a $k$-block, then $\swa ij$ is its left border, and 
we say that the descendant of
$[(\swa ij, \swa st),(\swa ij, \swa st)]$ is 
$[(\LB(\Su_k(\swa st)),{}$\linebreak$\Su_k(\swa st)),(\LB(\Su_k(\swa st)), \Su_k(\swa st))]$.

Finally, let $k\in\NN\cup 0$,
and let $\al[\mathfrak x, i\ldots j,\mathfrak y]_k$, where $\mathfrak x,\mathfrak y\in\{<,>\}$,
be a non-empty $k$-multiblock. It consists of consecutive letters of order $>k$ and (if $k>0$) $k$-blocks,
and their descendants according to the definitions above are also consecutive 
$k$-multiblocks.
We 
call
the concatenation of these $k$-multiblocks
the \textit{descendant} of 
$\al[\mathfrak x, i\ldots j,\mathfrak y]_k$.
Denote it by $\Su_k(\al[\mathfrak x, i\ldots j,\mathfrak y]_k)$.
One checks easily using the particular cases of the definition of 
the descendant of a $k$-multiblock above that 
$\Su_k(\al[\mathfrak x, i\ldots j,\mathfrak y]_k)$
can be written as $\al[\mathfrak x,s\ldots t,\mathfrak y]_k$, 
where the indices $s$ and $t$ may differ from $i$ and $j$, but the signs 
$\mathfrak x$ and $\mathfrak y$ stay the same.
If $l\in\NN$, we also write 
$\Su^l_k(\al[\mathfrak x, i\ldots j,\mathfrak y]_k)=\Su_k(\ldots\Su_k(\al[\mathfrak x, i\ldots j,\mathfrak y]_k)\ldots)$, 
where $\Su_k$ is repeated $l$ times. We call $\Su^l_k(\al[\mathfrak x, i\ldots j,\mathfrak y]_k)$ the 
\textit{$l$-th superdescendant} of $\al[\mathfrak x, i\ldots j,\mathfrak y]_k$. 

Observe that if $k=0$, then 
$\Su_0(\al[\mathfrak x, i\ldots j,\mathfrak y]_0)$ is just 
$\varphi(\al[\mathfrak x, i\ldots j,\mathfrak y]_0)$, but it will be useful to have 
$\Su_k$ as a uniform notation later, for example, when we will define atoms inside blocks.

\begin{remark}
The descendants of two consecutive $k$-multiblocks ($k\in\NN\cup 0$)
are always consecutive, even if they contain several $k$-blocks and letters of order $>k$ 
or one or two of them is empty.
\end{remark}

Consider the following example: let $\E=\{a,b,c,d\}$, $\varphi(a)=ab$, 
$\varphi(b)=cdcdd$, $\varphi(c)=cdd$, $\varphi(d)=d$. 
The orders of letters $a, b, c, d$ are $3, 2, 2, 1$, respectively, and 
$b$ is a preperiodic letter, all other letters are periodic. This 
morphism $\varphi$ is strongly 1-periodic, $\finmax=1$, so 
$\varphi$ is also a strongly 1-periodic morphism with long images.
Consider the corresponding pure morphic sequence 
$\al=\varphi^\infty (a)=a\,b\,cdcdd\,cdddcdddd\,cdddddcdddddd\ldots$ and 
a 1-multiblock $\al[>,1\ldots 1,<]_1$ consisting of a single letter $b$ of order $2>1$.
Here $\swa 10$ is an empty 1-block, and $\swa 21$ 
is also an empty 1-block, but we do not include them into the 
1-multiblock.
We have 
$\Su_1(\al[>,1\ldots 1,<]_1)=\al[>,2\ldots 4,<]_1$ ($\swa 24=cdc$)
and 
$\Su_1^2(\al[>,1\ldots 1,<]_1)=\Su_1(\al[>,2\ldots 4,<]_1)=\al[>,7\ldots 11,<]_1$ ($\swa 7{11}=cdddc$).
If we include both $\swa 10$ and $\swa 21$ into the 1-multiblock and 
consider a 1-multiblock $\al[<,1\ldots 1,>]_1$, 
we will get 
$\Su_1(\al[<,1\ldots 1,>]_1)=\al[<,2\ldots 6,>]_1$ ($\swa 26=cdcdd$)
and 
$\Su_1^2(\al[<,1\ldots 1,>]_1)=\Su_1(\al[<,2\ldots 6,>]_1)=\al[<,5\ldots 15,>]_1$ ($\swa 5{15}=ddcdddcdddd$).

\begin{lemma}\label{inclusionsurvivesbase}
If $k\in\NN$,
$\al[\mathfrak x,s\ldots t,\mathfrak y]_{k-1}$ is a $(k-1)$-multiblock
consisting of a single letter of order $\ge k$ or (if $k>1$) a single $(k-1)$-block,
$\swa st$ is a suboccurrence of a $k$-block $\swa ij$, 
and $\Su_{k-1}(\al[\mathfrak x,s\ldots t,\mathfrak y]_{k-1})=\al[\mathfrak x,s'\ldots t',\mathfrak y]_{k-1}$,
then $\swa {s'}{t'}$
is a suboccurrence of $\Su_k(\swa ij)$.
\end{lemma}
\begin{proof}
The claim follows directly from the definitions of the descendant of a
$k$-block
and of a $(k-1)$-multiblock consisting of a single $(k-1)$-block
or of a single letter of order $>(k-1)$.
\end{proof}

\begin{corollary}\label{inclusionsurvives}
If $k\in\NN$,
$\al[\mathfrak x,s\ldots t,\mathfrak y]_{k-1}$ is a $(k-1)$-multiblock,
$\swa st$ is a suboccurrence of a $k$-block $\swa ij$, 
and $\Su_{k-1}(\al[\mathfrak x,s\ldots t,\mathfrak y]_{k-1})=\al[\mathfrak x,s'\ldots t',\mathfrak y]_{k-1}$,
then $\swa {s'}{t'}$
is a suboccurrence of $\Su_k(\swa ij)$.\qed
\end{corollary}

Now we define \textit{atoms} inside $k$-blocks ($k\in\NN$). 
Let $\mathcal E=\mathcal E_0,\mathcal E_1,\mathcal E_2,\ldots$ be an evolution of $k$-blocks.
The $l$th left and right atoms exist in a $k$-block $\mathcal E_m$
iff $m\ge l>0$. We will also define the zeroth atom, but there will be 
only one zeroth atom in each $k$-block $\mathcal E_m$, it will not be left or right.
First, define the $l$-th atoms inside the $k$-block
$\mathcal E_l$ ($l>0$).
Let $\mathcal E_l=\al_{i\ldots j}$. 
Its ancestor $\swa st=\Pd_k(\swa ij)$ is a $k$-block, so it is a concatenation of 
letters of order $k$ and (if $k>1$) $(k-1)$-blocks, 
and we can consider a $(k-1)$-multiblock $\al[<,s\ldots t,>]_{k-1}$.
If $\swa s{s-1}$ or $\swa {t+1}t$ is a $(k-1)$-block, we 
include it into the 
$(k-1)$-block, 
so we are considering 
a $(k-1)$-block, which starts with a $(k-1)$-block and ends with a $(k-1)$-block. 
Now consider 
a $(k-1)$-block
$\Su_{k-1}(\al[<,s\ldots t,>]_{k-1})$
and denote it by 
$\al[<,i'\ldots j',>]_{k-1}$.

By Corollary \ref{inclusionsurvives}, $\swa {i'}{j'}$ is a suboccurrence of $\swa ij$.
It also follows from the definition of the descendant of a $(k-1)$-block 
that $\varphi(\swa st)$ is a suboccurrence of $\swa {i'}{j'}$ and that $\al_{i'-1}$ and $\al_{j'+1}$ are 
letters of order $>(k-1)$, more precisely, $\al_{i'-1}$ (resp.\ $\al_{j'+1}$) is the rightmost
(resp.\ the leftmost) letter of order $>(k-1)$ in $\varphi(\al_{s-1})$ (resp.\ in $\varphi(\al_{t+1})$).
The 
$(k-1)$-multiblock $\al[<,i\ldots i'-1,<]_{k-1}$
that
comes from the image of the left border of the ancestor, is called
\emph{the $l$th left atom} of the block and is denoted by $\LA_{k,l}(\swa ij)$.
\begin{remark}\label{atomsimplestructleft}
If $k>1$, then this $(k-1)$-multiblock is either empty (does not contain any letters of order $>(k-1)$ or $k$-blocks, 
even empty ones) if $i=i'$, or it begins with a (possibly empty) $(k-1)$-block of the form $\swa i{i''}$ and 
ends with a single letter $\al_{i'-1}$ of order $k$ if $i<i'$.
If $k=1$, then $\al_{i'-1}$ is the rightmost letter in $\varphi(\al_{s-1})$, and 
$\LA_{k,l}(\swa ij)=\swa i{i'-1}$ is an empty occurrence in $\al$ if and only if $i=i'$ if and 
only if the rightmost letter in $\varphi(\al_{s-1})$ is of order $>1$.
\end{remark}
Similarly, the 
$(k-1)$-multiblock $\al[>,j'+1\ldots j,>]_{k-1} = \RA_{k,l}(\swa ij)$
is called
\emph{the $l$-th right atom} of the $k$-block $\swa ij$. 
\begin{remark}\label{atomsimplestructright}
If $k>1$, then it is either an empty $(k-1)$-multiblock if $j'=j$, or 
it begins with a single letter $\al_{j'+1}$ of order $k$ and 
ends with a (possibly empty) $(k-1)$-block of the form $\swa {j''}j$ if 
$j'<j$. 
If $k=1$, then $\al_{j'+1}$ is the leftmost letter in $\varphi(\al_{t+1})$, and 
$\LA_{k,l}(\swa ij)=\swa {j'+1}j$ is an empty occurrence in $\al$ if and only if $j'=j$ if and 
only if the leftmost letter in $\varphi(\al_{t+1})$ is of order $>1$.
\end{remark}
Fig.~\ref{kblockstruct} illustrates this construction.

\begin{figure}[!h]
\centering
\includegraphics{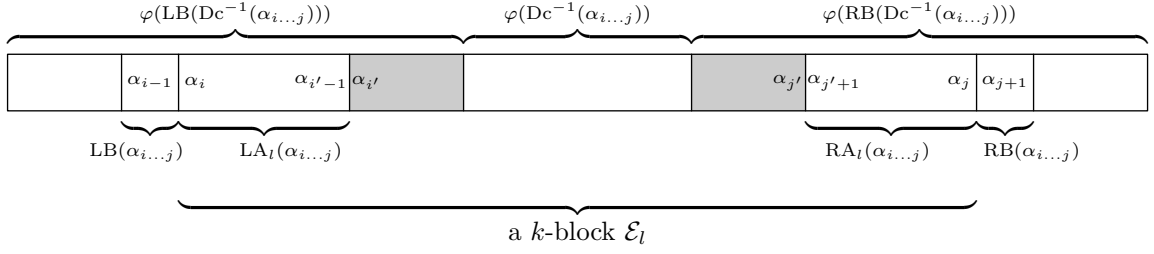}
\caption{Structure of a $k$-block: $\al_{i'-1}$ and $\al_{j'+1}$ are letters of order $k$, 
all letters in the grayed areas are of order $\le k-1$.}
\label{kblockstruct}
\end{figure}

Then,
if $l>m$, the $l$th left and right atoms of $\mathcal E_m$ are defined as follows:
$\LA_{k,l}(\mathcal E_m)=\Su_{k-1}^{m-l}(\LA_{k,l}(\mathcal E_l))$,
$\RA_{k,l}(\mathcal E_m)=\Su_{k-1}^{m-l}(\RA_{k,l}(\mathcal E_l))$.
Then, using Remarks \ref{atomsimplestructleft} and \ref{atomsimplestructright} 
and the definitions of the descendant of a $(k-1)$-block or of a 
$(k-1)$-multiblock that consists of a single letter of order $>(k-1)$, we note the following:
\begin{remark}\label{allatomsimplestruct}
If $k>1$, then each left (resp.\ right) atom in any $k$-block is either an empty $(k-1)$-multiblock 
(it does not contain any letters of order $>(k-1)$ or $k$-blocks, 
even empty ones), or it begins with a (possibly empty) $(k-1)$-block 
(resp.\ with a a single letter of order $k$) and 
ends with a single letter of order $k$
(resp.\ with a (possibly empty) $(k-1)$-block).
\end{remark}

Finally, if $\mathcal E_0=\swa ij$, then the zeroth atom of $\mathcal E_0$ is 
$\CA_{k,0}(\mathcal E_0)=\al[<,i\ldots j,>]_{k-1}$, i.~e.\ it is the largest (including 
all possible empty $(k-1)$-blocks if $k>1$) $(k-1)$-multiblock whose forgetful 
occurrence is $\mathcal E_0$. The zeroth atoms of other blocks in the 
evolution are defined by $\CA_{k,0}(\mathcal E_m)=\Su_{k-1}^m(\CA_{k,0}(\mathcal E_0))$.
\begin{remark}\label{zerothatomsimplestruct}
If $k>1$, then each zeroth atom begins with a (possibly empty) $(k-1)$-block
and ends with a (possibly empty) $(k-1)$-block. However, these two $(k-1)$-blocks may be the same, 
i.~e.\ the zeroth atom can consist of a single $(k-1)$-block. The zeroth 
atom is never empty as a $(k-1)$-multiblock, i.~e.\ it contains at least
one (maybe, empty) $(k-1)$-block, but the forgetful occurrence of 
the zeroth atom may be empty.
\end{remark}

Therefore, if $\mathcal E_m=\swa ij$, then $\al[<,i\ldots j,>]_{k-1}$ (the largest 
$(k-1)$-multiblock whose forgetful occurrence is $\mathcal E_m$) splits into the concatenation of all 
atoms in $\mathcal E_m$: 
$\al[<,i\ldots j,>]_{k-1}=
\LA_{k,m}(\mathcal E_m)\LA_{k,m-1}(\mathcal E_m)\ldots\LA_{k,1}(\mathcal E_m)\CA_{k,0}(\mathcal E_m)
\RA_{k,1}(\mathcal E_m)\ldots\RA_{k,m-1}(\mathcal E_m)\RA_{k,m}(\mathcal E_m)$.

\begin{lemma}\label{lastatomlocation}
Let $l\ge 0$ and $m\ge 1$. Consider an occurrence $\varphi^m(\LB(\mathcal E_l))$ in $\al$. 
Let $\al_{i-1}$ be the rightmost occurrence of a letter of order $>k$ in $\varphi^m(\LB(\mathcal E_l))$
and let $\al_j$ be the rightmost occurrence of a letter of order $\ge k$ in $\varphi^m(\LB(\mathcal E_l))$.

Then $\al_{i-1}=\LB(\mathcal E_{l+m})$ and $\swa ij=\Fg(\LA_{k,l+m}(\mathcal E_{l+m})\ldots\LA_{k,l+1}(\mathcal E_{l+m}))$
as occurrences in $\al$.
\end{lemma}

\begin{proof}
The first equality is proved directly by induction on $m$ using the definition of the descendant of a $k$-block 
and the fact that the image of a letter of order $\le k$ consists of letters of order $\le k$ only.

The second equality for $m=1$ it follows directly from the definitions of the $(l+1)$th atom and of the descendant
of a $(k-1)$-block (see Remark \ref{atomsimplestructleft}). 
The second equality in general will follow from the first one and the fact that either 
$\LA_{k,l+m}(\mathcal E_{l+m})\ldots\LA_{k,l+1}(\mathcal E_{l+m})$ is an empty $(k-1)$-multiblock and $j=i-1$, 
or $\LA_{k,l+m}(\mathcal E_{l+m})\ldots\LA_{k,l+1}(\mathcal E_{l+m})$ is not empty, and $\al_j$ is the 
rightmost occurrence of a letter of order $k$ in its forgetful occurrence.
We already know this for $m=1$, 
to prove 
this
in general, we use induction on $m$. By the definition of 
a descendant of a single letter $\al_s$ of order $>(k-1)$, the rightmost letter in the forgetful 
occurrence of $\Su_{k-1}(\al[>,s\ldots s,<]_{k-1})$ is the rightmost letter of order $>(k-1)$ in 
$\varphi(\al_s)$. Therefore, if 
$\LA_{k,l+m}(\mathcal E_{l+m})\ldots\LA_{k,l+1}(\mathcal E_{l+m})$
is not an empty $(k-1)$-multiblock and the rightmost letter of its forgetful occurrence is $\al_j$, 
a letter of order $k$, then 
$\LA_{k,l+m}(\mathcal E_{l+m+1})\ldots\LA_{k,l+1}(\mathcal E_{l+m+1})=
\Su_{k-1}(\LA_{k,l+m}(\mathcal E_{l+m})\ldots\LA_{k,l+1}(\mathcal E_{l+m}))$
is also a nonempty $k$-multiblock, and the rightmost letter in its forgetful 
occurrence is the rightmost letter of order $>(k-1)$ in $\varphi(\al_j)$.
By the induction hypothesis, $\al_j$ is the rightmost occurrence of a letter of order 
$\ge k$ in $\varphi^m(\LB(\mathcal E_l))$, so, since images of letters of order $\le(k-1)$ consist of 
letters of order $\le(k-1)$ only if $k>1$, we get that 
the rightmost occurrence of a letter of order $>(k-1)$ in $\varphi(\al_j)$
and rightmost occurrence of a letter of order $>(k-1)$ in $\varphi^{m+1}(\LB(\mathcal E_l))$ coincide.
If $\LA_{k,l+m}(\mathcal E_{l+m})\ldots\LA_{k,l+1}(\mathcal E_{l+m})$ 
is an empty $k$-multiblock, then 
$\LA_{k,l+m+1}(\mathcal E_{l+m+1})\LA_{k,l+m}(\mathcal E_{l+m+1})\ldots\LA_{k,l+1}(\mathcal E_{l+m+1})=
\LA_{k,l+m+1}(\mathcal E_{l+m+1})$, and (by the induction hypothesis) the rightmost occurrence of a letter of order $\ge k$ in 
$\varphi^m(\LB(\mathcal E_l))$ is $\LB(\mathcal E_{l+m})$. Now it suffices to use the claim for $l+m$ instead of $l$ and 1 instead of $m$, 
but we have already considered this case before.
\end{proof}

\begin{lemma}
Let $l\ge 0$ and $m\ge 1$. Consider an occurrence $\varphi^m(\LB(\mathcal E_l))$ in $\al$. 
Let $\al_{i+1}$ be the leftmost occurrence of a letter of order $>k$ in $\varphi^m(\LB(\mathcal E_l))$
and let $\al_j$ be the leftmost occurrence of a letter of order $\ge k$ in $\varphi^m(\LB(\mathcal E_l))$.

Then $\al_{i+1}=\RB(\mathcal E_{l+m})$ and $\swa ji=\Fg(\LA_{k,l+1}(\mathcal E_{l+m})\ldots\LA_{k,l+m}(\mathcal E_{l+m}))$
as occurrences in $\al$.
\end{lemma}
\begin{proof}
The proof is completely symmetric to the proof of the previous lemma.
\end{proof}

\begin{corollary}\label{atomsperiodicitybig}
If $l\ge 2$ and $m,n\ge 0$, then 
$\Fg(\LA_{k,l+m}(\mathcal E_{l+m})\ldots\LA_{k,l}(\mathcal E_{l+m}))$ 
is the same abstract word as 
$\Fg(\LA_{k,l+m+n}(\mathcal E_{l+m+n})\ldots\LA_{k,l+n}(\mathcal E_{l+m+n}))$
and
$\Fg(\RA_{k,l}(\mathcal E_{l+m})\ldots\RA_{k,l+m}(\mathcal E_{l+m}))$ 
is the same abstract word as 
$\Fg(\RA_{k,l+n}(\mathcal E_{l+m+n})\ldots\RA_{k,l+m+n}(\mathcal E_{l+m+n}))$.
In other words, if $l\ge 2$ and $m\ge 0$, then the abstract words
$\Fg(\LA_{k,l+m}(\mathcal E_{l+m})\ldots\LA_{k,l}(\mathcal E_{l+m}))$ 
and 
$\Fg(\RA_{k,l}(\mathcal E_{l+m})\ldots\RA_{k,l+m}(\mathcal E_{l+m}))$ 
do not depend on $l$.
\end{corollary}
\begin{proof}
Since $l\ge 2$ and $\varphi$ is a strongly 1-periodic morphism, 
$\LB(\mathcal E_{l-1})=\LB(\mathcal E_{l-1+n})$ as abstract letters.
Denote this abstract letter by $b$. Denote $\varphi^m(b)=\gamma$, this is 
a finite abstract word. Let $\gamma_{i-1}$ (resp.\ $\gamma_j$) be the rightmost occurrence of a letter
of order $>k$ (resp.\ $\ge k$) in $\gamma$. By the previous lemma, 
$\gamma_{i\ldots j}=\Fg(\LA_{k,l+m}(\mathcal E_{l+m})\ldots\LA_{k,l}(\mathcal E_{l+m}))$ 
as abstract words and 
$\gamma_{i\ldots j}=\Fg(\LA_{k,l+m+n}(\mathcal E_{l+m+n})\ldots\LA_{k,l+n}(\mathcal E_{l+m+n}))$ as abstract
words. The proof for right atoms is analogous.
\end{proof}

\begin{corollary}\label{atomsperiodicitysmallleft}
If $l\ge 2$ and $m,n\ge 0$, then 
$\Fg(\LA_{k,l}(\mathcal E_{l+m}))$ is the same abstract word as 
$\Fg(\LA_{k,l+n}(\mathcal E_{l+m+n}))$. Moreover, if $\mathcal E_{l+m}=\swa ij$, 
$\Fg(\LA_{k,l}(\mathcal E_{l+m}))=\swa st$, $\mathcal E_{l+m+n}=\swa{i'}{j'}$, and $\Fg(\LA_{k,l+n}(\mathcal E_{l+m+n}))=\swa{s'}{t'}$, 
then $s-i=s'-i'$. In other words, if $l\ge 2$ and $m\ge 0$, then $\Fg(\LA_{k,l}(\mathcal E_{l+m}))$ does not depend on $l$, 
as an abstract word, and the numbers of letters in $\al$ between $\LB(\mathcal E_l)$ and $\Fg(\LA_{k,l}(\mathcal E_{l+m}))$ 
also does not depend on $l$.
\end{corollary}
\begin{proof}
If $m=0$, then this is just the previous corollary. If $m>0$, then 
by the previous corollary, 
$\Fg(\LA_{k,l+m}(\mathcal E_{l+m})\ldots\LA_{k,l}(\mathcal E_{l+m}))$ 
is the same abstract word as 
$\Fg(\LA_{k,l+m+n}(\mathcal E_{l+m+n})\ldots\LA_{k,l+n}(\mathcal E_{l+m+n}))$, 
denote this abstract word by $\gamma$,
and 
$\Fg(\LA_{k,l+m}(\mathcal E_{l+m})\ldots\LA_{k,l+1}(\mathcal E_{l+m}))$ 
is the same abstract word as 
$\Fg(\LA_{k,l+m+n}(\mathcal E_{l+m+n})\ldots\LA_{k,l+n+1}(\mathcal E_{l+m+n}))$,
denote this abstract word by $\delta$. Clearly, $\delta$ is a prefix of $\gamma$, 
so write $\gamma=\delta\delta'$ for some finite abstract word $\delta'$. But then 
$\Fg(\LA_{k,l}(\mathcal E_{l+m}))=\delta'$ as abstract words, 
$\Fg(\LA_{k,l}(\mathcal E_{l+m+n}))=\delta'$ as abstract words, 
$s-i=|\delta|$ and $s'-i'=|\delta|$.
\end{proof}

\begin{corollary}\label{atomsperiodicitysmallright}
If $l\ge 2$ and $m,n\ge 0$, then 
$\Fg(\RA_{k,l}(\mathcal E_{l+m}))$ is the same abstract word as 
$\Fg(\RA_{k,l+n}(\mathcal E_{l+m+n}))$. Moreover, if $\mathcal E_{l+m}=\swa ij$, 
$\Fg(\RA_{k,l}(\mathcal E_{l+m}))=\swa st$, $\mathcal E_{l+m+n}=\swa{i'}{j'}$, and $\Fg(\RA_{k,l+n}(\mathcal E_{l+m+n}))=\swa{s'}{t'}$, 
then $j-t=j'-t'$. In other words, if $l\ge 2$ and $m\ge 0$, then $\Fg(\RA_{k,l}(\mathcal E_{l+m}))$ does not depend on $l$, 
as an abstract word, and the numbers of letters in $\al$ between $\Fg(\RA_{k,l}(\mathcal E_{l+m}))$ and $\LB(\mathcal E_l)$
also does not depend on $l$.
\end{corollary}
\begin{proof}
The proof is completely symmetric to the proof of the previous corollary.
\end{proof}

Observe that the condition $l\ge 2$ cannot be omitted
since in the proof of Corollary \ref{atomsperiodicitybig}
we used the fact that $\LB(\mathcal E_{l-1})=\LB(\mathcal E_{l-1+n})$
as abstract letters. Moreover, $\LA_{k,1}(\mathcal E_1)$ is a $(k-1)$-multiblock 
contained in the image of $\LB(\mathcal E_0)$, and $\LA_{k,l}(\mathcal E_l)$ for $l>1$
is contained in the image of $\LB(\mathcal E_l)$, a letter which does not have to be equal to 
$\LB(\mathcal E_0)$, so the letters of order $k$ and (if $k>1$) $(k-1)$-blocks
in $\LA_{k,1}(\mathcal E_1)$ and $\LA_{k,l}(\mathcal E_l)$ may be different.
And the first left atoms of other $k$-blocks in the evolution 
are superdescendants of $\LA_{k,1}(\mathcal E_1)$, while the $l$th 
atoms of other $k$-blocks in the evolution are superdescendants of $\LA_{k,l}(\mathcal E_l)$
for $l>1$. So, the $(k-1)$-blocks in $\LA_{k,1}(\mathcal E_m)$ may belong to 
totally different evolutions than $(k-1)$-blocks in $\LA_{k,l}(\mathcal E_n)$ for $l>1$ belong to, 
while the $(k-1)$-blocks in $\LA_{k,l}(\mathcal E_m)$ 
and in $(k-1)$-blocks in $\LA_{k,l'}(\mathcal E_n)$ by Corollary \ref{atomsperiodicitysmallleft} 
belong to the same evolutions 
if evolutions are understood as sequences of abstract words 
(as in Lemma \ref{finite-number-of-evolutions}).

These observations
and these corollaries
justify
the following definitions. If $l\ge 1$, we call the concatenation of the $(k-1)$-multiblocks 
$\LA_{k,1}(\mathcal E_l)\CA_{k,0}(\mathcal E_l)\RA_{k,1}(\mathcal E_l)$ 
the \emph{core} of 
$\mathcal E_l$. The core of $\mathcal E_l$ is denoted by $\Cr_k(\mathcal E_l)$.
If $l\ge 2$, the 
concatenation of the $(k-1)$-multiblocks 
$\LA_{k,l}(\mathcal E_l)\LA_{k,l-1}(\mathcal E_l)\ldots\LA_{k,2}(\mathcal E_l)$
(resp.\ $\RA_{k,2}(\mathcal E_l)\ldots\RA_{k,l-1}(\mathcal E_l)\RA_{k,l}(\mathcal E_l)$)
is called the \emph{left (resp.\ right) component}.

By Remark \ref{atomsimplestructleft}, $\LA_{k,l}(\mathcal E_l)$ is either an empty $(k-1)$-multiblock, 
or it contains (actually, the rightmost letter of its forgetful occurrence is)
a letter of order $k$. By Corollary \ref{atomsperiodicitysmallleft}, either for all $l>1$ 
$\LA_{k,l}(\mathcal E_l)$ is an empty $(k-1)$-multiblock, 
or for all $l>1$ $\LA_{k,l}(\mathcal E_l)$ contains a letter of order $k$. So, if 
each atom $\LA_{k,l}(\mathcal E_l)$ for $l>1$ contains a letter of order $k$, we say that 
\textit{Case I holds for $\mathcal E$ at the left}. If all atoms $\LA_{k,l}(\mathcal E_l)$
for $l>1$ are empty $(k-1)$-blocks, we say that 
\textit{Case II holds for $\mathcal E$ at the left}.
Similarly, cases I
and II
are defined for right 
atoms. These cases happen independently at right and at
left, in any combination. 
\begin{remark}
The left (resp.\ right) component 
is empty if and only if Case II holds at the left (resp.\ at the right).
If Case II holds both at the left and at the right for an evolution $\mathcal E$
of $k$-blocks and $l\ge 1$, then $\mathcal E_l=\Fg(\Cr_k(\mathcal E_l))$.
\end{remark}

Note that if $k\in\NN$, then $k$-blocks may exist by definition even if all letters in $\al$ 
have either order $<k$, or order $\infty$ (see also Lemma \ref{splitconcatenation}). In this situation, 
Case II holds for all evolutions of $k$-blocks both at the left and at the right.

\section{1-Blocks}\label{sectiononeblocks}

Now we will consider 1-blocks more accurately. The fact that 
$\varphi$ is a strongly 1-periodic morphism makes the structure of 
a 1-block quite easy. During this section, it will be useful to keep in mind that 
0-multiblocks are just occurrences in $\al$ and their descendants are just 
their images under $\varphi$.

\begin{lemma}\label{oneblockstructlemmai}
Let $\mathcal E$ be an evolution of 1-blocks. Then:

If $l>1$, then $\Cr_1(\mathcal E_l)$ does not depend on $l$ as an abstract word and consists of periodic letters only.

If $l>1$, then $\LA_{1,l}(\mathcal E_l)$ and $\RA_{1,l}(\mathcal E_l)$ do not depend on $l$ as abstract words.

If $l>1$ and $m\ge 1$, then $\LA_{1,l}(\mathcal E_{l+m})$ and $\RA_{1,l}(\mathcal E_{l+m})$
as abstract words
depend neither on $l$ nor on $m$. They equal $\varphi(\LA_{1,l}(\mathcal E_l))$ and 
$\varphi(\RA_{1,l}(\mathcal E_l))$ as abstract words, respectively
and consist of periodic letters of order 1 only.
\end{lemma}

\begin{proof}
Since $\varphi$ is (in particular) weakly 1-periodic, the image of a preperiodic letter of order 1
consists of periodic letters of oder 1 only. The image of a periodic letter of order 1 is a (single) 
periodic letter of order 1. We have $\Cr_1(\mathcal E_l)=\LA_{1,1}(\mathcal E_l)\CA_{1,0}(\mathcal E_l)\RA_{1,1}(\mathcal E_l)=
\Su_0^{l-1}(\LA_{1,1}(\mathcal E_1)\CA_{1,0}(\mathcal E_1)\RA_{1,1}(\mathcal E_1))=\varphi^{l-1}(\Cr_1(\mathcal E_1))$.
So, if $l>1$, all letters in $\varphi^{l-1}(\Cr_1(\mathcal E_1))$ are periodic letters of order 1.
By weak 1-periodicity again, $\varphi(\varphi^{l-1}(\Cr_1(\mathcal E_1)))=\varphi^{l-1}(\Cr_1(\mathcal E_1))$
as abstract words. But $\Cr_1(\mathcal E_l)=\varphi^l(\Cr_1(\mathcal E_1))=\varphi(\varphi^{l-1}(\Cr_1(\mathcal E_1)))$, 
so we have the first claim.

The second claim is just a particular case of Corollaries \ref{atomsperiodicitysmallleft} and \ref{atomsperiodicitysmallright}.
For the third claim, we write $\LA_{1,l}(\mathcal E_l+m)=\Su_0^m(\LA_{1,l}(\mathcal E_l))=\varphi^m(\LA_{1,l}(\mathcal E_l))$. 
Using the second claim, we see that it is sufficient to prove that $\varphi^m(\LA_{1,l}(\mathcal E_l))$
does not depend on $m$ as an abstract word if $m\ge 1$ (for $m=1$ it clearly equals $\varphi(\LA_{1,l}(\mathcal E_l))$). 
Again, since $\varphi$ is weakly 1-periodic, 
$\varphi^m(\LA_{1,l}(\mathcal E_l))$ consists of periodic letters of order 1 only if $m\ge 1$, and, 
by weak 1-periodicity again, $\varphi^{m+1}(\LA_{1,l}(\mathcal E_l))=\varphi^m(\LA_{1,l}(\mathcal E_l))$
as an abstract words if $m\ge 1$. The computation for the right atoms is the same.
\end{proof}

After we have this lemma, we can give the following definitions:

Given an evolution $\mathcal E$ of 1-blocks, we call the abstract word $\Cr_1(\mathcal E_l)$ for 
any $l>1$ the \textit{core of $\mathcal E$} and denote it by $\Cr_1(\mathcal E)$.
The abstract word $\LA_{1,l}(\mathcal E_l)$ (resp.\ $\RA_{1,l}(\mathcal E_l)$) for any $l>1$ is called 
the \textit{left (resp.\ right) preperiod of $\mathcal E$} and is denoted by $\LpreP_1(\mathcal E)$
(resp.\ by $\RpreP_1(\mathcal E)$). The $l$th left (resp.\ right) atom 
of a particular 1-block $\mathcal E_l$, where $l>1$ is called 
the \textit{left (resp.\ right) preperiod of $\mathcal E_l$}
and is denoted by
by $\LpreP_1(\mathcal E_l)$ (resp.\ by $\RpreP_1(\mathcal E_l)$.
The abstract word
$\LA_{1,l}(\mathcal E_{l+m})$ 
(resp.\ $\RA_{1,l}(\mathcal E_{l+m})$) for any $l>1$ and $m\ge 1$ is called 
the \textit{left (resp.\ right) period of $\mathcal E$} and is denoted by 
$\LP_1(\mathcal E)$ (resp.\ by $\RP_1(\mathcal E)$).
By Lemma \ref{oneblockstructlemmai}, it equals $\varphi(\LpreP_1(\mathcal E))$
(resp.\ $\varphi(\RpreP_1(\mathcal E))$).
If $l>1$, the occurrence between $\LpreP_1(\mathcal E_l)$ and $\Cr_1(\mathcal E_l)$
(resp.\ between $\Cr_1(\mathcal E_l)$ and $\RpreP_1(\mathcal E_l)$)
is called the \textit{left (resp.\ right) regular part} of $\mathcal E_l$
and is denoted by $\LR_1(\mathcal E_l)$ (resp.\ by $\RR_1(\mathcal E_l)$).
If $l=2$, it is an occurrence of the empty word, and if $l\ge 3$, it is the concatenation 
of left atoms $\LA_{1,l-1}(\mathcal E_l)\ldots \LA_{1,2}(\mathcal E_l)$
(resp.\ of right atoms $\RA_{1,2}(\mathcal E_l)\ldots \RA_{1,l-1}(\mathcal E_l)$),
all these atoms equal $\LP_1(\mathcal E)$ (resp.\ $\RP_1(\mathcal E)$) as abstract words.

Using this terminology, we formulate the following corollary.
\begin{corollary}
If $\mathcal E$ is an evolution of 1-blocks and $l>1$, 
then the 1-block $\mathcal E_l$ equals the following
abstract word: 
$\LpreP_1(\mathcal E)\LP_1(\mathcal E)\ldots \LP_1(\mathcal E)\Cr_1(\mathcal E)\RP_1(\mathcal E)\ldots \RP_1(\mathcal E)\RpreP_1(\mathcal E)$, 
where $\LP_1(\mathcal E)$ and $\RP_1(\mathcal E)$ are repeated $l-2$ times each.

$\LpreP_1(\mathcal E)$ (resp.\ $\RpreP_1(\mathcal E)$) is an empty word if and only if 
Case II holds at the left (resp.\ at the right) for $\mathcal E$.

$\LP_1(\mathcal E)$ (resp.\ $\RP_1(\mathcal E)$) is an empty word if and only if 
Case II holds at the left (resp.\ at the right) for $\mathcal E$.

The left (resp.\ right) regular part of $\mathcal E_l$ consists 
of periodic letters of order 1 only. It is an empty word if 
and only if Case II holds at the left (resp.\ at the right) for $\mathcal E$ or 
$l=2$.\qed
\end{corollary}

The terminology we introduced and the structure of a 1-block is illustrated by Fig. \ref{onestruct}.

\begin{figure}[!h]
\centering
\includegraphics{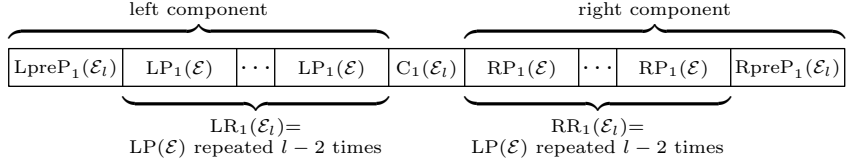}
\caption{Detailed structure of a 1-block $\mathcal E_l$.}
\label{onestruct}
\end{figure}

We call a 1-block $\mathcal E_l$ \textit{stable} if $l\ge 3$, 
otherwise it is called \textit{unstable}. If a 1-block is stable, then its 
left and right components, preperiods and regular parts, as well as its core, are defined.
The following corollary about lengths of subwords inside 1-blocks follows directly
from what we already know about the structure of 1-blocks and from 
Corollary \ref{finite-number-of-evolutions}.

\begin{corollary}\label{oneblocklengths}
The lengths of all unstable 1-blocks are bounded by a single constant that depends on 
$\E$ and $\varphi$ only. The lengths of all cores and left and right 
preperiods of all stable 1-blocks are bounded by a single constant that depends on 
$\E$ and $\varphi$ only.

The the left (resp.\ right) regular part
of a stable 1-block $\mathcal E_l$
is a nonempty word if and only if Case I holds at the left (resp.\ at the right).
Moreover, it is 
completely $\LP_1(\mathcal E)$-periodic 
(resp.\ $\RP_1(\mathcal E)$-periodic),
and the length of the left (resp.\ right) regular part equals $(l-2)|\LP_1(\mathcal E)|$
(resp.\ $(l-2)|\LP_1(\mathcal E)|$).

In particular, 
the length of the left (resp.\ right)
regular part of a 1-block $\mathcal E_l$, as well as 
the length of the left (resp.\ right) component
is either $\Theta(l)$ if Case I holds at the left (resp.\ at the right), 
or $0$ if Case II holds at the left (resp.\ at the right).

The length of the whole 1-block $\mathcal E_l$ is always $O(l)$. 
It is $\Theta(l)$ if Case I holds at the left or at the right, and 
is $O(1)$ if Case II holds both at the left and at the right.
All constants
in the $\Theta$- and $O$-notations in this corollary depend on $\E$ and $\varphi$ only.\qed
\end{corollary}

Now let us recall the definition of a strongly 1-periodic morphism 
with long images. Let $\mathcal E$ be an evolution of 1-blocks 
As we already noted, $\LB(\mathcal E_{l+1})=\RL_1(\varphi(\LB(\mathcal E_l)))$
for all $l\ge 0$. Moreover, suppose now that $l\ge 1$ and
$\LB(\mathcal E_{l+1})=\RL_1(\varphi(\LB(\mathcal E_l)))=\LB(\mathcal E)$ as
an abstract letter. Then $\varphi(\LB(\mathcal E_l))$
has a suffix $\LB(\mathcal E_{l+1})\LpreP_1(\mathcal E_{l+1})$.
Hence, 
the word $\gamma$ we used in the definition of a final 
period for $a=\LB(\mathcal E)$ is $\LpreP_1(\mathcal E)$,
and $\varphi(\gamma)=\varphi(\LpreP_1(\mathcal E))=\LP_1(\mathcal E)$
by Lemma \ref{oneblockstructlemmai} (and by the definitions of 
the left preperiod and the left period of an evolution).
So, the following lemma follows now directly from Lemma \ref{finalperiodsarecomplete}
and from the definition of a a strongly 1-periodic morphism 
with long images.

\begin{lemma}\label{regpartabslargeone}
If $\mathcal E$ is an evolution of 1-blocks and Case I holds at the left (resp.\ at the right), 
then $\psi(\LP_1(\mathcal E))$ (resp.\ $\psi(\RP_1(\mathcal E))$)
has a minimal complete period $\lambda$, and $\lambda$ is a final period. 
$|\LP_1(\mathcal E)|\ge 2\finmax$ and $|\RP_1(\mathcal E)|\ge 2\finmax$.

If $\mathcal E_l$ is a stable 1-block and Case I holds at the left (resp.\ at the right)
for $\mathcal E$, then $\lambda$ is the minimal complete period of
$\psi(\LR_1(\mathcal E))$ (resp.\ $\psi(\RR_1(\mathcal E))$).
$|\LR_1(\mathcal E)|\ge 2\finmax$ and $|\RR_1(\mathcal E)|\ge 2\finmax$.
\qed
\end{lemma}

The core of a stable 1-block is called its (unique) \emph{prime central kernel}.
It is also called its (unique) \emph{composite central kernel}.
If $\mathcal E$ is an evolution of 1-blocks and $l\ge 3$ (so that $\mathcal E_l$ is stable), 
then the prime (resp.\ composite) central kernel of $\mathcal E_{l+1}$ is called the 
\textit{descendant} of the prime (resp.\ composite) central kernel of $\mathcal E_l$.

\section{Stable $k$-Blocks}\label{sectionstableblocks}

Now we are going to consider $k$-blocks more accurately. 
In this section we mostly focus on $k$-blocks for $k>1$, referring 
to the previous section for similar results for $k=1$.
Through this section, we will give examples based on $\E=\{a, \mathfrak b, b, \mathfrak c, c, \mathfrak d, d, \mathfrak e, e, \mathfrak f, f\}$ 
and on the following morphism $\varphi$:
$\varphi(a)=a\mathfrak{bdb}$, $\varphi(\mathfrak b)=\mathfrak cb\mathfrak{ee}$, $\varphi(b)=\mathfrak cb\mathfrak{ee}$,
$\varphi(\mathfrak c)=\mathfrak{ee}c\mathfrak{ee}$, $\varphi(c)=\mathfrak{ee}c\mathfrak{ee}$, 
$\varphi(\mathfrak d)=\mathfrak{ff}d\mathfrak{ff}$, $\varphi(d)=\mathfrak{ff}d\mathfrak{ff}$, 
$\varphi(\mathfrak e)=e$, $\varphi(e)=e$, $\varphi(\mathfrak f)=f$, $\varphi(f)=f$. Then
$\varphi^\infty(a)=\alpha=a\,\mathfrak{bdb}\,\mathfrak cb\mathfrak{eeff}d\mathfrak{ffc}b\mathfrak{ee}
\,\mathfrak{ee}c\mathfrak{eec}b\mathfrak{ee}ee\!f\mskip-4.5mu f\mskip-1.5mu\mathfrak{ff}d
\mathfrak{ff}\mskip-1.5mu f\mskip-4.5mu f\!\mathfrak{ee}c\mathfrak{eec}b\mathfrak{ee}ee\ldots$.
Here $a$ is a 
periodic letter
of order 4, $\mathfrak b$ is a preperiodic letter of order 3, $b$ is a periodic letter of order 3, 
$\mathfrak c$ and $\mathfrak d$ are preperiodic letters of order 2, $c$ and $d$ are periodic letters of order 2, 
$\mathfrak e$ and $\mathfrak f$
are preperiodic letters of order 1,
and $e$ and $f$
are periodic letters of order 1. Consider an evolution $\mathcal E$ of 2-blocks, whose origin is
$\alpha_{2\ldots 2}=\mathfrak d$. A 2-block $\mathcal E_l$ where $l$ is large enough looks as follows:
$$
\underbrace{\mathfrak{ee}ee..e\!f\mskip-4.5mu f\!..\mskip-1.5mu f\mskip-1.5mu\mathfrak{ff}
d\mathfrak{ff}\mskip-1.5mu f\mskip-4.5mu f\!..\mskip-1.5mu f\!ee..ee\mathfrak{ee}c\mathfrak{ee}ee..ee\mathfrak{ee}}_{\mbox{\scriptsize core\strut}}
\underbrace{\vphantom{f}c\mathfrak{ee}ee..ee\mathfrak{ee}c\mathfrak{ee}ee..ee\mathfrak{ee}c \ldots 
c\mathfrak{ee}ee..ee\mathfrak{ee}c\ldots ee\mathfrak{ee}c\mathfrak{ee}ee\mathfrak{ee}c\mathfrak{eec}}_{\mbox{\scriptsize right component\strut}}.
$$
Here Case I holds at the right and Case II holds at the left (and the left component is empty). 
Intervals denoted by $\ldots$ may contain many intervals denoted by $..$
The (forgetful occurrence of) the zeroth atom is 
$\mathfrak{ee}ee..e\!f\mskip-4.5mu f\!..\mskip-1.5mu f\mskip-1.5mu\mathfrak{ff}
d\mathfrak{ff}\mskip-1.5mu f\mskip-4.5mu f\!..\mskip-1.5mu f\!ee..ee\mathfrak{ee}$,
the (forgetful occurrence of) the $m$th right atom, where $0<m<l-1$ is of the form $c\mathfrak{ee}e..e\mathfrak{ee}$, where $e$ is repeated 
$4(l-m)-2$ times,
the (forgetful occurrence of) the $(l-1)$th right atom is $c\mathfrak{ee}$, and the 
(forgetful occurrence of) the $l$th right atom is $\mathfrak c$, the $l$th atom itself also includes the empty 
1-block located immediately to the right of this $\mathfrak c$.

First, let us define \textit{stable $k$-blocks}. 
A $k$-block 
is called \textit{stable} if 
its evolutional sequence number is at least $3k$.
(For $k=1$ we get 
exactly the definition from
the 
previous section.) 
Let $\mathcal E=\mathcal E_0,\mathcal E_1,\mathcal E_2,\ldots$ be an evolution of $k$-blocks.
If $\mathcal E_l$ is a stable $k$-block (i.~e.\ if $l\ge 3k$), the concatenation of atoms 
$\LA_{k,l}(\mathcal E_l)\LA_{k,l-1}(\mathcal E_l)\ldots \LA_{k,l-3k+3}(\mathcal E_l)$
(resp.\ $\RA_{k,l-3k+3}(\mathcal E_l)\ldots\RA_{k,l-1}(\mathcal E_l)\RA_{k,l}(\mathcal E_l)$)
is called the \textit{left (resp.\ right) preperiod} of $\mathcal E_l$
and is denoted by $\LpreP_k(\mathcal E_l)$ (resp.\ by $\RpreP_k(\mathcal E_l)$).
The concatenation of all atoms between the left preperiod and the core
(resp.\ between the core and the right preperiod),
i.~e.\ the concatenation 
$\LA_{k,l-3k+2}(\mathcal E_l)\LA_{k,l-3k+1}(\mathcal E_l)\ldots \LA_{k,2}(\mathcal E_l)$
(resp.\ $\RA_{k,2}(\mathcal E_l)\ldots\RA_{k,l-3k+1}(\mathcal E_l)\RA_{k,l-3k+2}(\mathcal E_l)$)
is called the 
\textit{left (resp.\ right) regular part} of $\mathcal E_l$.
It is denoted by $\LR_k(\mathcal E_l)$ (resp.\ by $\RR_k(\mathcal E_l)$).
Again, these definitions for $k=1$ coincide with the definition from the previous section.
The following remark is a particular case of Corollary \ref{atomsperiodicitybig}.
\begin{remark}\label{preperiodfixed}
If $\mathcal E_l$ is a stable $k$-block, then 
$\Fg(\LpreP_k(\mathcal E_l))$ and $\Fg(\RpreP_k(\mathcal E_l))$ do not depend on $l$ as abstract words if $l\ge 3k$.
\end{remark}
So, we call the abstract word $\Fg(\LpreP_k(\mathcal E_l))$ (resp.\ $\Fg(\RpreP_k(\mathcal E_l))$) 
for any $l\ge 3k$ \textit{the left (resp.\ right) preperiod of $\mathcal E$} and denote it by 
$\LpreP_k(\mathcal E)$ (resp.\ by $\RpreP_k(\mathcal E)$). In the example above, 
$\LpreP_k(\mathcal E)$ is empty since Case II holds for $\mathcal E$ at the left, 
and $\RpreP_k(\mathcal E)=c\mathfrak{ee}eeeeee\mathfrak{ee}c\mathfrak{ee}ee\mathfrak{ee}c\mathfrak{eec}$.

\begin{corollary}\label{finitelrprep}
The lengths of all left and right preperiods 
of all evolutions of $k$-blocks arising in $\al$
are bounded by a single constant that depends on $\E$, $\varphi$, and $k$ only.
In particular, only finitely many abstract words can equal 
left and right preperiods of evolutions 
of $k$-blocks arising in $\al$.
\end{corollary}

\begin{proof}
By Corollary \ref{finite-number-of-evolutions}, only finitely many 
sequences of abstract words can be evolutions in $\al$. Therefore, there exists 
a single constant $x$ that depends on $\E$, $\varphi$ and $k$ only
such that if $\mathcal E$ is an evolution of $k$-blocks, then $|\mathcal E_{3k}|\le x$.
By Remark \ref{preperiodfixed}, $\LpreP_k(\mathcal E)$ and $\RpreP_k(\mathcal E)$ are 
subwords of 
$\mathcal E_{3k}$, so $|\LpreP_k(\mathcal E)|\le x$ and $|\RpreP_k(\mathcal E)|\le x$.
\end{proof}

Note that we \textit{\textbf{do not}} claim that if $\mathcal E$ and $\mathcal E'$ are two evolutions of $k$-blocks such 
that $\mathcal E_l=\mathcal E'_l$ as an abstract word for all $l\ge 0$, 
then $\LpreP_k(\mathcal E)=\LpreP_k(\mathcal E')$ and $\RpreP_k(\mathcal E)=\RpreP_k(\mathcal E')$.

Now let us prove some facts about atoms inside the regular parts of a stable $k$-block
(or about atoms of the form $\LA_{k,l}(\mathcal E_{l+m})$, where $m$ is large enough).
\begin{lemma}\label{klettersbecomeperiodic}
Let $\mathcal E$ be an evolution of $k$-blocks. If $l\ge 1$ and $m\ge 1$, 
then all letters of order $k$ in $\LA_{k,l}(\mathcal E_{l+m})$ (resp.\ in 
$\RA_{k,l}(\mathcal E_{l+m})$) are periodic, and there is at least one such 
letter if Case I holds at the left (resp.\ at the right). If $m\ge 1$, then all letters of 
order $k$ in $\CA_{k,0}(\mathcal E_m)$ are periodic.
\end{lemma}
\begin{proof}
For $k=1$ we already know this by Lemma \ref{oneblockstructlemmai}. Suppose that $k>1$.
By the definition of the descendant of a ($(k-1)$-multiblock that consists of a) 
single letter $\al_i$ of order $>(k-1)$, it is a $(k-1)$-multiblock that consists of 
$(k-1)$-blocks and letters of order $>(k-1)$ inside $\varphi(\al_i)$.
Hence, all letters of order $k$ in $\LA_{k,l}(\mathcal E_{l+1})$
and in $\RA_{k,l}(\mathcal E_{l+1})$ ($l\ge 1$) are periodic since they are 
contained in the images of letters of order $k$
in $\LA_{k,l}(\mathcal E_l)$ and in $\RA_{k,l}(\mathcal E_l)$. If Case I holds at the 
left (resp.\ at the right), then there is at least one letter of order $k$ in 
in $\LA_{k,l}(\mathcal E_l)$ (resp.\ in $\RA_{k,l}(\mathcal E_l)$), 
and its descendant gives at least one letter of order $k$ for 
$\LA_{k,l}(\mathcal E_{l+1})$
(resp.\ for $\RA_{k,l}(\mathcal E_{l+1})$).
Also, 
all letters of order $k$ in $\CA_{k,0}(\mathcal E_1)$
are periodic since they are 
contained in the images of letters of order $k$
in $\CA_{k,0}(\mathcal E_0)$. Now the claim follows from 
Remark \ref{descendantofperiodic} and the definition of a left (right, zeroth) atom.
\end{proof}

\begin{corollary}\label{amountkletterssame}
Let $\mathcal E$ be an evolution of $k$-blocks. If $l\ge 2$ and $m\ge 1$, then 
the amounts of letters of order $k$ in $\LA_{k,l}(\mathcal E_{l+m})$ and in 
$\RA_{k,l}(\mathcal E_{l+m})$ do not depend on $l$ and $m$.
The amounts of letters of order $k$ in $\LA_{k,1}(\mathcal E_{1+m})$, in 
$\RA_{k,1}(\mathcal E_{1+m})$ and in $\CA_{k,0}(\mathcal E_m)$
do not depend on $m$ (but may differ from the amounts of letters of 
order $k$ in $\LA_{k,l}(\mathcal E_{l+m})$ and in 
$\RA_{k,l}(\mathcal E_{l+m})$ for $l\ge 2$).
\end{corollary}
\begin{proof}
For $k=1$ this follows from Lemma \ref{oneblockstructlemmai}. If $k>1$, then 
the fact that these amounts do not depend on $m$ follows now from Remark 
\ref{descendantofperiodic}, and the fact that they do not depend on $l$ if $l\ge 2$
follows from Corollaries \ref{atomsperiodicitysmallleft} (for left atoms)
and \ref{atomsperiodicitysmallright} (for right atoms).
\end{proof}

\begin{corollary}\label{amountkminusoneblockssame}
Let $\mathcal E$ be an evolution of $k$-blocks, where $k>1$. 
If $l\ge 2$ and $m\ge 1$, then the amounts of (possibly empty) $(k-1)$-blocks 
in $\LA_{k,l}(\mathcal E_{l+m})$ and in $\RA_{k,l}(\mathcal E_{l+m})$ 
do not depend on $l$ and $m$ and equal the amounts of letters of order $k$
in $\LA_{k,l}(\mathcal E_{l+m})$ and in $\RA_{k,l}(\mathcal E_{l+m})$, respectively.
The amounts of (possibly empty) $(k-1)$-blocks 
in $\LA_{k,1}(\mathcal E_{1+m})$ and in $\RA_{k,1}(\mathcal E_{1+m})$ 
do not depend on $m$ and equal the amounts of letters of order $k$
in $\LA_{k,1}(\mathcal E_{1+m})$ and in $\RA_{k,1}(\mathcal E_{1+m})$, respectively
(but may differ from the amounts of $(k-1)$-blocks 
in $\LA_{k,l}(\mathcal E_{l+m})$ and in 
$\RA_{k,l}(\mathcal E_{l+m})$ for $l\ge 2$).
\end{corollary}
\begin{proof}
Recall that by Remark \ref{allatomsimplestruct}, a left (resp.\ right) atom 
is either an empty $(k-1)$-multiblock, or it begins with a (possibly empty) $(k-1)$-block 
(resp.\ with a a single letter of order $k$) and 
ends with a single letter of order $k$
(resp.\ with a (possibly empty) $(k-1)$-block).
It follows from the general definition of a $(k-1)$-multiblock
that $(k-1)$-blocks and letters of order $>(k-1)$
always alternate inside a $(k-1)$-multiblock.
Hence, the amount of letters 
of order $k$ in a left (resp.\ right) atom is 
always the same as the amount of $(k-1)$-blocks in it.
\end{proof}

\begin{corollary}\label{amountkminusoneblockssamezero}
Let $\mathcal E$ be an evolution of $k$-blocks, where $k>1$. 
If $m\ge 1$, then the amount of (possibly empty) $(k-1)$-blocks 
in $\CA_{k,0}(\mathcal E_m)$
does not depend on $m$ and equals one plus the amount of letters of order $k$
in $\CA_{k,0}(\mathcal E_m)$.
\end{corollary}
\begin{proof}
By Remark \ref{zerothatomsimplestruct}, the zeroth atom
either consists of a single (possibly empty) $(k-1)$-block, 
or it begins with a (possibly empty) $(k-1)$-block 
and 
ends 
with another (possibly empty) $(k-1)$-block.
Again, it follows from the general definition of a $(k-1)$-multiblock
that $(k-1)$-blocks and letters of order $>(k-1)$
always alternate inside a $(k-1)$-multiblock.
Hence, the amount of $(k-1)$-blocks 
in the zeroth atom always equals one plus 
the amount of letters 
of order $k$ in it.
\end{proof}

\begin{lemma}\label{evolseqnolarge}
Let $\mathcal E$ be an evolution of $k$-blocks, where $k>1$. 
Let $m\ge 1$.
Let $\swa ij$ be a $(k-1)$-block in a left atom $\LA_{k,l}(\mathcal E_{l+m})$, 
in a right atom $\RA_{k,l}(\mathcal E_{l+m})$, 
or in a zeroth atom $\CA_{k,0}(\mathcal E_m)$.
Then the evolutional sequence number of $\swa ij$ is either $m$ or $m-1$.
\end{lemma}

\begin{proof}
Observe first that all $(k-1)$-blocks in $\LA_{k,l}(\mathcal E_l)$
and in $\RA_{k,l}(\mathcal E_l)$ are origins since by Lemma 
\ref{lastatomlocation} they are contained in $\varphi(\LB(\mathcal E_{l-1}))$
and $\varphi(\RB(\mathcal E_{l-1}))$, respectively, and cannot be 
prefixes or suffixes of $\varphi(\LB(\mathcal E_{l-1}))$
and $\varphi(\RB(\mathcal E_{l-1}))$, respectively. Also, 
all $(k-1)$-blocks in $\CA_{k,0}(\mathcal E_0)$ are origins since 
$\Fg(\CA_{k,0}(\mathcal E_0))=\mathcal E_0$ is an origin, so it is contained 
in the image of a single letter and cannot be a prefix or a suffix there.

Now, let 
$\swa ij$
be a $(k-1)$-block in 
$\LA_{k,l+1}(\mathcal E_l)=\Su_{k-1}(\LA_{k,l}(\mathcal E_l))$, 
in $\RA_{k,l+1}(\mathcal E_l)=\Su_{k-1}(\RA_{k,l}(\mathcal E_l))$, 
or in $\CA_{k,0}(\mathcal E_1)=\Su_{k-1}(\CA_{k,0}(\mathcal E_0))$.
Then there are two possibilities for 
$\swa ij$:
The first possibility is that 
$\swa ij$
is the descendant of a $(k-1)$-block 
in $\LA_{k,l}(\mathcal E_l)$, in $\RA_{k,l}(\mathcal E_l)$, 
or in $\CA_{k,0}(\mathcal E_0)$, respectively, and
then 
the evolutional sequence number of $\swa ij$ is 1.
The second possibility is that
$\swa ij$
is a suboccurrence 
of the descendant of a letter $\al_s$ of order $k$ in 
$\LA_{k,l}(\mathcal E_l)$, in $\RA_{k,l}(\mathcal E_l)$, or in 
$\CA_{k,0}(\mathcal E_0)$, respectively. It follows from the definition 
of the descendant of a $(k-1)$-multiblock consisting of a single letter 
of order $>(k-1)$ only, that in this case 
$\swa ij$
is a 
suboccurrence of $\varphi(\al_s)$, and it cannot be a prefix or a suffix of 
$\varphi(\al_s)$. 
Then $\swa ij$ is an origin, and its evolutional sequence number is 0.

Finally, we do induction on $m$. By Lemma \ref{klettersbecomeperiodic}, 
all letters of order $k$ in $\LA_{k,l}(\mathcal E_{l+m})$, 
in $\RA_{k,l}(\mathcal E_{l+m})$, and in $\CA_{k,0}(\mathcal E_m)$
are periodic. 
Then it follows from Remark \ref{descendantofperiodic} that
all $(k-1)$-blocks in $\LA_{k,l}(\mathcal E_{l+m+1})$, 
in $\RA_{k,l}(\mathcal E_{l+m+1})$, and in $\CA_{k,0}(\mathcal E_{m+1})$
are the descendants of 
$(k-1)$-blocks in $\LA_{k,l}(\mathcal E_{l+m})$, 
in $\RA_{k,l}(\mathcal E_{l+m})$, and in $\CA_{k,0}(\mathcal E_m)$, 
respectively. So, if 
$\swa ij$
is a $(k-1)$-block in 
$\LA_{k,l}(\mathcal E_{l+m+1})$, 
in $\RA_{k,l}(\mathcal E_{l+m+1})$, or in $\CA_{k,0}(\mathcal E_{m+1})$, 
then 
$\Pd_{k-1}(\swa ij)$
is a $(k-1)$-block in $\LA_{k,l}(\mathcal E_{l+m})$, 
in $\RA_{k,l}(\mathcal E_{l+m})$, or in $\CA_{k,0}(\mathcal E_m)$.
By induction hypothesis, 
the evolutional sequence number of $\Pd_{k-1}(\swa ij)$ is $m$ or $m-1$, 
so the evolutional sequence number of $\swa ij$ is $m+1$ or $m$.
\end{proof}

Note that a $(k-1)$-block with evolutional sequence number $m-1$ can appear in 
the $l$th atom of $\mathcal E_{l+m}$ only if there is a letter $b$ of order $k$ in
the $l$th atom of $\mathcal E_l$ whose image contains several letters of order $k$. 
In particular, $b$ must be preperiodic. In our example of an evolution of $2$-blocks, 
the $l$th atoms of blocks $\mathcal E_l$ contain preperiodic letters of order 1, but their
images contain only one letter of order 1, so in our example, 
each $(k-1)$-block the $l$th atom of $\mathcal E_{l+m}$ has evolutional sequence number $m$, 
not $m-1$.

\begin{corollary}\label{regpartstablekminusone}
If $\mathcal E_l$ is a stable $k$-block, where $k>1$, then all letters of 
order $k$ in $\LR_k(\mathcal E_l)$, in $\RR_k(\mathcal E_l)$ and in $\Cr_k(\mathcal E_l)$ 
are periodic, and all $(k-1)$-blocks in $\LR_k(\mathcal E_l)$, in $\RR_k(\mathcal E_l)$ 
and in $\Cr_k(\mathcal E_l)$ are stable.\qed
\end{corollary}

\begin{lemma}\label{caseoneexistsleft}
Let $\mathcal E$ be an evolution of $k$-blocks ($k>1$) such 
that Case I holds at the left. Let $l\ge 2$ and $m\ge 2$. 
There exists a $(k-1)$-block $\swa ij$ in $\LA_{k,l}(\mathcal E_{l+m})$
such that Case I holds at the left or at the right for the evolution of $\swa ij$.
\end{lemma}

\begin{proof}
By Lemma \ref{klettersbecomeperiodic}, there is at least one letter of order $k$
in $\LA_{k,l}(\mathcal E_{l+m})$, and all these letters of order $k$ are periodic.

Let us first assume that there are at least
two periodic letters of order $k$ in $\LA_{k,l}(\mathcal E_{l+m})$. Let $\al_s$ 
be the leftmost of these letters. Then by Remark \ref{allatomsimplestruct}, 
$\LA_{k,l}(\mathcal E_{l+m})$ contains a $(k-1)$-block of the form 
$\swa i{s-1}$ and a $(k-1)$-block of the form $\swa {s+1}j$. 
We have $\RB(\swa i{s-1})=\al_s$ and $\LB(\swa {s+1}j)=\al_s$.
By Lemma \ref{evolseqnolarge}, the evolutional sequence numbers of these blocks are 
at least 1, and the evolutional sequence numbers 
of the blocks $\Su_{k-1}(\swa i{s-1})$ and $\Su_{k-1}(\swa {s+1}j)$ are at least 2. 
Since $\al_s$ is a periodic letter of order $k$, $\varphi(\al_s)$
contains at least one letter of order $k-1$. If it is located to the left 
from (the unique) letter of order $k$ in $\varphi(\al_s)$, then
by
Remark \ref{atomsimplestructright} 
the $n$th right atom of $\Su_{k-1}(\swa i{s-1})$, where $n\ge 2$ is 
the evolutional sequence number of $\Su_{k-1}(\swa i{s-1})$, contains a letter of order $k-1$, 
and Case I holds at the right for the evolution of $\swa i{s-1}$.
Similarly, if there is a letter of order $k-1$ in $\varphi(\al_s)$
located to the right
from the letter of order $k$ in $\varphi(\al_s)$, then 
by Remark \ref{atomsimplestructleft}, Case I holds at the left for the evolution of
$\swa {s+1}j$.

Now suppose that there is exactly one periodic letter of order $k$ in 
$\LA_{k,l}(\mathcal E_{l+m})$.
By Corollary \ref{amountkletterssame}, 
$\LA_{k,l}(\mathcal E_{l+m})$ contains exactly one $(k-1)$-block, denote it by $\swa ij$. 
By Lemma \ref{klettersbecomeperiodic} and by 
Corollaries \ref{amountkletterssame} and \ref{amountkminusoneblockssame}, 
$\LA_{k,l}(\mathcal E_{l+m-1})$ also consists of one 
$(k-1)$-block and one periodic letter of order $k$. Moreover, the 
periodic letter of order $k$ in 
$\LA_{k,l}(\mathcal E_{l+m})=\Su_{k-1}(\LA_{k,l}(\mathcal E_{l+m-1}))$
is the descendant of the periodic letter of order $k$ in 
$\LA_{k,l}(\mathcal E_{l+m-1})$, so they coincide as abstract letters. 
Recall also that by Remark \ref{allatomsimplestruct}, 
these letters of order $k$ are the rightmost letters in the forgetful 
occurrences of $\LA_{k,l}(\mathcal E_{l+m-1})$ and of 
$\LA_{k,l}(\mathcal E_{l+m})$. By Corollary \ref{atomsperiodicitysmallleft},
the rightmost letter in $\Fg(\LA_{k,l+1}(\mathcal E_{l+m}))$
is the same letter of order $k$. Therefore, 
$\LB(\swa ij)$ and $\RB(\swa ij)$ coincide as abstract letters, 
denote this abstract letter by $b\in\E$. By Lemma \ref{evolseqnolarge}, 
the evolutional sequence number of $\swa ij$ is at least 1.
So, again, $\varphi(b)$ contains at least one letter of order $k-1$.
If it is located to the left (resp.\ to the right) of the unique 
occurrence of $b$ in $\varphi(b)$, then by Remark \ref{atomsimplestructright}
(resp.\ \ref{atomsimplestructleft}) and by the definition of 
Case I, Case I holds at the right (resp.\ at the left) 
for the evolution of $\Su_{k-1}(\swa ij)$, i.~e.\ for the evolution of 
$\swa ij$.
\end{proof}

\begin{lemma}\label{caseoneexistsright}
Let $\mathcal E$ be an evolution of $k$-blocks ($k>1$) such 
that Case I holds at the right. Let $l\ge 2$ and $m\ge 2$. 
There exists a $(k-1)$-block $\swa ij$ in $\RA_{k,l}(\mathcal E_{l+m})$
such that Case I holds at the left or at the right for the evolution of $\swa ij$.
\end{lemma}
\begin{proof}
The proof is completely similar to the proof of the previous lemma.
\end{proof}

\begin{lemma}\label{regpartabslargek}
Let $\mathcal E$ be an evolution of $k$-blocks such 
that Case I holds at the left (resp.\ at the right).
Let $\mathcal E_l$ be a stable $k$-block. Then 
$|\Fg(\LR_k(\mathcal E_l))|\ge 2\finmax$ (resp.\ $|\Fg(\RR_k(\mathcal E_l))|\ge 2\finmax$).
\end{lemma}
\begin{proof}
For $k=1$ we already know this by Lemma \ref{regpartabslargeone}.
For $k>1$ note first that 
$\LR_k(\mathcal E_l)=\LA_{k,l-3k+2}(\mathcal E_l)\LA_{k,l-3k+1}(\mathcal E_l)\ldots \LA_{k,2}(\mathcal E_l)$
(resp.\ $\RR_k(\mathcal E_l)=\RA_{k,2}(\mathcal E_l)\ldots\RA_{k,l-3k+1}(\mathcal E_l)\RA_{k,l-3k+2}(\mathcal E_l)$)
is nonempty since $l\ge 3k$, and $l-3k+2\ge 2$. Now 
the claim follows by induction on $k$
from Corollary \ref{regpartstablekminusone} and from Lemmas \ref{caseoneexistsleft}
and \ref{caseoneexistsright}.
\end{proof}

\begin{lemma}\label{regpartabsgrows}
Let $\mathcal E$ be an evolution of $k$-blocks such 
that Case I holds at the left (resp.\ at the right).
Let $\mathcal E_l$ be a stable $k$-block. Then 
$|\Fg(\LR_k(\mathcal E_{l+1}))|>|\Fg(\LR_k(\mathcal E_l))|$ 
(resp.\ $|\Fg(\RR_k(\mathcal E_{l+1}))|>|\Fg(\RR_k(\mathcal E_l))|$).
\end{lemma}
\begin{proof}
We can write
$\LR_k(\mathcal E_l)=\LA_{k,l-3k+2}(\mathcal E_l)\LA_{k,l-3k+1}(\mathcal E_l)\ldots \LA_{k,2}(\mathcal E_l)$
(resp.\ $\RR_k(\mathcal E_l)=\RA_{k,2}(\mathcal E_l)\ldots\RA_{k,l-3k+1}(\mathcal E_l)\RA_{k,l-3k+2}(\mathcal E_l)$)
and
$\LR_k(\mathcal E_{l+1})=\LA_{k,l+1-3k+2}(\mathcal E_{l+1})\LA_{k,l+1-3k+1}(\mathcal E_{l+1})\ldots \LA_{k,2}(\mathcal E_{l+1})$
(resp.\ $\RR_k(\mathcal E_{l+1})=\RA_{k,2}(\mathcal E_{l+1})\ldots\RA_{k,l+1-3k+1}(\mathcal E_{l+1})\RA_{k,l+1-3k+2}(\mathcal E_{l+1})$).
By Corollary \ref{atomsperiodicitybig}, 
$\Fg(\LA_{k,l-3k+2}(\mathcal E_l)\ldots \LA_{k,2}(\mathcal E_l))$
and 
$\Fg(\LA_{k,l+1-3k+2}(\mathcal E_{l+1})\ldots \LA_{k,3}(\mathcal E_{l+1}))$
(resp.\ $\Fg(\RA_{k,2}(\mathcal E_l)\ldots\RA_{k,l-3k+2}(\mathcal E_l))$
and
$\Fg(\RA_{k,3}(\mathcal E_{l+1})\ldots\RA_{k,l+1-3k+2}(\mathcal E_{l+1}))$)
coincide as abstract words, and 
$\LA_{k,2}(\mathcal E_{l+1})$ (resp.\ $\RA_{k,2}(\mathcal E_{l+1})$)
contains at least one letter of order $k$ since Case I holds at the left (resp.\ at the right).
\end{proof}

\begin{lemma}\label{kblocklengtho}
Let $\mathcal E$ be an evolution of $k$-blocks. Then $|\mathcal E_l|$ is $O(l^k)$ and 
$|\Fg(\Cr_k(\mathcal E_l))|$ is $O(l^{k-1})$ (for $l\to\infty$), 
and the constants in the $O$-notation depend on $\varphi$, $\E$, and $k$ only, but not on $\mathcal E$.
\end{lemma}

\begin{proof}
For $k=1$ we already know this by Corollary \ref{oneblocklengths}. For $k>1$, we do induction on $k$.
Without loss of generality, we may consider only the values of $l$ grater than or equal to $3k$.
By Remark \ref{preperiodfixed}, the lengths of the forgetful occurrences of 
$\LpreP_{k,l}(\mathcal E_l)$ and of $\RpreP_{k,l}(\mathcal E_l)$
are constants (they do not depend on $l$), and since the total number of 
different evolutions present in $\al$, understood as sequences as abstract 
words, is finite (Corollary \ref{finite-number-of-evolutions}), 
the lengths of all left and right preperiods of all $k$-blocks
are bounded by a single constant that depends on $\E$, $\varphi$ and $k$ only.

By Corollaries \ref{amountkminusoneblockssame} and \ref{amountkminusoneblockssamezero}, 
the amount of $(k-1)$-blocks in $\Cr_k(\mathcal E_l)$
does not depend on $l$, and using Corollary \ref{finite-number-of-evolutions}
again, we conclude that all amounts of $(k-1)$-blocks in $\Cr_k(\mathcal E_l)$
are bounded by a single constant that depends on $\E$, $\varphi$ and $k$ only.
By Lemma \ref{evolseqnolarge}, the evolutional sequence numbers of these $(k-1)$-blocks
can be $l$, $l-1$ or $l-2$ (since $\Cr_k(\mathcal E_l)$ is the concatenation of the zeroth
and the first atoms). Similarly, it follows from Corollary \ref{amountkletterssame} and from 
Corollary \ref{finite-number-of-evolutions} that the amount of letters of order $k$ in 
$\Cr_k(\mathcal E_l)$ is bounded by a single constant 
that depends on $\E$, $\varphi$ and $k$ only. Therefore, it follows from the induction hypothesis 
for $k-1$ that $|\Fg(\Cr_k(\mathcal E_l))|$ is $O(l^{k-1})$.

Let us count the total amount of $(k-1)$-blocks in $\LR_k(\mathcal E_l)$ and in 
$\RR_k(\mathcal E_l)$. By Corollary \ref{amountkminusoneblockssame}, 
the amount of $(k-1)$-blocks in the $n$th left atom, if this atom is inside $\LR_k(\mathcal E_l)$
(i.~e.\ if $2\le n\le l-3k+2$), does not depend on $l$ and $n$, denote this amount
by $x$. Similarly, denote by $y$ the 
amount of $(k-1)$-blocks in the $n$th right atom if this atom is inside $\RR_k(\mathcal E_l)$.
The total amount of $(k-1)$-blocks in $\LR_k(\mathcal E_l)$ and in 
$\RR_k(\mathcal E_l)$ is $(x+y)(l-3k+1)$. By Corollary \ref{amountkminusoneblockssame}, 
the total amount of letters of order $k$ in $\LR_k(\mathcal E_l)$ and in 
$\RR_k(\mathcal E_l)$ is also $(x+y)(l-3k+1)$. Using 
Corollary \ref{finite-number-of-evolutions} 
again, we can write this number as $O(l)$. By Lemma \ref{evolseqnolarge}, 
the evolutional sequence numbers of all $(k-1)$-blocks in 
$\LR_k(\mathcal E_l)$ and in 
$\RR_k(\mathcal E_l)$ are at most $l-2$.

Now observe that it follows from the definition of the descendant of a $(k-1)$-block
and from the fact that $\varphi$ is nonerasing that the 
length of a $(k-1)$-block is less than or equal to the length of its descendant.
Hence, it follows from the induction hypothesis for $k-1$ that if we have 
a $(k-1)$-block whose evolutional sequence number is at most $l$, then 
its length is bounded by a constant (that depends on $\E$, $\varphi$ and $k$ only) 
multiplied by $l^{k-1}$. Therefore, the total length of 
$\Fg(\LR_k(\mathcal E_l))$ and $\Fg(\RR_k(\mathcal E_l))$ 
is $O(l)O(l^{k-1})+O(l)=O(l^k)$. Finally, $|\mathcal E_l|$ is $O(1)+O(l^{k-1})+O(l^k)=O(l^k)$.
\end{proof}

\begin{lemma}\label{regpartasym}
Let $\mathcal E$ be an evolution of $k$-blocks.
If Case I holds at the left (resp.\ at the right) for $\mathcal E$, then 
$|\Fg(\LR_k(\mathcal E_l))|$ (resp.\ $|\Fg(\RR_k(\mathcal E_l))|$) is $\Theta(l^k)$.
If Case I holds for $\mathcal E$ at least at one side (at the left or at the right), 
then $|\mathcal E_l|$ is $\Theta(l^k)$.
The constants in the $\Theta$-notation here depend on $\varphi$, $\E$, and $k$ only, but not on $\mathcal E$.
\end{lemma}

\begin{proof}
For $k=1$ this is true by Corollary \ref{oneblocklengths}. For $k>1$ we are going to prove this by induction on $k$.
Since we already have Lemma \ref{kblocklengtho}, 
it is sufficient to prove that 
if Case I holds at the left (resp.\ at the right) for $\mathcal E$,
then 
$|\Fg(\LR_k(\mathcal E_l))|$ (resp.\ $|\Fg(\RR_k(\mathcal E_l))|$) is $\Omega(l^k)$.

By the induction hypothesis for $k-1$, there exist numbers $l_0\in\NN$ and $x\in\R_{>0}$
such that if the evolutional sequence number of a $(k-1)$-block is $l\ge l_0$ and 
Case I holds at the left or at the right for its evolution, then the 
length of this $(k-1)$-block is at least $xl^{k-1}$. Set $l_1=6k+2l_0+4$.
Suppose that we are considering a $k$-block $\mathcal E_l$ such that $l\ge l_1$.
Note first that $l-\lfloor l/2\rfloor\ge l/2\ge l_1/2\ge 3k$ and $\lfloor l/2\rfloor\ge\lfloor l_1/2\rfloor\ge 2$.
Hence, if Case I holds at the left (resp.\ at the right) for $\mathcal E$, 
then the concatenation of left atoms
$\LA_{k,\lfloor l/2\rfloor}(\mathcal E_l)\ldots \LA_{k,2}(\mathcal E_l)$
(resp.\ $\RA_{k,2}(\mathcal E_l)\ldots \RA_{k,\lfloor l/2\rfloor}(\mathcal E_l)$)
is contained in the left (resp.\ right) regular part of $\mathcal E_l$.
By Lemma \ref{evolseqnolarge}, the smallest possible evolutional sequence number
of a $(k-1)$-block contained in one of these atoms is 
$l-\lfloor l/2\rfloor-1\ge l/2-1\ge l_1/2-1=3k+l_0+1\ge l_0$.
By Lemma \ref{caseoneexistsleft} (resp.\ \ref{caseoneexistsright}), each
of these atoms contains at least one $(k-1)$-block such 
that Case I holds at the left or at the right for its evolution. 
So, we have at least $\lfloor l/2\rfloor-2+1$ such $(k-1)$-blocks in this 
concatenation of atoms, and by the induction hypothesis, each of these $(k-1)$-blocks
has length at least $x(l-\lfloor l/2\rfloor-1)^{k-1}$. Hence, 
the length of the forgetful occurrence of the whole left (resp.\ right) regular part is 
at least $x(l-\lfloor l/2\rfloor-1)^{k-1}(\lfloor l/2\rfloor-1)=\Omega(l^k)$.
\end{proof}

These two lemmas explain why letters of order $k$ were called letters of order $k$, not letter 
of order $k-1$. Despite the growth rates of individual letters of order $k$ (in the sense of 
their repeated images under $\varphi$) are $\Theta(l^{k-1})$, the "growth rates" of $k$-blocks 
(i.~e.\ occurrences consisting of letters of order at most $k$) in the sense of their 
superdescendants and evolutions are $O(l^k)$ and sometimes $\Theta(l^k)$.

Now we are going to define stable $k$-multiblocks and (prime and composite) kernels in 
stable $k$-multiblocks. Here we also allow $k=0$.
We call a $k$-multiblock \textit{stable} if 
it consists of periodic letters of order $k+1$ and (if $k\ge 1$) stable $k$-blocks only. 
In particular, an empty $k$-multiblock is always stable.
Note that 
it is not true in general that if $l$ is large enough, then the $l$th superdescendant of 
a $k$-multiblock is stable, namely, if a $k$-multiblock contains a letter of order $>k+1$, then 
its superdescendants never become stable. On the other hand, we can say the following:
\begin{remark}
If a $k$-multiblock $\al[\mathfrak x,i\ldots j,\mathfrak y]_k$, where $\mathfrak x,\mathfrak y\in\{<,>\}$, 
is stable, then
each letter of order $>k$ in $\Su_k([\mathfrak x,i\ldots j,\mathfrak y]_k)$ is 
periodic, has order $k+1$ and is the descendant of (more precisely, is the only 
letter of order $>k$ or $k$-block in the descendant of) a letter of order $k$ in 
$[\mathfrak x,i\ldots j,\mathfrak y]_k$.
If $\al[\mathfrak x,i\ldots j,\mathfrak y]_k$ is stable and $k\ge 1$, then
each $k$-block in $\Su_k(\al[\mathfrak x,i\ldots j,\mathfrak y]_k)$
is stable and is the descendant of a $k$-block in $[\mathfrak x,i\ldots j,\mathfrak y]_k$.
In other words, the operation of taking the descendant of a 
$k$-block or of a letter of order $>k$ establishes a bijection between the letters of order $>k$ 
and (if $k\ge 1$) the $k$-blocks in $[\mathfrak x,i\ldots j,\mathfrak y]_k$
and the letters of order $>k$ and (if $k\ge 1$) the $k$-blocks 
in $\Su_k([\mathfrak x,i\ldots j,\mathfrak y]_k)$.
In particular, $\Su_k([\mathfrak x,i\ldots j,\mathfrak y]_k)$ is also a stable 
$k$-multiblock.
\end{remark}
Note that if $k\ge 1$, then the requirement that all letters of order $>k$ in a stable $k$-multiblock 
are of order exactly $k+1$ and are periodic is essential in the sense that 
if we only know that all $k$-blocks in a given $k$-multiblock are stable, then 
it is not true in general that all $k$-blocks in its descendant are also stable. 
The descendants of letters of order $>(k+1)$ (in the sense of the definition of the descendant of 
a $k$-multiblock) or of preperiodic letters of order $k+1$ can contain $k$-blocks with 
evolutional sequence number 0 (i.~e.\ origins), and they are unstable.

In particular, a stable 0-multiblock is just an occurrence in $\al$ consisting of periodic letters of order 1 
only.

It will also be convenient for us now to introduce the notion of an evolution of 
stable $k$-multiblocks,
but 
we will not introduce ancestors and origins. Instead, we say the following:
A sequence of \textit{\textbf{stable}} 
$k$-multiblocks $\mathcal F_0,\mathcal F_1,\mathcal F_2,\ldots$
is called an \textit{evolution} if $\mathcal F_{l+1}=\Su_k(\mathcal F_l)$
for all $l\ge 0$. 
An evolution containing a given stable $k$-multiblock
always exists (for example, one can take this block as $\mathcal F_0$ 
and set $\mathcal F_l=\Su_k^l(\mathcal F_0)$ for $l>0$),
but is not necessarily unique, for example, 
if $\mathcal F$ is an evolution of $k$-multiblocks, then 
a $k$-multiblock $\mathcal F_l$ with $l>0$ 
is also contained in the following evolution $\mathcal F'$: $\mathcal F'_m=\mathcal F_{m+1}$
for all $m\ge 0$.

We call two evolutions $\mathcal F$ and $\mathcal F'$ of stable $k$-multiblocks 
\textit{consecutive}
if $\mathcal F_l$ and $\mathcal F'_l$ are consecutive $k$-multiblocks for all $l\ge 0$.
If $\mathcal F$ and $\mathcal F'$ are two consecutive evolutions of stable $k$-multiblocks, 
we call the evolution $\mathcal F''$, where $\mathcal F''_l$ is the concatenation 
of $\mathcal F_l$ and $\mathcal F'_l$, the \textit{concatenation} of $\mathcal F$
and $\mathcal F'$.

Now we define prime kernels in a stable $k$-multiblock ($k\ge 0$). 
They will be suboccurrences in the forgetful occurrence of the 
$k$-multiblock. More precisely, we are going to define them 
by induction on $k$. 
The only prime kernel of a stable 0-multiblock is just the 
0-multiblock itself (which is already an occurrence in $\al$.

Suppose that $k\ge 1$.
Let 
$\mathcal F_l$
be a stable $k$-multiblock. 
We say that a suboccurrence $\swa st$ in 
its forgetful occurrence 
$\Fg(\mathcal F_l)$
is a 
prime kernel of 
$\mathcal F_l$
if one of the following conditions holds:
\begin{enumerate}
\item $\swa st$ is the forgetful occurrence of the left (resp.\ right)
preperiod of a $k$-block in 
$\mathcal F_l$
such that 
Case I holds at the left (resp.\ at the right) for its evolution.
\item $\swa st$ is a single letter of order $k+1$.
\item There exists a 
$k$-block 
$\swa ij$
in 
$\mathcal F_l$
such that 
$\swa st$ is a prime kernel of 
$\Cr_k(\swa ij)$ 
(recall that the core of a 
$k$-block is by definition a $(k-1)$-multiblock, so its prime kernels are already defined 
by the induction hypothesis).
\end{enumerate}

\begin{lemma}\label{primekernelsdiffer}
Each suboccurrence of 
$\mathcal F_l$
is listed in this list at most once, and kernels of a $k$-multiblock do 
not overlap.
\end{lemma}
\begin{proof}
For $k=0$ the claim is trivial since there is only one kernel in each 0-multiblock.

If $k>0$, then a stable $k$-multiblock consists of periodic letters of order $k+1$ and stable $k$-blocks,
and prime kernels are either these letters of order $k+1$ themselves (and they are listed only once), 
or are suboccurrences of the $k$-blocks. Clearly, suboccurrences of the $k$-blocks cannot overlap with 
letters of order $k+1$. Suppose that an occurrence $\swa st$ is a prime kernel contained in a $k$-block 
$\swa ij$.
It can be listed in the list above as a left (resp.\ right) preperiod only if Case I holds at 
the left (resp.\ at the right) for the evolution of 
$\swa ij$.
But then the left (resp.\ right) 
regular part of 
$\swa ij$
is nonempty by Lemmas \ref{regpartabslargeone} and \ref{regpartabslargek},
so $\swa st$ cannot overlap or coincide with the prime kernels of 
$\Cr_k(\swa ij)$. 
By the induction hypothesis, the prime kernels of 
$\Cr_k(\swa ij)$ 
also do not overlap, and each of them is mentioned in the definition exactly once.
\end{proof}
Note that we do not claim (and this is not true in general) that prime kernels are nonempty 
occurrences. The most trivial counterexample is an empty 0-multiblock, i.~e.\ an empty occurrence 
in $\al$, then it is its prime kernel. A bit more general example is a 1-multiblock consisting 
of a single 1-block whose core is empty. This can happen if both borders of the origin 
of an evolution of 1-blocks are, for example, preperiodic letters of order 2, and their 
images consist of exactly one periodic letter of order 2 each. More generally, 
it can happen that the core of a $k$-block 
$\swa ij$, 
where $k>1$, consists of a single empty 
$(k-1)$-block, then using our definition of a prime kernel, we get by induction that 
this empty $(k-1)$-block is a prime kernel inside 
$\swa ij$. 
Moreover, if 
Case II holds, say, at the right for the evolution of 
$\swa ij$, 
and the $k$-multiblock
whose prime kernels we are defining contains the right border 
$\al_{j+1}$
of 
$\swa ij$ 
as well as 
$\swa ij$, 
then this empty prime kernel (which is 
$\swa{j+1}j$ 
in this case) 
and 
$\swa{j+1}{j+1}$ 
are both prime kernels. Thus, it is possible that two prime kernels are consecutive
occurrences in $\al$, 
and one of them is an empty occurrence. However, it is not possible (and the above lemma proves 
that this is not possible) that this empty 
occurrence is called a prime kernel twice, 
and this is guaranteed (in particular) by the fact that we call the forgetful occurrence of 
the right preperiod a prime kernel only if Case I holds at the right. Otherwise in the example
above, we would have called 
$\swa{j+1}j$ 
a prime kernel twice: first as a prime kernel 
of 
$\Cr_k(\swa ij)$ 
and second as the forgetful occurrence of 
$\RpreP_k(\swa ij)$.

One more remark we make about this definition is the following.
\begin{remark}\label{primekernelsexist}
An empty 0-multiblock (i.~e.\ an empty occurrence in $\al$) 
has a prime kernel, which is the 0-multiblock itself, 
while an empty $k$-multiblock for $k>0$ (which is really empty in the $k$-multiblock sense, 
i.~e.\ it does not contain any letters of order $>k$ and $k$-blocks, even empty ones) does not have 
any prime kernels according to this definition. However, if a $k$-multiblock, where $k>0$, 
consists of a single empty $k$-block (which must be stable, otherwise prime kernels are not defined anyway)
does have one prime kernel, which is this $k$-block itself.
\end{remark}

For example, let us find all prime kernels in the 2-multiblock consisting of a single stable 2-block in the 
example above.
The core of this 2-block consists the following 
1-blocks and letters of order 2:
\begin{enumerate}
\item A 1-block $\mathfrak{ee}ee..e\!f\mskip-4.5mu f\!..\mskip-1.5mu f\mskip-1.5mu\mathfrak{ff}$, 
where the amount of letters $e$ equals the amount of letters $f$ and is at least 10 if 
the ambient 2-block is stable.
\item A letter $d$ of order 2.
\item A 1-block $\mathfrak{ff}\mskip-1.5mu f\mskip-4.5mu f\!..\mskip-1.5mu f\!ee..ee\mathfrak{ee}$, 
where the amount of letters $e$ equals two plus the amount of letters $f$. There are at least 8 letters $f$ 
if the ambient 2-block is stable.
\item A letter $c$ of order 2.
\item A 1-block $\mathfrak{ee}ee..ee\mathfrak{ee}$, where $e$ is repeated an even 
number of times, and at least 18 times if the 2-block is stable.
\end{enumerate}
So, the prime kernels of the 2-multiblock are:
\begin{enumerate}
\item The left preperiod of the 1-block $\mathfrak{ee}ee..e\!f\mskip-4.5mu f\!..\mskip-1.5mu f\mskip-1.5mu\mathfrak{ff}$, 
which is $\mathfrak{ee}$.
\item The core of this 1-block, which is the occurrence $ee\!f\mskip-4.5mu f$ in the middle.
\item The right preperiod of this 1-block, which is $\mathfrak{ff}$.
\item The letter $d$ of order 2.
\item The left preperiod of the 1-block $\mathfrak{ff}\mskip-1.5mu f\mskip-4.5mu f\!..\mskip-1.5mu f\!ee..ee\mathfrak{ee}$, 
which is $\mathfrak{ff}$.
\item The core of this 1-block, which is the rightmost occurrence of the word $f\mskip-4.5mu f$ in this 1-block.
\item The right preperiod of this 1-block, which is $\mathfrak{ee}$.
\item The letter $c$ of order 2.
\item The left preperiod of the 1-block $\mathfrak{ee}ee..ee\mathfrak{ee}$, which is $\mathfrak{ee}$.
\item The core of this 1-block, which is the occurrence of the word $ee$ located exactly 
in the middle of this 1-block (recall that the number of letters $e$ in this 1-block is even).
\item The right preperiod of this 1-block, which is $\mathfrak{ee}$.
\item The right preperiod of the whole 2-block, which is $c\mathfrak{ee}eeeeee\mathfrak{ee}c\mathfrak{ee}ee\mathfrak{ee}c\mathfrak{eec}$.
\end{enumerate}

Now we define \textit{descendants} of prime kernels of a stable $k$-multiblock 
$\mathcal F_l$, where $\mathcal F$ is an evolution of stable $k$-multiblocks. 
In general, 
they are prime kernels of 
$\mathcal F_{l+1}=\Su_k(\mathcal F_l)$.
More precisely, if $k=0$, then the only prime kernel of 
$\mathcal F_l$
is 
$\mathcal F_l$ itself,
and we say that 
$\mathcal F_{l+1}=\varphi(\mathcal F_l)$ 
is the 
descendant of the prime kernel 
$\mathcal F_l$.
For $k>0$, we define the descendants of 
prime kernels by induction on $k$. Let $\swa st$ be a prime kernel of 
$\mathcal F_l$.
Consider three cases (we can do that since we know the statement of Lemma \ref{primekernelsdiffer}).

If $\swa st$ is the forgetful occurrence of the left (resp.\ right) preperiod of a $k$-block 
$\swa ij$
contained in 
$\mathcal F_l$
and such that Case I holds at the left (resp.\ at the right)
for its evolution, then we say that the descendant of the prime kernel $\swa st$ is the forgetful occurrence
of 
$\LpreP_k(\Su_k(\swa ij))$ 
(resp.\ of 
$\RpreP_k(\Su_k(\swa ij))$).

If $\swa st$ is a single periodic letter of order $k+1$, then its descendant in the sense of $k$-multiblocks, 
i.~e.\ the only periodic letter of order $k+1$ in $\varphi(\swa st)$, is called the descendant of $\swa st$ as 
a prime kernel.

Finally, if $\swa st$ is a prime kernel of the $(k-1)$-block 
$\Cr_k(\swa ij)$, 
where 
$\swa ij$
is a $k$-block inside 
$\mathcal F_l$,
then the descendant of 
$\swa st$ as of a prime kernel of a $(k-1)$-multiblock is already defined by the induction hypothesis
(and is a suboccurrence in 
$\Fg(\Su_{k-1}(\Cr_k(\swa ij)))=\Fg(\Cr_k(\Su_k(\swa ij)))$), 
and we say that the descendant of $\swa st$ as of a prime kernel of 
$\mathcal F_l$
is the same suboccurrence of 
$\Fg(\Su_{k-1}(\Cr_k(\swa ij)))$.

The descendant of a prime kernel $\swa st$ of a stable $k$-multiblock is denoted 
by $\kSu_k(\swa st)$. 

\begin{remark}\label{suprimekerbijection}
A trivial induction on $k$ shows that the operation of taking the descendant 
of a prime kernel establishes a bijection between the prime kernels of a stable $k$-multiblock and 
the prime kernels of its descendant.
\end{remark}

\begin{lemma}\label{borderkernelsleft}
Let $\al[\mathfrak x,i\ldots j,\mathfrak y]_k$, where $\mathfrak x,\mathfrak y\in\{<,>\}$,
be a stable $k$-multiblock. If $k>0$, suppose also that it is nonempty 
in the $k$-multiblock sense.
Then there exists a prime kernel of $\al[\mathfrak x,i\ldots j,\mathfrak y]_k$
of the form $\swa i{i'}$.

If $\Su_k(\al[\mathfrak x,i\ldots j,\mathfrak y]_k)=\al[\mathfrak x,s\ldots t,\mathfrak y]_k$, then 
$\kSu_k(\swa i{i'})$ is a prime kernel of the form $\swa s{s'}$.
\end{lemma}
\begin{proof}
For $k=0$ the claim follows directly from the definition. Suppose that $k>0$.
We use induction on $k$.

If the leftmost $k$-block or letter of order $k+1$ contained 
in $\al[\mathfrak x,i\ldots j,\mathfrak y]_k$ is actually 
a letter of order $k+1$, then this letter is $\al_i$, and it is a prime kernel itself.
If $\Su_k(\al[\mathfrak x,i\ldots j,\mathfrak y]_k)=\al[\mathfrak x,s\ldots t,\mathfrak y]_k$, 
then $\al_s$ is a periodic letter of order $k+1$,
and $\kSu_k(\al_i)=\al_s$.

If the leftmost $k$-block or letter of order $k+1$ contained 
in $\al[\mathfrak x,i\ldots j,\mathfrak y]_k$ is actually 
a $k$-block, denote it by $\swa i{i''}$. If Case I holds at the left 
for the evolution of $\swa i{i''}$, then the forgetful occurrence of $\LpreP_k(\swa i{i''})$
is a prime kernel, and it starts from position $i$ in $\al$.
Again, if $\Su_k(\al[\mathfrak x,i\ldots j,\mathfrak y]_k)=\al[\mathfrak x,s\ldots t,\mathfrak y]_k$, 
then $\kSu_k(\Fg(\LpreP_k(\swa i{i''})))=\Fg(\LpreP_k(\Su_k(\swa i{i''})))$ starts from 
position $s$ in $\al$.

If Case II holds at the left, then the left component of $\swa i{i''}$ 
is empty, and the core of $\swa i{i''}$ is a $(k-1)$-multiblock of the 
form $\al[<,i\ldots i''',>]_{k-1}$ for some $i'''$, and, by Remark \ref{zerothatomsimplestruct}, 
it is a non-empty $(k-1)$-multiblock if $k-1>0$. By the induction hypothesis, 
there exists a prime kernel of $\al[<,i\ldots i''',>]_{k-1}$
of the form $\swa i{i'}$. Then 
$\swa i{i'}$ is also a kernel of 
$\al[\mathfrak x,i\ldots j,\mathfrak y]_k$.
And again, if $\Su_k(\al[\mathfrak x,i\ldots j,\mathfrak y]_k)=\al[\mathfrak x,s\ldots t,\mathfrak y]_k$, 
then $\Su_{k-1}(\al[<,i\ldots i''',>]_{k-1})=\al[<,s\ldots s''',>]_{k-1}$ for some $s'''$, and 
by the induction hypothesis, $\kSu_{k-1}(\swa i{i'})$
is a prime kernel of $\al[<,s\ldots s''',>]_{k-1}$ of the form $\swa s{s'}$
for some $s'$. By the definition of the descendant of a prime kernel in this case, 
we also have $\kSu_k(\swa i{i'})=\swa s{s'}$.
\end{proof}

\begin{lemma}\label{borderkernelsright}
Let $\al[\mathfrak x,i\ldots j,\mathfrak y]_k$, where $\mathfrak x,\mathfrak y\in\{<,>\}$,
be a stable $k$-multiblock. If $k>0$, suppose also that it is nonempty 
in the $k$-multiblock sense.
Then there exists a prime kernel of $\al[\mathfrak x,i\ldots j,\mathfrak y]_k$
of the form $\swa{j'}j$.

If $\Su_k(\al[\mathfrak x,i\ldots j,\mathfrak y]_k)=\al[\mathfrak x,s\ldots t,\mathfrak y]_k$, then 
$\kSu_k(\swa {j'}j)$ is a prime kernel of the form $\swa {t'}t$.
\end{lemma}
\begin{proof}
The proof is completely symmetric to the proof of the previous lemma.
\end{proof}

\begin{lemma}\label{kernelfixed}
If $\swa st$ is a prime kernel of 
a stable $k$-multiblock 
$\mathcal F_l$,
then $\kSu_k(\swa st)$
coincides with $\swa st$ as an abstract word.
\end{lemma}

\begin{proof}
For $k=0$ this is true by the definitions of a periodic letter of order 1 and a weakly 1-periodic morphism.
If $k>0$, we again use induction on $k$. Consider the three cases from the definition of a prime kernel.

If there exists a $k$-block 
$\swa ij$ 
in 
$\mathcal F_l$
such that 
$\swa st=\Fg(\LpreP_k(\swa ij))$ 
(resp.\ 
$\swa st=\Fg(\RpreP_k(\swa ij))$) 
and Case I holds at the left (resp.\ at the right) for the evolution of 
$\swa ij$, 
then 
$\kSu_k(\swa st)=\Fg(\LpreP_k(\Su_k(\swa ij)))$ 
(resp.\ 
$\kSu_k(\swa st)=\Fg(\RpreP_k(\Su_k(\swa ij)))$). 
By Remark \ref{preperiodfixed},
$\Fg(\LpreP_k(\swa ij))$ and $\Fg(\LpreP_k(\Su_k(\swa ij)))$ 
(resp.\ 
$\Fg(\RpreP_k(\swa ij))$ and $\Fg(\RpreP_k(\Su_k(\swa ij)))$)
coincide as abstract words.

If $\swa st$ is a single periodic letter of order $k+1$, then by Remark \ref{descendantofperiodic}, 
$\kSu_k(\swa st)$
is also a single letter, and it coincides with $\swa st$ as an abstract letter.

Finally, if there is a $k$-block 
$\swa ij$ 
in 
$\mathcal F_l$
such that $\swa st$ is a kernel of 
$\Cr_k(\swa ij)$, 
then the claim follows from the induction hypothesis.
\end{proof}

\begin{lemma}\label{conskernelsurvives}
Let 
$\mathcal F_l$
be a stable $k$-multiblock, 
$\swa ij$ and $\swa st$ 
be its two prime kernels. 
Suppose that 
$\swa ij$ 
is located to the left from 
$\swa st$ 
and that there are 
no other prime kernels between 
$\swa ij$ and $\swa st$.

Then they are either consecutive occurrences in $\al$, or 
the occurrence 
$\swa{j+1}{s-1}$
between them is the 
forgetful occurrence of the (left or right) regular part of 
a stable $k'$-block $\swa uv$ ($1\le k'\le k$) such that Case I holds for its evolution 
at the left or at the right, respectively.

$\swa ij$ and $\swa st$ 
are consecutive occurrences in 
$\al$ if and only if 
$\kSu_k(\swa ij)$ and $\kSu_k(\swa st)$ 
are consecutive occurrences.
If 
$\swa{j+1}{s-1}=\Fg(\LR_{k'}(\swa uv))$ 
(resp.\ 
$\swa{j+1}{s-1}=\Fg(\RR_{k'}(\swa uv))$)
and Case I holds at the left (resp.\ at the right) for the evolution of $\swa uv$, then 
the occurrence between 
$\kSu_k(\swa ij)$ and $\kSu_k(\swa st)$ 
is $\Fg(\LR_{k'}(\Su_{k'}(\swa uv)))$ (resp.\ $\Fg(\RR_{k'}(\Su_{k'}(\swa uv)))$).
\end{lemma}
\begin{proof}
For $k=0$ the statement is clear since each 0-multiblock has only one prime kernel.
For $k>0$, we prove the statement by induction on $k$.

$\swa ij$ and $\swa st$ 
cannot be occurrences in two different $k$-blocks, otherwise there would be 
a letter of order $k+1$ between these two $k$-blocks, and this letter of order $k+1$ would also be located between 
$\swa ij$ and $\swa st$.
So, there are two possible cases: either 
$\swa ij$ and $\swa st$ 
are both 
occurrences in the same $k$-block, or one of these occurrences is located in a $k$-block, and the other 
is the left or the right border of this $k$-block.

Suppose that 
$\swa ij$
is an occurrence in a $k$-block 
$\swa{i'}{j'}$, 
and 
$\swa st$
is the right border of $\swa{i'}{j'}$, i.~e.\ 
$s=t=j'$. 
By Lemma \ref{borderkernelsright}, 
the $k$-multiblock consisting of the $k$-block $\swa{i'}{j'}$ has a prime kernel of the form 
$\swa{j''}{j'}$, i.~e.\ 
$i=j''$ and $j=j'$. 
Then 
$\swa ij$ and $\swa st$ 
are consecutive, 
and 
$\kSu_k(\swa{j''}{j'})$ and $\kSu_k(\swa{j'}{j'})=\RB(\Su_k(\swa{i'}{j'}))$ are consecutive 
by the second part of Lemma \ref{borderkernelsright}.

The case when 
$\swa st$
is an occurrence in a $k$-block $\swa{i'}{j'}$, and 
$\swa ij$ 
is the left border of $\swa{i'}{j'}$ is similar to the previous one. In this case, 
we have 
$i=j=i'$, 
and, by Lemma \ref{borderkernelsleft}, 
$\swa st$ 
starts from position $i'$ in 
$\al$, i.~e.\ 
$s=i'$. 
Then 
$\swa ij$ and $\swa st$ 
are consecutive occurrences in $\al$, 
and, by the second part of Lemma \ref{borderkernelsleft}, 
$\kSu_k(\swa ij)=\RB(\Su_k(\swa{i'}{j'}))$ and $\kSu_k(\swa st)$
are also consecutive.

Suppose now that 
$\swa ij$ and $\swa st$ 
are both occurrences in a $k$-block $\swa{i'}{j'}$. Again, there are several 
possibilities:

First, it is possible that Case I holds at the left for the evolution of $\swa{i'}{j'}$, and 
$\swa ij=\Fg(\LpreP_k(\swa{i'}{j'}))$.
Then 
$\swa st$ 
is a prime kernel of $\Cr_k(\swa{i'}{j'})$ 
(recall that there are no prime kernels between 
$\swa ij$ and $\swa st$,
that 
$\Cr_k(\swa{i'}{j'})$ is a nonempty $(k-1)$-multiblock if $k-1>0$ by Remark \ref{zerothatomsimplestruct}, 
so by Remark \ref{primekernelsexist}, $\Cr_k(\swa{i'}{j'})$ has at least one prime kernel).
By Lemma \ref{borderkernelsleft}, then the forgetful occurrence of $\Cr_k(\swa{i'}{j'})$
starts from position 
$s$
in $\al$, so 
$\swa{j+1}{s-1}=\Fg(\LR_k(\swa {i'}{j'}))$.
We also have 
$\kSu_k(\swa ij)=\Fg(\LR_k(\swa{i'}{j'}))$,
and 
$\kSu_k(\swa st)$ 
is the leftmost prime kernel of $\Cr_k(\Su_k(\swa{i'}{j'}))$, so 
the occurrence between 
$\kSu_k(\swa ij)$ and $\kSu_k(\swa st)$
is $\Fg(\LR_k(\Su_k(\swa {i'}{j'})))$.

Second, it is possible that Case I holds at the right for the evolution of $\swa{i'}{j'}$, 
and 
$\swa st=\Fg(\RR_k(\swa{i'}{j'}))$, 
but this case is completely 
symmetric to the previous one.

The remaining possibility is that both 
$\swa ij$ and $\swa st$
are prime kernels of 
$\Cr_k(\swa{i'}{j'})$, which is a $(k-1)$-multiblock, 
but then the claim follows from the induction hypothesis.
\end{proof}

This lemma (together with Lemma \ref{regpartabslargek}, which implies that 
the forgetful occurrence of the left or right regular part is nonempty if Case I holds at the left or at the right, 
respectively) enables us to define \textit{composite kernels} of stable $k$-multiblocks as 
maximal (by inclusion) concatenations of consecutive prime kernels. In other words, 
if 
$\mathcal F_l$
is a stable $k$-multiblock, then an occurrence $\swa st$ is called a composite kernel, 
if it is a concatenation of consecutive prime kernels of 
$\mathcal F_l$,
and letters $\al_{s-1}$ (if $s>0$) and $\al_{t+1}$ do not belong to any prime kernels
of 
$\mathcal F_l$.
(Empty occurrences are allowed by this definition, so, if $\swa s{s-1}$ is an empty prime kernel, 
and letters $\al_{s-1}$ and $\al_s$ do not belong to any prime kernel, then $\swa s{s-1}$ is also a composite kernel).
If 
$\mathcal F_l$
is a stable $k$-multiblock, we can write its composite kernels in a 
list, as they occur in 
$\Fg(\mathcal F_l)$
from the right to the left. 
Denote the number of these composite kernels by 
$\nker_k(\mathcal F_l)$
We refer to the elements of this list as to the first, the second, \textellipsis, 
the 
$\nker_k(\mathcal F_l)$th
composite 
kernel of 
$\mathcal F_l$,
and denote them by 
$\Ker_{k,1}(\mathcal F_l), \Ker_{k,2}(\mathcal F_l), 
\ldots, \Ker_{k,\nker_k(\mathcal F_l)}(\mathcal F_l)$.

We also can define the 
\textit{descendant} of a composite kernel as the concatenation of the 
descendants of all prime kernels inside this composite kernel, they are consecutive 
by Lemma \ref{conskernelsurvives}. In other words, 
if $\Ker_{k,m}(\mathcal F_l)=\swa{s_1}{t_1}\ldots \swa{s_n}{t_n}$, where $1\le m\le \nker_k(\mathcal F_l)$ 
and $\swa{s_i}{t_i}$ is a prime 
kernel for $1\le i\le n$,
then we say that the descendant of $\mathcal F_l$ is 
$\kSu_k(\swa{s_1}{t_1})\ldots \kSu_k(\swa{s_n}{t_n})$, these descendants are consecutive 
by Lemma \ref{conskernelsurvives}.
Lemma \ref{kernelfixed} guarantees that empty prime kernels do not lead to any ambiguity in the notation 
here since their descendants are also empty.
Denote the descendant of a 
composite kernel $\Ker_{k,m}(\mathcal F_l)$ by $\kSu_k(\Ker_{k,m}(\mathcal F_l))$.

So, now we have split each stable $k$-multiblock into a concatenation of alternating 
(possibly empty) composite kernels and (nonempty) forgetful occurrences of 
left or right regular parts of $k'$-blocks ($1\le k'\le k$) such that Case I 
holds for their evolutions at the left or at the right, respectively.

We are going to call these left or right regular parts, as well as some concatenations, 
the inner pseudoregular parts of the $k$-multiblock. Here is the precise definition:
Let $\mathcal F_l$ be a stable $k$-multiblock. Suppose that $\Fg(\mathcal F_l)=\swa ij$.
First, if $1\le m<\nker_k(\mathcal F_l)$, $\Ker_{k,m}(\mathcal F_l)=\swa st$ and 
$\Ker_{k,m+1}(\mathcal F_l)=\swa{s'}{t'}$, then we call the occurrence 
$\swa{t+1}{s'-1}$ between these two composite kernels the \textit{$(m,m+1)$th inner pseudoregular
part} of $\mathcal F_l$ and denote it by $\IpR_{k,m,m+1}(\mathcal F_l)$. Second, we say that the \textit{$(0,1)$th} 
(resp.\ the \textit{$(\nker_k(\mathcal F_l),\nker_k(\mathcal F_l)+1)$th}) \textit{inner pseudoregular part}
of $\mathcal F_l$ is the empty occurrence $\swa i{i-1}$ (resp.\ $\swa {j+1}j$) at the beginning
(resp.\ at the end) of the forgetful occurrence of $\mathcal F_l$. Denote it by $\IpR_{k,0,1}(\mathcal F_l)$
(resp.\ by $\IpR_{k,\nker_k(\mathcal F_l),\nker_k(\mathcal F_l)+1}$).
Finally, choose indices $m$ and $m'$ so that $0\le m<m'\le \nker_k(\mathcal F_l)+1$.
We call the concatenation $\IpR_{k,m,m+1}(\mathcal F_l)\Ker_{k,m+1}(\mathcal F_l)\ldots 
\Ker_{k,m'-1}(\mathcal F_l)\IpR_{k,m'-1,m'}(\mathcal F_l)$ the 
\textit{$(m,m')$th inner pseudoregular part} of $\mathcal F_l$ and denote it by 
$\IpR_{k,m,m'}(\mathcal F_l)$. (By Lemmas \ref{borderkernelsleft} and \ref{borderkernelsright}, 
these words are really consecutive even if $m=0$ or $m'=\nker_k(\mathcal F_l)+1$.)
In particular, $\IpR_{k,0,\nker_k(\mathcal F)+1}(\mathcal F_l)=\Fg(\mathcal F_l)$.
If $\Ker_{k,1}(\mathcal F_l)$ (resp.\ $\Ker_{k,\nker_k(\mathcal F_l)}(\mathcal F_l)$)
is an empty occurrence, then it coincides with $\IpR_{k,0,1}(\mathcal F_l)$
(resp.\ with $\IpR_{k,\nker_k(\mathcal F_l),\nker_k(\mathcal F_l)+1}$)
as an occurrence in $\al$, and
in this case $\IpR_{k,0,m'}(\mathcal F_l)=\IpR_{k,1,m'}(\mathcal F_l)$
(resp.\ $\IpR_{k,m,\nker_k(\mathcal F_l)}(\mathcal F_l)=\IpR_{k,m,\nker_k(\mathcal F_l)+1}(\mathcal F_l)$)
for $1<m'\le\nker_k(\mathcal F_l)+1$ (resp.\ for $0\le m < \nker_k(\mathcal F_l)$) 
as an occurrence in $\al$.

For example, let us list the composite kernels and the regular parts between them for the 2-multiblock
consisting of a single 2-block from the example above. We have already listed its prime kernels,
and its composite kernels and regular parts between them are:

\begin{enumerate}
\item The first prime kernel from the list above, which is $\mathfrak{ee}$.
\item The left regular part of $\mathfrak{ee}ee..e\!f\mskip-4.5mu f\!..\mskip-1.5mu f\mskip-1.5mu\mathfrak{ff}$, 
which is $ee..e$.
\item The second prime kernel from the list above, which is $ee\!f\mskip-4.5mu f$.
\item The left regular part of $\mathfrak{ee}ee..e\!f\mskip-4.5mu f\!..\mskip-1.5mu f\mskip-1.5mu\mathfrak{ff}$, 
which is $f\mskip-4.5mu f\!..\mskip-1.5mu f$.
\item The 
concatenation of the third, fourth and fifth prime kernels, which is $\mathfrak{ff}d\mathfrak{ff}$.
\item The left regular part of $\mathfrak{ff}\mskip-1.5mu f\mskip-4.5mu f\!..\mskip-1.5mu f\!ee..ee\mathfrak{ee}$, 
which is $f\mskip-4.5mu f\!..\mskip-1.5mu f$.
\item The sixth element of the list above, $f\mskip-4.5mu f$.
\item The right regular part of $\mathfrak{ff}\mskip-1.5mu f\mskip-4.5mu f\!..\mskip-1.5mu f\!ee..ee\mathfrak{ee}$, 
which is $ee..e$.
\item The 
concatenation of the seventh, eighth and ninth prime kernels, $\mathfrak{ee}c\mathfrak{ee}$.
\item The left regular part of $\mathfrak{ee}ee..ee\mathfrak{ee}$, 
which is $ee..e$.
\item The tenth prime kernel, $ee$.
\item The right regular part of $\mathfrak{ee}ee..ee\mathfrak{ee}$, 
which is $ee..e$.
\item The eleventh prime kernel, $\mathfrak{ee}$.
\item The right regular part of the whole 2-block, an occurrence of the form 
$c\mathfrak{ee}ee..ee\mathfrak{ee}c\mathfrak{ee}ee..ee\mathfrak{ee}c \ldots 
c\mathfrak{ee}ee..ee\mathfrak{ee}c\ldots ee\mathfrak{ee}c\mathfrak{ee}eeeeeeeeee\mathfrak{ee}$.
\item The twelfth prime kernel, $c\mathfrak{ee}eeeeee\mathfrak{ee}c\mathfrak{ee}ee\mathfrak{ee}c\mathfrak{eec}$.
\end{enumerate}

We already know (Remark \ref{suprimekerbijection}) that the operation of taking the descendant 
of a prime kernel establishes a bijection between the prime kernels of a stable $k$-multiblock and 
the prime kernels of its descendant. The same is true for composite kernels:
\begin{remark}\label{sucompkerbijection}
The operation of taking the descendant 
of a composite kernel establishes a bijection between the composite kernels of a stable $k$-multiblock and 
the composite kernels of its descendant.

In other words, 
if $\mathcal F$ is an evolution of stable $k$-multiblocks, then 
$\nker_k(\mathcal F_l)$ does not depend on $l$
and $\Ker_{k,m}(\mathcal F_{l+1})=\kSu_k(\Ker_{k,m}(\mathcal F_l))$ 
for $l\ge 0$ and for $1\le m\le\nker_k(\mathcal F_l)$.
\end{remark}

So if $\mathcal F$ is an evolution of stable $k$-multiblocks, we can denote the number 
$\nker_k(\mathcal F_l)$ for arbitrary $l$ by $\nker_k(\mathcal F)$.

\begin{lemma}\label{compkernelfixed}
If $\mathcal F$ is an evolution of stable $k$-multiblocks, then 
$\Ker_{k,m}(\mathcal F_l)$ does not depend on $l$
as an abstract word 
for $1\le m\le\nker_k(\mathcal F)$.
\end{lemma}

\begin{proof}
This follows directly from Lemma \ref{kernelfixed} and the definitions of a composite kernel and its descendant.
\end{proof}

So, we can denote the abstract word $\Ker_{k,m}(\mathcal F_l)$
for arbitrary $l\ge 0$ by $\Ker_{k,m}(\mathcal F)$
and call it \textit{the $m$th composite kernel of $\mathcal F$}. 
The number $\nker_k(\mathcal F)$ is then called 
\textit{the number of composite kernels of $\mathcal F$}.

\begin{lemma}\label{betweenkernelsreg}
Let $\mathcal F$ be an evolution of stable $k$-multiblocks, $k>0$. Let $1\le m<\nker_k(\mathcal F)$.

Then there exist an evolution of $k'$-blocks $\mathcal E$ ($1\le k'\le k$) and a number $l_0\ge 3k'$ such that 
one of the following statement holds:
\begin{enumerate}
\item Case I holds for $\mathcal E$ at the left, and for all $l\ge 0$
we have $\IpR_{k,m,m+1}(\mathcal F_l)=\Fg(\LR_{k'}(\mathcal E_{l+l_0}))$.
\item Case I holds for $\mathcal E$ at the right, and for all $l\ge 0$ 
we have $\IpR_{k,m,m+1}(\mathcal F_l)=\Fg(\RR_{k'}(\mathcal E_{l+l_0}))$.
\end{enumerate}
\end{lemma}

\begin{proof}
This follows directly from the last statement of Lemma \ref{conskernelsurvives}.
\end{proof}

\begin{lemma}\label{betweenkernelgrows}
Let $\mathcal F$ be an evolution of stable $k$-multiblocks, $k>0$. Let $1\le m<m'\le\nker_k(\mathcal F)$.
For each $l\ge 0$, denote 
$n_l=|\IpR_{k,m,m'}(\mathcal F_l)|$.

Then $n_l\ge 2\finmax$, $n_l$ strictly grows as $l$ grows, and there exists $k'$ ($1\le k'\le k$) such that 
$n_l$ is $\Theta (l^{k'})$ for $l\to\infty$.
\end{lemma}

\begin{proof}
By the previous lemma, there exist numbers $k_m, \ldots, k_{m'-1}$ ($1\le k_i\le k$), evolutions 
$\mathcal E^{(m)},\ldots \mathcal E^{(m'-1)}$ ($\mathcal E^{(i)}$ is an evolution of 
$k_i$-blocks), and numbers $l_m,\ldots, l_{m'-1}$ ($l_i\ge 3k_i$) such that 
for each $i$ ($m\le i<m'$) one of the following holds:
\begin{enumerate}
\item Case I holds for $\mathcal E^{(i)}$ at the left, and for each $l\ge 0$
we have $\IpR_{k,i,i+1}(\mathcal F_l)=\Fg(\LR_{k_i}(\mathcal E^{(i)}_{l+l_i}))$.
\item Case I holds for $\mathcal E^{(i)}$ at the right, and for each $l\ge 0$
we have $\IpR_{k,i,i+1}(\mathcal F_l)=\Fg(\RR_{k_i}(\mathcal E^{(i)}_{l+l_i}))$.
\end{enumerate}
By Lemmas \ref{regpartabslargek}, \ref{regpartabsgrows}, and \ref{regpartasym}, 
we see that in each of these cases, 
$|\IpR_{k,i,i+1}(\mathcal F_l)|\ge 2\finmax$, that $|\IpR_{k,i,i+1}(\mathcal F_l)|$
strictly grows as $l$ grows, and that
$|\IpR_{k,i,i+1}(\mathcal F_l)|$ is 
$\Theta((l+l_i)^{k_i})=\Theta(l^{k_i})$ for $l\to\infty$.

Now, $n_l$ is the sum of $m'-m-1\ge 0$ summands $|\Ker_{k,i}(\mathcal F)|$ that do not depend of 
$l$, and of $m'-m\ge 1$ summands
$|\IpR_{k,i,i+1}(\mathcal F_l)|$.
Therefore, $n_l\ge 2\finmax$, and $n_l$ strictly grows as $l$ grows. The asymptotic
of $n_l$ for $l\to\infty$ is 
$$
\sum_{i=m+1}^{m'-1} |\Ker\nolimits_{k,i}(\mathcal F)|+\sum_{i=m}^{m'-1}\Theta(l^{k_i})=\Theta(l^{k'}),
$$
where $k'=\max(k_m,\ldots,k_{m'-1})$.
\end{proof}

%
%

\begin{corollary}\label{betweenkernelgrowsgood}
Let $\mathcal F$ be an evolution of stable $k$-multiblocks, $k>0$. Suppose that $\nker_k(\mathcal F)>1$.
Let $m,m'\in\ZZ$ be two indices such that $0\le m<\nker_k(\mathcal F)$, $m<m'$ and $1<m'\le\nker_k(\mathcal F)+1$.
For each $l\ge 0$, denote 
$n_l=|\IpR_{k,m,m'}(\mathcal F_l)|$. (In particular, if $m=0$ and $m'=\nker_k(\mathcal F)+1>2$, then $n_l=|\Fg(\mathcal F_l)|$.)

Then $n_l\ge 2\finmax$, $n_l$ strictly grows as $l$ grows, and there exists $k'$ ($1\le k'\le k$) such that 
$n_l$ is $\Theta (l^{k'})$ for $l\to\infty$.
\end{corollary}

\begin{proof}
The only cases in the statement of this corollary not covered by the previous lemma are the cases when 
$m=0$ or $m'=\nker_k(\mathcal F)+1$.

If $m=0$ and $m'<\nker_k(\mathcal F)+1$, then denote 
$n'_l=|\IpR_{k,1,m'}(\mathcal F_l)|$ (recall that we assume that $m'>1$). By the previous lemma, 
$n'_l\ge 2\finmax$, $n'_l$ strictly grows as $l$ grows, and there exists $k'$ ($1\le k'\le k$) such that 
$n'_l$ is $\Theta (l^{k'})$ for $l\to\infty$. But $\IpR_{k,0,1}(\mathcal F_l)$ is always an empty occurrence, 
so $\IpR_{k,0,m'}(\mathcal F_l)=\Ker_{k,1}(\mathcal F_l)\IpR_{k,1,m'}(\mathcal F_l)$, and 
$n_l=|\Ker_{k,1}(\mathcal F)|+n'_l$, and the first summand does not depend on $l$. 
Hence, $n_l\ge 2\finmax+|\Ker_{k,1}(\mathcal F)|\ge 2\finmax$, $n_l$ strictly grows as $l$ grows, 
and $n_l$ is $\Theta (l^{k'})+|\Ker_{k,1}(\mathcal F)|=\Theta (l^{k'})$ 
for $l\to\infty$.

The case when $m>0$, but $m'=\nker_k(\mathcal F)+1$ is completely analogous.

Finally, if $m=0$ and $m'=\nker_k(\mathcal F)+1$, then, since $\nker_k(\mathcal F)>1$ by assumption, 
we can apply the previous lemma to $n'_l=|\IpR_{k,1,\nker_k(\mathcal F)}(\mathcal F_l)|$. 
Again, $n'_l\ge 2\finmax$, $n'_l$ strictly grows as $l$ grows, and there exists $k'$ ($1\le k'\le k$) such that 
$n'_l$ is $\Theta (l^{k'})$ for $l\to\infty$. And again, since 
$\IpR_{k,0,1}(\mathcal F_l)$ and $\IpR_{k,\nker_k(\mathcal F),\nker_k(\mathcal F)+1}(\mathcal F_l)$ 
are always empty occurrences, 
we have $n_l=|\Ker_{k,1}(\mathcal F)|+n'_l+|\Ker_{k,\nker_k(\mathcal F)}(\mathcal F)|$.
The first and the last summands do not depend on $l$, so 
$n_l\ge 2\finmax+|\Ker_{k,1}(\mathcal F)|+|\Ker_{k,\nker_k(\mathcal F)}(\mathcal F)|\ge 2\finmax$, $n_l$ strictly grows as $l$ grows, 
and $n_l$ is $\Theta (l^{k'})+|\Ker_{k,1}(\mathcal F)|+|\Ker_{k,\nker_k(\mathcal F)}(\mathcal F)|=\Theta (l^{k'})$ 
for $l\to\infty$.
\end{proof}

The following lemma shows how composite kernels and 
inner pseudoregular parts behave for concatenations of consecutive evolutions.

\begin{lemma}\label{concatcompker}
Let $\mathcal F$ and $\mathcal F'$ be two consecutive evolutions of \textbf{nonempty}
stable $k$-multiblocks ($k\ge 0$), and let $\mathcal F''$ be the concatenation of 
$\mathcal F$ and $\mathcal F'$. Then:
\begin{enumerate}
\item $\nker_k(\mathcal F'')=\nker_k(\mathcal F)+\nker_k(\mathcal F')-1$
\item If $1\le m< \nker_k(\mathcal F)$, then $\Ker_{k,m}(\mathcal F''_l)=\Ker_{k,m}(\mathcal F_l)$ for all $l\ge 0$
as an occurrence in $\al$ and $\Ker_{k,m}(\mathcal F'')=\Ker_{k,m}(\mathcal F)$ as an 
abstract word. If $m=\nker_k(\mathcal F)$, then $\Ker_{k,m}(\mathcal F''_l)=\Ker_{k,m}(\mathcal F_l)\Ker_{k,1}(\mathcal F'_l)$ for all $l\ge 0$
as an occurrence in $\al$ and $\Ker_{k,m}(\mathcal F'')=\Ker_{k,m}(\mathcal F)\Ker_{k,1}(\mathcal F')$
as an abstract word. If $\nker_k(\mathcal F)<m\le \nker_k(\mathcal F'')$, then 
$\Ker_{k,m}(\mathcal F''_l)=\Ker_{k,m-\nker_k(\mathcal F)+1}(\mathcal F'_l)$ for all $l\ge 0$
as an occurrence in $\al$ and 
$\Ker_{k,m}(\mathcal F'')=\Ker_{k,m-\nker_k(\mathcal F)+1}(\mathcal F')$ 
as an abstract word.
\item If $0\le m<m'\le\nker_k(\mathcal F)$, then $\IpR_{k,m,m'}(\mathcal F'')=\IpR_{k,m,m'}(\mathcal F)$
for all $l\ge 0$
as an occurrence in $\al$.
If $\nker_k(\mathcal F)\le m<m'\le \nker_k(\mathcal F'')+1$, then 
$\IpR_{k,m,m'}(\mathcal F'')=\IpR_{k,m-\nker_k(\mathcal F)+1,m'-\nker_k(\mathcal F)+1}(\mathcal F')$
for all $l\ge 0$
as an occurrence in $\al$.
If $0\le m<\nker_k(\mathcal F)<m'\le \nker_k(\mathcal F'')+1$, then 
$\IpR_{k,m,m'}(\mathcal F'')=\IpR_{k,m,\nker_k(\mathcal F)+1}(\mathcal F)\IpR_{k,0,m'-\nker_k(\mathcal F)+1}(\mathcal F')$
for all $l\ge 0$
as an occurrence in $\al$.
\end{enumerate}
\end{lemma}

\begin{proof}
The first two claims follow directly from Lemmas \ref{borderkernelsleft} and \ref{borderkernelsright} and the definition of a composite central kernel.
The third claim follows from the first two claims and the definition of an inner pseudoregular part.
\end{proof}

A particular case of prime and composite kernels will be especially important for us.
If $\swa ij$ is a stable $k$-block ($k\ge 1$), we call the prime (resp.\ composite) 
kernels of $\Cr_k(\swa ij)$ the \textit{prime (resp.\ composite) central kernels} of $\swa ij$.
Denote the number of the composite central kernels of a stable $k$-block $\swa ij$ by 
$\ncker_k(\swa ij)$.
Observe that this definition coincides with the definition we gave in the previous 
section for 1-blocks. Again, we can write the composite central kernels of 
$\swa ij$ in a list as they occur in $\mathcal E_l$, from the 
left to the right. We call the elements of this list the first, the second, \textellipsis, 
the $\ncker_k(\swa ij)$th composite central kernel of $\swa ij$ and denote them 
by $\cKer_{k,1}(\swa ij), \cKer_{k,2}(\swa ij), \ldots, \cKer_{k,\ncker_k(\swa ij)}(\swa ij)$.
It follows from Lemma \ref{borderkernelsleft} and from Remark \ref{zerothatomsimplestruct} that $\ncker_k(\swa ij)\ge 1$.
Note that, for example, the first composite central kernel of $\swa ij$ can be the first 
or the second composite kernel of the $k$-multiblock consisting of the $k$-block $\swa ij$ 
only, depending on whether Case II or Case I holds for the evolution of $\swa ij$ at the left.

Let $\mathcal E$ is an evolution of $k$-blocks, and let $\mathcal E_l$ be a stable $k$-block.
We have defined the descendants of 
these composite central kernels,
and they are composite central kernels of $\mathcal E_{l+1}$.
By Remark \ref{sucompkerbijection}, 
$\ncker_k(\mathcal E_l)=\nker_{k-1}(\Cr_k(\mathcal E_l))=\nker_{k-1}(\Cr_k(\mathcal E_{l+1}))=\ncker_k(\mathcal E_{l+1})$, and
by Lemma \ref{compkernelfixed}, 
$\cKer_{k,m}(\mathcal E_{l+1})$ is the same abstract word as $\cKer_{k,m}(\mathcal E_l)$
for $1\le m\le \ncker_k(\mathcal E_l)$. In other words, the number
$\ncker_k(\mathcal E_l)$ and the abstract words 
$\cKer_{k,1}(\mathcal E_l), \cKer_{k,2}(\mathcal E_l), \ldots, \cKer_{k,\mathcal E_l}(\mathcal E_l)$
do not depend on $l$ if $l\ge 3k$. 
We call these abstract words \textit{the composite central kernels of $\mathcal E$},
denote the number of them by $\ncker_k(\mathcal E)$, and 
denote the composite central kernels of $\mathcal E$ 
themselves by $\cKer_{k,1}(\mathcal E), \cKer_{k,2}(\mathcal E), \ldots, \cKer_{k,\mathcal E}(\mathcal E)$.

In the example above, the composite central kernels of the evolution of 2-blocks are: 
$\mathfrak{ee}$, $ee\!f\mskip-4.5mu f$, $\mathfrak{ff}d\mathfrak{ff}$,
$f\mskip-4.5mu f$, $\mathfrak{ee}c\mathfrak{ee}$, $ee$, $\mathfrak{ee}$.

The structure of a 2-block in a bit more general case is shown by Fig. \ref{twostruct}.

\begin{figure}[!h]
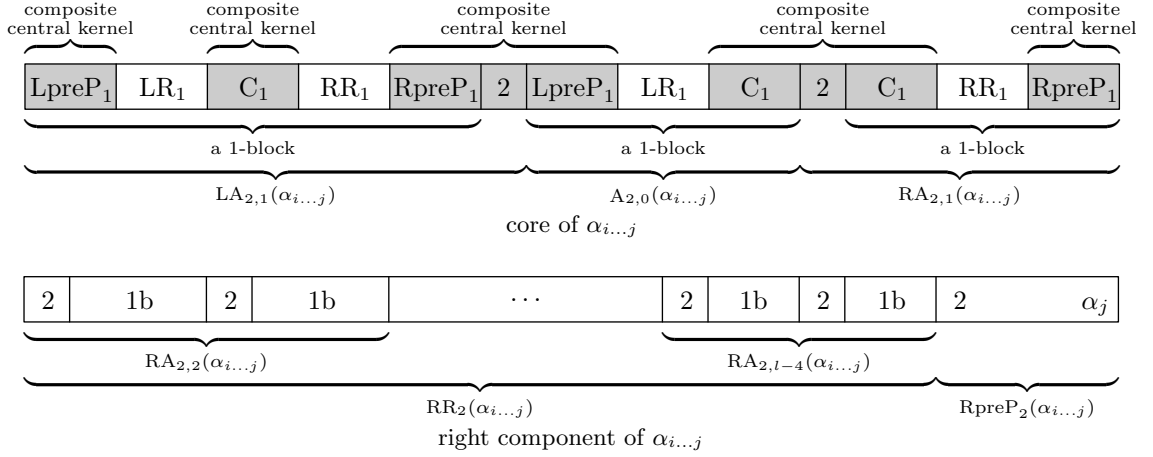

\centering
\includegraphics{sequence-5.mps}\\[\bigskipamount]
\includegraphics{sequence-6.mps}
\caption{Detailed structure of a 2-block $\mathcal E_l=\swa ij$, where Case II holds at the left and Case I
holds at the right: 2 denotes a letter of order 2, 1b denotes a 1-block. Each individual grayed box is a prime 
central kernel.}
\label{twostruct}
\end{figure}

\begin{lemma}\label{finitecker}
The lengths of all composite central kernels 
of all evolutions of $k$-blocks arising in $\al$
are bounded by a single constant that depends on $\E$, $\varphi$, and $k$ only.
In particular, only finitely many abstract words can equal 
central kernels of evolutions 
of $k$-blocks arising in $\al$.
\end{lemma}

\begin{proof}
The proof is similar to the proof of Corollary \ref{finitelrprep}.
By Corollary \ref{finite-number-of-evolutions}, there are only finitely many 
sequences of abstract words that can be evolutions in $\al$, so there exists 
a single constant $x$ that depends on $\E$, $\varphi$ and $k$ only
such that if $\mathcal E$ is an evolution of $k$-blocks, then $|\mathcal E_{3k}|\le x$.
By definition, $\cKer_{k,m}(\mathcal E)=\cKer_{k,m}(\mathcal E_{3k})$ 
for $1\le m\le \ncker_k(\mathcal E)$, and
$\cKer_{k,m}(\mathcal E_{3k})$ is a subword of 
$\mathcal E_{3k}$. Therefore, $|\cKer_{k,m}(\mathcal E)|\le x$.
\end{proof}

Again, as in the proof of Corollary \ref{finitelrprep}, we do not 
claim that if 
two evolutions of $k$-blocks equal as sequences of abstract words,
then their composite central kernels are equal.

\section{Continuously Periodic Evolutions}\label{sectionevolutions}

In this section we will define and study continuously periodic evolutions, which will enable us 
to formulate a criterion for subword complexity of morphic sequences.
We will use the coding $\psi$ a lot as well as the morphism $\varphi$, 
so we will use some obvious properties of codings without mentioning every time 
that $\psi$ is a coding. For example, if $\gamma$ is a finite word, then $|\gamma|=|\psi(\gamma)|$, 
and if $0\le i\le |\gamma|-1$, then $\psi(\gamma_i)=\psi(\gamma)_i$.
Also, if $i\in\ZZ_{\ge 0}$, then $\psi(\al_i)=\psi(\al)_i$. If $\gamma_{i\ldots j}$ 
is an occurrence in a finite word $\gamma$, then $\psi(\gamma_{i\ldots j})=\psi(\gamma)_{i\ldots j}$.
And if $\swa ij$ is an occurrence in $\al$, then $\psi(\swa ij)=\psi(\al)_{i\ldots j}$.

Before we will be able to
define continuously periodic evolutions, we need to introduce two more technical notions, namely, we need to define 
left and right bounding sequences of an evolution of $k$-blocks and 
weak left and right evolutional periods of $k$-multiblocks.

The construction of the left and the right bounding sequences is similar to the construction
of pure morphic sequences themselves. Let us construct the right bounding sequence, 
the construction for the left bounding sequence is symmetric.
Let $\mathcal E$ be an evolution of $k$-blocks such that Case II holds at the right. 
First, consider the following sequence of abstract words:
$\RB(\mathcal E), \varphi(\RB(\mathcal E)), \varphi^2(\RB(\mathcal E)), \ldots$
Since $\varphi$ is strongly 1-periodic and images of letters of order $\le k$ 
consist of letters of order $\le k$, the leftmost letter of order $>k$ in each 
of these words is $\RB(\mathcal E)$. Temporarily denote by $\gamma_l$ (resp.\ $\delta_l$) the prefix (resp.\ the suffix)
of $\varphi^l(\RB(\mathcal E))$ to the left (resp.\ to the right) from the leftmost occurrence of $\RB(\mathcal E)$
(not including this occurrence of $\RB(\mathcal E)$). In other words, write 
$\varphi^l(\RB(\mathcal E))=\gamma_l\RB(\mathcal E)\delta_l$, where $\gamma_l$ consists 
of letters of order $\le k$ only. In particular, $\gamma_0$ and $\delta_0$ are the empty word.

\begin{remark}\label{rbsremarki}
If $\RB(\mathcal E_m)=\al_i$, where $m\ge 1$, $\RB(\mathcal E_{m_l})=\al_j$, where $l\ge 0$, 
and $\varphi^l(\al_i)=\swa st$ as an occurrence in $\al$, then $\swa{j+1}t=\delta_l$ as an abstract word.
\end{remark}

\begin{lemma}\label{rbslemmai}
For all $l\ge 0$ we have $\delta_{l+1}=\delta_1\varphi(\delta_l)$.
\end{lemma}
\begin{proof}
We have $\gamma_{l+1}\RB(\mathcal E)\delta_{l+1}=\varphi^{l+1}(\RB(\mathcal E))=
\varphi(\varphi^l(\RB(\mathcal E)))=\varphi(\gamma_l\RB(\mathcal E)\delta_l)=
\varphi(\gamma_l)\varphi(\RB(\mathcal E))\varphi(\delta_l)=\varphi(\gamma_l)\gamma_1\RB(\mathcal E)\delta_1\varphi(\delta_l)$.
Note that $\varphi(\gamma_l)$ and $\gamma_1$ do not contain letters of order $>k$,
so the leftmost occurrence of $\RB(\mathcal E)$ in 
$\varphi(\gamma_l)\gamma_1\RB(\mathcal E)\delta_1\varphi(\delta_l)$
is the occurrence mentioned in this formula explicitly.
On the other hand, by the definition of $\gamma_{l+1}$ and $\delta_{l+1}$, 
the leftmost occurrence of $\RB(\mathcal E)$ in 
$\gamma_{l+1}\RB(\mathcal E)\delta_{l+1}$
is also the occurrence mentioned in this formula explicitly.
Hence, $\delta_{l+1}=\delta_1\varphi(\delta_l)$.
\end{proof}

\begin{lemma}
For all $l\ge 0$, $\delta_l$ is a prefix of $\delta_{l+1}$.
\end{lemma}
\begin{proof}
Let us prove this by induction on $l$. For $l=0$ this is clear since 
$\delta_0$ is the empty word. Suppose that 
$\delta_l$ is a prefix of $\delta_{l+1}$. Then 
$\varphi(\delta_l)$ is a prefix of $\varphi(\delta_{l+1})$.
By Lemma \ref{rbslemmai}, $\delta_{l+1}=\delta_1\varphi(\delta_l)$
and $\delta_{l+2}=\delta_1\varphi(\delta_{l+1})$, 
so $\delta_{l+1}$ is a prefix of $\delta_{l+2}$.
\end{proof}

\begin{lemma}
For all $l\ge 0$, we have $|\delta_{l+1}|>|\delta_l|$.
\end{lemma}
\begin{proof}
Recall that we have started with an evolution $\mathcal E$ such that Case II holds at the right.
By the definitions of Case II and of right atoms, $\gamma_1$ cannot contain 
letters of order $k$, it consists of letters of smaller orders (or is empty if $k=1$).
But then, if $\delta_1$ also had consisted of letters of order $<k$ only,
$\RB(\mathcal E)$ would have been a letter of order $k$ or less. So, 
$\delta_1$ contains at least one letter of order $\ge k$, in particular, 
$\delta_1$ is nonempty. Since $\varphi$ is nonerasing, 
$|\varphi(\delta_l)|\ge |\delta_l|$, and 
$|\delta_{l+1}|=|\delta_1|+|\varphi(\delta_l)|>|\delta_l|$.
\end{proof}

So, we have constructed an infinite sequence of words $\delta_l$, whose lengths strictly 
increase, and each of them is a prefix of the next one. Let us add $\RB(\mathcal E)$
at the left of each of these words. We get an infinite sequence 
of words $\RB(\mathcal E), \RB(\mathcal E)\delta_1, \RB(\mathcal E)\delta_2, \ldots, \RB(\mathcal E)\delta_l, \ldots$
whose lengths strictly increase, and each of them is a prefix of the next one.
So, there
exists a unique
infinite (to the right) word such that all these 
words $\RB(\mathcal E)\delta_l$ (for all $l\ge 0$) are its prefixes. 
We call this infinite word \textit{the right bounding sequence of $\mathcal E$}
and denote it by $\RBS_k(\mathcal E)$.

With this definition, Remark \ref{rbsremarki} 
can be reformulated as follows:

\begin{remark}\label{rbsremarkii}
Let $\mathcal E$ be an evolution of $k$-blocks such that Case II holds at the right, 
$m\ge 1$ and $l\ge 0$. Suppose that $\RB(\mathcal E_m)=\al_i$, 
$\RB(\mathcal E_{m+l})=\al_j$, and $\varphi^l(\al_i)=\swa st$
as an occurrence in $\al$. Then $\swa{j+1}t$ as an abstract word is a prefix of $\RBS_k(\mathcal E)$.
\end{remark}

$\RBS_k(\mathcal E)$ is an abstract infinite word, it is not an occurrence in $\al$. However, we can prove
the following lemma.
\begin{lemma}
For all $l\ge 0$, if $\mathcal E_{l+2}=\swa ij$ as an occurrence in $\al$, 
then $\swa i{j+|\RB(\mathcal E)\delta_l|}=\mathcal E_{l+2}\RB(\mathcal E)\delta_l$.
\end{lemma}
\begin{proof}
We prove this by induction on $l$. If $l=0$, then $|\delta_l|=0$ and, by the definition 
of $\RB(\mathcal E)$, $\RB(\mathcal E_2)=\RB(\mathcal E)$, and the claim is clear.

Suppose that $\mathcal E_{l+2}=\swa ij$ and $\swa i{j+|\RB(\mathcal E)\delta_l|}=\mathcal E_{l+2}\RB(\mathcal E)\delta_l$.
Let $s$ and $t$ be the indices such that $\swa st=\Su_k(\swa ij)=\mathcal E_{l+3}$ as an occurrence in $\al$.
Denote also by $i'$ and $j'$ the indices such that $\swa{i'}{j'}=\varphi(\swa ij)$.
By the definition of the descendant of a $k$-block, $\swa{i'}{j'}$ is a suboccurrence of $\swa st$.
Consider also the following occurrence in $\al$ starting from position $i'$: 
$\varphi(\swa i{j+|\delta_l|})=\varphi(\swa ij\RB(\mathcal E)\delta_l)=\swa{i'}{j'}\gamma_1\RB(\mathcal E)\delta_1\varphi(\delta_l)=
\swa{i'}{j'}\gamma_1\RB(\mathcal E)\delta_{l+1}$. Since $\gamma_1$ consists of letters of order $\le k$ only, 
and $\RB(\mathcal E)$ is a letter of order $>k$, the occurrence of $\RB(\mathcal E)$
mentioned explicitly in this formula is the right border of $\mathcal E_{l+3}$.
In other words, the occurrence of $\gamma_1$ mentioned explicitly here is $\swa{j'+1}t$, 
and the occurrence of $\RB(\mathcal E)$ mentioned explicitly here is $\al_{t+1}$.
Hence, $\swa{t+1}{t+|\RB(\mathcal E)\delta_{l+1}|}=\RB(\mathcal E)\delta_{l+1}$, 
and $\swa s{t+|\RB(\mathcal E)\delta_{l+1}|}=\mathcal E_{l+3}\RB(\mathcal E)\delta_{l+1}$
\end{proof}

This lemma still uses the notation $\delta_l$ we have introduced temporarily, 
but the next corollary does not use any temporary notation anymore.

\begin{corollary}\label{presenceofrbs}
Let $\mathcal E$ be an evolution of $k$-blocks such that Case II holds at the right, and 
let $\delta$ be an arbitrary finite prefix of $\RBS_k(\mathcal E)$.

Then there exists $l_0\in\NN$ (that depends on $\mathcal E$ and $\delta$)
such that if $l\ge l_0$ and $\mathcal E_l=\swa ij$ as an occurrence in $\al$, 
then 
$\swa i{j+|\delta|}=\mathcal E_l\delta$.
\end{corollary}
\begin{proof}
Since the lengths of the words $\delta_l$ strictly grow as $l\to \infty$, 
there exists $l'\in\NN$ such that if $l\ge l'$, then $\delta$ is a prefix 
of $\RB(\mathcal E)\delta_l$.

Set $l_0=l'+2$. Then if $\mathcal E_l=\swa ij$ as an occurrence in $\al$,
then by the previous lemma, $\swa i{j+|\RB(\mathcal E)\delta_{l-2}|}=\mathcal E_l\RB(\mathcal E)\delta_{l-2}$.
Since $\delta$ is a prefix of $\RB(\mathcal E)\delta_{l-2}$, 
we have $\swa i{j+|\delta|}=\mathcal E_l\delta$.
\end{proof}

Similarly, if Case II holds at the left for an evolution $\mathcal E$ of $k$-blocks, 
we define its \textit{left bounding sequence} and denote it by $\LBS(\mathcal E)$.
$\LBS_k(\mathcal E)$ is a word infinite to the left, and the symmetric version of Corollary \ref{presenceofrbs}
for left bounding sequences can be formulated as follows.

\begin{corollary}\label{presenceoflbs}
Let $\mathcal E$ be an evolution of $k$-blocks such that Case II holds at the left, and 
let $\delta$ be an arbitrary finite suffix of $\RBS_k(\mathcal E)$.

Then there exists $l_0\in\NN$ (that depends on $\mathcal E$ and $\delta$)
such that if $l\ge l_0$ and $\mathcal E_l=\swa ij$ as an occurrence in $\al$,
then $i\ge|\delta|$ and
$\swa {i-|\delta|}j=\delta\mathcal E_l$.\qed
\end{corollary}

The symmetric version of Remark \ref{rbsremarkii} can be formulated as follows.

\begin{remark}\label{lbsremarkii}
Let $\mathcal E$ be an evolution of $k$-blocks such that Case II holds at the left, 
$m\ge 1$ and $l\ge 0$. Suppose that $\LB(\mathcal E_m)=\al_i$, 
$\LB(\mathcal E_{m+l})=\al_j$, and $\varphi^l(\al_i)=\swa st$
as an occurrence in $\al$. Then $\swa s{j-1}$ as an abstract word is a suffix of $\LBS_k(\mathcal E)$.
\end{remark}

Now let us define left and right weak evolutional periods for evolutions $\mathcal F$ of stable nonempty $k$-multiblocks 
($k\ge 1$). The 
definition will use two indices, $m$ and $m'$ ($0\le m < m'\le \nker_k(\mathcal F)+1$). 
A final period $\lambda$ is called a
\textit{left} (resp.\ \textit{right}) \textit{weak evolutional period}
of an evolution $\mathcal F$ of stable $k$-multiblocks ($k\ge 1$) for a pair of indices $(m,m')$ 
($0\le m < m'\le \nker_k(\mathcal F)+1$) if
\begin{enumerate}
\item For each $l\ge 0$, $\psi(\IpR_{k,m,m'}(\mathcal F_l))$ is a weakly 
left (resp.\ right) $\lambda$-periodic word.
\item The residue of 
$|\IpR_{k,m,m'}(\mathcal F_l)|$
(for $l\ge 0$) modulo $|\lambda|$, 
i.~e.\ the length of the incomplete occurrence in the previous condition, does not depend on $l$.
\end{enumerate}

The second condition here enables us to formulate the following remark:
\begin{remark}\label{contpermultcycle}
If $\lambda$ is a left weak evolutional period
of an evolution $\mathcal F$ of stable nonempty $k$-multiblocks ($k\ge 1$) for a pair of indices $(m,m')$
($0\le m < m'\le \nker_k(\mathcal F)+1$) and $r$ is the residue of 
$|\IpR_{k,m,m'}(\mathcal F_l)|$ modulo $|\lambda|$ (for any $l$), then 
$\lambda'=\Cyc_r(\lambda)=\Cyc_{|\IpR_{k,m,m'}(\mathcal F_l)|}(\lambda)$
is a right weak evolutional period of $\mathcal F$ for the pair $(m,m')$.

If $\lambda$ is a right weak evolutional period
of an evolution $\mathcal F$ of stable nonempty $k$-multiblocks ($k\ge 1$) for a pair of indices $(m,m')$
($0\le m < m'\le \nker_k(\mathcal F)+1$) and $r$ is the residue of 
$|\IpR_{k,m,m'}(\mathcal F_l)|$ modulo $|\lambda|$ (for any $l$), then 
$\lambda'=\Cyc_{-r}(\lambda)=\Cyc_{-|\IpR_{k,m,m'}(\mathcal F_l)|}(\lambda)$
is a left weak evolutional period of $\mathcal F$ for the pair $(m,m')$.
\end{remark}

\begin{lemma}\label{contperweakuniquewideleft}
Let $\mathcal F$ be an evolution of stable $k$-multiblocks ($k\ge 1$). Suppose that $\nker_k(\mathcal F)>1$. Let 
$m,m',m''\in\ZZ$ be three indices such that $0\le m<\nker_k(\mathcal F)$, $m<m'$, $m<m''$, 
$1<m'\le\nker_k(\mathcal F)+1$, and $1<m''\le\nker_k(\mathcal F)+1$.

If $\lambda$ is a left weak evolutional period of $\mathcal F$ for the pair $(m,m')$, 
and $\lambda'$ is a left weak evolutional period of $\mathcal F$ for the pair $(m,m'')$, 
then $\lambda=\lambda'$, and $\lambda$ is the minimal left period of 
$\psi(\IpR_{k,m,m'}(\mathcal F_l))$ for all $l\ge 0$.
\end{lemma}

\begin{proof}
Both $\lambda$ and $\lambda'$ are final periods, so $|\lambda|\le\finmax$ and $|\lambda'|\le\finmax$.
Without loss of generality, $m'\le m''$, so $\IpR_{k,m,m'}(\mathcal F_l)$
is a prefix of $\IpR_{k,m,m''}(\mathcal F_l)$ for all $l\ge 0$.
Hence, for all $l\ge 0$, $\psi(\IpR_{k,m,m'}(\mathcal F_l))$ is both 
a weakly left $\lambda$-periodic word 
and
a weakly left $\lambda'$-periodic word. 
By Corollary \ref{betweenkernelgrowsgood}, 
$|\IpR_{k,m,m'}(\mathcal F_l)|\ge 2\finmax$.
Since $\lambda$ and $\lambda'$ are final periods, 
none of them can be written as a word repeated more than once by Lemma \ref{finalnomorethanonce}. 
Then Lemma \ref{finwordperiods}
implies that $\lambda=\lambda'$.

Moreover, if $\psi(\IpR_{k,m,m'}(\mathcal F_l))$ were (for some $l\ge 0$) a weakly left $\lambda''$-periodic 
word, where $|\lambda''|<|\lambda|$, then Lemma \ref{finwordperiods}
would again imply that $\lambda$ can be written as another word repeated several times, 
a contradiction.
Therefore, $\lambda$ is the minimal left period of
$\psi(\IpR_{k,m,m'}(\mathcal F_l))$ for all $l\ge 0$.
\end{proof}

\begin{corollary}\label{contperweakuniqueleft}
Let $\mathcal F$ be an evolution of stable $k$-multiblocks ($k\ge 1$). Suppose that $\nker_k(\mathcal F)>1$. Let 
$m,m'\in\ZZ$ be two indices such that $0\le m<\nker_k(\mathcal F)$, $m<m'$ and $1<m'\le\nker_k(\mathcal F)+1$.

If there exists a left weak evolutional period of $\mathcal F$ for the pair $(m,m')$, 
then it is unique and is the minimal left period of
$\psi(\IpR_{k,m,m'}(\mathcal F_l))$ for all $l\ge 0$.\qed
\end{corollary}

\begin{lemma}\label{contperweakuniquewideright}
Let $\mathcal F$ be an evolution of stable $k$-multiblocks ($k\ge 1$). Suppose that $\nker_k(\mathcal F)>1$. Let 
$m'',m',m\in\ZZ$ be three indices such that $0\le m''<\nker_k(\mathcal F)$, $0\le m'<\nker_k(\mathcal F)$, $m''<m$, $m'<m$, 
and $1<m\le\nker_k(\mathcal F)+1$.

If $\lambda$ is a right weak evolutional period of $\mathcal F$ for the pair $(m'',m)$, 
and $\lambda'$ is a right weak evolutional period of $\mathcal F$ for the pair $(m',m)$, 
then $\lambda=\lambda'$, and $\lambda$ is the minimal right period of 
$\psi(\IpR_{k,m,m'}(\mathcal F_l))$ for all $l\ge 0$.
\end{lemma}

\begin{proof}
The proof is completely symmetric to the proof of Lemma \ref{contperweakuniquewideleft}.
\end{proof}

\begin{corollary}\label{contperweakuniqueright}
Let $\mathcal F$ be an evolution of stable $k$-multiblocks ($k\ge 1$). Suppose that $\nker_k(\mathcal F)>1$. Let 
$m',m\in\ZZ$ be two indices such that $0\le m'<\nker_k(\mathcal F)$, $m'<m$ and $1<m\le\nker_k(\mathcal F)+1$.

If there exists a right weak evolutional period of $\mathcal F$ for the pair $(m',m)$, 
then it is unique and is the minimal right period of 
$\psi(\IpR_{k,m,m'}(\mathcal F_l))$ for all $l\ge 0$.\qed
\end{corollary}

Now we define left and right pseudoregular parts and continuous evolutional periods 
for evolutions of $k$-blocks. The definition is similar to the definition of inner 
pseudoregular parts and weak evolutional periods for evolutions of $k$-multiblocks, 
but is not entirely the same. Let $\mathcal E$ be an evolution of $k$-blocks, 
and let $m$ be an index ($1\le m\le\ncker_k(\mathcal E)$). 
For each 
$l\ge 0$ we define the left and the right pseudoregular parts of a stable $k$-block $\mathcal E_l$ 
for index $m$ 
as follows. Let $i$ and $j$ be indices such that $\mathcal E_l=\swa ij$. Suppose also that 
$\cKer_{k,m}(\mathcal E_l)=\swa st$, 
$\Fg(\LpreP_k(\swa ij))=\swa i{i'}$, 
and $\Fg(\RpreP_k(\swa ij))=\swa {j'}j$. Then \textit{the left (resp.\ right) pseudoregular part of 
$\mathcal E_l$ for index $m$} is $\swa {i'+1}{s-1}$ (resp.\ $\swa {t+1}{j'-1}$). In other words,
the left (resp.\ right) pseudoregular part of $\mathcal E_l$ for index $m$ is 
the suboccurrence of $\mathcal E_l$ between the forgetful occurrence of the 
left preperiod and the $m$th composite central kernel (resp.\ between the $m$th 
composite central kernel and the forgetful occurrence of the right preperiod).
Denote it by $\LpR_{k,m}(\mathcal E_l)$ (resp.\ by $\RpR_{k,m}(\mathcal E_l)$).

The following remark shows how the definitions of left and right pseudoregular parts of $k$-multiblocks
and inner pseudoregular parts of $k$-multiblocks are connected.

\begin{remark}
Let $\mathcal E$ be an evolution of $k$-blocks, $l_0\ge 3k$, and let $\mathcal F$ be the evolution of 
$k$-multiblocks defined by $\mathcal F_l=\mathcal E_{l+l_0}$. Let $m$ be an index 
($1\le m\le\ncker_k(\mathcal E)$). Let $m'$ ($1\le m'\le \nker_k(\mathcal F)$)
be the index such that $\cKer_{k,m}(\mathcal E_{l+l_0})=\Ker_{k,m'}(\mathcal F_l)$ 
as an occurrence in $\al$ for all $l\ge 0$. In other words, 
if Case I holds for $\mathcal E$ at the left, then $m'=m+1$, 
otherwise $m'=m$.

Then
\begin{enumerate}
\item
\begin{enumerate}
\item If Case I holds for $\mathcal E$ at the left, then $\LpR_{k,m}(\mathcal E_{l+l_0})=\IpR_{k,1,m'}(\mathcal F_l)$.
\item If Case II holds for $\mathcal E$ at the left, then $\LpR_{k,m}(\mathcal E_{l+l_0})=\IpR_{k,0,m'}(\mathcal F_l)$.
\end{enumerate}
\item
\begin{enumerate}
\item If Case I holds for $\mathcal E$ at the right, then $\LpR_{k,m}(\mathcal E_{l+l_0})=\IpR_{k,m',\nker_k(\mathcal F)}(\mathcal F_l)$.
\item If Case II holds for $\mathcal E$ at the right, then $\LpR_{k,m}(\mathcal E_{l+l_0})=\IpR_{k,m',\nker_k(\mathcal F)+1}(\mathcal F_l)$.
\end{enumerate}
\end{enumerate}
\end{remark}

%
%
%

A final period $\lambda$ is called a \textit{left} (resp.\ \textit{right}) \textit{continuous evolutional period}
of an evolution of $k$-blocks $\mathcal E$ for an index $m$ if the following two conditions hold:
\begin{enumerate}
\item\label{contperstd}
\begin{enumerate}
\item For each $l\ge 3k$ (i.~e.\ if $\mathcal E_l$ is a stable $k$-block), 
$\psi(\LpR_{k,m}(\mathcal E_l))$ (resp.\ $\psi(\RpR_{k,m}(\mathcal E_l))$)
is a weakly 
left (resp.\ right) $\lambda$-periodic word.
\item The residue of 
$|\LpR_{k,m}(\mathcal E_l)|$ (resp.\ of $|\RpR_{k,m}(\mathcal E_l)|$)
(for $l\ge 3k$) modulo $|\lambda|$, 
i.~e.\ the length of the incomplete occurrence in the previous condition, does not depend on $l$.
\end{enumerate}
\item\label{contperlrbs} If Case II holds for $\mathcal E$ at the left (resp.\ at the right)
and $m>1$ (resp.\ $m<\ncker_k(\mathcal E)$) (i.~e.\ $\cKer_{k,m}(\mathcal E)$ is not the 
leftmost (resp.\ rightmost) composite central kernel for $\mathcal E$), 
then
$\psi(\LBS_k(\mathcal E))$ (resp.\ $\psi(\RBS_k(\mathcal E))$)
(this is a sequence infinite to the 
left (resp.\ to the right)) is 
periodic 
with 
period $\lambda$, 
i.~e.\ $\psi(\LBS_k(\mathcal E))=\ldots\lambda\ldots\lambda\lambda$
(resp.\ $\psi(\RBS_k(\mathcal E))=\lambda\lambda\ldots\lambda\ldots$).
\end{enumerate}

The following remark shows connection between the definitions of a continuous evolutional
period of an evolution of $k$-blocks and a weak evolutional period of 
an evolution of stable $k$-multiblocks.

\begin{remark}\label{evperblockmulrel}
Let $\mathcal E$ be an evolution of $k$-blocks, let $\lambda$ be a final period, 
and let $\mathcal F$ be the evolution of 
$k$-multiblocks defined by $\mathcal F_l=\mathcal E_{l+3k}$. Let $m$ be an index 
($1\le m\le\ncker_k(\mathcal E)$). Again, let $m'$ ($1\le m'\le \nker_k(\mathcal F)$)
be the index such that $\cKer_{k,m}(\mathcal E_{l+3k})=\Ker_{k,m'}(\mathcal F_l)$ 
as an occurrence in $\al$ for all $l\ge 0$. In other words, 
if Case I holds for $\mathcal E$ at the left, then $m'=m+1$, 
otherwise $m'=m$.

Then
\begin{enumerate}
\item
\begin{enumerate}
\item If Case I holds for $\mathcal E$ at the left, then $\lambda$ is a left continuous evolutional period 
of $\mathcal E$ for index $m$ if and only if 
$\lambda$ is a left weak evolutional period of $\mathcal F$ for pair $(1,m')$.
\item If Case II holds for $\mathcal E$ at the left, then Condition \ref{contperstd} in the definition 
of a left continuous evolutional period 
is satisfied for $\lambda$ if and only if 
$\lambda$ is a left weak evolutional period of $\mathcal F$ for pair $(0,m')$.
\end{enumerate}
\item
\begin{enumerate}
\item If Case I holds for $\mathcal E$ at the right, then $\lambda$ is a right continuous evolutional period 
of $\mathcal E$ for index $m$ if and only if 
$\lambda$ is a right weak evolutional period of $\mathcal F$ for pair $(m',\nker_k(\mathcal F))$.
\item If Case II holds for $\mathcal E$ at the right, then Condition \ref{contperstd} in the definition 
of a right continuous evolutional period 
is satisfied for $\lambda$ if and only if 
$\lambda$ is a right weak evolutional period of $\mathcal F$ for pair $(m',\nker_k(\mathcal F)+1)$.
\end{enumerate}
\end{enumerate}
\end{remark}

%
%

Let us make several obvious remarks about this definition. First, if Case II holds at the left 
for $\mathcal E$, then the left 
component is empty, in particular, the left preperiod is empty. 
So, $\LpR_{k,m}(\mathcal E_l)$ is a prefix of $\mathcal E_l$, and 
by Corollary \ref{presenceoflbs}, a suffix of the sequence 
$\psi(\LBS_k(\mathcal E)\LpR_{k,m}(\mathcal E_l))$ 
mentioned in Condition \ref{contperlrbs}
is a subword of $\psi(\al)$ if $l$ is large enough.
Informally speaking, Condition \ref{contperlrbs} says that
$\psi(\LBS_k(\mathcal E))$ is periodic and that the periods of $\psi(\LpR_{k,m}(\mathcal E_l))$ and 
of $\psi(\LBS_k(\mathcal E))$ "agree well".

If Case II holds at the left 
for $\mathcal E$ (and the left 
component is empty), and we are trying to figure out whether $\lambda$ is a 
left continuous evolutional period of $\mathcal E$ for index 1 (for the first composite central kernel), then 
it follows from Lemma \ref{borderkernelsleft}
that 
the word 
$\psi(\LpR_{k,m}(\mathcal E_l))$,
whose periodicity is required by Condition \ref{contperstd}, 
is actually empty in this case, and the condition is always satisfied. 
Condition \ref{contperlrbs} does not say anything about this case, it only applies if 
$m>1$.

Finally, we are ready to give the definition of a continuously periodic evolution.
An evolution $\mathcal E$ of $k$-blocks is called \textit{continuously periodic for an 
index $m$ ($1\le m\le\ncker_k(\mathcal E)$)} if there exist two final periods $\lambda$ and 
$\lambda'$ such that $\lambda$ is a left continuous evolutional period
of $\mathcal E$ for the index $m$ and 
$\lambda'$ is a right continuous evolutional period
of $\mathcal E$ for the index $m$.

An evolution $\mathcal E$ of $k$-blocks is called \textit{continuously periodic}
if there exists an index $m$ ($1\le m\le\ncker_k(\mathcal E)$)
such that $\mathcal E$ is continuously periodic for $m$.

\begin{remark}\label{onecontper}
By Lemma \ref{regpartabslargeone}, all evolutions of 1-blocks are continuously periodic for index 1
(recall that there is only one composite central kernel of a 1-block). 
\end{remark}

One more important 
case when an evolution of $k$-blocks is continuously
periodic for index 1 is the case when the number of 
the composite central kernels of an evolution is one, 
and Case II holds both at the left and at the right. Then all blocks 
in this evolution in fact consist of letters of order 1, but 
this does not necessarily mean that $k=1$, $k$ may be bigger than 1 
if the left and the right border of this evolution are letters of order $>2$.
Condition \ref{contperstd} in the definition of a left and a right continuous 
evolutional period
says in this case that some empty words have to be weakly periodic, 
and this is always true, and Condition
\ref{contperlrbs}
does not apply since there is only one composite central kernel, 
so such an evolution is always continuously periodic for index 1.

After we gave the definition of a continuously periodic evolution,
the statements of 
Propositions~\ref{largecompl}--\ref{infordercomplstop}
are completely formulated. Before we will be able to prove them, 
we need to give some more definitions.

A sequence $\mathcal H=\mathcal H_0,\mathcal H_1,\mathcal H_2,\ldots$ of 
occurrences in $\al$ is called a $k$-series of obstacles if there exists 
a number $p\le \finmax$
such that:
\begin{enumerate}
\item Each word $\psi(\mathcal H_l)$ is a weakly 
$p$-periodic
word.
\item The length of $\mathcal H_l$ strictly grows as $l$ grows, 
$|\mathcal H_l|\ge 2\finmax$ for all $l\ge 0$, and there exists $k'$ ($1\le k'\le k$)
such that $|\mathcal H_l|$ is $\Theta(l^{k'})$ for $l\to\infty$.
\item If $\mathcal H_l=\swa ij$,
then 
$\psi(\swa i{j+1})$ and (if $i>0$) $\psi(\swa{i-1}j)$ 
are not weakly 
$p$-periodic words.
\end{enumerate}

We will also need some notion of weak and continuous periodicity for $k$-multiblocks.
Let $\mathcal F$ be an evolution of stable nonempty $k$-multiblocks ($k\ge 1$).
The multiblocks $\mathcal F_l$ are nonempty in the multiblock sense, 
in other words, it is not allowed that $\mathcal F_l$ contains no $k$-blocks and 
no letters of order $k$, but it is allowed that $\mathcal F_l$ consists 
of a single empty $k$-block if this $k$-block is stable.

We call $\mathcal F$ \textit{weakly periodic for a sequence of 
indices $m_0=0,m_1,\ldots, m_{n-1},m_n=\nker_k(\mathcal F)+1$} ($1\le m_1<m_2<\ldots<m_{n-1}\le \nker_k(\mathcal F)$) 
if there exist final periods $\lambda^{(0)},\ldots,\lambda^{(n-1)}$ 
such that $\lambda^{(i)}$ is a left weak evolutional period of $\mathcal F$ 
for pair $(m_i,m_{i+1})$ for all $i$ ($0\le i<n$). By Remark \ref{contpermultcycle}, 
the left weak evolutional periods in this definition can be replaced by right
weak evolutional periods, moreover, this can be done independently for each index $i$.
We call $\mathcal F$ \textit{weakly periodic}
if there exists a sequence of indices 
$m_0=0,m_1,\ldots, m_{n-1},m_n=\nker_k(\mathcal F)+1$
($1\le m_1<m_2<\ldots<m_{n+1}\le \nker_k(\mathcal F)$) 
such that $\mathcal F$ is weakly periodic for 
$m_0,m_1,\ldots, m_{n-1},m_n$.

We call $\mathcal F$ 
\textit{continuously periodic for an index $m$} ($1\le m\le \nker_k(\mathcal F)$)
if it is weakly periodic for indices $m_0=0,m_1=m,m_2=\nker_k(\mathcal F)+1$, i.~e.\ 
if there exist two final periods $\lambda$ and 
$\lambda'$ such that $\lambda$ is a left weak evolutional period of 
$\mathcal F$ for pair $(0,m)$, and 
$\lambda'$ is a right weak evolutional period of 
$\mathcal F$ for pair $(m,\nker_k(\mathcal F)+1)$.
We call $\mathcal F$ 
\textit{continuously periodic} if there exists an index $m$ 
($1\le m\le \nker_k(\mathcal F)$)
such that 
$\mathcal F$ 
is continuously periodic for $m$.

Note that in this definition, unlike in the definition for $k$-blocks, we don't have any left and 
right preperiods or a replacement for them, and the words whose periodicity we require 
(the $(0,m)$th and the $(m,\nker_k(\mathcal F)+1)$th inner pseudoregular parts)
are a prefix and a suffix (respectively) of (the forgetful occurrence of) a $k$-multiblock.
Also, we don't introduce any left and right bounding sequence, and speak only about periodicity
of subwords of (the forgetful occurrence of) a $k$-multiblock itself.


Finally, 
a final period $\lambda$ is called 
a \textit{total left (resp.\ right) evolutional period}
of $\mathcal F$
if it is a weak left (resp.\ right) evolutional period of $\mathcal F$ 
for the pair $(0, \nker_k(\mathcal F)+1)$.
$\mathcal F$ is called \textit{totally periodic} 
if it is weakly periodic for the indices $m_0=0,m_1=\nker_k(\mathcal F)+1$.

\begin{remark}\label{nogrowthcontper}
Let $\mathcal F$ be an evolution of stable nonempty $k$-blocks ($k\ge 1$)
such that $\nker_k(\mathcal F)=1$. Then $\mathcal F$ is weakly periodic for sequence 
$0,1,2$ and is continuously periodic for index 1. The corresponding final periods can be chosen 
arbitrarily since the inner pseudoregular parts in question in this case are empty occurrences.
\end{remark}

%
%
%
%
%

\begin{lemma}\label{weakcontperinsleft}
Let $\mathcal F$ be an evolution of stable nonempty $k$-multiblocks ($k\ge 1$).
Suppose that 
$\mathcal F$ is weakly periodic for a sequence of 
indices $m_0=0,m_1,\ldots, m_{n-1},m_n=\nker_k(\mathcal F)+1$.
Suppose that $m_1>1$. Then 
$\mathcal F$ is also weakly periodic for sequence
$m_0=0,1,m_1,\ldots, m_{n-1},m_n=\nker_k(\mathcal F)+1$.
\end{lemma}

\begin{proof}
Let 
$\lambda$ be a right weak evolutional period of $\mathcal F$ 
for pair $(0,m_1)$. By definition this means that 
$\psi(\IpR_{k,0,m_1}(\mathcal F_l))$
is a weakly right $\lambda$-periodic word
for all $l\ge 0$, and the residue of 
$|\IpR_{k,0,m_1}(\mathcal F_l)|$ modulo 
$|\lambda|$ does not depend on $l$. But 
$\IpR_{k,0,m_1}(\mathcal F_l)=
\Ker_{k,1}(\mathcal F_l)\IpR_{k,1,m_1}(\mathcal F_l)$, 
so $\psi(\IpR_{k,1,m_1}(\mathcal F_l))$
is also 
a weakly right $\lambda$-periodic word, 
and the length of 
$\Ker_{k,1}(\mathcal F_l)$ does not depend on $l$, so the residue of 
$|\IpR_{k,1,m_1}(\mathcal F_l)|=
|\IpR_{k,0,m_1}(\mathcal F_l)|-|\Ker_{k,1}(\mathcal F_l)|$
modulo $|\lambda|$
does not depend on $l$ either. So, 
$\lambda$ is also 
a right weak evolutional period of $\mathcal F$ 
for pair $(1,m_1)$.
Also, any final period is a (right) weak evolutional period of $\mathcal F$ for pair $(0,1)$
since $\IpR_{k,0,1}(\mathcal F_l)$ is always an empty occurrence.
Hence, we can insert an index 1 in the sequence for which $\mathcal F$ is weakly periodic.
\end{proof}

\begin{lemma}\label{weakcontperinsright}
Let $\mathcal F$ be an evolution of stable nonempty $k$-multiblocks ($k\ge 1$).
Suppose that 
$\mathcal F$ is weakly periodic for a sequence of 
indices $m_0=0,m_1,\ldots, m_{n-1},m_n=\nker_k(\mathcal F)+1$.
Suppose that $m_{n-1}<\nker_k(\mathcal F)$. Then 
$\mathcal F$ is also weakly periodic for sequence
$m_0=0,m_1,\ldots, m_{n-1},\nker_k(\mathcal F),m_n=\nker_k(\mathcal F)+1$.
\end{lemma}

\begin{proof}
The proof is completely similar to the proof of the previous lemma.
\end{proof}

Our next goal is to prove for every $k\in\NN$ that if all evolutions of $k$-blocks present in $\al$
are continuously periodic, then either there exists a $k$-series of obstacles in $\al$, or 
all evolutions of $(k+1)$-blocks present in $\al$ are continuously periodic. In order to prove this, 
we prove several lemmas.

%
%

Let $\mathcal F$ and $\mathcal F'$ be two consecutive evolutions of stable nonempty $k$-multiblocks ($k\ge 1$).
Suppose that $\mathcal F$ is weakly periodic for a sequence of indices 
$m_0=0,m_1,\ldots,m_{n-1}=\nker_k(\mathcal F),m_n=\nker_k(\mathcal F)+1$ and 
$\mathcal F'$ is weakly periodic for a sequence of indices 
$m'_0=0,m'_1=1,\ldots,m'_{n'-1},m'_{n'}=\nker_{k,m}(\mathcal F')+1$.
Denote the concatenation of $\mathcal F$ and $\mathcal F'$ by $\mathcal F''$.

Consider the following sequence: $m''_0=0,m''_1=m_1,\ldots, m''_{n-1}=m_{n-1}=\nker_k(\mathcal F)=\nker_k(\mathcal F)-1+m'_1,
m''_n=\nker_k(\mathcal F)-1+m'_2,\ldots,m''_{n-2+n'-1}=\nker_k(\mathcal F)-1+m'_{n'-1},m''_{n-2+n'}=\nker_k(\mathcal F)-1+m'_{n'}=
\nker_k(\mathcal F)-1+\nker_k(\mathcal F')+1=\nker_k(\mathcal F'')+1$ (the last equality is Lemma \ref{concatcompker}).
In other words, we did the following. We removed the last entry $\nker_k(\mathcal F)+1$ from the sequence $m$ 
and the first entry 0 from the sequence $m'$. Then we added $\nker_k(\mathcal F)-1$ to each remaining entry 
of the second sequence. The last remaining entry 
in the first sequence now coincides with the first remaining entry in the second sequence and 
equals $\nker_k(\mathcal F)$. We remove one of these two coinciding entries 
and take the concatenation of the resulting two sequences.

We are going to prove that $\mathcal F''$ is weakly periodic for the 
sequence $m''_0,\ldots,m''_{n-2+n'}$.

\begin{lemma}\label{concatcontper}
Let $\mathcal F$ and $\mathcal F'$ be two consecutive evolutions of stable nonempty $k$-multiblocks ($k\ge 1$)
such that $\mathcal F$ is weakly periodic for a sequence of indices 
$m_0=0,m_1,\ldots,m_{n-1}=\nker_k(\mathcal F),m_n=\nker_k(\mathcal F)+1$ and 
$\mathcal F'$ is weakly periodic for a sequence of indices 
$m'_0=0,m'_1=1,\ldots,m'_{n'-1},m'_{n'}=\nker_{k,m}(\mathcal F')+1$.

Then the concatenation $\mathcal F''$ of $\mathcal F$ and $\mathcal F'$ is weakly periodic
for the sequence 
$m''_0=0,m''_1=m_1,\ldots, m''_{n-1}=m_{n-1},
m''_n=\nker_k(\mathcal F)-1+m'_2,\ldots,m''_{n-2+n'-1}=\nker_k(\mathcal F)-1+m'_{n'-1},m''_{n-2+n'}=\nker_k(\mathcal F)-1+m'_{n'}=
\nker_k(\mathcal F)-1+\nker_k(\mathcal F')+1=\nker_k(\mathcal F'')+1$.
\end{lemma}

\begin{proof}
Let $\lambda^{(0)},\ldots,\lambda^{(n-1)}$ be final periods
such that $\lambda^{(i)}$ is a left weak evolutional period of $\mathcal F$ 
for pair $(m_i,m_{i+1})$, and let 
$\lambda'^{(0)},\ldots,\lambda'^{(n'-1)}$ 
be final periods 
such that $\lambda'^{(i)}$ is a left weak evolutional period of $\mathcal F'$ 
for pair $(m'_i,m'_{i+1})$.

Consider the following sequence of final periods: 
$\lambda''^{(0)}=\lambda^{(0)},\ldots,\lambda''^{(n-2)}=\lambda^{(n-2)},\lambda''^{(n-2+1)}=\lambda'^{(1)},\lambda''^{(n-2+n'-1)}=\lambda'^{(n'-1)}$.
In other words, we the sequence of final periods
$\lambda^{(0)},\ldots,\lambda^{(n-1)}$, removed the last entry, 
took the sequence of final periods
$\lambda'^{(0)},\ldots,\lambda'^{(n'-1)}$,
removed the first entry, and took the concatenation of the resulting two sequences.

Set $n''=n-2+n'$.
By Lemma \ref{concatcompker} we see that if $0\le i<n-1$, then $\IpR_{k,m''_i,m''_{i+1}}(\mathcal F''_l)=
\IpR_{k,m_i,m_{i+1}}(\mathcal F_l)$ as an occurrence in $\al$ for all $l\ge 0$. And if $n-1\le i<n''$, 
then $\IpR_{k,m''_i,m''_{i+1}}(\mathcal F''_l)=\IpR_{k,m'_{i-(n-2)},m'_{i+1-(n-2)}}(\mathcal F'_l)$
as an occurrence in $\al$ for all $l\ge 0$.
So, if $0\le i<n-1$, then 
$\lambda''^{(i)}=\lambda^{(i)}$ 
is a left weak evolutional period of $\mathcal F$ for pair 
$(m_i,m_{i+1})$, hence it is also a 
a left weak evolutional period of $\mathcal F''$ for pair 
$(m''_i,m''_{i+1})$. And if $n-1\le i<n''$, 
then $\lambda''^{(i)}=\lambda'^{(i-(n-2))}$ 
is a left weak evolutional period of $\mathcal F'$ for pair 
$(m'_{i-(n-2)},m'_{i+1-(n-2)})$, hence 
it is also a 
a left weak evolutional period of $\mathcal F''$ for pair 
$(m''_i,m''_{i+1})$.
Therefore, $\mathcal F''$ is weakly periodic 
for the sequence 
$m''_0,\ldots,m''_{n''}$.
\end{proof}

\begin{corollary}\label{concatcontpergen}
Let $\mathcal F$ and $\mathcal F'$ be two consecutive evolutions of stable nonempty $k$-multiblocks ($k\ge 1$).
If $\mathcal F$ and $\mathcal F'$ are weakly periodic, then 
the concatenation of $\mathcal F$ and $\mathcal F'$ is also weakly periodic.
\end{corollary}

\begin{proof}
Let $m_0=0,m_1,\ldots,m_{n-1},m_n=\nker_k(\mathcal F)+1$ be a sequence of indices such that 
$\mathcal F$ is weakly periodic for $m_0,\ldots, m_n$.
Let $m'_0=0,m'_1,\ldots,m'_{n'-1},m'_{n'}=\nker_{k,m}(\mathcal F')+1$ be a sequence of indices such that 
$\mathcal F'$ is weakly periodic for $m'_0,\ldots, m'_{n'}$.
By Lemma \ref{weakcontperinsright},
without loss of generality (possibly increasing $n$ by 1), 
we may suppose that $m_{n-1}=\nker_k(\mathcal F)$.
Similarly, by Lemma \ref{weakcontperinsleft}, possibly increasing $n'$ by 1, without loss of generality 
we may suppose that $m'_1=1$. The claim now follows from Lemma \ref{concatcontper}.
\end{proof}

\begin{lemma}\label{multweakper}
Let $k\in\NN$.
Suppose that all evolutions of $k$-blocks in $\al$ are continuously periodic.
Let $\mathcal F$ be an evolution of nonempty stable $k$-multiblocks.
Then $\mathcal F$ is weakly periodic.
\end{lemma}

\begin{proof}
If each $k$-multiblock in $\mathcal F$ consists of a single periodic letter of order $k+1$, 
then the claim is clear by Remark \ref{nogrowthcontper}.

If each $k$-multiblock in $\mathcal F$ consists of a single $k$-block, then there exists 
an evolution $\mathcal E$ of $k$-blocks and a number $l_0\ge 3k$ such that 
$\mathcal F_l$ consists of $\mathcal E_{l+l_0}$ for all $l\ge 0$. 
By assumption, $\mathcal E$ is continuously periodic, and there exists an 
index $m$ ($1\le m\le \ncker_k(\mathcal E)$)
such that $\mathcal E$ is continuously periodic for the index $m$.
Now it follows from Remark \ref{evperblockmulrel}
that there exists an index $m'$ (it can equal $m$ or $m+1$)
such that $\mathcal F$ is weakly periodic for the sequence $0, m', \nker_k(\mathcal F)+1$.

Finally, the claim for general $\mathcal F$ follows from Lemma \ref{concatcontper}.
\end{proof}

\begin{lemma}\label{contperexpandright}
Let $\mathcal F$ be an evolution of stable $k$-multiblocks.
Suppose that there exists a final period $\lambda$ such that 
$\lambda$ is a left weak evolutional period of $\mathcal F$ 
for a pair $(m,m')$ ($0\le m<m'\le \nker_k(\mathcal F)$). 
Suppose also that there exists a final period $\mu$ such that 
$\mu$ is a left weak evolutional period of $\mathcal F$ 
for a pair $(m',m'')$ ($m'<m''\le \nker_k(\mathcal F)+1$). 
Then exactly one of the following two statements is true:
\begin{enumerate}
\item $\lambda$ is a left weak evolutional period of $\mathcal F$ 
for pair $(m,m'')$.
\item There exists a number $s\in\NN$ such that the following is true for all $l\ge 0$.
Suppose that $\Fg(\mathcal F_l)=\swa ij$ and $\IpR_{k,m,m'}(\mathcal F_l)=\swa {i'}{j'}$. Then:
\begin{enumerate}
\item $\psi(\swa{i'}{j'+s-1})$ is a 
weakly left $\lambda$-periodic word, 
and 
$\psi(\swa{i'}{j'+s})$ is not a 
weakly left $\lambda$-periodic word.
\item $s\le |\Ker_{k,m'}(\mathcal F)|+2\finmax$.
\item $j'+s\le j$, i.~e.\ $\al_{j'+s}$ is a letter in $\Fg(\mathcal F_l)$.
\end{enumerate}
\end{enumerate}
\end{lemma}

\begin{proof}
Denote the remainder of $|\IpR_{k,m,m'}(\mathcal F_l)|$ modulo $|\lambda|$ by $r$ 
(by the definition of a weak left evolutional period, $r$ does not depend on $l$).
Set $\lambda'=\Cyc_r(\lambda)$.
By Remark \ref{contpermultcycle}, this is a right weak evolutional 
period of $\mathcal F$ for the pair $(m,m')$.

We have two possibilities for $\Ker_{k,m'}(\mathcal F)$: either 
$\psi(\Ker_{k,m'}(\mathcal F))$ is a 
weakly left $\lambda'$-periodic,
or $\psi(\Ker_{k,m'}(\mathcal F))$ is not 
weakly left $\lambda'$-periodic.

If $\psi(\Ker_{k,m'}(\mathcal F))$ is not 
weakly left $\lambda'$-periodic, 
denote by 
$\delta$ the 
weakly left $\lambda'$-periodic word of length $|\Ker_{k,m'}(\mathcal F)|$,
and denote by 
$t$ the minimal 
("the leftmost") index such that 
$\delta_t\ne \psi(\Ker_{k,m'}(\mathcal F))_t$. Then 
$\psi(\Ker_{k,m'}(\mathcal F)_{0\ldots t-1})$ is 
a weakly left $\lambda'$-periodic word,
and 
$\psi(\Ker_{k,m'}(\mathcal F)_{0\ldots t})$ is not a 
weakly left $\lambda'$-periodic word.
Hence, 
$\psi(\IpR_{k,m,m'}(\mathcal F_l)\Ker_{k,m'}(\mathcal F)_{0\ldots t-1})$ is a 
weakly left $\lambda$-periodic word
for all $l\ge 0$, and 
$\psi(\IpR_{k,m,m'}(\mathcal F_l)\Ker_{k,m'}(\mathcal F)_{0\ldots t})$ is not a 
weakly left $\lambda$-periodic word
for all $l\ge 0$.
Set $s=t+1$. 
If $\IpR_{k,m,m'}(\mathcal F_l)=\swa {i'}{j'}$, then 
$\IpR_{k,m,m'}(\mathcal F_l)\Ker_{k,m'}(\mathcal F)_{0\ldots t-1}=\swa{i'}{j'+s-1}$
and 
$\IpR_{k,m,m'}(\mathcal F_l)\Ker_{k,m'}(\mathcal F)_{0\ldots t}=\swa{i'}{j'+s}$
as an occurrence in $\al$. The largest possible value of $t$ is 
$|\Ker_{k,m'}(\mathcal F)|-1$, so $s\le |\Ker_{k,m'}(\mathcal F)|\le |\Ker_{k,m'}(\mathcal F)|+2\finmax$.
$\al_{j'+s}$ is a letter in $\Ker_{k,m'}(\mathcal F_l)$, so $j'+s\le j$.

Suppose now that $\psi(\Ker_{k,m'}(\mathcal F))$ is 
weakly left $\lambda'$-periodic.
Then for all $l\ge 0$, 
$\psi(\IpR_{k,m,m'}(\mathcal F_l)\Ker_{k,m'}(\mathcal F))$ 
is 
a weakly left $\lambda$-periodic word.
If $m'=\nker_k(\mathcal F)$, then $m''=\nker_k(\mathcal F)+1$, and we are done since 
in this case $\IpR_{k,m,m'}(\mathcal F_l)\Ker_{k,m'}(\mathcal F)=\IpR_{k,m,m''}(\mathcal F_l)$
for all $l\ge 0$, and $|\Ker_{k,m'}(\mathcal F)|$ and the remainder of 
$|\Ker_{k,m'}(\mathcal F)|$ modulo $|\lambda|$ do not depend on $l$.

Otherwise, 
denote 
$\lambda''=\Cyc_{|\Ker_{k,m'}(\mathcal F)|}(\lambda')$.
Then $\psi(\Ker_{k,m'}(\mathcal F))$ is a 
weakly right $\lambda''$-periodic word
and 
$\psi(\IpR_{k,m,m'}(\mathcal F_l)\Ker_{k,m'}(\mathcal F))$ 
also is a 
weakly right $\lambda''$-periodic word.

Denote by $\delta'$ 
the weakly left $\lambda''$-periodic word of length $2\finmax$, 
and denote by $\gamma$ 
the weakly left $\mu$-periodic word 
of length $2\finmax$
(recall that $\mu$ is a left weak evolutional period of $\mathcal F$ 
for the pair $(m',m'')$). Now we are considering the 
case when $m'<\nker_k(\mathcal F)$, and we also have $m''>m'>m$, so
$m''>1$, and 
by Corollary \ref{betweenkernelgrowsgood},
$|\IpR_{k,m',m''}(\mathcal F_l)|\ge 2\finmax$ 
for all $l\ge 0$.
Since $\mu$ is a left weak evolutional period of $\mathcal F$ 
for the pair $(m',m'')$, $\gamma$ is a prefix of $\psi(\IpR_{k,m',m''}(\mathcal F_l))$ for all 
$l\ge 0$. Again, we have two possibilities: $\gamma=\delta'$ or $\gamma\ne\delta'$.

If $\gamma\ne\delta'$, denote by $t$ the smallest index such that $\gamma_t\ne \delta'_t$.
Then 
$\gamma_{0\ldots t-1}$ is 
weakly left $\lambda''$-periodic,
and $\gamma_{0\ldots t}$ is not 
weakly left $\lambda''$-periodic.
Fix a number $l\ge 0$. Let $i$ and $j$ be the indices such that $\Fg(\mathcal F_l)=\swa ij$, 
and let $i'$ and $j'$ be the indices 
such that $\IpR_{k,m,m'}(\mathcal F_l)=\swa{i'}{j'}$.
Recall that $\gamma$ is a prefix of $\psi(\IpR_{k,m',m''}(\mathcal F_l))$, 
that $\psi(\IpR_{k,m,m'}(\mathcal F_l)\Ker_{k,m'}(\mathcal F))$ 
is a weakly right $\lambda''$-periodic word 
and is a 
weakly left $\lambda$-periodic word for all $l\ge 0$.
Hence, $\psi(\swa{i'}{j'+|\Ker_{k,m'}(\mathcal F)|+t})=\psi(\IpR_{k,m,m'}(\mathcal F_l)\Ker_{k,m'}(\mathcal F))\gamma_{0\ldots t-1}$
is 
a weakly left $\lambda$-periodic word, 
and 
$\psi(\swa{i'}{j'+|\Ker_{k,m'}(\mathcal F)|+t+1})=\psi(\IpR_{k,m,m'}(\mathcal F_l)\Ker_{k,m'}(\mathcal F))\gamma_{0\ldots t}$
is 
not a weakly left $\lambda$-periodic word.
So, we can set $s=|\Ker_{k,m'}(\mathcal F)|+t+1$ and see that the claim is true in this case
since $t$ does not depend on $l$, $s=|\Ker_{k,m'}(\mathcal F)|+t+1\le |\Ker_{k,m'}(\mathcal F)|+|\gamma|=|\Ker_{k,m'}(\mathcal F)|+2\finmax$,
and $\gamma$ and hence $\gamma_{0\ldots t}$ are prefixes of 
$\psi(\IpR_{k,m',m''}(\mathcal F_l))$, so $j'+s\le j$.

Finally, suppose that $\gamma=\delta'$.
Since $\lambda''$ and $\mu$ are final periods, 
none of them can be written as a word repeated more than once by Lemma \ref{finalnomorethanonce}. 
Then Lemma \ref{finwordperiods}
implies that $\lambda''=\mu$.
Hence, $\psi(\IpR_{k,m',m''}(\mathcal F_l))$
is weakly left $\lambda''$-periodic
for all $l\ge 0$
and the remainder of $|\IpR_{k,m',m''}(\mathcal F_l)|$ modulo $|\lambda''|$ does not depend on 
$l$. 
Again, recall that 
$\psi(\IpR_{k,m,m'}(\mathcal F_l)\Ker_{k,m'}(\mathcal F))$ 
is a weakly right $\lambda''$-periodic word
and 
is a 
weakly left $\lambda$-periodic word for all $l\ge 0$.
Therefore, 
$\psi(\IpR_{k,m,m'}(\mathcal F_l)\Ker_{k,m'}(\mathcal F)\IpR_{k,m',m''}(\mathcal F_l))=\psi(\IpR_{k,m,m''}(\mathcal F_l))$
is also a 
weakly left $\lambda$-periodic word.
The remainder of 
$\IpR_{k,m,m''}(\mathcal F_l)=|\IpR_{k,m,m'}(\mathcal F_l)|+|\Ker_{k,m'}(\mathcal F)|+|\IpR_{k,m',m''}(\mathcal F_l)|$
modulo $|\lambda|$ does not depend on $l$ since the remainder of each summand modulo
$|\lambda|=|\lambda''|$ does not depend on $l$.
So, in this case $\lambda$ is a weak left evolutional period of $\mathcal F$ for the pair $(m,m'')$
\end{proof}

\begin{corollary}\label{contperexpandstartright}
Let $\mathcal F$ be an evolution of stable $k$-multiblocks.
Suppose that there exist a final period $\lambda$ such that 
$\lambda$ is a right weak evolutional period of $\mathcal F$ 
for a pair $(m,m')$ ($0\le m<m'\le \nker_k(\mathcal F)$). 
and $\lambda$ is a right weak evolutional period of $\mathcal F$ 
for a pair $(m_0,m')$ ($0\le m_0<m'\le \nker_k(\mathcal F)$). 

Suppose also that there exists a final period $\mu$ such that 
$\mu$ is a left weak evolutional period of $\mathcal F$ 
for a pair $(m',m'')$ ($m'<m''\le \nker_k(\mathcal F)+1$). 

Then there exists a left weak evolutional period of $\mathcal F$ 
for pair $(m,m'')$ if and only if 
there exists a left weak evolutional period of $\mathcal F$ 
for pair $(m_0,m'')$.
\end{corollary}

\begin{proof}

If $m=m_0$, then everything is clear, so suppose that $m\ne m_0$.
Then $m'$ cannot be equal to 1, and $m'>1$, hence, $m''>1$. Also, $m'<m''$, so 
$m'\le \nker_k(\mathcal F)$, hence $m<\nker_k(\mathcal F)$ 
and $m_0<\nker_k(\mathcal F)$.

Denote by $r$ the remainder of $|\IpR_{k,m,m'}(\mathcal F_l)|$
modulo $|\lambda|$ (for any $l\ge 0$) and set 
$\lambda'=\Cyc_{-r}(\lambda)$.
By Remark \ref{contpermultcycle},
$\lambda'$ is a left weak evolutional period of $\mathcal F$ 
for the pair $(m,m')$.
By Lemma \ref{contperweakuniqueleft}, if there exists a weak left evolutional
period of $\mathcal F$ for the pair $(m,m'')$, it equals $\lambda'$.
By Lemma \ref{contperexpandright}, 
$\lambda'$ is not a left weak evolutional period of $\mathcal F$ 
for the pair $(m,m'')$ if and only if there exists $s\in\NN$ such that the following is true:

For each $l\ge 0$, suppose that 
$\Fg(\mathcal F_l)=\swa ij$ and $\IpR_{k,m,m'}(\mathcal F_l)=\swa {i'}{j'}$. Then:
\begin{enumerate}
\item\label{perstucks} $\psi(\swa{i'}{j'+s-1})$ is 
a weakly left $\lambda'$-periodic word, 
and 
$\psi(\swa{i'}{j'+s})$ is not 
a weakly left $\lambda'$-periodic word.
\item $s\le |\Ker_{k,m'}(\mathcal F)|+2\finmax$.
\item $j'+s\le j$, i.~e.\ $\al_{j'+s}$ is a letter in $\Fg(\mathcal F_l)$.
\end{enumerate}

Since $r$ is the remainder of $|\IpR_{k,m,m'}(\mathcal F_l)|$
modulo $|\lambda|=|\lambda'|$, 
$\lambda=\Cyc_r(\lambda')$, 
and 
$\psi(\swa{i'}{j'})$ is 
a weakly left $\lambda'$-periodic 
word, 
Condition \ref{perstucks}
in the list above is equivalent to the following condition:
$\psi(\swa{j'+1}{j'+s-1})$ is 
a weakly left $\lambda$-periodic word, 
and 
$\psi(\swa{j'+1}{j'+s})$ is not 
a weakly left $\lambda$-periodic word.

Therefore, a weak left evolutional period of $\mathcal F$ for the pair $(m,m'')$
does not exist if and only if there exists $s\in\NN$ such that the following is true:

For each $l\ge 0$, suppose that 
$\Fg(\mathcal F_l)=\swa ij$ and $\IpR_{k,m,m'}(\mathcal F_l)=\swa {i'}{j'}$. Then:
\begin{enumerate}
\item 
$\psi(\swa{j'+1}{j'+s-1})$ is 
a weakly left $\lambda$-periodic word, 
and 
$\psi(\swa{j'+1}{j'+s})$ is not 
a weakly left $\lambda$-periodic word.
\item $s\le |\Ker_{k,m'}(\mathcal F)|+2\finmax$.
\item $j'+s\le j$, i.~e.\ $\al_{j'+s}$ is a letter in $\Fg(\mathcal F_l)$.
\end{enumerate}

But these conditions do not use the indices $m$ and $i'$, so we can repeat the 
arguments above for $m_0$ instead of $m$ and conclude that the nonexistence of a 
weak left evolutional period of $\mathcal F$ for the pair $(m_0,m'')$
is equivalent to the same list of conditions.
\end{proof}

\begin{lemma}\label{contperexpandleft}
Let $\mathcal F$ be an evolution of stable $k$-multiblocks.
Suppose that there exists a final period $\lambda$ such that 
$\lambda$ is a right weak evolutional period of $\mathcal F$ 
for a pair $(m',m)$ ($1\le m'<m\le \nker_k(\mathcal F)+1$). 
Suppose also that there exists a final period $\mu$ such that 
$\mu$ is a right weak evolutional period of $\mathcal F$ 
for a pair $(m'',m')$ ($0\le m''<m'$). 
Then one of the following two statements is true:
\begin{enumerate}
\item $\lambda$ is a right weak evolutional period of $\mathcal F$ 
for pair $(m'',m)$.
\item There exists a number $s\in\NN$ such that the following is true for all $l\ge 0$.
Suppose that $\Fg(\mathcal F_l)=\swa ij$ and $\IpR_{k,m',m}(\mathcal F_l)=\swa {i'}{j'}$. Then:
\begin{enumerate}
\item $\psi(\swa{i'-s+1}{j'})$ is 
a weakly right $\lambda$-periodic word, 
and 
$\psi(\swa{i'-s}{j'})$ is not 
a weakly right $\lambda$-periodic word.
\item $s\le |\Ker_{k,m}(\mathcal F)|+2\finmax$.
\item $i'-s\ge i$, i.~e.\ $\al_{i'-s}$ is a letter in $\Fg(\mathcal F_l)$.
\end{enumerate}
\end{enumerate}
\end{lemma}

\begin{proof}
The proof is completely symmetric to the proof of Lemma \ref{contperexpandright}.
\end{proof}

\begin{corollary}\label{contperexpandstartleft}
Let $\mathcal F$ be an evolution of stable $k$-multiblocks.
Suppose that there exist a final period $\lambda$ such that 
$\lambda$ is a left weak evolutional period of $\mathcal F$ 
for a pair $(m',m)$ ($1\le m'<m\le \nker_k(\mathcal F)+1$). 
and $\lambda$ is a left weak evolutional period of $\mathcal F$ 
for a pair $(m',m_0)$ ($1\le m'<m_0\le \nker_k(\mathcal F)+1$). 

Suppose also that there exists a final period $\mu$ such that 
$\mu$ is a right weak evolutional period of $\mathcal F$ 
for a pair $(m'',m')$ ($0\le m''<m'$). 

Then there exists a right weak evolutional period of $\mathcal F$ 
for pair $(m'',m)$ if and only if 
there exists a right weak evolutional period of $\mathcal F$ 
for pair $(m'',m_0)$.
\end{corollary}

\begin{proof}
The proof is completely symmetric to the proof of Corollary \ref{contperexpandstartright}.
\end{proof}

\begin{lemma}\label{contperexpandmultiright}
Let $\mathcal F$ be an evolution of stable nonempty $k$-multiblocks ($k\ge 1$). 
Suppose that $\mathcal F$ is weakly periodic for 
a sequence of 
indices $m_0=0,m_1,\ldots, m_{n-1},m_n=\nker_k(\mathcal F)+1$, where $n\ge 2$.
Let $q$ be an index ($1\le q\le n-1$) and let $\lambda$ be a right weak 
evolutional period of $\mathcal F$ for pair $(m_q,m_{q+1})$.
Suppose also that $\lambda$ is not a 
right weak 
evolutional period of $\mathcal F$ for pair $(m_{q-1},m_{q+1})$.
Then there are two possibilities:
\begin{enumerate}
\item $\mathcal F$ is weakly periodic for the following sequence of indices:
$m_0=0,m_1,\ldots, m_{q-1},m_q,\nker_k(\mathcal F)+1$.
\item There exists a $k$-series of obstacles in $\al$.
\end{enumerate}
\end{lemma}

\begin{proof}
If $q=n-1$, then everything is clear. Suppose that $q<n-1$.

Since $\lambda$ is not a 
right weak 
evolutional period of $\mathcal F$ for the pair $(m_{q-1},m_{q+1})$, it follows from 
Lemma \ref{contperexpandleft} that there exists $s\in\NN$ ($s$ does not depend on $l$) 
such that for all $l\ge 0$, if $\IpR_{k,m_q,m_{q+1}}(\mathcal F_l)=\swa{i'}{j'}$, then 
$\psi(\swa{i'-s+1}{j'})$ is 
a weakly right $\lambda$-periodic word, 
and 
$\psi(\swa{i'-s}{j'})$ is not 
a weakly right $\lambda$-periodic word.
Again, denote the residue of $|\IpR_{k,m_q,m_{q+1}}(\mathcal F_l)|$ for any $l\ge 0$
by $r$ and denote 
$\lambda'=\Cyc_{-r}(\lambda)$.
Then $\lambda'$ is a weak left evolutional period of $\mathcal F$ for pair $(m_q,m_{q+1})$, and,
if $\IpR_{k,m_q,m_{q+1}}(\mathcal F_l)=\swa{i'}{j'}$, then 
$\psi(\swa{i'-s+1}{i'-1})$ is 
a weakly right $\lambda'$-periodic word, 
and 
$\psi(\swa{i'-s}{i'-1})$ is not 
a weakly right $\lambda'$-periodic word.

Now let $q'$ be the maximal index ($q+1\le q'\le n$) such that $\lambda'$ is a weak left 
evolutional period of $\mathcal F$ for pair $(m_q,m_{q'})$. If $q'=n$, then
$\mathcal F$ is weakly periodic for the sequence 
$m_0=0,m_1,\ldots, m_{q-1},m_q,\nker_k(\mathcal F)+1$, and 
we are done. 

Suppose that $q'<n$. Then, since $q'$ is maximal,
it follows from Lemma \ref{contperexpandright} that there exists $s'\in\NN$ ($s'$ does not depend on $l$) 
such that for all $l\ge 0$, if $\IpR_{k,m_q,m_{q'}}(\mathcal F_l)=\swa{i''}{j''}$, 
then 
$\psi(\swa{i''}{j''+s'-1})$ is 
a weakly left $\lambda'$-periodic word, 
and 
$\psi(\swa{i''}{j''+s'})$ is not 
a weakly left $\lambda'$-periodic word.
Note also that if $\IpR_{k,m_q,m_{q+1}}(\mathcal F_l)=\swa{i'}{j'}$ and $\IpR_{k,m_q,m_{q'}}(\mathcal F_l)=\swa{i''}{j''}$, 
then, by the definition of an inner pseudoregular part, $i'=i''$.

Summarizing, we have the following weakly periodic and non-weakly periodic words.
Fix $l\ge 0$ and let $i'$ and $j''$ be the indices such that $\IpR_{k,m_q,m_{q'}}(\mathcal F_l)=\swa{i'}{j''}$.
The following two occurrences in $\psi(\al)$ are weakly $|\lambda'|$-periodic: 
$\psi(\swa{i'-s+1}{i'-1})$ with right period $\lambda'$ and 
$\psi(\swa{i'}{j''+s'-1})$ with left period $\lambda'$.
And the following two occurrences are \textbf{not} weakly $|\lambda'|$-periodic with right and left period $\lambda'$, respectively:
$\psi(\swa{i'-s}{i'-1})$ with right period $\lambda'$ and 
$\psi(\swa{i'}{j''+s'})$ with left period $\lambda'$.

Denote $\lambda''=\Cyc_{-(s-1)}(\lambda')$.
Then $\psi(\swa{i'-s+1}{j''+s'-1})$ is 
a weakly left $\lambda''$-periodic word,
$\al_{i'-s}\ne\lambda''_{|\lambda''|-1}$, 
and $\psi(\swa{i'-s+1}{j''+s'})$ is not 
a weakly left $\lambda''$-periodic word.
In other words, if $r'$ is the residue of $(j''+s')-(i'-s+1)$ modulo
$|\lambda''|$, then $\psi(\al_{j''+s'})\ne \lambda''_{r'}$.

Set $p=|\lambda''|=|\lambda|$ and set $\mathcal H_l=\swa{i'-s+1}{j''+s'-1}$. Let us prove that 
$\mathcal H=\mathcal H_0,\mathcal H_1,\mathcal H_2,\ldots$ is a $k$-series of obstacles in $\al$.
We have $|\mathcal H_l|=|\swa{i'-s+1}{i'-1}|+|\swa{i'}{j''}|+|\swa{j''+1}{j''+s'-1}|=(s-1)+|\IpR_{k,m_q,m_{q'}}(\mathcal F_l)|+(s'-1)$.
By Lemma \ref{betweenkernelgrows} (here we use the fact that $0<q<q'<n$, so $1\le m_q<m_{q'}\le\nker_k(\mathcal F)$), 
$|\IpR_{k,m_q,m_{q'}}(\mathcal F_l)|$ strictly grows as $l$ grows, 
$|\IpR_{k,m_q,m_{q'}}(\mathcal F_l)|\ge 2\finmax$, and 
there exists $k'\in \NN$ ($1\le k'\le k$) such that $|\IpR_{k,m_q,m_{q'}}(\mathcal F_l)|$ is $\Theta(l^{k'})$ for $l\to \infty$.
Since $s,s'\in\NN$ do not depend on $l$, $|\mathcal H_l|$ also strictly grows as $l$ grows, 
$|\mathcal H_l|\ge 2\finmax$, and $|\mathcal H_l|$ is $\Theta(l^{k'})$ for $l\to \infty$.

Now let us check the required weak $p$-periodicity for the definition of a $k$-series of obstacles.
We already know that $\psi(\mathcal H_l)$ is a weakly $p$-periodic word with left period $\lambda''$.
Since $\lambda$ is a final period, $|\mathcal H_l|\ge 2\finmax\ge|\lambda|=|\lambda''|=p$, 
and $\psi(\al_{i'-s+1+p-1})=\lambda''_{p-1}$, while $\psi(\al_{i'-s})\ne\lambda''_{p-1}$, so 
$\psi(\swa{i'-s}{j''+s'-1})$ is not a weakly $p$-periodic word. 
Similarly, if $r'$ is the residue of $(j''+s')-(i'-s+1)$ modulo
$|\lambda''|$, then $\psi(\al_{j''+s'})\ne \lambda''_{r'}$, but 
since $|\mathcal H_l|\ge p$, we have $\psi(\al_{j''+s'-p})=\lambda''_{r'}$, 
so $\psi(\swa{i'-s+1}{j''+s'})$ is not a weakly $p$-periodic word. 
(We knew before that 
$\psi(\swa{i'-s+1}{j''+s'})$ is not a weakly $p$-periodic word with 
left period $\lambda''$, but it is important here that $|\mathcal H_l|\ge p$, 
otherwise we could have $|\psi(\swa{i'-s+1}{j''+s'})|\le p$, and then 
$\psi(\swa{i'-s+1}{j''+s'})$ would be a weakly $p$-periodic word with another left period.)
\end{proof}

\begin{lemma}\label{contperexpandmultileft}
Let $\mathcal F$ be an evolution of stable nonempty $k$-multiblocks ($k\ge 1$). 
Suppose that $\mathcal F$ is weakly periodic for 
a sequence of 
indices $m_0=0,m_1,\ldots, m_{n-1},m_n=\nker_k(\mathcal F)+1$, where $n\ge 2$.
Let $q$ be an index ($1\le q\le n-1$) and let $\lambda$ be a left weak 
evolutional period of $\mathcal F$ for pair $(m_{q-1},m_q)$.
Suppose also that $\lambda$ is not a 
left weak 
evolutional period of $\mathcal F$ for pair $(m_{q-1},m_{q+1})$.
Then there are two possibilities:
\begin{enumerate}
\item $\mathcal F$ is weakly periodic for the following sequence of indices:
$0,m_q,m_{q+1},\ldots,m_{n-1},m_n=\nker_k(\mathcal F)+1$.
\item There exists a $k$-series of obstacles in $\al$.
\end{enumerate}
\end{lemma}

\begin{proof}
The proof is completely symmetric to the proof of the previous lemma.
\end{proof}

\begin{lemma}\label{contperexpandtotalright}
Let $\mathcal F$ and $\mathcal F'$ be two consecutive evolutions of stable nonempty $k$-multiblocks ($k\ge 1$), 
and let $\mathcal F''$ be their concatenation. 
Suppose that $\nker_k(\mathcal F)>1$, and $\mathcal F$ is weakly periodic for 
a sequence of 
indices $m_0=0,m_1,\ldots, m_{n-1},m_n=\nker_k(\mathcal F)+1$, where 
$m_{n-1}<\nker_k(\mathcal F)$
($n\in \NN$, and $n=1$ is allowed).
Suppose that $\mathcal F''$ is also weakly 
periodic for the following sequence:
$m'_0=m_0=0,m'_1=m_1,\ldots, m'_{n-1}=m_{n-1},m'_n=\nker_k(\mathcal F'')+1$.

Then 
$\mathcal F'$ is totally periodic, moreover, 
if $\lambda$ is the (unique by Corollary \ref{contperweakuniqueright}) 
right weak evolutional period of $\mathcal F$ for 
the pair $(m_{n-1},\nker_k(\mathcal F)+1)$, then 
$\lambda$ is also 
a total left evolutional period 
of $\mathcal F'$.
\end{lemma}

\begin{proof}
Let $\lambda'$ be the (unique by Corollary \ref{contperweakuniqueleft})
left weak evolutional period of $\mathcal F$ for the pair 
$(m_{n-1},\nker_k(\mathcal F)+1)$. By Lemma \ref{weakcontperinsright}, 
$\mathcal F$ is weakly periodic for the sequence
$m_0,m_1,\ldots, m_{n-1},\nker_k(\mathcal F),\nker_k(\mathcal F)+1$, 
and by Lemma \ref{contperweakuniquewideleft}, 
the left weak evolutional period of $\mathcal F$ for the pair 
$(m_{n-1},\nker_k(\mathcal F))$
also equals $\lambda'$. By Lemma \ref{concatcompker}, 
$\IpR_{k,m_{n-1}, \nker_k(\mathcal F)}(\mathcal F_l)=\IpR_{k,m_{n-1}, \nker_k(\mathcal F)}(\mathcal F''_l)$
as an occurrence in $\al$ for all $l\ge 0$, so 
$\lambda'$ is also 
the left weak evolutional period of $\mathcal F''$ for the pair 
$(m_{n-1},\nker_k(\mathcal F))$.
Now, by Lemma \ref{contperweakuniquewideleft} again, 
the left weak evolutional period of $\mathcal F''$ for the pair 
$(m_{n-1},\nker_k(\mathcal F'')+1)$ (it exists by assumption)
also equals $\lambda'$.

Denote the residue of $|\IpR_{k,m_{n-1},\nker_k(\mathcal F)+1}(\mathcal F_l)|$
modulo $|\lambda'|$ by $r$.
By Remark \ref{contpermultcycle}, 
$\lambda=\Cyc_r(\lambda')$.
By Lemma \ref{concatcompker}, 
$\IpR_{k,m_{n-1},\nker_k(\mathcal F'')+1}(\mathcal F''_l)=
\IpR_{k,m_{n-1},\nker_k(\mathcal F)+1}(\mathcal F_l)\Fg(\mathcal F'_l)$.
We know that $\psi(\IpR_{k,m_{n-1},\nker_k(\mathcal F'')+1}(\mathcal F''_l))$ 
is 
a weakly left $\lambda'$-periodic word, 
hence $\psi(\Fg(\mathcal F'_l))$ is 
a weakly left $\lambda$-periodic word.
Since $\lambda'$ is both the weak left evolutional period of $\mathcal F$ 
for the pair 
$(m_{n-1},\nker_k(\mathcal F)+1)$
and 
the weak left evolutional period of $\mathcal F''$
for the pair
$(m_{n-1},\nker_k(\mathcal F'')+1)$, 
the residues of 
$|\IpR_{k,m_{n-1},\nker_k(\mathcal F)+1}(\mathcal F_l)|$
and of
$|\IpR_{k,m_{n-1},\nker_k(\mathcal F'')+1}(\mathcal F''_l)|$
modulo $|\lambda'|=|\lambda|$
do not depend on $l$. Hence, the residue of 
$|\Fg(\mathcal F'_l)|=
|\IpR_{k,m_{n-1},\nker_k(\mathcal F'')+1}(\mathcal F''_l)|-
|\IpR_{k,m_{n-1},\nker_k(\mathcal F)+1}(\mathcal F_l)|$
modulo $|\lambda|$ does not depend on $l$, 
and $\lambda$ is a 
left weak evolutional period of $\mathcal F'$ 
for the pair $(0,\nker_k(\mathcal F')+1)$.
\end{proof}

\begin{lemma}\label{contperexpandtotalleft}
Let $\mathcal F'$ and $\mathcal F$ be two consecutive evolutions of stable nonempty $k$-multiblocks ($k\ge 1$), 
and let $\mathcal F''$ be their concatenation. 
Suppose that $\nker_k(\mathcal F)>1$, and $\mathcal F$ is weakly periodic for 
a sequence of 
indices $m_0=0,m_1,\ldots, m_{n-1},m_n=\nker_k(\mathcal F)+1$, where 
$m_1>1$
($n\in \NN$, and $n=1$ is allowed).
Suppose that $\mathcal F''$ is also weakly 
periodic for the following sequence:
$m'_0=m_0=0,m'_1=m_1+\nker_k(\mathcal F)-1,\ldots, m'_{n-1}=m_{n-1}+\nker_k(\mathcal F)-1,m'_n=\nker_k(\mathcal F'')+1$.

Then 
$\mathcal F'$ is totally periodic, moreover, 
if $\lambda$ is the (unique by Corollary \ref{contperweakuniqueleft}) 
left weak evolutional period of $\mathcal F$ for 
the pair $(0,m_1)$, then 
$\lambda$ is also 
a total right evolutional period 
of $\mathcal F'$.
\end{lemma}

\begin{proof}
The proof is completely similar to the proof of the previous lemma.
\end{proof}

\begin{lemma}\label{contperexpandtotalrightgood}
Let $k\in\NN$.
Suppose that all evolutions of $k$-blocks in $\al$ are continuously periodic.
Let $\mathcal F$ and $\mathcal F'$ be two consecutive evolutions of stable nonempty $k$-multiblocks, 
and let $\mathcal F''$ be their concatenation. 
Suppose that $\nker_k(\mathcal F)>1$, and $\mathcal F$ is weakly periodic for 
a sequence of 
indices $m_0=0,m_1,\ldots, m_{n-1},m_n=\nker_k(\mathcal F)+1$, where $n\ge 2$, and 
$m_{n-1}<\nker_k(\mathcal F)$.
Let $\lambda$ be the (unique by Corollary \ref{contperweakuniqueright}) 
right weak evolutional period of $\mathcal F$ for 
the pair $(m_{n-1},\nker_k(\mathcal F)+1)$
Suppose also that $\lambda$ is not a 
right weak 
evolutional period of $\mathcal F$ for pair $(m_{n-2},m_n)$.
Then there are two possibilities:
\begin{enumerate}
\item $\mathcal F'$ is totally periodic, moreover, 
$\lambda$ is 
a left total evolutional period 
of $\mathcal F'$.
\item There exists a $k$-series of obstacles in $\al$.
\end{enumerate}
\end{lemma}

\begin{proof}
By Lemma \ref{multweakper},
$\mathcal F'$ is weakly periodic for some 
sequence of indices $m'_0=0,m'_1,\ldots, m'_{n'-1},m'_{n'}=\nker_k(\mathcal F')+1$.
By Lemma \ref{weakcontperinsleft}, 
without loss of generality we may suppose that $m'_1=1$. 
By Lemma \ref{weakcontperinsright}, 
$\mathcal F$ is also weakly periodic for the 
sequence 
$0,m_1,\ldots, m_{n-1},\nker_k(\mathcal F),\nker_k(\mathcal F)+1$. 
By Lemma \ref{concatcontper}, 
$\mathcal F''$ is weakly periodic for the sequence 
$m''_0=0,m''_1=m_1,\ldots, m''_{n-1}=m_{n-1},m''_n=\nker_k(\mathcal F),
m''_{n+1}=\nker_k(\mathcal F)-1+m'_2,\ldots,m''_{n-1+n'-1}=\nker_k(\mathcal F)-1+m'_{n'-1},m''_{n-1+n'}=\nker_k(\mathcal F'')+1$.
In particular, there exists 
a right weak evolutional period $\lambda'$ of 
$\mathcal F''$ for the pair $(m_{n-1},\nker_k(\mathcal F))$.

We are going to use Lemma \ref{contperexpandmultiright}.
To use it, we have to prove that 
$\lambda'$ is not a weak right 
evolutional period of $\mathcal F''$ for the pair $(m_{n-2},\nker_k(\mathcal F))$.
Assume the contrary.
By Lemma \ref{concatcompker}, 
$\IpR_{k,m_{n-2}, \nker_k(\mathcal F)}(\mathcal F_l)=\IpR_{k,m_{n-2}, \nker_k(\mathcal F)}(\mathcal F''_l)$
as an occurrence in $\al$ for all $l\ge 0$.
Hence, $\lambda'$ is then also a 
a weak right 
evolutional period of $\mathcal F$ for the pair $(m_{n-2},\nker_k(\mathcal F))$.
We already know that 
$\mathcal F$ is weakly periodic for the 
sequence 
$0,m_1,\ldots, m_{n-1},\nker_k(\mathcal F),\nker_k(\mathcal F)+1$, 
so there exists a weak right evolutional period of $\mathcal F$ 
for the pair $(m_{n-1},\nker_k(\mathcal F))$, 
and by Lemma \ref{contperweakuniquewideright},
this period must also be $\lambda'$. 
Since $\IpR_{k,\nker_k(\mathcal F), \nker_k(\mathcal F)+1}(\mathcal F_l)$
is always an empty occurrence, any final period $\mu$ is a 
left weak evolutional period of 
$\mathcal F$ 
for the pair $(\nker_k(\mathcal F),\nker_k(\mathcal F)+1)$.
Now $\lambda'$ and $\mu$ satisfy the conditions of 
Corollary \ref{contperexpandstartright}, 
and it implies that 
there exists a left weak evolutional period of $\mathcal F$ 
for the pair $(m_{n-2},\nker_k(\mathcal F)+1)$ if and only if 
there exists a left weak evolutional period of $\mathcal F$ 
for pair $(m_{n-1},\nker_k(\mathcal F)+1)$.
By Remark \ref{contpermultcycle},
we can replace left periods with right ones in this statement, 
in other words, 
there exists a right weak evolutional period of $\mathcal F$ 
for the pair $(m_{n-2},\nker_k(\mathcal F)+1)$ if and only if 
there exists a right weak evolutional period of $\mathcal F$ 
for pair $(m_{n-1},\nker_k(\mathcal F)+1)$.
But we know that 
$\lambda$ is a 
weak right evolutional period of $\mathcal F$ 
for the pair $(m_{n-1},\nker_k(\mathcal F)+1)$, 
so a 
weak right evolutional period of $\mathcal F$ 
for the pair $(m_{n-2},\nker_k(\mathcal F)+1)$
also exists, and by Lemma \ref{contperweakuniquewideright},
it also equals $\lambda$, but this contradicts the conditions of the lemma.

Therefore, $\lambda'$ is not a weak right 
evolutional period of $\mathcal F''$ for the pair $(m_{n-2},\nker_k(\mathcal F))$.
By Lemma \ref{contperexpandmultiright},
either there exists a $k$-sequence of obstacles in $\al$, 
or $\mathcal F''$ is weakly periodic 
for the sequence of indices 
$m''_0=0,m''_1=m_1,\ldots, m''_{n-1}=m_{n-1},\nker_k(\mathcal F'')+1$.
In the latter case the claim follows directly from Lemma \ref{contperexpandtotalright}.
\end{proof}

\begin{lemma}\label{contperexpandtotalleftgood}
Let $k\in\NN$.
Suppose that all evolutions of $k$-blocks in $\al$ are continuously periodic.
Let $\mathcal F'$ and $\mathcal F$ be two consecutive evolutions of stable nonempty $k$-multiblocks, 
and let $\mathcal F''$ be their concatenation. 
Suppose that $\nker_k(\mathcal F)>1$, and $\mathcal F$ is weakly periodic for 
a sequence of 
indices $m_0=0,m_1,\ldots, m_{n-1},m_n=\nker_k(\mathcal F)+1$, where $n\ge 2$, and 
$m_1>1$.
Let $\lambda$ be the (unique by Corollary \ref{contperweakuniqueleft}) 
left weak evolutional period of $\mathcal F$ for 
the pair $(0,m_1)$
Suppose also that $\lambda$ is not a 
left weak 
evolutional period of $\mathcal F$ for pair $(m_0,m_2)$.
Then there are two possibilities:
\begin{enumerate}
\item $\mathcal F'$ is totally periodic, moreover, 
$\lambda$ is 
a right total evolutional period 
of $\mathcal F'$.
\item There exists a $k$-series of obstacles in $\al$.
\end{enumerate}
\end{lemma}

\begin{proof}
The proof is completely symmetric to the proof of the previous lemma.
\end{proof}

\begin{lemma}\label{contperallmultiblocks}
Let $k\in\NN$.
Suppose that all evolutions of $k$-blocks in $\al$ are continuously periodic.
Let $\mathcal F$ be an evolution of stable nonempty $k$-multiblocks.
Then at least one of 
the following is true:
\begin{enumerate}
\item $\mathcal F$ is continuously periodic.
\item There exists a $k$-series of obstacles in $\al$.
\end{enumerate}
\end{lemma}

\begin{proof}
By Lemma \ref{multweakper}, $\mathcal F$ is weakly periodic for some 
sequence of indices $m_0=0,m_1,\ldots, m_{n-1},m_n=\nker_k(\mathcal F)+1$.
Let $q$ ($1\le q\le n$) be the maximal index such that 
$\mathcal F$ is weakly periodic for the sequence 
$m_0,m_q,m_{q+1},\ldots, m_{n-1},m_n$.

If $q=n$, in other 
words, if $\mathcal F$ is weakly periodic for the sequence 
$m_0=0,m_n=\nker_k(\mathcal F)+1$, then 
by Lemma \ref{weakcontperinsright}, 
$\mathcal F$ is also weakly periodic for the sequence 
$m_0=0,\nker_k(\mathcal F),m_n=\nker_k(\mathcal F)+1$, 
and this by definition means that $\mathcal F$ is continuously periodic.


Now suppose that $q\le n-1$, then $n\ge 2$. Let $\lambda'$ be a final period such that 
$\lambda'$ is a weak right evolutional period of $\mathcal F$ for 
the pair $(m_q,m_{q+1})$. Then $\lambda'$ cannot be a weak right evolutional
period for the pair $(0,m_{q+1})$ as well, otherwise $\mathcal F$ would 
also be weakly periodic for the sequence 
$m_0,m_{q+1},\ldots, m_{n-1},m_n$, and this is a contradiction 
with the fact that $q$ was chosen as a maximal index such that 
$\mathcal F$ is weakly periodic for the sequence 
$m_0,m_q,m_{q+1},\ldots, m_{n-1},m_n$. So, we can use Lemma \ref{contperexpandmultiright}.
It implies that either there exists a $k$-series of obstacles in $\al$, 
or 
$\mathcal F$ is 
weakly periodic for the sequence 
$m_0=0,m_q,m_n=\nker_k(\mathcal F)+1$, so $\mathcal F$ is continuously 
periodic for the index $m_q$ directly by the definition.
\end{proof}

\begin{corollary}\label{contpercr}
Let $k\in\NN$.
Suppose that all evolutions of $k$-blocks in $\al$ are continuously periodic.
Let $\mathcal E$ be an evolution of $(k+1)$-blocks, and let $\mathcal F$ be the following evolution of
stable nonempty $k$-multiblocks: $\mathcal F_l=\Cr_k(\mathcal E_{l+3k})$.
Then at least one of 
the following is true:
\begin{enumerate}
\item $\mathcal F$ is continuously periodic.
\item There exists a $k$-series of obstacles in $\al$.
\end{enumerate}\qed
\end{corollary}

Now we are going to prove some periodicity properties for the (left and right) regular parts of 
$(k+1)$-blocks or to find a $k$-series of obstacles 
provided that we know that all evolutions of $k$-blocks are continuously periodic.

\begin{lemma}\label{atleastonetotal}
Let $k\in\NN$.
Suppose that all evolutions of $k$-blocks in $\al$ are continuously periodic.
Let $\mathcal F$, $\mathcal F'$ and $\mathcal F''$ be three
consecutive evolutions of stable nonempty $k$-multiblocks.
Suppose that $\nker_k(\mathcal F')>1$.
Then at least one of 
the following is true:
\begin{enumerate}
\item At least one of the evolutions $\mathcal F$, $\mathcal F'$, and $\mathcal F''$ is totally
periodic.
\item There exists a $k$-series of obstacles in $\al$.
\end{enumerate}
\end{lemma}

\begin{proof}
First, apply Lemma \ref{contperallmultiblocks} to $\mathcal F'$. 
Either there exists a $k$-series of obstacles in $\al$
(and then we are done), or there exists an index $m$ ($1\le m\le \nker_k(\mathcal F')$)
such that $\mathcal F'$ is continuously periodic for $m$.

Suppose that $m>1$. Denote the (unique by Corollary \ref{contperweakuniqueleft}) 
left weak evolutional period of $\mathcal F'$ for the pair $(0,m)$
by $\lambda$. If $\lambda$ is also a weak left 
evolutional period of $\mathcal F'$ for the pair $(0,\nker_k(\mathcal F')+1)$,
$\mathcal F'$ is totally periodic, and we are done. Otherwise, by 
Lemma \ref{contperexpandtotalleftgood},
either $\mathcal F$ is totally periodic, or there is a $k$-series of obstacles in $\al$.

Now let us consider the case when $m=1$.
Denote the (unique by Corollary \ref{contperweakuniqueright}) 
right weak evolutional period of $\mathcal F'$ for the pair $(1,\nker_k(\mathcal F)+1)$
by $\lambda'$.
If $\lambda'$ is also a weak right 
evolutional period of $\mathcal F'$ for the pair $(0,\nker_k(\mathcal F')+1)$,
$\mathcal F'$ is totally periodic.
Otherwise, by 
Lemma \ref{contperexpandtotalrightgood},
either $\mathcal F''$ is totally periodic, or there is a $k$-series of obstacles in $\al$.
\end{proof}

\begin{lemma}\label{atleastonereallytotal}
Let $\mathcal F$, $\mathcal F'$ and $\mathcal F''$ be three
consecutive evolutions of stable nonempty $k$-multiblocks ($k\ge 1$).
Suppose that $\mathcal F$, $\mathcal F'$, and $\mathcal F''$ are totally periodic. 
Suppose also that $\nker_k(\mathcal F)>1$, $\nker_k(\mathcal F')>1$, and $\nker_k(\mathcal F'')>1$.
Let $\lambda$ (resp.\ $\lambda'$) be the (unique by Corollary \ref{contperweakuniqueright})
total right evolutional period 
of $\mathcal F$ (resp.\ of $\mathcal F'$),
and let $\mu'$ (resp $\mu''$) 
be the (unique by Corollary \ref{contperweakuniqueleft})
total left evolutional period 
of $\mathcal F'$ (resp.\ of $\mathcal F''$).

Then at least one of 
the following is true:
\begin{enumerate}
\item $\lambda=\mu'$.
\item $\lambda'=\mu''$.
\item There exists a $k$-series of obstacles in $\al$.
\end{enumerate}
\end{lemma}

\begin{proof}
By Lemma \ref{weakcontperinsleft}, 
$\mathcal F$ is also weakly periodic for the sequence $0,\nker_k(\mathcal F),\nker_k(\mathcal F)+1$.
By Lemma \ref{weakcontperinsright}, 
$\mathcal F'$ is also weakly periodic for the sequence $0,1,\nker_k(\mathcal F')+1$.
Denote the concatenation of $\mathcal F$ and $\mathcal F'$ by $\mathcal F'''$.
By Lemma \ref{concatcontper}, $\mathcal F'''$ is 
weakly periodic for the sequence $0,\nker_k(\mathcal F), \nker_k(\mathcal F''')+1=\nker_k(\mathcal F)+\nker_k(\mathcal F')$.

By Lemma \ref{contperweakuniquewideright},
the right weak evolutional period of $\mathcal F'$ 
for the pair $(1,\nker_k(\mathcal F')+1)$
equals $\lambda'$.
By Lemma \ref{concatcompker}, 
$\IpR_{k,\nker_k(\mathcal F), \nker_k(\mathcal F''')+1}(\mathcal F'''_l)=
\IpR_{k,1, \nker_k(\mathcal F)+1}(\mathcal F''_l)$
as an occurrence in $\al$ for all $l\ge 0$, 
so $\lambda'$
is also a
right weak evolutional period of $\mathcal F'''$ 
for the pair
$(\nker_k(\mathcal F), \nker_k(\mathcal F''')+1)$.

Suppose first that $\lambda'$ is also a weak right evolutional 
period of $\mathcal F'''$ 
for the pair 
$(0, \nker_k(\mathcal F''')+1)$. Then, since $\mathcal F'$ 
is totally periodic, we can use Lemma \ref{contperexpandtotalleft}.
It implies that $\mu'$, which is a left weak evolutional period 
of $\mathcal F'$ for the pair $(0,\nker_k(\mathcal F')+1)$, 
is also a right total evolutional period of $\mathcal F$.
Then it follows from Corollary \ref{contperweakuniqueright}
that $\lambda=\mu'$.

Now let us consider the case when $\lambda'$ is not a 
weak right evolutional 
period of $\mathcal F'''$ 
for the pair 
$(0, \nker_k(\mathcal F''')+1)$.
Then, by Lemma \ref{contperexpandtotalrightgood},
either there exists a $k$-series of obstacles in $\al$, 
or $\lambda'$ is a left total evolutional period of $\mathcal F''$.
In the latter case, $\lambda'=\mu''$ by Corollary \ref{contperweakuniqueleft}.
\end{proof}


\begin{lemma}\label{oneperenoughfortotleft}
Let $k\in\NN$.
Suppose that all evolutions of $k$-blocks in $\al$ are continuously periodic.
Let $\mathcal F$ be an evolution of stable nonempty $k$-multiblocks, and let $\lambda$ 
be a final period.
Suppose that there exists $l_0\ge 0$ such that $\psi(\Fg(\mathcal F_{l_0}))$
is 
a weakly left $\lambda$-periodic word. 
Then $\lambda$ is 
a left total period of $\mathcal F$.
\end{lemma}

\begin{proof}
By Lemma \ref{multweakper},
$\mathcal F$ is weakly periodic for a
sequence of indices $m_0=0,m_1,\ldots, m_{n-1},m_n=\nker_k(\mathcal F)+1$.
By Lemma \ref{weakcontperinsleft}, 
without loss of generality we may suppose that $m_1=1$.
Since $\IpR_{k,0,1}(\mathcal F_l)$ is always an empty word, 
$\lambda$ is a weak left evolutional period of $\mathcal F$ for the pair $(0,1)$.
Let $q$ ($1\le q\le n$) be the maximal index such that 
$\lambda$ is a weak left evolutional period of $\mathcal F$ for the pair $(0,m_q)$.
If $q=n$, we are done.

Otherwise, we are going to get a contradiction using Lemma \ref{contperexpandright}.
Denote $\Fg(\mathcal F_{l_0})=\swa ij$ and $\IpR_{k,0,m_q}(\mathcal F_{l_0})=\swa i{j'}$.
If $\lambda$ is not a weak left evolutional period of $\mathcal F$
for the pair $(0,m_{q+1})$,
then Lemma \ref{contperexpandright} implies that there exists $s\in\NN$ 
such that $\psi(\swa i{j'+s})$ is not 
a weakly left $\lambda$-periodic word, 
and $j'+s\le j$, i.~e.\ $\al_{j'+s}$ 
is a letter in $\Fg(\mathcal F_{l_0})$. But then 
$\psi(\swa i{j'+s})$ is a prefix of $\psi(\Fg(\mathcal F_{l_0}))$,
and $\psi(\Fg(\mathcal F_{l_0}))$
is
weakly left $\lambda$-periodic, 
a contradiction.
\end{proof}

The following lemma can be proved symmetrically.

\begin{lemma}\label{oneperenoughfortotright}
Let $k\in\NN$.
Suppose that all evolutions of $k$-blocks in $\al$ are continuously periodic.
Let $\mathcal F$ be an evolution of stable nonempty $k$-multiblocks, and let $\lambda$ 
be a final period.
Suppose that there exists $l_0\ge 0$ such that $\psi(\Fg(\mathcal F_{l_0}))$
is a weakly right $\lambda$-periodic word. 
Then $\lambda$ is 
a right total period of $\mathcal F$.\qed
\end{lemma}

\begin{lemma}\label{atleasttwokersleft}
Let $k\in\NN$.
Let $\mathcal E$ be an evolution of $(k+1)$-blocks such that Case I holds at the left.
Let $m\ge 0, m'\ge 0$.
Then $\mathcal F=(\mathcal F_l)_{l\ge 0}$, where 
$\mathcal F_l=\LA_{k+1,2+m}(\mathcal E_{l+3(k+1)+m+m'})$,
is an 
evolution of stable nonempty
$k$-multiblocks, and
$\nker_k(\mathcal F)>1$.
\end{lemma}

\begin{proof}
The stability follows from 
Corollary \ref{regpartstablekminusone}
(and from the definition of the left regular part of $\mathcal E_{l+3k}$), 
and the nonemptiness follows from the definition of Case I.

By Lemma \ref{caseoneexistsleft}, 
there exists a $k$-block 
$\swa ij$ in $\LA_{k+1,2+m}(\mathcal E_{3(k+1)+m+m'})$ such that Case I holds at the left or at the right 
for the evolution of $\swa ij$. Denote the evolution $\swa ij$ belongs to 
by $\mathcal E'$. 
Then, by the definitions of the descendant of a $k$-multiblock and of a left atom, 
there exists $l_0\ge 0$ such that $\LA_{k+1,2+m}(\mathcal E_{l+3(k+1)+m+m'})$ contains
$\mathcal E'_{3(k+1)+l-l_0}$ for all $l\ge 0$.
By Lemma \ref{regpartasym}, $|\mathcal E'_{3(k+1)+l-l_0}|$ is $\Theta(l^k)$ for $l\to\infty$.
Hence, $|\Fg(\mathcal F_l)|$ cannot be bounded for $l\to\infty$. But if $\nker_k(\mathcal F)=1$, 
then $\Fg(\mathcal F_l)=\Ker_{k,1}(\mathcal F)$ for all $l\ge 0$, and 
in particular, $|\Fg(\mathcal F_l)|$ is bounded for $l\to\infty$.
Therefore, $\nker_k(\mathcal F)>1$.
\end{proof}

The proof of the following lemma is completely symmetric.

\begin{lemma}\label{atleasttwokersright}
Let $k\in\NN$.
Let $\mathcal E$ be an evolution of $(k+1)$-blocks such that Case I holds at the right.
Let $m\ge 0, m'\ge 0$.
Then $\mathcal F=(\mathcal F_l)_{l\ge 0}$, where 
$\mathcal F_l=\RA_{k+1,2+m}(\mathcal E_{l+3(k+1)+m+m'})$,
is an 
evolution of stable nonempty
$k$-multiblocks, and
$\nker_k(\mathcal F)>1$.\qed
\end{lemma}

\begin{lemma}\label{onelatotal}
Let $k\in\NN$.
Suppose that all evolutions of $k$-blocks in $\al$ are continuously periodic.
Let $\mathcal E$ be an evolution of $(k+1)$-blocks such that Case I holds at the left.
Consider the following evolution of stable nonempty (by Lemma \ref{atleasttwokersleft})
$k$-multiblocks: $\mathcal F_l=\LA_{k+1,2}(\mathcal E_{l+3(k+1)})$.

At least one of 
the following is true:
\begin{enumerate}
\item $\mathcal F$ is totally periodic.
\item There exists a $k$-series of obstacles in $\al$.
\end{enumerate}
\end{lemma}

\begin{proof}
First, consider the following three evolutions of stable nonempty $k$-multiblocks:
$\mathcal F'_l=\LA_{k+1,4}(\mathcal E_{l+3(k+1)+2})$, 
$\mathcal F''_l=\LA_{k+1,3}(\mathcal E_{l+3(k+1)+2})$, and
$\mathcal F'''_l=\LA_{k+1,2}(\mathcal E_{l+3(k+1)+2})$.
By Lemma \ref{atleasttwokersleft} for $m=2,m'=0$ (resp.\ for $m=1,m'=1$, for $m=0,m'=2$), 
we have $\nker_k(\mathcal F')>1$ (resp.\ $\nker_k(\mathcal F'')>1$, $\nker_k(\mathcal F''')>1$).
Also, these three evolutions are consecutive, so, by Lemma \ref{atleastonetotal},
either there exists a $k$-series of obstacles in $\al$ (and then we are done), or 
at least one of these three evolutions is totally periodic.

Suppose now that 
at least one of the evolutions $\mathcal F'$, $\mathcal F''$, and $\mathcal F'''$ is totally periodic.
In other words, 
there exists a final period $\lambda$ and a number $m$ ($m$ can equal 0, 1, or 2)
such that for all $l\ge 0$, 
$\psi(\Fg(\LA_{k+1,2+m}(\mathcal E_{l+3(k+1)+2})))$ is 
a weakly left $\lambda$-periodic word, 
and the residue of $|\Fg(\LA_{k+1,2+m}(\mathcal E_{l+3(k+1)+2}))|$
modulo $|\lambda|$
does not depend on $l$.
By Corollary \ref{atomsperiodicitysmallleft}, 
$\Fg(\LA_{k+1,2+m}(\mathcal E_{l+3(k+1)+2}))=\Fg(\LA_{k+1,2}(\mathcal E_{l+3(k+1)+(2-m)}))$
as an abstract word. Therefore, $\lambda$ is a left total evolutional period of the following evolution of 
stable $k$-multiblocks: $\mathcal F_{2-m}, \mathcal F_{2-m+1}, \mathcal F_{2-m+2},\ldots$

Therefore, $\lambda$ is a total left evolutional period of $\mathcal F$ by Lemma \ref{oneperenoughfortotleft}.
\end{proof}

\begin{lemma}\label{oneratotal}
Let $k\in\NN$.
Suppose that all evolutions of $k$-blocks in $\al$ are continuously periodic.
Let $\mathcal E$ be an evolution of $(k+1)$-blocks such that Case I holds at the right.
Consider the following evolution of stable nonempty (by Lemma \ref{atleasttwokersright})
$k$-multiblocks: $\mathcal F_l=\RA_{k+1,2}(\mathcal E_{l+3(k+1)})$.

At least one of 
the following is true:
\begin{enumerate}
\item $\mathcal F$ is totally periodic.
\item There exists a $k$-series of obstacles in $\al$.
\end{enumerate}
\end{lemma}

\begin{proof}
The proof is completely symmetric to the proof of the previous lemma.
\end{proof}

\begin{lemma}\label{alllatotal}
Let $k\in\NN$.
Suppose that all evolutions of $k$-blocks in $\al$ are continuously periodic.
Let $\mathcal E$ be an evolution of $(k+1)$-blocks such that Case I holds at the left.

At least one of 
the following is true:
\begin{enumerate}
\item There exists a unique final period $\lambda$ and a number $r$ ($0\le r<|\lambda|$)
such that for all $m\ge 0$, $\lambda$ is a left total period of the
evolution $\mathcal F$ of stable nonempty (by Lemma \ref{atleasttwokersright})
$k$-multiblocks defined by $\mathcal F_l=\LA_{k+1,2+m}(\mathcal E_{l+3(k+1)+m})$, 
moreover, the residue of $|\Fg(\LA_{k+1,2+m}(\mathcal E_{l+3(k+1)+m}))|$ always equals $r$
(i.~e.\ it does not depend on $m$).
\item There exists a $k$-series of obstacles in $\al$.
\end{enumerate}
\end{lemma}

\begin{proof}
By Lemma \ref{onelatotal}, either there exists a 
$k$-series of obstacles in $\al$ (and then we are done), 
or
the
evolution $\mathcal F'$ of stable nonempty
$k$-multiblocks defined by $\mathcal F'_l=\LA_{k+1,2}(\mathcal E_{l+3(k+1)})$
is totally periodic.

Suppose that $\mathcal F'$ is totally periodic.
Since $\nker_k(\mathcal F')>1$ by Lemma \ref{atleasttwokersleft}, 
it follows from Corollary \ref{contperweakuniqueleft} that the left 
total evolutional period of $\mathcal F'$ is unique, denote it by $\lambda$. 
Denote by $r$ the remainder of $|\Fg(\mathcal F'_l)|$ modulo $|\lambda|$
(it does not depend on $l$ by the definition of a weak left evolutional period).
Then $\psi(\Fg(\LA_{k+1,2}(\mathcal E_{l+3(k+1)})))$ is 
a weakly left $\lambda$-periodic word 
for all $l\ge 0$.
By Corollary \ref{atomsperiodicitysmallleft}, 
$\Fg(\LA_{k+1,2+m}(\mathcal E_{l+3(k+1)+m}))=\Fg(\LA_{k+1,2}(\mathcal E_{l+3(k+1)}))$
as an abstract word.
Hence, for all $l\ge 0$ and $m\ge 0$,
$\psi(\Fg(\LA_{k+1,2+m}(\mathcal E_{l+3(k+1)+m})))$ is 
a weakly left $\lambda$-periodic word,
and 
the residue of $|\Fg(\LA_{k+1,2+m}(\mathcal E_{l+3(k+1)+m}))|$ equals $r$.
Therefore, 
$\lambda$ is a left total period of the
evolution $\mathcal F$ of stable nonempty
$k$-multiblocks defined by $\mathcal F_l=\LA_{k+1,2+m}(\mathcal E_{l+3(k+1)+m})$.
\end{proof}

\begin{lemma}\label{allratotal}
Let $k\in\NN$.
Suppose that all evolutions of $k$-blocks in $\al$ are continuously periodic.
Let $\mathcal E$ be an evolution of $(k+1)$-blocks such that Case I holds at the right.

At least one of 
the following is true:
\begin{enumerate}
\item There exists a unique final period $\lambda$ and a number $r$ ($0\le r<|\lambda|$)
such that for all $m\ge 0$, $\lambda$ is a right total period of the
evolution $\mathcal F$ of stable nonempty (by Lemma \ref{atleasttwokersright})
$k$-multiblocks defined by $\mathcal F_l=\RA_{k+1,2+m}(\mathcal E_{l+3(k+1)+m})$, 
moreover, the residue of $|\Fg(\RA_{k+1,2+m}(\mathcal E_{l+3(k+1)+m}))|$ always equals $r$
(i.~e.\ it does not depend on $m$).
\item There exists a $k$-series of obstacles in $\al$.
\end{enumerate}
\end{lemma}

\begin{proof}
The proof is completely symmetric to the proof of the previous lemma.
\end{proof}

\begin{lemma}\label{leftrightcompleteperiod}
Let $\lambda$ be a final period. Let $\gamma$ be a finite word.
Suppose that $\gamma$ is weakly $|\lambda|$-periodic with 
both left and right period $\lambda$, and $|\gamma|\ge 2\finmax$. Then 
$\gamma$ is 
a completely $\lambda$-periodic word.
\end{lemma}

\begin{proof}
Denote the remainder of $|\gamma|$ modulo $|\lambda|$ by $r$.
Set $\lambda'=\Cyc_r(\lambda)=\Cyc_{|\gamma|}(\lambda)$.
Then $\gamma$ is 
a weakly right $\lambda'$-periodic word.
Since $\lambda$ is a final period, 
by Lemma \ref{finalnomorethanonce},
$\lambda$ cannot be written as a word repeated more than once, so, 
by Lemma \ref{finwordperiods}, $\lambda'=\lambda$.

Assume that $r>0$. Then the equality $\lambda=\lambda'$ means that 
$\lambda_{0\ldots r-1}=\lambda_{|\lambda|-r\ldots |\lambda|-1}$
and
$\lambda_{0\ldots |\lambda|-r-1}=\lambda_{r\ldots |\lambda|-1}$.
Consider the word $\delta=\lambda\lambda$. 
Let us check that $\delta$ is a weakly $r$-periodic word. To see this, 
we have to check that $\delta_i=\delta_{i+r}$ for $0\le i<|\delta|-r$ 
(in other words, $0\le i<2|\lambda|-r$).
Consider the following three cases for $i$:

1. $0\le i<|\lambda|-r$. Then $\delta_i=\lambda_i=\lambda_{i+r}=\delta_{i+r}$ since 
$\lambda_{0\ldots |\lambda|-r-1}=\lambda_{r\ldots |\lambda|-1}$.

2. $|\lambda|-r\le i<\lambda$. Then $\delta_i=\lambda_i=\lambda_{i-(|\lambda|-r)}=\lambda_{i+r-|\lambda|}=\delta_{i+r}$
since $\lambda_{0\ldots r-1}=\lambda_{|\lambda|-r\ldots |\lambda|-1}$.

3. $|\lambda|\le i<2|\lambda|-r$. Then $\delta_i=\lambda_{i-|\lambda|}=\lambda_{i-|\lambda|+r}=\delta_{i+r}$
since 
$\lambda_{0\ldots |\lambda|-r-1}=\lambda_{r\ldots |\lambda|-1}$.

Therefore, $\delta$ is a weakly $r$-periodic word, and $|\delta|=2|\lambda|\ge 2r$.
Clearly, $\delta$ is also a 
weakly left $\lambda$-periodic word
and $|\delta|\ge 2|\lambda|$.
Then Lemma \ref{finwordperiods} implies that there exists a finite word 
$\mu$ such that $\lambda$ is $\mu$ repeated several times (an integer number of times, so $|\lambda|$ is 
divisible by $|\mu|$), and 
$r$ is also divisible by $|\mu|$. But then $|\mu|\le r<|\lambda|$, and $\lambda$ is $\mu$ repeated 
more than once. But this contradicts 
Lemma \ref{finalnomorethanonce}.
\end{proof}

\begin{lemma}\label{alllareallytotal}
Let $k\in\NN$.
Suppose that all evolutions of $k$-blocks in $\al$ are continuously periodic.
Let $\mathcal E$ be an evolution of $(k+1)$-blocks such that Case I holds at the left.

At least one of 
the following is true:
\begin{enumerate}
\item There exists a unique final period $\lambda$ 
such that for all $l\ge 0$ and $m\ge 0$, 
$\psi(\Fg(\LA_{k+1,2+m}(\mathcal E_{l+3(k+1)+m})))$
is 
a \textbf{completely} $\lambda$-periodic 
word.
\item There exists a $k$-series of obstacles in $\al$.
\end{enumerate}
\end{lemma}

\begin{proof}
Suppose that $k$-series of obstacles do not exist in $\al$.
Then, by Lemma \ref{alllatotal}, there exists a final period $\lambda$ and a number $r$ ($0\le r<|\lambda|$)
such that 
for all $l\ge 0$ and $m\ge 0$, 
$\psi(\Fg(\LA_{k+1,2+m}(\mathcal E_{l+3(k+1)+m})))$
is a weakly left $\lambda$-periodic word, 
and the residue of $|\Fg(\LA_{k+1,2+m}(\mathcal E_{l+3(k+1)+m}))|$ 
modulo $|\lambda|$
equals $r$. It is sufficient to prove that $r=0$.

Again, consider the following three evolutions of stable nonempty $k$-multiblocks: 
$\mathcal F'_l=\LA_{k+1,4}(\mathcal E_{l+3(k+1)+2})$, 
$\mathcal F''_l=\LA_{k+1,3}(\mathcal E_{l+3(k+1)+2})$, and
$\mathcal F'''_l=\LA_{k+1,2}(\mathcal E_{l+3(k+1)+2})$.
Then $\lambda$ is the total left evolutional period of each of them.
By Lemma \ref{atleastonereallytotal}, 
$\lambda$ is also the 
total right evolutional period of at least one of the evolutions $\mathcal F'$ or $\mathcal F''$.
So, at least one of the words
$\psi(\Fg(\LA_{k+1,4}(\mathcal E_{3(k+1)+2})))$ 
and 
$\psi(\Fg(\LA_{k+1,3}(\mathcal E_{3(k+1)+2})))$
is a weakly $|\lambda|$-periodic word with both left and right 
period $\lambda$. Denote this word by $\gamma$.
By Corollary \ref{betweenkernelgrowsgood}, 
$|\gamma|\ge 2\finmax$.
By Lemma \ref{leftrightcompleteperiod}, 
$\gamma$ is 
a completely $\lambda$-periodic word,
and $|\gamma|$ is divisible by $|\lambda|$. 
But we also know that the residue of $|\gamma|$
modulo $|\lambda|$ equals $r$, so $r=0$.

Therefore, all words 
$\psi(\Fg(\LA_{k+1,2+m}(\mathcal E_{l+3(k+1)+m})))$
for all $l\ge 0$ and $m\ge 0$ 
are completely $\lambda$-periodic.
\end{proof}

\begin{corollary}\label{lrperiodic}
Let $k\in\NN$.
Suppose that all evolutions of $k$-blocks in $\al$ are continuously periodic.
Let $\mathcal E$ be an evolution of $(k+1)$-blocks such that Case I holds at the left.

At least one of 
the following is true:
\begin{enumerate}
\item There exists a unique final period $\lambda$ 
such that for all $l\ge 0$, 
$\psi(\Fg(\LR_{k+1}(\mathcal E_{l+3(k+1)})))$
is a completely $\lambda$-periodic word. 
$\lambda$ is also a total left and a total right period of 
the evolution $\mathcal F$ of stable nonempty 
$k$-multiblocks defined by $\mathcal F_l=\LA_{k+1,2}(\mathcal E_{l+3(k+1)})$.
\item There exists a $k$-series of obstacles in $\al$.\qed
\end{enumerate}
\end{corollary}

\begin{lemma}\label{allrareallytotal}
Let $k\in\NN$.
Suppose that all evolutions of $k$-blocks in $\al$ are continuously periodic.
Let $\mathcal E$ be an evolution of $(k+1)$-blocks such that Case I holds at the right.

At least one of 
the following is true:
\begin{enumerate}
\item There exists a unique final period $\lambda$ 
such that for all $l\ge 0$ and $m\ge 0$, 
$\psi(\Fg(\RA_{k+1,2+m}(\mathcal E_{l+3(k+1)+m})))$
is a \textbf{completely} $\lambda$-periodic 
word.
\item There exists a $k$-series of obstacles in $\al$.
\end{enumerate}
\end{lemma}

\begin{proof}
The proof is completely symmetric to the proof of Lemma \ref{alllareallytotal}.
\end{proof}

\begin{corollary}\label{rrperiodic}
Let $k\in\NN$.
Suppose that all evolutions of $k$-blocks in $\al$ are continuously periodic.
Let $\mathcal E$ be an evolution of $(k+1)$-blocks such that Case I holds at the right.

At least one of 
the following is true:
\begin{enumerate}
\item There exists a unique final period $\lambda$ 
such that for all $l\ge 0$, 
$\psi(\Fg(\RR_{k+1}(\mathcal E_{l+3(k+1)})))$
is a completely $\lambda$-periodic word. 
$\lambda$ is also a total left and a total right period of 
the evolution $\mathcal F$ of stable nonempty 
$k$-multiblocks defined by $\mathcal F_l=\RA_{k+1,2}(\mathcal E_{l+3(k+1)})$.
\item There exists a $k$-series of obstacles in $\al$.\qed
\end{enumerate}
\end{corollary}

Now we are going to prove some facts about the periodicity of left and right bounding sequences
of evolutions of $(k+1)$-blocks such that Case II holds at the right or at the left.

\begin{lemma}\label{rbsperiodic}
Let $k\in\NN$.
Let $\mathcal E$ be an evolution of $(k+1)$-blocks such that Case II holds at the right.
Let $l_0\ge 3(k+1)$, and let $\mathcal F$ be an evolution of stable nonempty $k$-multiblocks 
such that $\Fg(\mathcal F_l)$ is a suffix of $\mathcal E_{l+l_0}$ for all $l\ge 0$.
Suppose that $\mathcal F$ is continuously periodic for an index $m$ ($1\le m<\nker_k(\mathcal F)$).
Let $\lambda$ be the (unique by Corollary \ref{contperweakuniqueright})
weak right evolutional period of $\mathcal F$ for the pair $(m,\nker_k(\mathcal F)+1)$.

Then there are three possibilities:
\begin{enumerate}
\item $\mathcal F$ is totally periodic.
\item $\psi(\RBS_{k+1}(\mathcal E))$ is periodic with period $\lambda$.
\item There exists a $k$-series of obstacles in $\al$.
\end{enumerate}
\end{lemma}

\begin{proof}
Suppose that $k$-series of obstacles do not exist in $\al$ and that 
$\mathcal F$ is not totally periodic. Then, since $\mathcal F$ is 
continuously periodic for the index $m$, Lemma \ref{contperexpandleft}
implies that there exists a number $s\in\NN$ such that for all $l\ge 0$, 
if $\Fg(\mathcal F_l)=\swa ij$ and $\IpR_{k,m,\nker_k(\mathcal F)+1}(\mathcal F_l)=\swa{i'}j$, 
then
$\psi(\swa{i'-s+1}j)$ is 
a weakly right $\lambda$-periodic word, 
and 
$\psi(\swa{i'-s}j)$ is not 
a weakly right $\lambda$-periodic word.

Assume that $\psi(\RBS_{k+1}(\mathcal E))$ is not an infinite periodic sequence with period $\lambda$.
Then there exists a number $s'\ge 0$ such that $\psi(\RBS_{k+1}(\mathcal E))_{0\ldots s'-1}$
is a weakly left $\lambda$-periodic word, 
and
$\psi(\RBS_{k+1}(\mathcal E))_{0\ldots s'}$
is not 
a weakly left $\lambda$-periodic word.
We are going to find a $k$-series of obstacles in $\al$.

By Corollary \ref{presenceofrbs}, there exists $l_1\ge 0$ such that if $l\ge 0$ and 
$\mathcal E_{l+l_1}=\swa{i''}j$, then $\swa{i''}{j+s'+1}=\mathcal E_{l+l_1}\RBS_{k+1}(\mathcal E)_{0\ldots s'}$.
Without loss of generality, $l_1\ge l_0$.

Fix a number $l\ge 0$. Suppose that $\mathcal E_{l+l_1}=\swa{i''}j$. 
Then $\Fg(\mathcal F_{l+l_1-l_0})$ is a suffix of $\mathcal E_{l+l_1}$, 
so there exists an index $i\ge i''$ such that $\Fg(\mathcal F_{l+l_1-l_0})=\swa ij$.
Let $i'\ge i$ be the index such that 
$\IpR_{k,m,\nker_k(\mathcal F)+1}(\mathcal F_{l+l_1-l_0})=\swa{i'}j$.
Set $\mathcal H_l=\swa{i'-s+1}{j+s'}$.

First, as an abstract word, $\psi(\mathcal H_l)=\psi(\swa{i'-s+1}j\swa{j+1}{j+s'})=
\psi(\swa{i'-s+1}j)\psi(\RBS_{k+1}(\mathcal E)_{0\ldots s'-1})$, 
and $\psi(\swa{i'-s+1}j)$ 
(resp.\ $\psi(\RBS_{k+1}(\mathcal E)_{0\ldots s'-1})$) 
is a weakly right (resp.\ left) $\lambda$-periodic
word,
hence, $\psi(\mathcal H_l)$
is weakly left $\lambda'$-periodic,
where 
$\lambda'=\Cyc_{-(j-(i'-s+1)+1)}(\lambda)$.
Since 
$\psi(\swa{j+1}{j+s'+1})=\psi(\RBS_{k+1}(\mathcal E))_{0\ldots s'}$
is not 
a weakly left $\lambda$-periodic word,
$\psi(\swa{i'-s+1}j\swa{j+1}{j+s'+1})=\psi(\swa{i'-s+1}{j+s'+1})$
is not a weakly left $\lambda'$-periodic word either.
And since 
$\psi(\swa{i'-s}j)$ is not 
a weakly right $\lambda$-periodic word,
$\psi(\al_{i'-s})\ne \lambda'_{|\lambda|-1}$.

Now, since $1\le m<\nker_k(\mathcal F)$,
by Corollary \ref{betweenkernelgrowsgood}, 
$|\IpR_{k,m,\nker_k(\mathcal F)+1}(\mathcal F_{l+l_1-l_0})|\ge 2\finmax$, hence, 
$|\mathcal H_l|=(s-1)+|\IpR_{k,m,\nker_k(\mathcal F)+1}(\mathcal F_{l+l_1-l_0})|+s'\ge 2\finmax$.
In particular, $|\mathcal H_l|\ge|\lambda|=|\lambda'|$.
Since $\psi(\swa{i'-s+1}{j+s'})$
is a weakly left $\lambda'$-periodic word, 
$\psi(\al_{i'-s+|\lambda|})=\lambda'_{|\lambda|-1}$.
Since
$\psi(\al_{i'-s})\ne \lambda'_{|\lambda|-1}$,
$\psi(\swa{i'-s}{j+s'})$ is not a 
weakly $|\lambda|$-periodic word (with any period).
Let $r$ be the residue of $|\swa{i'-s+1}{j+s'}|$
modulo $|\lambda|$. Then, since 
$\psi(\swa{i'-s+1}{j+s'})$
is a weakly left $\lambda'$-periodic word, 
$\psi(\al_{j+s'-|\lambda|+1})=\lambda'_r$.
And since 
$\psi(\swa{i'-s+1}{j+s'+1})$
is not a weakly left $\lambda'$-periodic word,
$\psi(\al_{j+s'+1})\ne\lambda'_r$.
Therefore, 
$\psi(\swa{i'-s+1}{j+s'+1})$
is not a weakly $|\lambda|$-periodic word (with any period).

Finally, let $l\ge 0$ be arbitrary again. 
By Corollary \ref{betweenkernelgrowsgood}, 
$|\IpR_{k,m,\nker_k(\mathcal F)+1}(\mathcal F_{l+l_1-l_0})|$
strictly grows as $l$ grows, and there exists 
$k'\in\NN$ ($1\le k'\le k$) such that 
$|\IpR_{k,m,\nker_k(\mathcal F)+1}(\mathcal F_{l+l_1-l_0})|$
is $\Theta(l^{k'})$ for $l\to\infty$.
Then, since $s$ and $s'$ do not depend on $l$, 
$|\mathcal H_l|=(s-1)+|\IpR_{k,m,\nker_k(\mathcal F)+1}(\mathcal F_{l+l_1-l_0})|+s'$
also strictly grows as $l$ grows, and 
$|\mathcal H_l|$
is $\Theta(l^{k'})$ for $l\to\infty$.
\end{proof}

\begin{lemma}\label{lbsperiodic}
Let $k\in\NN$.
Let $\mathcal E$ be an evolution of $(k+1)$-blocks such that Case II holds at the left.
Let $l_0\ge 3(k+1)$, and let $\mathcal F$ be an evolution of stable nonempty $k$-multiblocks 
such that $\Fg(\mathcal F_l)$ is a prefix of $\mathcal E_{l+l_0}$ for all $l\ge 0$.
Suppose that $\mathcal F$ is continuously periodic for an index $m$ ($1<m\le\nker_k(\mathcal F)$).
Let $\lambda$ be the (unique by Corollary \ref{contperweakuniqueleft})
weak left evolutional period of $\mathcal F$ for the pair $(0,m)$.

Then there are three possibilities:
\begin{enumerate}
\item $\mathcal F$ is totally periodic.
\item $\psi(\LBS_{k+1}(\mathcal E))$ is periodic with period $\lambda$.
\item There exists a $k$-series of obstacles in $\al$.
\end{enumerate}
\end{lemma}

\begin{proof}
The proof is completely symmetric to the proof of the previous lemma.
\end{proof}

\begin{lemma}\label{lbsrbsperiodic}
Let $k\in\NN$.
Let $\mathcal E$ be an evolution of $(k+1)$-blocks such that Case II holds both 
at the left and at the right. Suppose that $\ncker_{k+1}(\mathcal E)>1$. 
Let $\mathcal F$ be the evolution of stable nonempty $k$-blocks defined 
by $\mathcal F_l=\Cr_{k+1}(\mathcal E_{l+3(k+1)})$.
Suppose that $\mathcal F$ is totally periodic, and let $\lambda$ (resp.\ $\lambda'$)
be the left (resp.\ the right) total evolutional period of $\mathcal F$.

Then there are three possibilities:
\begin{enumerate}
\item $\psi(\LBS_{k+1}(\mathcal E))$ is periodic with period $\lambda$.
\item $\psi(\RBS_{k+1}(\mathcal E))$ is periodic with period $\lambda'$.
\item There exists a $k$-series of obstacles in $\al$.
\end{enumerate}
\end{lemma}

\begin{proof}
Suppose that $\psi(\LBS_{k+1}(\mathcal E))$ is not an infinite periodic sequence with period $\lambda$, 
and $\psi(\RBS_{k+1}(\mathcal E))$ is not an infinite periodic sequence with period $\lambda'$.
We are going to find a $k$-series of obstacles in $\al$.

Let $s$ be the length of the maximal 
weakly right $\lambda$-periodic suffix of 
$\psi(\LBS_{k+1}(\mathcal E))$, 
in 
other words, $s\ge 0$ is the number such that
$\psi(\LBS_{k+1}(\mathcal E))_{-s+1\ldots 0}$ is a 
weakly right $\lambda$-periodic word, 
and 
$\psi(\LBS_{k+1}(\mathcal E))_{-s\ldots 0}$ is not a 
weakly right $\lambda$-periodic word.
Similarly, let 
$s'$ be the length of the maximal 
weakly left $\lambda'$-periodic prefix of 
$\psi(\RBS_{k+1}(\mathcal E))$, 
in 
other words, $s'\ge 0$ is the number such that
$\psi(\RBS_{k+1}(\mathcal E))_{0\ldots s'-1}$ is a 
weakly left $\lambda'$-periodic word, 
and 
$\psi(\RBS_{k+1}(\mathcal E))_{0\ldots s'}$ is not a 
weakly left $\lambda'$-periodic word.
By Corollaries \ref{presenceofrbs} and \ref{presenceoflbs}, 
there exists $l_0\ge 0$ such that if $l\ge 0$ and 
$\mathcal E_{l+l_0}=\swa ij$, then $\swa{i-s-1}{j+s'+1}=
\LBS_{k+1}(\mathcal E)_{-s\ldots 0}\mathcal E_{l+l_0}\RBS_{k+1}(\mathcal E)_{0\ldots s'}$.
Without loss of generality, $l_0\ge 3(k+1)$.

Fix $l\ge 0$. Suppose that $\mathcal E_{l+l_0}=\Fg(\mathcal F_{l+l_0-3(k+1)})=\swa ij$.
Set $\mathcal H_l=\swa{i-s-1}{j+s'+1}$. We are going to prove that 
all $\mathcal H_l$ for $l\ge 0$ form a $k$-series of obstacles.
The argument is similar to the proof of Lemma \ref{rbsperiodic}.

First, it follows from Remark \ref{contpermultcycle} and from Corollary \ref{contperweakuniqueright}
that $\lambda'=\Cyc_{|\Fg(\mathcal F_{l+l_0-3(k+1)})|}(\lambda)$.
We also know that $\psi(\swa ij)$ is 
a weakly left $\lambda$-periodic word, 
$\psi(\swa{i-s}{i-1})$
is a weakly right $\lambda$-periodic word, 
and 
$\psi(\swa{i-s-1}{i-1})$
is not a weakly right $\lambda$-periodic word.
Hence, 
$\psi(\swa{i-s}j)$ is 
a weakly right $\lambda'$-periodic word, 
and 
$\psi(\swa{i-s-1}j)$ is not 
a weakly right $\lambda'$-periodic word.

Now, 
$\psi(\swa{j+1}{j+s'})$ is a 
weakly left $\lambda'$-periodic word, 
and 
$\psi(\swa{j+1}{j+s'+1})$ is not a 
weakly left $\lambda'$-periodic word.
Therefore, if 
$\lambda''=\Cyc_{-|\swa{i-s}j|}(\lambda')$, 
then 
$\psi(\swa{i-s}{j+s'})=\psi(\mathcal H_l)$ is a 
weakly left $\lambda''$-periodic word, 
$\psi(\swa{i-s}{j+s'+1})$ is not a 
weakly left $\lambda''$-periodic word,
and $\psi(\al_{i-s-1})\ne \lambda''_{|\lambda''|-1}$.

Since $\ncker_{k+1}(\mathcal E)=\nker_k(\mathcal F)>1$,
by Corollary \ref{betweenkernelgrowsgood}, 
$|\Fg(\mathcal F_{l+l_0-3(k+1)})|\ge 2\finmax$, hence, 
$|\mathcal H_l|=s+|\Fg(\mathcal F_{l+l_0-3(k+1)})|+s'\ge 2\finmax$.
In particular, $|\mathcal H_l|\ge|\lambda''|=|\lambda|$.
Since $\psi(\swa{i-s}{j+s'})$
is a weakly left $\lambda''$-periodic word, 
$\psi(\al_{i-s+|\lambda|-1})=\lambda''_{|\lambda''|-1}$.
Since
$\psi(\al_{i-s-1})\ne \lambda''_{|\lambda''|-1}$,
$\psi(\swa{i-s-1}{j+s'})$ is not a 
weakly $|\lambda''|$-periodic word (with any period).
Let $r$ be the residue of $|\swa{i-s}{j+s'}|=(j'+s)-(i-s)+1$
modulo $|\lambda''|$. Then, since 
$\psi(\swa{i-s}{j+s'})$
is a weakly left $\lambda''$-periodic word, 
$\psi(\al_{j+s'-|\lambda''|+1})=\lambda''_r$.
And since 
$\psi(\swa{i-s}{j+s'+1})$
is not a weakly left $\lambda''$-periodic word,
$\psi(\al_{j+s'+1})\ne\lambda''_r$.
Therefore, 
$\psi(\swa{i-s}{j+s'+1})$
is not a weakly $|\lambda''|$-periodic word (with any period).

Finally, let $l\ge 0$ be arbitrary again. 
By Corollary \ref{betweenkernelgrowsgood}, 
$|\Fg(\mathcal F_{l+l_0-3(k+1)})|$
strictly grows as $l$ grows, and there exists 
$k'\in\NN$ ($1\le k'\le k$) such that 
$|\Fg(\mathcal F_{l+l_0-3(k+1)})|$
is $\Theta(l^{k'})$ for $l\to\infty$.
Then, since $s$ and $s'$ do not depend on $l$, 
$|\mathcal H_l|=s+|\Fg(\mathcal F_{l+l_0-3(k+1)})|+s'$
also strictly grows as $l$ grows, and 
$|\mathcal H_l|$
is $\Theta(l^{k'})$ for $l\to\infty$.
\end{proof}

Finally, we are ready to prove that 
if all evolutions of $k$-blocks are continuously periodic, 
then either there is a $k$-series of obstacles, or 
all evolutions of $(k+1)$-blocks are continuously periodic.

\begin{lemma}\label{sameperiodcontper}
Let $k\in\NN$.
Let $\mathcal E$ be an evolution of $(k+1)$-blocks such that Case I holds at the 
left (resp.\ at the right). Let $\lambda$ be a final period, and 
let $m$ be an index $1\le m\le\ncker_{k+1}(\mathcal E)$. 
Denote by $\mathcal F$ the following evolution of 
$k$-multiblocks: $\mathcal F_l=\Cr_{k+1}(\mathcal E_{l+3(k+1)})$.

Suppose that for all $l\ge 0$, 
$\psi(\Fg(\LR_{k+1}(\mathcal E_{l+3(k+1)})))$
(resp.\ $\psi(\Fg(\RR_{k+1}(\mathcal E_{l+3(k+1)})))$)
is a completely $\lambda$-periodic word.
Suppose also that $\lambda$ is a left (resp.\ right)
weak evolutional period of $\mathcal F$ for the pair $(0,m)$
(resp.\ $(m,\nker_k(\mathcal F)+1)$).

Then $\lambda$ is a left continuous evolutional period of $\mathcal E$ for the index $m$.
\end{lemma}

\begin{proof}
We prove the lemma for the situation when Case I holds at the left. If Case I holds at the right, 
the proof is completely symmetric.

For all $l\ge 0$, $\psi(\IpR_{k,0,m}(\mathcal F_l))$
is a weakly left $\lambda$-periodic word, 
and the residue of $|\psi(\IpR_{k,0,m}(\mathcal F_l))|$
modulo $|\lambda|$ does not depend on $l$.
$\psi(\Fg(\LR_{k+1}(\mathcal E_{l+3(k+1)})))$
is a completely $\lambda$-periodic word, and
$\LpR_{k+1,m}(\mathcal E_{l+3(k+1)})=\Fg(\LR_{k+1}(\mathcal E_{l+3(k+1)}))\IpR_{k,0,m}(\mathcal F_l)$,
so $\psi(\LpR_{k+1,m}(\mathcal E_{l+3(k+1)}))$ is also a 
weakly left $\lambda$-periodic word.
$|\Fg(\LR_{k+1}(\mathcal E_{l+3(k+1)}))|$ is divisible by $|\lambda|$ for all $l\ge 0$, 
so the residue of 
$|\LpR_{k+1,m}(\mathcal E_{l+3(k+1)})|=|\Fg(\LR_{k+1}(\mathcal E_{l+3(k+1)}))|+|\IpR_{k,0,m}(\mathcal F_l)|$
modulo $|\lambda|$
equals the residue of $|\IpR_{k,0,m}(\mathcal F_l)|$
modulo $|\lambda|$
and does not depend on $l$.
Therefore, $\lambda$ is a left continuous evolutional period of $\mathcal E$ for the index $m$.
\end{proof}

\begin{lemma}\label{corenottotcontper}
Let $k\in\NN$.
Suppose that all evolutions of $k$-blocks in $\al$ are continuously periodic.
Let $\mathcal E$ be an evolution of $(k+1)$-blocks such that Case I holds at the 
left (resp.\ at the right). 
Denote by $\mathcal F$ the following evolution of 
$k$-multiblocks: $\mathcal F_l=\Cr_{k+1}(\mathcal E_{l+3(k+1)})$.
Suppose that $\mathcal F$ 
is continuously periodic for an index $m$ ($1\le m\le \nker_k(\mathcal F)$), but
is \textbf{not} totally periodic.

Then there are two possibilities:
\begin{enumerate}
\item There exists a left (resp.\ right) continuous evolutional period of $\mathcal E$ for the index $m$.
\item There exists a $k$-series of obstacles in $\al$.
\end{enumerate}
\end{lemma}

\begin{proof}
We prove the lemma for the situation when Case I holds at the left. If Case I holds at the right, 
the proof is completely symmetric.
Suppose that $k$-series of obstacles do not exist in $\al$. We have to prove that 
there exists a left continuous evolutional period of $\mathcal E$ for the index $m$.

By Corollary \ref{lrperiodic}
there exists a final period $\lambda$ such that 
for all $l\ge 0$, $\psi(\Fg(\LR_{k+1}(\mathcal E_{l+3(k+1)})))$
is a completely $\lambda$-periodic word,
moreover,
$\lambda$ is the unique right total evolutional period of the 
evolution $\mathcal F'$ of stable nonempty $k$-multiblocks defined by 
$\mathcal F'_l=\LA_{k+1,2}(\mathcal E_{l+3(k+1)})$.
Let us check that $\lambda$ is also a weak left evolutional period of $\mathcal F$ 
for the pair $(0,m)$.
If $m=1$, this is already clear since 
$\IpR_{k,0,1}(\mathcal F_l)$ is always an empty occurrence.
If $m>1$, then 
since $\mathcal F'$ and $\mathcal F$ are consecutive, 
the fact that 
the left weak evolutional period of $\mathcal F$ 
for the pair $(0,m)$
also equals $\lambda$
follows from
Lemma \ref{contperexpandtotalleftgood}.
Now, $\lambda$ is a left continuous evolutional period of $\mathcal E$ for the index $m$
by Lemma \ref{sameperiodcontper}.
\end{proof}

\begin{lemma}\label{contperbothone}
Let $k\in\NN$.
Suppose that all evolutions of $k$-blocks in $\al$ are continuously periodic.
Let $\mathcal E$ be an evolution of $(k+1)$-blocks such that Case I holds both 
at the left and at the right.

Then there are two possibilities:
\begin{enumerate}
\item $\mathcal E$ is continuously periodic.
\item There exists a $k$-series of obstacles in $\al$.
\end{enumerate}
\end{lemma}

\begin{proof}
Suppose that $k$-series of obstacles do not exist in $\al$. We have to prove that 
$\mathcal E$ is continuously periodic.

If $\ncker_{k+1}(\mathcal E)=1$, then the claim follows from 
Corollaries \ref{lrperiodic}
and \ref{rrperiodic}. Suppose that $\ncker_{k+1}(\mathcal E)>1$.

Denote by $\mathcal F$ the evolution of stable nonempty $k$-multiblocks 
defined by $\mathcal F_l=\Cr_{k+1}(\mathcal E_{l+3(k+1)})$. Our assumption 
$\ncker_{k+1}(\mathcal E)>1$
means that 
$\nker_k(\mathcal F)>1$.
By Corollary \ref{contpercr}, $\mathcal F$ is continuously periodic. If $\mathcal F$ 
is not totally periodic, then the claim follows from Lemma \ref{corenottotcontper}.

Let us consider the case when $\mathcal F$ is totally periodic.
Again, 
By Corollaries \ref{lrperiodic}
and \ref{rrperiodic}, there exist final periods $\lambda$ and $\mu$ such that 
for all $l\ge 0$, $\psi(\Fg(\LR_{k+1}(\mathcal E_{l+3(k+1)})))$
is a completely $\lambda$-periodic word
and
$\psi(\Fg(\RR_{k+1}(\mathcal E_{l+3(k+1)})))$
is a completely $\mu$-periodic word.
Moreover,
$\lambda$ (resp.\ $\mu$) is the unique right (resp.\ left) total evolutional period of the 
evolution $\mathcal F'$ (resp.\ $\mathcal F''$) of stable nonempty $k$-multiblocks defined by 
$\mathcal F'_l=\LA_{k+1,2}(\mathcal E_{l+3(k+1)})$
(resp.\ by $\mathcal F''_l=\RA_{k+1,2}(\mathcal E_{l+3(k+1)})$).
So, $\mathcal F'$, $\mathcal F$, and $\mathcal F''$ are three consecutive 
totally periodic evolutions of stable nonempty $k$-multiblocks, and
by Lemmas \ref{atleasttwokersleft} and \ref{atleasttwokersright}, 
$\nker_k(\mathcal F')>1$ and $\nker_k(\mathcal F'')>1$. We have also assumed that 
$\nker_k(\mathcal F)>1$. Now we can use Lemma \ref{atleastonereallytotal}.
It implies that either $\lambda$ is a total left evolutional period of 
$\mathcal F$, or
$\mu$ is a total right evolutional period of 
$\mathcal F$. If $\lambda$ is a total left evolutional period of 
$\mathcal F$, set $m=\nker_k(\mathcal F)$. Then $\lambda$ is also a 
left weak evolutional period of $\mathcal F$ 
for the pair $(0,m)$ by Lemmas \ref{weakcontperinsright} and \ref{contperweakuniquewideleft}, 
and $\mu$ is a weak right 
evolutional period of $\mathcal F$ 
for the pair $(m,\nker_k(\mathcal F)+1)$
since 
$\IpR_{k,\nker_k(\mathcal F),\nker_k(\mathcal F)+1}(\mathcal F_l)$ is always an empty occurrence.
Similarly, if 
$\mu$ is a total right evolutional period of 
$\mathcal F$, then set $m=1$.
Then $\mu$ is also a 
right weak evolutional period of $\mathcal F$ 
for the pair $(m,\nker_k(\mathcal F)+1)$ by Lemmas \ref{weakcontperinsleft} and \ref{contperweakuniquewideright}, and
$\lambda$ is a 
left weak evolutional period of $\mathcal F$ 
for the pair $(0,m)$
since 
$\IpR_{k,0,1}(\mathcal F_l)$ is always an empty occurrence.

%
%

The claim now follows from Lemma \ref{sameperiodcontper}.
\end{proof}

\begin{lemma}\label{corenottotcontperii}
Let $k\in\NN$.
Let $\mathcal E$ be an evolution of $(k+1)$-blocks such that Case II holds at the 
left (resp.\ at the right). 
Denote by $\mathcal F$ the following evolution of 
$k$-multiblocks: $\mathcal F_l=\Cr_{k+1}(\mathcal E_{l+3(k+1)})$.

Suppose that $\mathcal F$ 
is continuously periodic for an index $m$ ($1\le m\le \nker_k(\mathcal F)$), but
is \textbf{not} totally periodic.

Then there are two possibilities:
\begin{enumerate}
\item There exists a left (resp.\ right) continuous evolutional period of $\mathcal E$ for the index $m$.
\item There exists a $k$-series of obstacles in $\al$.
\end{enumerate}
\end{lemma}

\begin{proof}
We prove the lemma for the situation when Case II holds at the right. If Case II holds at the left, 
the proof is completely symmetric.
Suppose that $k$-series of obstacles do not exist in $\al$. We have to prove that 
there exists a right continuous evolutional period of $\mathcal E$ for the index $m$.

If $m=\nker_k(\mathcal F)=\ncker_{k+1}(\mathcal E)$, then,
as we have already noted after the definition of a continuous evolutional period of 
an evolution of $k$-blocks,
any final period is a right continuous evolutional period of $\mathcal E$ for the index $m$.

Suppose that $m<\nker_k(\mathcal F)$. Denote by $\lambda$ the (unique by Corollary \ref{contperweakuniqueright})
weak right evolutional period of $\mathcal F$ for the pair $(m,\nker_k(\mathcal F)+1)$.
Since Case II holds for $\mathcal E$ at the right, $\Fg(\Cr_{k+1}(\mathcal E_l))$ 
is a suffix of $\mathcal E_l$ for all $l\ge 1$, 
and $\IpR_{k,m,\nker_k(\mathcal F)+1}(\mathcal F_l)=\RpR_{k+1,m}(\mathcal E_{l+3(k+1)})$.
By Lemma \ref{rbsperiodic}, $\psi(\RBS_{k+1}(\mathcal E))$ is periodic with period $\lambda$, 
so $\lambda$ is a right continuous evolutional period of $\mathcal E$ 
directly by definition.
\end{proof}

\begin{lemma}\label{contperonetwo}
Let $k\in\NN$.
Suppose that all evolutions of $k$-blocks in $\al$ are continuously periodic.
Let $\mathcal E$ be an evolution of $(k+1)$-blocks such that Case I holds 
at the left and Case II holds at the right.

Then there are two possibilities:
\begin{enumerate}
\item $\mathcal E$ is continuously periodic.
\item There exists a $k$-series of obstacles in $\al$.
\end{enumerate}
\end{lemma}

\begin{proof}
Suppose that $k$-series of obstacles do not exist in $\al$. We have to prove that 
$\mathcal E$ is continuously periodic.

If $\ncker_{k+1}(\mathcal E)=1$, then the claim follows from 
Corollary \ref{lrperiodic}.
Suppose that $\ncker_{k+1}(\mathcal E)>1$.

Again, denote by $\mathcal F$ the evolution of stable nonempty $k$-multiblocks 
defined by $\mathcal F_l=\Cr_{k+1}(\mathcal E_{l+3(k+1)})$. 
By Corollary \ref{contpercr}, $\mathcal F$ is continuously periodic.
If $\mathcal F$ 
is not totally periodic, then the claim follows from Lemmas \ref{corenottotcontper} and \ref{corenottotcontperii}.

Suppose that $\mathcal F$ is totally periodic.
Denote by $\lambda$ and $\mu$
the (unique by Corollaries \ref{contperweakuniqueleft} and \ref{contperweakuniqueright} 
since $\nker_k(\mathcal F)=\ncker_{k+1}(\mathcal E)>1$)
left and right (respectively) total evolutional periods of $\mathcal F$.
It follows from Lemmas \ref{weakcontperinsright} and \ref{contperweakuniquewideleft}
that $\lambda$ is also a weak left evolutional period of $\mathcal F$ 
for the pair $(0,\nker_k(\mathcal F))$, and 
it follows from Lemmas \ref{weakcontperinsleft} and \ref{contperweakuniquewideright}
that $\mu$ is also a weak right evolutional period of $\mathcal F$ 
for the pair $(1,\nker_k(\mathcal F)+1)$.

Consider the 
evolution $\mathcal F'$ of stable nonempty $k$-multiblocks defined by 
$\mathcal F'_l=\LA_{k+1,2}(\mathcal E_{l+3(k+1)})$. $\mathcal F'$ and $\mathcal F$
are consecutive, denote their concatenation by $\mathcal F''$.
By Corollary \ref{lrperiodic}, there exists a final period $\lambda'$ 
such that $\lambda'$ is both left and right total evolutional 
period of $\mathcal F'$.
By Lemma \ref{weakcontperinsright}, $\mathcal F'$ is also 
weakly periodic for the sequence $0,\nker_k(\mathcal F'),\nker_k(\mathcal F')+1$.
Now Lemma \ref{concatcontper} says 
that $\mathcal F''$ is weakly periodic for the sequence $0,\nker_k(\mathcal F'),\nker_k(\mathcal F'')+1$.

First, let us consider the case when $\mathcal F''$ is not totally periodic.
We know that $\mu$ is a weak right evolutional period of $\mathcal F$ 
for the pair $(1,\nker_k(\mathcal F)+1)$. By Lemma \ref{concatcompker},
$\IpR_{k,\nker_k(\mathcal F'),\nker_k(\mathcal F'')+1}(\mathcal F'')=\IpR_{k,1,\nker_k(\mathcal F)+1}(\mathcal F)$
for all $l\ge 0$
as an occurrence in $\al$.
Hence,
$\mu$ is also a weak right evolutional period of $\mathcal F''$ 
for the pair $(\nker_k(\mathcal F'),\nker_k(\mathcal F'')+1)$.
By Lemma \ref{rbsperiodic}, $\psi(\RBS_{k+1}(\mathcal E))$
is periodic with period $\mu$.
Since Case II holds for $\mathcal E$ at the right, 
$\IpR_{k,1,\nker_k(\mathcal F)+1}(\mathcal F_l)=\RpR_{k+1,1}(\mathcal E_{l+3(k+1)})$.
Therefore, $\mu$ is a right continuous evolutional period of $\mathcal E$ for index 1
by definition. It also follows from Corollary \ref{lrperiodic}
that $\lambda'$ is a left continuous evolutional period of $\mathcal E$ for index 1, 
and $\mathcal E$ is continuously periodic.

Now suppose that $\mathcal F''$ is totally periodic.
We know that $\nker_k(\mathcal F)=\ncker_{k+1}(\mathcal E)>1$, 
that $\mathcal F'$ and $\mathcal F$ are totally periodic, 
and that
$\lambda$ is a weak left evolutional period of $\mathcal F$ 
for the pair $(0,\nker_k(\mathcal F)+1)$. By Lemma \ref{contperexpandleft},
$\lambda$ is also a total right evolutional period of $\mathcal F'$.
By Lemma \ref{atleasttwokersleft}, $\nker_k(\mathcal F')>1$,
so by Corollary \ref{contperweakuniqueright}, 
$\lambda=\lambda'$. Now recall that 
$\lambda$ is also a weak left evolutional period of $\mathcal F$ 
for the pair $(0,\nker_k(\mathcal F))$
and that
Corollary \ref{lrperiodic} also says that
for all $l\ge 0$, 
$\psi(\Fg(\LR_{k+1}(\mathcal E_{l+3(k+1)})))$
is a completely $\lambda$-periodic word.
Therefore, by Lemma \ref{sameperiodcontper},
$\lambda$ is a left continuous evolutional period of $\mathcal E$ for the index $\nker_k(\mathcal F)$.
And again, since Case II holds for $\mathcal E$ at the right, any final period is 
a right continuous evolutional period 
of $\mathcal E$ for the index $\ncker_{k+1}(\mathcal F)=\nker_k(\mathcal F)$, 
and $\mathcal E$ is continuously periodic.
\end{proof}

\begin{lemma}\label{contpertwoone}
Let $k\in\NN$.
Suppose that all evolutions of $k$-blocks in $\al$ are continuously periodic.
Let $\mathcal E$ be an evolution of $(k+1)$-blocks such that Case II holds 
at the left and Case I holds at the right.

Then there are two possibilities:
\begin{enumerate}
\item $\mathcal E$ is continuously periodic.
\item There exists a $k$-series of obstacles in $\al$.
\end{enumerate}
\end{lemma}

\begin{proof}
The proof is completely symmetric to the proof of the previous lemma.
\end{proof}

\begin{lemma}\label{contperbothtwo}
Let $k\in\NN$.
Suppose that all evolutions of $k$-blocks in $\al$ are continuously periodic.
Let $\mathcal E$ be an evolution of $(k+1)$-blocks such that Case II holds 
both at the left and at the right.

Then there are two possibilities:
\begin{enumerate}
\item $\mathcal E$ is continuously periodic.
\item There exists a $k$-series of obstacles in $\al$.
\end{enumerate}
\end{lemma}

\begin{proof}
Suppose that $k$-series of obstacles do not exist in $\al$. We have to prove that 
$\mathcal E$ is continuously periodic.

Again, consider the evolution $\mathcal F$ of stable nonempty $k$-multiblocks 
defined by $\mathcal F_l=\Cr_{k+1}(\mathcal E_{l+3(k+1)})$. 
By Corollary \ref{contpercr}, $\mathcal F$ is continuously periodic.
If $\mathcal F$ 
is not totally periodic, then the claim follows from Lemma \ref{corenottotcontperii}.

Suppose that $\mathcal F$ is totally periodic. If $\ncker_k(\mathcal E)=\nker_k(\mathcal F)=1$, then 
$\mathcal E$ is automatically continuously periodic, 
as we have noted right after the definition of a continuously
periodic evolution of $k$-multiblocks.

If $\ncker_k(\mathcal E)=\nker_k(\mathcal F)>1$, denote the 
(unique by Corollaries \ref{contperweakuniqueleft} and \ref{contperweakuniqueright})
left and right total evolutional periods of $\mathcal F$ by $\lambda$ and $\mu$, 
respectively.
By Lemma \ref{lbsrbsperiodic}, either $\psi(\LBS_{k+1}(\mathcal E))$
is periodic with period $\lambda$, 
or 
$\psi(\RBS_{k+1}(\mathcal E))$
is periodic with period $\mu$.
If $\psi(\LBS_{k+1}(\mathcal E))$
is periodic with period $\lambda$, 
then $\lambda$ is a left continuous
evolutional period of $\mathcal E$ for the index $\ncker_{k+1}(\mathcal E)$, 
and any final period is a right continuous 
evolutional period of $\mathcal E$ for the index $\ncker_{k+1}(\mathcal E)$.
If $\psi(\RBS_{k+1}(\mathcal E))$
is periodic with period $\mu$,
then $\mu$ is a right continuous
evolutional period of $\mathcal E$ for the index 1, 
and any final period is a left continuous 
evolutional period of $\mathcal E$ for the index 1.
\end{proof}

\begin{proposition}\label{contpergreatalt}
Let $k\in\NN$.
Suppose that all evolutions of $k$-blocks in $\al$ are continuously periodic.
Then either all evolutions of $(k+1)$-blocks in $\al$ are continuously periodic, 
or there exists a $k$-series of obstacles in $\al$.
\end{proposition}

\begin{proof}
This follows directly from Lemmas \ref{contperbothone}, \ref{contperonetwo}, 
\ref{contpertwoone}, and \ref{contperbothtwo}.
\end{proof}

\section{Subword complexity}\label{sectionsubwordcompl}

In this section, we will prove 
Propositions \ref{largecompl}--\ref{infordercomplstop}
and Theorem \ref{maintheorem}.

\begin{lemma}\label{obstcompl}
Suppose that there exists a $k$-series of obstacles $\mathcal H$ in $\al$.
Then the subword complexity of $\beta=\psi(\al)$ is $\Omega(n^{1+1/k})$.
\end{lemma}

\begin{proof}
Let $k'\in\NN$ ($1\le k'\le k$) be the number such
that $|\mathcal H_l|=\Theta(l^{k'})$ for $l\to\infty$.
This means that there exist 
$l_0\in\NN$ and $x,y\in\R_{>0}$ such that 
if $l\ge l_0$, then $xl^{k'}<|\mathcal H_l|<yl^{k'}$.

Fix an arbitrary $n\in\NN$, $n>4yl_0^{k'}$. We are going to find a 
lower estimate for the amount of different subwords of $\beta$ of length $n$.
Set 
$$
l_1=\sqrt[k']{\frac n{4y}},\quad l_2=\sqrt[k']{\frac n{2y}},\quad 
l_3=l_2-l_1=\left(\sqrt[k']{\frac12}-\sqrt[k']{\frac14}\right)\sqrt[k']{\frac ny}.
$$
Then $l_0<l_1<l_2$, and there exist at least $l_3-1$ indices $l\in\NN$ such that $l_1\le l\le l_2$.
Consider the occurrences $\mathcal H_l$ in $\al$ for $l_1\le l\le l_2$. 
Since $|\mathcal H_l|$
strictly grows as $l$ grows, all these occurrences have different lengths. 
Moreover, 
if $l_1\le l\le l_2$, then
$$
|\mathcal H_l|>xl^{k'}\ge xl_1^{k'}=\frac x{4y}n,
$$
and 
$$
|\mathcal H_l|<yl^{k'}\le yl_2^{k'}=\frac12 n.
$$
Denote
$$
n_0=\frac x{4y}n.
$$
Then if $l_1\le l\le l_2$, then $n_0<|\mathcal H_l|<n/2$.

For each $l\in\NN$, $l_1\le l\le l_2$, denote by $i_l$ and 
$j_l$ the indices such that $\mathcal H_l=\swa{i_l}{j_l}$.
Since $\mathcal H_l$ is a series of obstacles, there exists 
$p\in\NN$ ($p\le\finmax$) such that 
all words $\psi(\swa{i_l}{j_l})$
are weakly $p$-periodic, and all words 
$\psi(\swa{i_l}{j_l+1})$
and 
(if $i_l>0$)
$\psi(\swa{i_l-1}{j_l})$
are not.
Since all words $\mathcal H_l$ have different lengths, $i_l$ cannot
coincide with $i_{l'}$ if $l\ne l'$ ($l_1\le l,l'\le l_2$).
Denote by $m$ ($l_1\le m\le l_2$) the index such that 
$i_m=\min_{l_1\le l\le l_2}i_l$. Let us check
that if $l\ne m$, $l_1\le l\le l_2$, then 
$i_l>n_0-2\finmax$.

Indeed, assume that $i_l\le n_0-2\finmax$. Then $i_m<n_0-2\finmax$
since $i_m<i_l$. But $j_m=j_m-i_m+1+i_m-1=|\mathcal H_m|+i_m-1>n_0-1$, 
Similarly, $j_l>n_0-1$. So, if $t=\min(j_m,j_l)$, then $t>n_0-1$, 
so $t\ge i_l$ and $|\swa{i_l}t|=t-i_l+1>n_0-1-n_0+2\finmax+1=2\finmax\ge 2p$.
Denote $t'=\max(j_m,j_l)$. By Corollary \ref{overlapperiod},
$\psi(\swa{i_m}{t'})$ is a weakly $p$-periodic word.
In particular, $\psi(\swa{i_m}{j_l})$ is a weakly 
$p$-periodic word, 
but this contradicts the assumption that 
$\psi(\swa{i_l-1}{j_l})$ is not a $p$-periodic word.

Consider the following occurrences in $\al$: $\swa{i_l-s}{i_l-s+n-1}$, where 
$0\le s\le n_0-2\finmax$ and $l_1\le l\le l_2$, $l\ne m$. We already know that if 
$l_1\le l\le l_2$, $l\ne m$, then $i_l>n_0-2\finmax$, so if 
$0\le s\le n_0-2\finmax$ and $l_1\le l\le l_2$, $l\ne m$, then $i_l-s>0$.
Clearly, all these occurrences have length $n$. Let us prove that 
all words $\psi(\swa{i_l-s}{i_l-s+n-1})$ are different \textit{abstract words}.
(If $n_0-2\finmax<0$, then we have no occurrences, but $n_0-2\finmax\ge0$ if $n$ is 
large enough. During the proof that all these abstract words are different, we suppose that 
$n_0-2\finmax\ge0$, and we have at least one word.)

Denote $\mathcal T_{s,l}=\psi(\swa{i_l-s}{i_l-s+n-1})$.
Temporarily fix an index $s$ ($0\le s\le n_0-2\finmax$)
and an index $l$ ($l_1\le l\le l_2$, $l\ne m$). Denote 
$\gamma=\mathcal T_{s,l}$.
Then $\gamma_v=\psi(\al_{i_l-s+v})$
for $0\le v\le n-1$.
We have $j_l-i_l+1=|\mathcal H_l|<n/2$ and 
$s\le n_0-2\finmax<n/2$, so $n/2<n-s$, $j_l-i_l+1<n/2<n-s$, 
and $s+j_l-i_l+1<n$.
Hence, $\gamma_{s\ldots s+j_l-i_l}$, 
$\gamma_{s\ldots s+j_l-i_l+1}$,
and (if $s>0$)
$\gamma_{s-1\ldots s+j_l-i_l}$
are occurrences in $\gamma$.
We have 
$\gamma_{s\ldots s+j_l-i_l}=\psi(\swa{i_l-s+s}{i_l-s+s+j_l-i_l})=\psi(\swa{i_l}{j_l})$.
Similarly,
$\gamma_{s\ldots s+j_l-i_l+1}=\psi(\swa{i_l}{j_l+1})$
and
(if $s>0$)
$\gamma_{s-1\ldots s+j_l-i_l}=\psi(\swa{i_l-1}{j_l})$
(the notation $\swa{i_l-1}{j_l}$ is well-defined since $i_l>n_0-2\finmax\ge 0$).
Therefore, $\gamma_{s\ldots s+j_l-i_l}$
is a weakly $p$-periodic word, and 
$\gamma_{s\ldots s+j_l-i_l+1}$
and (if $s>0$)
$\gamma_{s-1\ldots s+j_l-i_l}$
are not.

Now assume that $\mathcal T_{s,l}=\mathcal T_{s',l'}$ as an abstract word,
where $s\ne s'$ or $l\ne l'$. 
(Here $0\le s,s'\le n_0-2\finmax$, $l_1\le l,l'\le l_2$, $l\ne m$, and $l'\ne m$.)
Denote $\gamma=\mathcal T_{s,l}=\mathcal T_{s',l'}$.
First, let us consider the case when $s=s'$ and $l\ne l'$.
Without loss of generality, $l<l'$, so 
$j_l-i_l+1=|\mathcal H_l|<|\mathcal H_{l'}|=j_{l'}-i_{l'}+1$, 
and $j_l-i_l+1\le j_{l'}-i_{l'}$.
Then $\gamma_{s\ldots s+j_l-i_l+1}$ is a prefix of 
$\gamma_{s\ldots s+j_{l'}-i_{l'}}$, 
but $\gamma_{s\ldots s+j_{l'}-i_{l'}}$ is a $p$-periodic word, 
and $\gamma_{s\ldots s+j_l-i_l+1}$ is not, so we have a contradiction.

Now suppose that $s\ne s'$. Without loss of generality, $s'<s$, so $s>0$.
Since $|\mathcal H_l|=j_l-i_l+1\ge 2\finmax$, $(s+j_l-i_l)-s+1\ge 2\finmax$.
Since $|\mathcal H_{l'}|>n_0$ and $s'\ge 0$, we also have 
$s'+j_{l'}-i_{l'}+1>n_0$. Since $s\le n_0-2\finmax$, 
$(s'+j_{l'}-i_{l'})-s+1>n_0-n_0+2\finmax=2\finmax$. Therefore, if 
$t=\min(s+j_l-i_l,s'+j_{l'}-i_{l'})$,
then $t-s+1\ge 2\finmax\ge 2p$. Now we can use Corollary \ref{overlapperiod}.
Recall that $\gamma_{s\ldots s+j_l-i_l}$
and $\gamma_{s'\ldots s'+j_{l'}-i_{l'}}$
are weakly $p$-periodic words. Denote $t'=\max(s+j_l-i_l,s'+j_{l'}-i_{l'})$.
By Corollary \ref{overlapperiod},
$\gamma_{s'\ldots t'}$ is a $p$-periodic word. Hence, 
$\gamma_{s-1\ldots t'}$ is also a $p$-periodic word (recall that $s'<s$ and $s>0$),
and $\gamma_{s-1\ldots s+j_l-i_l}$ is also a $p$-periodic word.
But previously we have seen that 
$\gamma_{s-1\ldots s+j_l-i_l}$ is not a $p$-periodic word,
so we have a contradiction.

Let us count how many words $\mathcal T_{s,l}$ we have. If $n_0<2\finmax$, 
then we have none of them,  and if
$$
n_0\ge 2\finmax\Leftrightarrow \frac x{4y}n\ge 2\finmax\Leftrightarrow n\ge \frac{8y\finmax}x,
$$
then there are $n_0-2\finmax+1$ possibilities for $s$ and at least $l_2-l_1-2=l_3-2$
possibilities for $l$. Hence, we have at least
$$
(n_0-2\finmax+1)(l_2-l_1-2)=
\left(\frac x{4y}n-2\finmax+1\right)\left(\left(\sqrt[k']{\frac12}-\sqrt[k']{\frac14}\right)\sqrt[k']{\frac ny}-2\right)
$$
different subwords of $\beta=\psi(\al)$, and the subword complexity of $\beta$ is $\Omega(n^{1+1/k'})$ for $n\to\infty$.
But $k'\le k$, so the subword complexity of $\beta$ is also $\Omega(n^{1+1/k})$ for $n\to\infty$.
\end{proof}

Note that in the proof of this lemma, we proved in fact that if $|\mathcal H_l|$ is 
$\Theta(l^{k'})$ for $l\to\infty$, 
then
the subword complexity of $\beta=\psi(\al)$ is $\Omega(n^{1+1/k'})$ for $n\to\infty$, 
and $n^{1+1/k'}$ is $\omega(n^{1+1/k})$ if $k'<k$.
Later, after we prove Proposition \ref{smallcompl},
we will see that this is not possible if 
evolutions of $k$-blocks really exist in $\al$ and all of them are continuously 
periodic.
Therefore, if $k$-blocks really exist in $\al$, then 
all $k$-series of obstacles obtained from Proposition \ref{contpergreatalt}
actually satisfy $|\mathcal H_l|=\Theta(l^{k})$ for $l\to\infty$. However, it was not very 
convenient to prove this directly, so in the definition of a $k$-series of obstacles
we allowed $|\mathcal H_l|=\Theta(l^{k'})$ for some $k'\le k$.

\begin{proof}[Proof of Proposition \ref{largecompl}]
We know that there exists a non-continuously periodic 
evolution of $k$-blocks, and we also know (see Remark \ref{onecontper})
that all evolutions of 1-blocks arising in $\al$ are continuously
periodic. Let $k'\in\NN$ be the largest
number such that all evolutions of $k'$-blocks arising in $\al$ are continuously
periodic. Then $k'\le k-1$. By Proposition \ref{contpergreatalt}, 
there exists a $k'$-series of obstacles in $\al$.
By Lemma \ref{obstcompl}, the subword complexity of 
$\be=\psi(\al)$ is $\Omega(n^{1+1/k'})$. But $k'\le k-1$, so 
$n^{1+1/k'}\ge n^{1+1/(k-1)}$, and 
the subword complexity of 
$\be$ is $\Omega(n^{1+1/(k-1)})$.
\end{proof}

The proof of Proposition \ref{smallcompl} is based on the following lemma.
\begin{lemma}\label{smallcompllemma}
Let $k\in\NN$.
Suppose that $a\in\E$ is a letter of order at least $k+2$ such that $\varphi(a)=a\gamma$ for some $\gamma\in\E^*$, 
and all evolutions of $k$-blocks arising in $\al=\varphi^\infty(a)$ are continuously periodic.

Let 
$f\colon \NN\to \R$ 
be a 
function and $n_0\in\NN$, $n_0>1$ be a number such that:
\begin{enumerate}
\item If $n\ge n_0$, then $f(n)\ge 3k$.
\item If $\mathcal E$ is an evolution of $k$-blocks such that Case I holds at the left (resp.\ at the right)
and $l\ge f(n)$ for some $l,n\in\NN$, $n\ge n_0$, then 
$|\Fg(\LR_k(\mathcal E_l))|>n$ (resp.\ $|\Fg(\RR_k(\mathcal E_l))|>n$).
\item If $b$ is a letter of order $>k$ 
and $l\ge f(n)$ for some $l,n\in\NN$, $n\ge n_0$, then 
$|\varphi^l(b)|>n$.
\end{enumerate}

Then the subword complexity of $\be=\psi(\al)$ is $O(nf(n))$.
\end{lemma}

\begin{proof}
Denote the total number of all abstract words that can equal the 
forgetful occurrences of all left and right preperiods of stable $k$-blocks or composite central kernels of $k$-blocks 
by $M$ (by Corollary \ref{finitelrprep} and by Lemma \ref{finitecker}, this number is finite). 
Denote the maximal length of the forgetful occurrence of a left or a right 
preperiod or of a composite central kernel of a stable $k$-block by $P$.
Denote the number of all final periods we have by $N$.

Fix a number $n\in \NN$ ($n\ge n_0$ and $n>P$.)
Let $\swa ij$ be an occurrence in $\al$ of length $n\ge n_0$. Set $l_0=\lceil f(n)\rceil+1$.
Since $\al=\varphi^\infty(a)$, $\al$ can be written as $\al=\varphi(\al)$
and as $\al=\varphi^l(\al)=\varphi^l(\al_0)\varphi^l(\al_1)\varphi^l(\al_2)\ldots$
for all $l\ge 0$.

Let $s'$ and $t'$ be the indices such that $\al_i$ (resp.\ $\al_j$) is contained in
$\varphi^{l_0}(\al_{s'})$ (resp.\ in $\varphi^{l_0}(\al_{t'})$) as an occurrence in $\al$.
Clearly, $s'\le t'$. If $s'<t'$, then set $s=s'$, $t=t'$, and $q=l_0$.

If $s'=t'$, then for each $l$ ($0\le l\le l_0$) denote by $s''_l$ and $t''_l$ the indices 
such that
$\al_i$ (resp.\ $\al_j$) is contained in
$\varphi^l(\al_{s''_l})$ (resp.\ in $\varphi^l(\al_{t''_l})$) as an occurrence in $\al$.
Let $q$ be the maximal value of $l$ such that $s''_l<t''_l$. Then $s''_{q+1}=t''_{q+1}$, 
and both $\al_{s''_{q}}$ and $\al_{t''_{q}}$ are contained in $\varphi(\al_{s''_{q+1}})$
as an occurrence in $\al$. So, $|\swa{s''_q}{t''_q}|=t''_q-s''_q+1\le |\varphi|$. Set $s=s''_q$ and $t=t''_q$.

Summarizing, we have found indices $s$ and $t$ and a number $q\in\ZZ_{\ge 0}$ ($0\le q\le l_0$)
such that:
\begin{enumerate}
\item $s<t$.
\item $\al_i$ (resp.\ $\al_j$) is contained in
$\varphi^q(\al_s)$ (resp.\ in $\varphi^q(\al_t)$) as an occurrence in $\al$.
\item If $q<l_0$, then $t-s<|\varphi|$.
\end{enumerate}

We are going to estimate the amount of different words that can be equal to $\psi(\swa ij)$
as abstract words (for different $i$ and $j$ such that $|\swa ij|=j-i+1=n$). We will consider the cases 
$q=l_0$ and $q<l_0$ separately.

First, suppose that $q=l_0$. Then we don't have any explicit upper estimates for $|\swa st|$
so far, but we can say that if $|\swa st|>2$, then $\varphi^q(\swa{s+1}{t-1})$
is a suboccurrence in $\swa ij$. So, $|\varphi^q(\swa{s+1}{t-1})|\le n$, and 
$\swa{s+1}{t-1}$ cannot contain letters of order $>k$. Since $\al_0=a$ is a letter of order at least 
$k+2$, by Lemma \ref{splitconcatenation},
$\al$ can be split into a concatenation of $k$-blocks and letters of order $>k$, so
$\swa{s+1}{t-1}$
is a suboccurrence in a $k$-block. 
Denote this $k$-block by $\swa{u'}{v'}$. Then $\varphi^q(\swa{s+1}{t-1})$
is a nonempty suboccurrence in both $\swa ij$ and $\Su_k^q(\swa{u'}{v'})$.
Let $u$ and $v$ be the indices such that $\Su_k^q(\swa{u'}{v'})=\swa uv$.
If $u>i$, then 
$\al_{u-1}$ is a letter of order $>k$, and this letter must be contained 
in $\varphi^q(\al_s)$, so $\al_s$ must be a letter of order $>k$, and $u'=s+1$.
Similarly, if $v<j$, then 
$\al_{v+1}$ is a letter of order $>k$, and this letter must be contained 
in $\varphi^q(\al_t)$, so $\al_t$ must be a letter of order $>k$, and $v'=t-1$.

Summarizing, we have the following cases:
\begin{enumerate}
\item\label{smallcomplcasei} $|\swa st|=2$, and $t=s+1$.
\item\label{smallcomplcaseiigen} $|\swa st|>2$. There exists a (unique) nonempty 
$k$-block $\swa{u'}{v'}$ such that $\swa{s+1}{t-1}$ is a suboccurrence 
in $\swa{u'}{v'}$. Denote $\Su_k^q(\swa{u'}{v'})=\swa uv$. Then there are the following possibilities:
\begin{enumerate}
\item\label{smallcomplcaseii} $u\le i$ and $v\ge j$, so $\swa ij$ is a suboccurrence in $\swa uv$.
\item\label{smallcomplcaseiii} $u>i$, but $v\ge j$, then $u'=s+1$ and $\al_s$ is a letter of order $>k$.
\item\label{smallcomplcaseiv} $u\le i$, but $v<j$, then $v'=t-1$ and $\al_t$ is a letter of order $>k$.
\item\label{smallcomplcasev} $u>i$ and $v<j$, then $u'=s+1$, $v'=t-1$, and both $\al_s$ and $\al_t$ are letters of order $>k$.
\end{enumerate}
\end{enumerate}
Let us consider these cases one by one.

\underline{Case \ref{smallcomplcasei}.} There exists $x\in\NN$ ($1\le x\le n-1$)
such that $\swa ij$ is the concatenation of the suffix of $\varphi^q(\al_s)$
of length $x$ and the prefix of $\varphi^q(\al_t)$ of length $n-x$. 
There are at most $|\E|^2(n-1)$ possibilities for $\swa ij$ as an abstract 
word.
Therefore, there are at most $|\E|^2(n-1)$ possibilities for $\psi(\swa ij)$ as an
abstract word.

\underline{Case \ref{smallcomplcaseiigen}.} Observe that, since $\swa uv=\Su_k^q(\swa uv)$, 
the evolutional sequence number of $\swa uv$ is at least $q=l_0>f(n)$. In particular, 
$q\ge 3k$, and $\swa uv$ is a stable $k$-block. 
Denote the evolution $\swa uv$ belongs to by $\mathcal E$.
If Case I holds for $\mathcal E$
at the left (resp.\ at the right), then $|\Fg(\LR_k(\swa uv))|>n$
(resp.\ $|\Fg(\RR_k(\swa uv))|>n$). 
If Case I holds for $\mathcal E$
at the left or at the right, then 
$|\Fg(\LR_k(\swa uv))\Fg(\Cr_k(\swa uv))\Fg(\RR_k(\swa uv))|>n$,
and 
$|\swa uv|=v-u+1>n$.
We know that all evolutions of $k$-blocks in $\al$ are continuously periodic, 
so let $m$ ($1\le m\le \ncker_k(\mathcal E)$) be an index such that 
$\mathcal E$ is continuously periodic for the index $m$.
Let $\lambda$ (resp.\ $\mu$) be a left (resp.\ right) continuous evolutional 
period of $\mathcal E$ for the index $m$.

Let us check that if Case II holds for $\mathcal E$ at the left and $u>i$, then 
$\swa i{u-1}$ is a suffix of $\LBS_k(\mathcal E)$. 
Recall that in this case, $u'-1=s$.
Denote $\swa{u''}{v''}=\Su_k(\swa{u'}{v'})$. Then $\al_{u''-1}$
is the rightmost letter of order $>k$ in $\varphi(\al_s)$.
Let $\al_w$ be the rightmost letter in $\varphi^{q-1}(\al_{u''-1})$, 
and let $\al_{w'}$ be the leftmost letter in $\varphi^q(\al_{u'})$.
Then $\al_w$ is contained in $\varphi^q(\al_s)$, so 
$w<w'$. We have $s<u'<t$, so $\varphi^q(\al_{u'})$ is a suboccurrence in $\swa ij$, 
hence $\al_{w'}$ is contained in $\swa ij$, 
$w'\le j$, and $w<j$. On the other hand, $\al_{w+1}$ is contained in 
$\varphi^{q-1}(\al_{u''})$, so $\al_{w+1}$ is contained in 
$\Su_k^{q-1}(\swa{u''}{v''})=\Su_k^q(\swa{u'}{v'})=\swa uv$, 
$w+1\ge u>i$, and $w\ge i$. Therefore, either $\varphi^{q-1}(\al_{u''-1})$
is a subword in $\swa ij$, or $\swa iw$ is a suffix of $\varphi^{q-1}(\al_{u''-1})$.
But $\varphi^{q-1}(\al_{u''-1})$ cannot be a subword of $\swa ij$ since 
$\al_{u''-1}$ is a letter of order $>k$ and $q-1=\lceil f(n)\rceil\ge f(n)$, 
so $|\varphi^{q-1}(\al_{u''-1})|>n$. Therefore, 
$\swa iw$ is a suffix of $\varphi^{q-1}(\al_{u''-1})$.
We have $\al_{u''-1}=\LB(\swa{u''}{v''})$, $\al_{u-1}=\LB(\swa uv)$, and 
$\swa uv=\Su_k^{q-1}(\swa{u''}{v''})$. We also have $\swa{u''}{v''}=\Su_k(\swa{u'}{v'})$, 
so the evolutional sequence number of $\swa{u''}{v''}$ is at least 1. 
Now it follows from Remark \ref{lbsremarkii} that 
$\swa i{u-1}$ is a suffix of $\LBS_k(\mathcal E)$.

Note that we could not use $\swa{u'}{v'}$
instead of $\swa{u''}{v''}$ in this argument since the evolutional sequence number of $\swa{u'}{v'}$
could equal 0.

Similarly, if Case II holds for $\mathcal E$ at the right and $v<j$, then 
$\swa{v+1}j$ is a prefix of $\RBS_k(\mathcal E)$.

If $u>i$, but Case I holds for $\mathcal E$ at the left, then we did not define 
any left bounding sequence, but we know that independently on whether Case I or II 
holds for $\mathcal E$ at the left, if $u>i$, then $\al_{u-1}$ is the rightmost letter 
of order $>k$ in $\varphi(\al_s)$. So, if $u>i$ (and $\al_s=\LB(\swa{u'}{v'})$ is a letter
of order $>k$), denote by $\gamma$ the prefix of $\varphi^q(\al_s)$ that ends with 
the rightmost letter of order $>k$ in $\varphi^q(\al_s)$. Then $\swa i{u-1}$ is a suffix of $\gamma$. 
Clearly, 
$\gamma$ as an abstract word depends only on $q$ and on $\al_s$ as an 
abstract letter.

Similarly, if $v<j$, denote by $\gamma'$
the suffix of $\varphi^q(\al_t)$ that begins with the leftmost letter of 
order $>k$ in $\varphi^q(\al_t)$. Then $\swa{v+1}j$ is a prefix of $\gamma'$, and 
$\gamma'$ as an abstract word depends only on $q$ and on $\al_t$ as an 
abstract letter.

Now let us consider cases \ref{smallcomplcaseii}--\ref{smallcomplcasev} one by one.

\underline{Case \ref{smallcomplcaseii}.}
We can write $\swa uv$ as 
$$
\swa uv=\Fg(\LpreP_k(\swa uv))\LpR_{k,m}(\swa uv)\cKer_{k,m}(\swa uv)\RpR_{k,m}(\swa uv)\Fg(\RpreP_k(\swa uv)).
$$
If Case I holds at the left, then $|\LpR_{k,m}(\swa uv)|\ge|\Fg(\LR_k(\swa uv))|>n$, 
and
$\swa ij$ is a suboccurrence either in 
$\Fg(\LpreP_k(\swa uv))\LpR_{k,m}(\swa uv)$, or in 
$\LpR_{k,m}(\swa uv)\cKer_{k,m}(\swa uv)\RpR_{k,m}(\swa uv)\Fg(\RpreP_k(\swa uv))$.
If Case II holds at the left, then 
$\Fg(\LpreP_k(\swa uv))$ is empty, and 
$\swa ij$ is a suboccurrence in 
$\LpR_{k,m}(\swa uv)\cKer_{k,m}(\swa uv)\RpR_{k,m}(\swa uv)\Fg(\RpreP_k(\swa uv))$.
So, independently on whether Case I or Case II holds at the left, 
$\swa ij$ is a suboccurrence either in 
$\Fg(\LpreP_k(\swa uv))\LpR_{k,m}(\swa uv)$, or in 
$\LpR_{k,m}(\swa uv)\cKer_{k,m}(\swa uv)\RpR_{k,m}(\swa uv)\Fg(\RpreP_k(\swa uv))$.

If 
$\swa ij$ is a suboccurrence in 
$\LpR_{k,m}(\swa uv)\cKer_{k,m}(\swa uv)\RpR_{k,m}(\swa uv)\Fg(\RpreP_k(\swa uv))$, 
and Case I holds at the right, then 
$|\RpR_{k,m}(\swa uv)|\ge|\Fg(\RR_k(\swa uv))|>n$, 
and
$\swa ij$ is a suboccurrence either in 
$\LpR_{k,m}(\swa uv)\cKer_{k,m}(\swa uv)\RpR_{k,m}(\swa uv)$, or in 
$\RpR_{k,m}(\swa uv)\Fg(\RpreP_k(\swa uv))$.
If 
$\swa ij$ is a suboccurrence in 
$\LpR_{k,m}(\swa uv)\cKer_{k,m}(\swa uv)\RpR_{k,m}(\swa uv)\Fg(\RpreP_k(\swa uv))$, 
and Case II holds at the right, then 
$\Fg(\RpreP_k(\swa uv))$ is empty, and 
$\swa ij$ is a suboccurrence in 
$\LpR_{k,m}(\swa uv)\cKer_{k,m}(\swa uv)\RpR_{k,m}(\swa uv)$ anyway.

Therefore, independently on whether Case I or II holds at the left or at the right,
there are three possibilities for $\swa ij$:
$\swa ij$ is a suboccurrence either in 
$\Fg(\LpreP_k(\swa uv))\LpR_{k,m}(\swa uv)$, or in 
$\LpR_{k,m}(\swa uv)\cKer_{k,m}(\swa uv)\RpR_{k,m}(\swa uv)$, or in 
$\RpR_{k,m}(\swa uv)\Fg(\RpreP_k(\swa uv))$.

If $\swa ij$ is a suboccurrence of 
$\Fg(\LpreP_k(\swa uv))\LpR_{k,m}(\swa uv)$, then 
$\psi(\swa ij)$ is the concatenation of a suffix 
of $\psi(\LpreP_k(\mathcal E))$ of length $x\le P$ and the 
weakly $|\lambda|$-periodic word of length $n-x$ with a left 
period $\lambda'$, which is a cyclic shift of $\lambda$ (this cyclic shift can be nontrivial if $x=0$, and 
$\swa ij$ is actually a suboccurrence in $\LpR_{k,m}(\swa uv)$), so 
$\lambda'$ is a final period as well. Recall that $n>P$, so 
$\swa ij$ cannot be a suboccurrence of $\Fg(\LpreP_k(\swa uv))$. We have at most 
$M(P+1)N$ different words that can equal $\psi(\swa ij)$.

The situation when 
$\swa ij$ is a suboccurrence of
$\RpR_{k,m}(\swa uv)\Fg(\RpreP_k(\swa uv))$
is considered similarly and gives us at most 
$M(P+1)N$ more possibilities for $\psi(\swa ij)$ as an abstract word.

If $\swa ij$ is a suboccurrence of
$\LpR_{k,m}(\swa uv)\cKer_{k,m}(\swa uv)\RpR_{k,m}(\swa uv)$,
but is not a suboccurrence
of $\Fg(\LpreP_k(\swa uv))\LpR_{k,m}(\swa uv)$ or $\RpR_{k,m}(\swa uv)\Fg(\RpreP_k(\swa uv))$, 
then $\psi(\swa ij)$ is a subword of the 
word $\delta\psi(\cKer_{k,m}(\mathcal E))\delta'$,
where $\delta$ (resp.\ $\delta'$) is the 
weakly $|\lambda|$-periodic (resp.\ $|\mu|$-periodic) word
of length $n$ with right (resp.\ left) period $\lambda'$ (resp.\ $\mu'$), which 
is a cyclic shift of $\lambda$ (resp.\ of $\mu$), and is a final period as well.
We have $|\delta\psi(\cKer_{k,m}(\mathcal E))\delta'|=2n+|\psi(\cKer_{k,m}(\mathcal E))|\le 2n+P$, 
and there are at most $N^2M(n+P+1)$ different possibilities for $\psi(\swa ij)$.

Totally, we have $2M(P+1)N+N^2M(n+P+1)$ possibilities for $\psi(\swa ij)$ in Case \ref{smallcomplcaseii}.

\underline{Case \ref{smallcomplcaseiii}.} Again write 
$$
\swa uv=\Fg(\LpreP_k(\swa uv))\LpR_{k,m}(\swa uv)\cKer_{k,m}(\swa uv)\RpR_{k,m}(\swa uv)\Fg(\RpreP_k(\swa uv)).
$$
This time $u>i$, so if Case I holds at the right, then 
$|\Fg(\LpreP_k(\swa uv))\LpR_{k,m}(\swa uv)\cKer_{k,m}(\swa uv)\RpR_{k,m}(\swa uv)|\ge |\Fg(\RR_k(\swa uv))|>n$, 
and $\swa uj$ is a 
suboccurrence in 
$\Fg(\LpreP_k(\swa uv))\LpR_{k,m}(\swa uv)\cKer_{k,m}(\swa uv)\RpR_{k,m}(\swa uv)$.
And again, if Case II holds at the right, then 
$\swa uv=\Fg(\LpreP_k(\swa uv))\LpR_{k,m}(\swa uv)\cKer_{k,m}(\swa uv)\RpR_{k,m}(\swa uv)$,
and 
$\swa uj$ is also a 
suboccurrence in 
$\Fg(\LpreP_k(\swa uv))\LpR_{k,m}(\swa uv)\cKer_{k,m}(\swa uv)\RpR_{k,m}(\swa uv)$.
So, independently on whether Case I or Case II holds at the right, 
$\swa uj$ is always a 
suboccurrence in 
$\Fg(\LpreP_k(\swa uv))\LpR_{k,m}(\swa uv)\cKer_{k,m}(\swa uv)\RpR_{k,m}(\swa uv)$.

If Case I holds at the left, then 
$|\LpR_{k,m}(\swa uv)|\ge|\Fg(\LR_k(\swa uv))|>n$, 
and
$\swa uj$ is a 
suboccurrence in 
$\Fg(\LpreP_k(\swa uv))\LpR_{k,m}(\swa uv)$.
Therefore, $\psi(\swa ij)$ is the concatenation 
of the suffix of 
$\psi(\gamma)$
of length $u-i<n$ and 
the prefix of length $n-(u-i)$
of the word 
$\psi(\LpreP_k(\mathcal E))\delta$, where $\delta$ is 
the weakly left $\lambda$-periodic
word of length $n$.
We have at most $|\E|(n-1)MN$ possibilities for $\psi(\swa ij)$.

If Case II holds at the left, then 
$\swa uj$ is a 
suboccurrence in 
$\LpR_{k,m}(\swa uv)\cKer_{k,m}(\swa uv)\RpR_{k,m}(\swa uv)$.
If $m=1$, then 
$\swa uj$ is a 
suboccurrence in 
$\cKer_{k,m}(\swa uv)\RpR_{k,m}(\swa uv)$.
Then 
$\psi(\swa ij)$
is the concatenation 
of the suffix of 
$\psi(\gamma)$
of length $u-i<n$ and 
the prefix of length $n-(u-i)$
of the word 
$\psi(\cKer_{k,m}(\mathcal E))\delta$, where $\delta$ is the weakly $|\mu|$-periodic
word of length $n$ with left period $\mu'$, which is a cyclic shift of $\mu$.
Again, we have at most $|\E|(n-1)MN$ possibilities for $\psi(\swa ij)$.

If Case II holds at the left and $m>1$, then $\psi(\LBS_k(\mathcal E))$ is a periodic sequence 
infinite to the left with period $\lambda$. As we have checked previously, 
$\swa i{u-1}$ is a suffix of $\psi(\LBS_k(\mathcal E))$, so $\psi(\swa i{u-1})$
is 
a weakly right $\lambda$-periodic word. 
$\psi(\LpR_{k,m}(\swa uv))$ is a 
weakly left $\lambda$-periodic word, 
so $\psi(\swa ij)$ is a subword in a word of the form 
$\delta\psi(\cKer_{k,m}(\mathcal E))\delta'$, 
where $\delta$ (resp.\ $\delta'$) is the weakly $|\lambda|$-periodic 
(resp.\ $|\mu|$-periodic)
word of length $n$ with a right (resp.\ a left) period $\lambda'$ (resp.\ $\mu'$), 
which is a cyclic shift of $\lambda$ (resp.\ of $\mu$), 
(so $\lambda'$ and $\mu'$ are final periods). We have at most 
$N^2M(n+P+1)$ possibilities for $\psi(\swa ij)$.

In total, Case \ref{smallcomplcaseiii} gives us at most 
$2|\E|(n-1)MN+N^2M(n+P+1)$ possibilities for $\psi(\swa ij)$.

\underline{Case \ref{smallcomplcaseiv}} is symmetric to Case \ref{smallcomplcaseiii} and
gives at most 
$2|\E|(n-1)MN+N^2M(n+P+1)$ more possibilities for $\psi(\swa ij)$.

\underline{Case \ref{smallcomplcasev}.} This time $\swa uv$ is a suboccurrence in 
$\swa ij$, so Case II must hold for $\mathcal E$ both at the left and at the right, 
otherwise, $|\swa uv|>n$. So, $\swa uv=\Fg(\Cr_k(\swa uv))=\LpR_{k,m}(\swa uv)\cKer_{k,m}(\swa uv)\RpR_{k,m}(\swa uv)$,
and we have several possibilities for the value of $m$.

First, if $m=1$ and $\ncker_k(\mathcal E)=1$, then $\swa uv=\cKer_{k,1}(\mathcal E)$
as an abstract word, and there exists a number $x\in\NN$ ($1\le x <n$) 
such that $\swa ij$ is the concatenation of the suffix of 
$\gamma$
of length $x$, the abstract word $\cKer_{k,1}(\mathcal E)$, 
and the prefix of 
$\gamma'$
of length $n-x-|\cKer_{k,1}(\mathcal E)|$.
So, there are at most $|\E|^2(n-1)M$ possibilities for $\swa ij$ and at most
$|\E|^2(n-1)M$ possibilities for $\psi(\swa ij)$ in this case.

If $m=1$, but $\ncker_k(\mathcal E)>1$, then 
$\swa uv$ can be written as
$\swa uv=\cKer_{k,1}(\swa uv)\RpR_{k,1}(\swa uv)$,
and $\psi(\RBS_k(\mathcal E))$ is a periodic sequence infinite to the right 
with period $\mu$. $\psi(\RpR_{k,1}(\swa uv))$ is 
a weakly right $\mu$-periodic word.
As we have checked previously, 
$\swa{v+1}j$ is a prefix of $\RBS_k(\mathcal E)$, 
so $\psi(\swa uj)$ is a prefix of a word of the form $\psi(\cKer_{k,1}(\mathcal E))\delta$, 
where $\delta$ is the weakly $|\mu|$-periodic word of length $n$ 
with left period $\mu'$, which is a cyclic shift of $\mu$ (and is a final period as well).
And $\psi(\swa i{u-1})$ is the suffix of length $u-i$ ($0<u-i<n$)
of $\psi(\gamma)$. We get at most $|\E|(n-1)MN$ possibilities for $\psi(\swa ij)$.

The situation when $m=\ncker_k(\mathcal E)$ and $\ncker_k(\mathcal E)>1$
is symmetric to the situation when $m=1$ and $\ncker_k(\mathcal E)>1$, 
so it gives us at most $|\E|(n-1)MN$ more possibilities for $\psi(\swa ij)$.

Finally, if $1<m<\ncker_k(\mathcal E)$, then $\psi(\LBS_k(\mathcal E))$
is a periodic sequence infinite to the left with period $\lambda$,
$\psi(\RBS_k(\mathcal E))$
is a periodic sequence infinite to the right with period $\mu$, 
$\swa i{u-1}$ is a suffix of $\LBS_k(\mathcal E)$, and 
$\swa{v+1}j$ is a prefix of $\RBS_k(\mathcal E)$.
Also, $\psi(\LpR_{k,m}(\swa uv))$ is 
a weakly left $\lambda$-periodic word, 
and 
$\psi(\RpR_{k,m}(\swa uv))$ is 
a weakly right $\mu$-periodic word.
Therefore, $\psi(\swa ij)$ as an abstract word is a subword of a 
word of the form $\delta\psi(\cKer_{k,m}(\mathcal E))\delta'$, 
where $\delta$ is the weakly $|\lambda|$-periodic word of length $n$ 
with right period $\lambda'$, which is a cyclic shift of $\lambda$,
and 
where $\delta'$ is the weakly $|\mu|$-periodic word of length $n$ 
with left period $\mu'$, which is a cyclic shift of $\mu$.
$\lambda'$ and $\mu'$ are final periods, $|\delta\psi(\cKer_{k,m}(\mathcal E))\delta'|\le 2n+P$, 
so we have at most 
$N^2M(n+P+1)$ possibilities for $\psi(\swa ij)$.

Totally, in Case \ref{smallcomplcasev} we have at most 
$|\E|^2(n-1)M+2|\E|(n-1)MN+N^2M(n+P+1)$ possibilities for $\psi(\swa ij)$.

Summarizing, if $q=l_0$, then we have at most 
$|\E|^2(n-1)+2M(P+1)N+N^2M(n+P+1)+2(2|\E|(n-1)MN+N^2M(n+P+1))+|\E|^2(n-1)M+2|\E|(n-1)MN+N^2M(n+P+1)$
possibilities for $\psi(\swa ij)$ as an abstract word. Since $|\E|$, $M$, $N$, and $P$ 
do not depend on $n$, this number is $O(n)$ for $n\to\infty$.

Now let us consider the case when $q<l_0$. Recall that in this case,
$2\le |\swa st|\le |\varphi|$. Denote by $w$ the index such that $\al_w$ 
is the rightmost letter in $\varphi(\al_s)$. Then $\swa ij$ is 
the concatenation of the suffix of $\varphi^q(\al_s)$ of length $x=|\swa iw|=w-i+1<n$
and the prefix of $\varphi^q(\swa{s+1}t)$ of length $n-x$. So, 
$\swa ij$ as an abstract word is determined by the following data: 
a word of length at least two and at most $|\varphi|$, which will be $\swa st$, 
and two numbers, $q\in\ZZ_{\ge 0}$ ($q\le \lceil f(n)\rceil+1\le f(n)+2$) and $x\in\NN$ ($1\le x<n$).
(This time we need to know $q$ since it is not determined by $n$ uniquely anymore.)
There are at most $(|\E|^2+|\E|^3+\ldots+|\E|^{|\varphi|})(f(n)+2)n$ possibilities for $\swa ij$,
and hence at most $(|\E|^2+\ldots+|\E|^{|\varphi|})(f(n)+2)n$ possibilities for $\psi(\swa ij)$.
$|\E|$ and $|\varphi|$ do not depend on $n$, so this number is $O(n(f(n)+1))$ for $n\to\infty$.

Overall, we have at most $O(n)+O(n(f(n)+1))=O(n(f(n)+1))$ possibilities for $\psi(\swa ij)$ as an abstract word.
Since $f(n)\ge 3k$ if $n\ge n_0$, we can say that the function $n\mapsto 1$ is also $O(f(n))$ for $n\to\infty$, 
and $O(n(f(n)+1))=O(nf(n))$.
\end{proof}

\begin{proof}[Proof of Proposition \ref{smallcompl}]
Consider the following sequences depending on $l\in\ZZ_{\ge 0}$. Previously we have seen that 
all of them have asymptotic $\Omega(l^k)$ for $l\to\infty$.
\begin{enumerate}
\item The sequences 
$|\Fg(\LR_k(\mathcal E_l))|$ (resp.\ $|\Fg(\RR_k(\mathcal E_l))|$) 
for all evolutions of $k$-blocks such that Case I holds for 
$\mathcal E$ at the left (resp.\ at the right) (see Lemma \ref{regpartasym}).
\item The sequences $|\varphi^l(b)|$, where $b\in\E$ is a letter of order $>k$,
including letters of order $\infty$ (see the definition of the order of a letter).
\end{enumerate}
For the sequences 
$|\Fg(\LR_k(\mathcal E_l))|$ and
$|\Fg(\RR_k(\mathcal E_l))|$
mentioned here, Lemma \ref{regpartasym}
actually says that the asymptotic of these 
sequences is $\Theta(l^k)$, and
constants in the $\Theta$-notation
do not depend on $\mathcal E$. So, we may suppose that the constants in the $\Omega$-notation
do not depend on $\mathcal E$. 
As for the sequences $|\varphi^l(b)|$, where $b\in\E$ is a letter of order $>k$,
there are only finitely many of them since there are only finitely many letters in $\al$, 
so we may also suppose that the constants in the $\Omega$-notation do not depend on $b$.
Therefore, there exist $l_0\in\ZZ_{\ge 0}$ and $x\in\R_{>0}$
such that for all $l\ge l_0$ the following is true:
\begin{enumerate}
\item If $\mathcal E$ is an 
evolution of $k$-blocks such that Case I holds for 
$\mathcal E$ at the left (resp.\ at the right),
then
$|\Fg(\LR_k(\mathcal E_l))|>xl^k$ (resp.\ $|\Fg(\RR_k(\mathcal E_l))|>xl^k$).
\item If $b\in\E$ is a letter of order $>k$, then $|\varphi(b)|>xl^k$.
\end{enumerate}
Without loss of generality, we will suppose that $l_0\ge 3k$.
Set $n_0=\lceil xl_0^k\rceil$. Consider the following 
function $f\colon \NN\to\R$: $f(n)=\sqrt[k]{n/x}$.
If $n\ge n_0$, then $f(n)=\sqrt[k]{n/x}\ge \sqrt[k]{n_0/x}\ge \sqrt[k]{(xl_0^k)/x}=l_0\ge 3k$.
If $n\ge n_0$, $l\in\NN$, and $l\ge f(n)$, then 
$l\ge \sqrt[k]{n/x}$, so $l^k\ge n/x$, $xl^k\ge n$, and, since $l\ge f(n)\ge l_0$, we have the following 
inequalities:
\begin{enumerate}
\item If $\mathcal E$ is an 
evolution of $k$-blocks such that Case I holds for 
$\mathcal E$ at the left (resp.\ at the right),
then
$|\Fg(\LR_k(\mathcal E_l))|>xl^k\ge n$ (resp.\ $|\Fg(\RR_k(\mathcal E_l))|>xl^k\ge n$).
\item If $b\in\E$ is a letter of order $>k$, then $|\varphi(b)|>xl^k\ge n$.
\end{enumerate}
Therefore, $f$ satisfies the conditions of Lemma \ref{smallcompllemma}, and the subword complexity
of $\be=\psi(\al)$ is $O(nf(n))=O(n^{1+1/k})$.
\end{proof}

Now we can check that if evolutions of $k$-blocks really exist in $\al$ and all of them are continuously
periodic, then all $k$-series of 
obstacles $\mathcal H$ in $\al$ actually satisfy $|\mathcal H_l|=\Theta(l^k)$.
(Actually, we do not need this fact to prove any subsequent lemmas, propositions or theorems.)
Indeed, if $\al=\varphi^\infty(a)$ and $k$-blocks exist in $\al$, then by Lemma \ref{splitconcatenation}, 
$a$ is a letter of order $\ge k+2$. If, in addition, all evolutions of $k$-blocks
are continuously periodic, 
then by Proposition \ref{smallcompl}, the subword complexity of $\be=\psi(\al)$
is $O(n^{1+1/k})$. But we also have seen in the proof of Lemma \ref{obstcompl}
that if $\mathcal H$ is a $k$-series of obstacles such that 
$|\mathcal H_l|$ is $\Theta(l^{k'})$ for $l\to\infty$ and $k'<k$, 
then the subword complexity of $\be$ is $\Omega(n^{1+1/k'})$, so it is 
$\omega(n^{1+1/k})$, and we have a contradiction. Therefore, 
$|\mathcal H_l|$ is in fact $\Theta(l^k)$ for $l\to\infty$.

To prove Proposition \ref{finordercomplstop}, we first prove the following lemma.

\begin{lemma}\label{finorderstoplemma}
Let $k\in\NN$.
Suppose that $\al=\varphi^\infty(a)$, where $a$ is a letter of order $k+2$.
Let $l\in\ZZ_{\ge 0}$.
Let $\al_s$ (resp.\ $\al_i$,$\al_j$) be the rightmost letter of order $\ge k+1$ in $\varphi(a)$
(resp.\ in $\varphi^l(a)$, in $\varphi^{l+1}(a)$).

Then $i<j$, 
$\al[<,i+1\ldots j,<]_k$ and $\al[<,1\ldots s,<]_k$
are $k$-multiblocks, each of them begins with 
a (possibly empty) $k$-block and ends with a letter of order $k+1$, and 
$\al[<,i+1\ldots j,<]_k=\Su_k^l(\al[<,1\ldots s,<]_k)$.
\end{lemma}

\begin{proof}
First, note that the $k$-multiblocks in the statement of the lemma really exist and are nonempty.
Indeed, Lemma \ref{splitconcatenation} says in this case that the only letter of 
order $k+2$ in $\al$ is $\al_0=a$, and all other letters in $\al$ are of order $\le k+1$.
If $\varphi(a)=a\gamma$, then for all $n\in\NN$ we have
$\varphi^n(a)=\varphi^{n-1}(a)\varphi^{n-1}(\gamma)=a\gamma\varphi(\gamma)\ldots\varphi^{n-1}(\gamma)$.
Since $\varphi(a)$ contains $a$, $a$ can only be a periodic letter of order $k+2$, 
not a preperiodic letter of order $k+2$. Then $\gamma$ contains at least one letter of 
order $k+1$, and $\varphi^n(\gamma)$ contains at least one letter of order $k+1$ for all $n\in\NN$.
Hence, $\varphi^n(a)$ contains at least one letter of order $k+1$. Moreover, for all $n\in\NN$, 
the rightmost letter of 
order $k+1$ in $\varphi^n(a)$ is the rightmost letter of order $k+1$ in 
$\varphi^{n-1}(\gamma)$ since $\varphi^{n-1}(\gamma)$ contains at least one letter of order $k+1$.
Hence, $i<j$, $1<s$. 

Since $\al_0$ and $\al_i$ are letters of order $\ge k+1$, there really 
exist (possibly empty) $k$-blocks of the form $\swa{i+1}{j'}$ and $\swa1{s'}$, so
the notations $<,i+1$ and $<,1$ really denote delimiters. Since $\al_j$ and $\al_s$ 
are also letters of order $\ge k+1$, $j,<$ and $s,<$ also denote delimiters, and 
the delimiter $j,<$ (resp.\ $s,<$) is located at the right-hand side of $<,i+1$
(resp.\ of $<,1$). Therefore, $\al[<,i+1\ldots j,<]_k$ and $\al[<,1\ldots s,<]_k$
are really nonempty $k$-multiblocks, and each of them begins with a (possibly empty) 
$k$-block and ends with a letter of order $k+1$.

If $l=0$, then $i=0$ and $j=s$, so the statement of the lemma is trivial. 
Otherwise, we prove the statement by induction on $l$. Let $\al_t$ be the rightmost letter of order $\ge k+1$ in 
$\varphi^{l+2}(a)$. 
We have to prove that $\al[<,j+1\ldots t,<]_k=\Su_k(\al[<,i+1\ldots j,<]_k)$.
$\al_t$ is also the rightmost letter of order $\ge k+1$ in $\varphi^{l+1}(\gamma)$, 
so, since $\al_j$ is the rightmost letter of order $\ge k+1$ in $\varphi^l(\gamma)$, and 
the image of a letter of order $\le k$ consists of letters of order $\le k$ only, 
$\al_t$ is also the rightmost letter of order $\ge k+1$ in $\varphi(\al_j)$. Similarly, 
$\al_j$ is the rightmost letter of order $\ge k+1$ in $\varphi(\al_i)$.

Therefore, $\Su_k(\al[>,j\ldots j,<]_k)$ is a $k$-multiblock that ends with the letter $\al_t$ 
of order $\ge k+1$. If $\swa{i+1}{j'}$ is a $k$-block, then $\al_i=\LB(\swa{i+1}{j'})$, 
and $\LB(\Su_k(\swa{i+1}{j'}))$ is the rightmost letter of order $>k$ in $\varphi(\al_i)$,
i.~e.\ $\LB(\Su_k(\swa{i+1}{j'}))=\al_j$. Hence, $\Su_k(\swa{i+1}{j'})$ is a $k$-block of the 
form $\swa{j+1}{t'}$ for some $t'\ge j$, and, by the definition of the descendant of a $k$-multiblock, 
$\Su_k(\al[<,i+1\ldots j,<]_k)$ is the $k$-multiblock that begins with the $k$-block 
$\swa{j+1}{t'}$ and ends with the letter $\al_t$ of order $\ge k+1$.
Therefore, $\al[<,j+1\ldots t,<]_k=\Su_k(\al[<,i+1\ldots j,<]_k)$.
\end{proof}

\begin{proof}[Proof of Proposition \ref{finordercomplstop}]
For each $l\ge 0$, denote by $s_l$ the index such that $\al_{s_l}$
is the rightmost letter of order $\ge k+1$ in $\varphi^l(a)$.
By Lemma \ref{finorderstoplemma}, 
$\al[<,s_l+1\ldots s_{l+1},<]_k=\Su_k^l(\al[<,1\ldots s_1,<]_k)$ for all $l\ge 0$.
By Lemma \ref{splitconcatenation}, the only letter of order $k+2$ in $\al$ is 
$\al_0=a$, so all letters of order $\ge k+1$ in 
$\al[<,1\ldots s_1,<]_k$ are actually of order $k+1$.
All letters of order $\ge k+1$ in 
$\Su_k^l(\al[<,1\ldots s_1,<]_k)$ are contained in the images 
under $\varphi^l$
of the letters of order $k+1$ from $\al[<,1\ldots s_1,<]_k$, 
so, if $l\ge 1$, then all letters of order 
$\ge k+1$ in 
$\Su_k^l(\al[<,1\ldots s_1,<]_k)=\al[<,s_l+1\ldots s_{l+1},<]_k$
are actually \textit{periodic} letters of order $k+1$.

In particular, this is true for $l=1$. Therefore, if $l\ge 1$, then 
all $k$-blocks in
$\al[<,s_l+1\ldots s_{l+1},<]_k=\Su_k^l(\al[<,1\ldots s_1,<]_k)=\Su_k^{l-1}(\Su_k(\al[<,1\ldots s_1,<]_k))=
\Su_k^{l-1}(\al[<,s_1+1\ldots s_2,<]_k)$ are the $(l-1)$th 
superdescendants of the $k$-blocks in
$\al[<,s_1+1,s_2,<]_k$. Hence, if $l\ge 1$, then the evolutional sequence 
number of each $k$-block contained in $\al[<,s_l+1\ldots s_{l+1},<]_k$
is always at least $l-1$. So, if $l\ge 3k+1$, then 
all $k$-blocks contained in $\al[<,s_l+1\ldots s_{l+1},<]_k$
are stable, and all letters of order $\ge k+1$ contained in 
$\al[<,s_l+1\ldots s_{l+1},<]_k$ are periodic letters of order $k+1$. 
Therefore, if $l\ge 3k+1$, then 
$\al[<,s_l+1\ldots s_{l+1},<]_k$ is a stable $k$-multiblock. It is nonempty by Lemma \ref{finorderstoplemma}.

Consider the following evolution $\mathcal F$ of stable nonempty 
$k$-multiblocks: $\mathcal F_0=\al[<,s_{3k+1}+1\ldots$\linebreak$s_{3k+2},<]_k$, 
$\mathcal F_l=\Su_k^l(\al[<,s_{3k+1}+1\ldots s_{3k+2},<]_k)=\Su_k^l(\Su_k^{3k+1}(\al[<,1\ldots s_1,<]_k))=
\Su_k^{l+3k+1}(\al[<,1\ldots s_1,<]_k)=\al[<,s_{l+3k+1}+1\ldots s_{l+3k+2},<]_k$
for $l\ge 0$.
Then we can write $\al=\swa0{s_{3k+1}}\swa{s_{3k+1}+1}{s_{3k+2}}\swa{s_{3k+2}+1}{s_{3k+3}}\ldots\swa{s_{l+3k+1}+1}{s_{l+3k+2}}\ldots=
\swa0{s_{3k+1}}\Fg(\mathcal F_0)\Fg(\mathcal F_1)\ldots\Fg(\mathcal F_l)\ldots$. 
Note that $\Fg(\mathcal F_l)$ contains a letter (namely, the rightmost letter of 
order $\ge k+1$ in $\varphi^{l+3k+2}(a)$, which we denote by $\al_{s_{l+3k+2}}$), 
which does not belong to $\varphi^{l+3k+1}(a)$ by Lemma \ref{finorderstoplemma}.
Hence, $|\swa0{s_{3k+1}}\Fg(\mathcal F_0)\Fg(\mathcal F_1)\ldots\Fg(\mathcal F_l)|>|\varphi^{l+3k+1}(a)|$, 
and the infinite concatenation
$\swa0{s_{3k+1}}\Fg(\mathcal F_0)\Fg(\mathcal F_1)\ldots\Fg(\mathcal F_l)\ldots$
is really an infinite word, and it covers the\linebreak whole $\al$.

First, let us check that $\nker_k(\mathcal F)>1$.
Indeed, otherwise $\Fg(\mathcal F_l)=\Ker_{k,1}(\mathcal F_l)$ for all $l\ge 0$, and
by Lemma \ref{compkernelfixed}, all words $\Fg(\mathcal F_l)$
coincide as abstract words. But then 
$|\swa0{s_{3k+1}}\Fg(\mathcal F_0)\ldots\Fg(\mathcal F_l)|$ is $O(l)$ for $l\to\infty$.
On the other hand, we know that
$|\swa0{s_{3k+1}}\Fg(\mathcal F_0)\ldots\Fg(\mathcal F_l)|>|\varphi^{l+3k+1}(a)|$, 
and $a$ is a letter of order $k+2\ge 3$ since $k\in\NN$, so
$|\varphi^{l+3k+1}(a)|$ is $\Theta(l^{k-1})$ for $l\to\infty$, 
and we have a contradiction.

So, $\nker_k(\mathcal F)>1$. 
Note that the indices $\al_i$ and $\al_j$ from the statement of the proposition can now 
be written as $i=s_{3k+1}$ and $j=s_{3k+2}$, so $\Fg(\mathcal F_0)=\Fg(\al[<,s_{3k+1}+1\ldots s_{3k+2},<]_k)=
\swa{s_{3k+1}+1}{s_{3k+2}}=\swa{i+1}j$.
First, let us consider the case when there exists a final 
period $\lambda$ such that $\psi(\swa{i+1}j)=\psi(\Fg(\mathcal F_0))$
is a completely $\lambda$-periodic word. 
By Lemma \ref{oneperenoughfortotleft}, in this case 
$\lambda$ is a total left evolutional period of $\mathcal F$.
By the definition of a total left evolutional period, 
the remainder of $|\mathcal F_l|$ modulo $|\lambda|$ does not depend 
on $l$, so it equals
the remainder of $|\mathcal F_0|$ modulo $|\lambda|$, 
which is zero. Therefore, all words $\psi(\Fg(\mathcal F_l))$
are completely $\lambda$-periodic words, 
and 
$\be=\psi(\al)=\psi(\swa0{s_{3k+1}})\psi(\Fg(\mathcal F_0))\ldots\psi(\Fg(\mathcal F_l))\ldots$
is an eventually periodic word with period $\lambda$, and its subword complexity is $O(1)$.

Now suppose that there exist no final 
period $\lambda$ such that $\psi(\swa{i+1}j)$
is a completely $\lambda$-periodic word.
Since all evolutions of $k$-blocks present in $\al$ are continuously periodic, 
Proposition \ref{smallcompl} guarantees that the subword complexity 
of $\be$ is $O(n^{1+1/k})$.
We are going to prove that there exists a $k$-series of 
obstacles in $\al$. Assume the contrary.

Consider the following evolutions of $k$-multiblocks:
$\mathcal F'$ defined by $\mathcal F'_l=\mathcal F_{l+1}$,
and $\mathcal F''$ defined by $\mathcal F''_l=\mathcal F_{l+2}$.
Each $k$-multiblock $\mathcal F'_l=\mathcal F_{l+1}$ begins with a
$k$-block, so $\mathcal F_l$ and $\mathcal F'_l$ are consecutive
as $k$-multiblocks (there cannot be an empty $k$-block between them).
So, the evolutions $\mathcal F$ and $\mathcal F'$ are consecutive. 
Similarly, $\mathcal F'$ and $\mathcal F''$ are consecutive.
By Lemma \ref{atleastonetotal}, at least one of the evolutions $\mathcal F$, 
$\mathcal F'$, and $\mathcal F''$ is totally periodic.
Hence, there exists a final period $\lambda$ such that 
at least one of the words 
$\psi(\Fg(\mathcal F_0))$, 
$\psi(\Fg(\mathcal F'_0))=\psi(\Fg(\mathcal F_1))$, 
or
$\psi(\Fg(\mathcal F''_0))=\psi(\Fg(\mathcal F_2))$
is a weakly left $\lambda$-periodic word.

By Lemma \ref{oneperenoughfortotleft}, this means that $\lambda$
is a total left evolutional period of $\mathcal F$. Now it follows directly 
form the definition of a total left evolutional period that 
$\lambda$
is also a total left evolutional period of $\mathcal F'$ and of $\mathcal F''$.
Then Lemma \ref{atleastonereallytotal} implies that 
$\lambda$ is either a right total evolutional period of $\mathcal F$, 
or a right total evolutional period of $\mathcal F'$. In particular, 
at least one of the words 
$\psi(\Fg(\mathcal F_0))$ or
$\psi(\Fg(\mathcal F'_0))=\psi(\Fg(\mathcal F_1))$
is a weakly right $\lambda$-periodic word.
By Lemma \ref{oneperenoughfortotright}, $\lambda$
is a total right evolutional period of $\mathcal F$.

Now we know that $\lambda$ is both left and right total evolutional
period of $\mathcal F$.
In particular, 
$\psi(\Fg(\mathcal F_0))$
is weakly $|\lambda|$-periodic word with both left and right period $\lambda$.
Since $\nker_k(\mathcal F)>1$, by Corollary \ref{betweenkernelgrowsgood}, 
$|\Fg(\mathcal F_0)|\ge 2\finmax$. So, by Lemma \ref{leftrightcompleteperiod}, 
$\psi(\Fg(\mathcal F_0))=\psi(\swa{i+1}j)$
is a completely $\lambda$-periodic word,
and we have a contradiction (we are considering the case when
there exist no final period $\lambda$ such that
$\psi(\swa{i+1}j)$
is a completely $\lambda$-periodic word).

Therefore, there exists a $k$-series of obstacles in $\al$, and, by 
Lemma \ref{obstcompl}, the subword complexity of $\be=\psi(\al)$ 
is $\Omega(n^{1+1/k})$. We already know that 
the subword complexity of $\be$ 
is $O(n^{1+1/k})$, so the subword complexity of $\be$ is $\Theta(n^{1+1/k})$.
\end{proof}

\begin{proof}[Proof of Proposition \ref{ordertwocompl}]
Write $\varphi(a)=a\gamma$. Then for all $l\in\NN$ we have
$\varphi^l(a)=\varphi^{l-1}(a)\varphi^{l-1}(\gamma)=a\gamma\varphi(\gamma)\ldots \varphi^{l-1}(\gamma)$
and $\al=\varphi^\infty(a)=a\gamma\varphi(\gamma)\ldots \varphi^l(\gamma)\ldots$.
Since $a$ is contained in $\varphi(a)$, $a$ must be a periodic letter of order 2, 
not a preperiodic letter of order 2. Then all letters in $\gamma$ have order 1.
Since we have assumed that $\varphi$ is (in particular)
weakly 1-periodic morphism, $\varphi(\gamma)$ consists of periodic letters of order 1 only, and 
$\varphi^2(\gamma)=\varphi(\gamma)$. Therefore, 
$\al=a\gamma\varphi(\gamma)\varphi(\gamma)\ldots\varphi(\gamma)\ldots$, 
$\al$ is an eventually periodic sequence with period 
$\varphi(\gamma)$, so $\be=\psi(\al)$ is 
an eventually periodic sequence with period 
$\psi(\varphi(\gamma))$, and the subword complexity of $\be$ is $O(1)$.
\end{proof}

\begin{proof}[Proof of Proposition \ref{infordercomplstop}]
By Lemma \ref{splitconcatenation}, $\al$ can be split into a concatenation 
of $k$-blocks and letters of order $>k$ (i.~e.\ letters of order $\infty$ in this case). 
In particular, $k$-blocks still exist (although it is possible that all of them
are empty occurrences), but Case II must hold for all evolutions
of $k$-blocks since there are no letters of order $k$.

We are going to use Lemma \ref{smallcompllemma}.
Recall that if $b$ is a letter of order $\infty$, then there exists $q\in\R$, $q>1$
such that $|\varphi^l(b)|$ is $\Omega(q^l)$ for $l\to\infty$.
Let $q_0$ be the minimal number $q$ for all letters of order $\infty$. 
Since there are finitely many letters in $\E$, there exists $l_0\in \NN$ 
and $x\in\R_{>0}$ such that if $l\ge l_0$ and $b$ is a letter of order $\infty$, then $|\varphi^l(b)|>xq_0^l$.
Without loss of generality, $l_0\ge 3k$. Set $n_0=\lceil xq_0^{l_0}\rceil$, 
and consider the following function $f\colon \NN\to\R$: 
$f(n)=\log_{q_0}(n/x)$. If $n\ge n_0$, then 
$f(n)=\log_{q_0}(n/x)\ge \log_{q_0}(n_0/x)\ge\log_{q_0}((xq_0^{l_0})/x)=\log_{q_0}(q_0^{l_0})=l_0\ge 3k$, 
so the first condition in Lemma \ref{smallcompllemma}
is satisfied. We do not have to check the second condition in Lemma \ref{smallcompllemma}
since Case II holds both at the left and at the right for all evolutions of $k+1$-blocks
in $\al$. For the third condition we observe that if 
$l,n\in\NN$, $n\ge n_0$, $l\ge f(n)$, and $b\in\E$ is a letter of order $>k+1$, then 
$b$ is a letter of order $\infty$, $l\ge l_0$, so $|\varphi^l(b)|>xq_0^l$.
We also have $l\ge \log_{q_0}(n/x)$, so $q_0^l\ge n/x$, $xq_0^l\ge n$, and 
$|\varphi^l(b)|>n$. Therefore, we can use Lemma \ref{smallcompllemma}.
By Lemma \ref{smallcompllemma}, the subword complexity of $\be=\psi(\al)$ is 
$O(n\log_{q_0}(n/x))=O(n\log_{q_0}n-n\log_{q_0}x)=O(n\log n)$.
\end{proof}

\begin{proof}[Proof of Theorem \ref{maintheorem}]
Let $a\in\E$ be a letter such that $\varphi(a)=a\gamma$ for a nonempty word $\gamma$, 
and let $\al=\varphi^\infty(a)$. Consider the following three cases: $a$ can be either a letter of order 2, or 
a letter of finite order $K>2$, 
or a letter of order $\infty$.

If $a$ is a letter of order 2, then the subword complexity of $\be=\psi(\al)$ is $O(1)$ by Proposition \ref{ordertwocompl}.

If $a$ is a letter of a finite order $K>2$, 
then denote by $k\in\NN$ the maximal number among the numbers $1,2,\ldots,K-2$
such that all evolutions of $k$-blocks arising in $\al$ are continuously
periodic. (Recall that all evolutions of 1-blocks are always continuously periodic, see 
Remark \ref{onecontper}.) By Proposition \ref{smallcompl}, the subword complexity of $\be=\psi(\al)$ is 
$O(n^{1+1/k})$.
If $k<K-2$, then 
by Lemma \ref{splitconcatenation}, 
$\al$ can be split into a concatenation of $(k+1)$-blocks and letters of order $>k+1$, 
so there exist evolutions of $(k+1)$-blocks in $\al$, and there exists 
a non-continuously periodic evolution of $(k+1)$-blocks in $\al$. 
So,
by Proposition \ref{largecompl}, the subword complexity of $\be$ is $\Omega(n^{1+1/k})$, therefore 
it is $\Theta(n^{1+1/k})$. If $k=K-2$, then by Proposition \ref{finordercomplstop}, 
the subword complexity of $\be$ is either $\Theta(n^{1+1/k})$, or it is $O(1)$.

Finally, suppose that $a$ is a letter of order $\infty$. Let $K$ be the maximal \textit{finite} order
of letters occurring in $\al$ (i.~e.\ all letters occurring in $\al$ are either of order $\le K$, or of order $\infty$).
Denote by $k\in\NN$ the maximal number among the numbers $1,2,\ldots,K+1$
such that all evolutions of $k$-blocks in $\al$ are continuously
periodic. Again, by Proposition \ref{smallcompl}, the subword complexity 
of $\be=\psi(\al)$ is $O(n^{1+1/k})$. If $k<K+1$, then there exists 
an evolution $\mathcal E$ of $k+1$-blocks such that $\mathcal E$ is 
not continuously periodic. By Proposition \ref{largecompl}, 
the subword complexity of $\be$ is $\Omega(n^{1+1/k})$.
Therefore, the subword complexity of $\be$ is $\Theta(n^{1+1/k})$.
If $k=K+1$, then by Proposition \ref{infordercomplstop}, the subword complexity of $\be$ is $O(n\log n)$.
\end{proof}

\section{An Example}\label{An_Example}

Here we give an example of a sequence with complexity $\T(n^{3/2})$.
The easiest way to construct such a sequence is to use Proposition \ref{finordercomplstop}.

Let $\E = \{1,2,3,4\}$. Consider the following morphism $\varphi$:
$\varphi(4) = 43, \varphi(3) = 32, \varphi(2) = 21, \varphi(1) = 1$
and the following coding $\psi$:
$\psi(4)=4,\psi(3)=3,\psi(2)=\psi(1)=1$. 
Here all letters are periodic, $4$ is a letter of order 4, 
$3$ is a letter of order 3, $2$ is a letter of order 2, and $1$ is a letter of order 1.
$\varphi$ is a weakly 1-periodic morphism, but it is not a strongly
1-periodic morphism.
However, for $\varphi^4$ we have 
$\varphi^4(4)=433232213221211$,
$\varphi^4(3)=32212112111$,
$\varphi^4(2)=21111$, and
$\varphi^4(1)=1$, 
so $\varphi^4$ is a strongly 1-periodic morphism.
The only final period we have here consists of a single letter $1$, 
so $\finmax=1$, and $\varphi^4$ is also a strongly 1-periodic morphism with long images.

To construct the pure morphic sequence, we can use $\varphi$ as well as $\varphi^4$, the resulting 
sequence will be the same.

So, we have $\varphi^\infty(4) =
\al = 433232213221211322121121113\ldots$, and 
$\psi(\varphi^\infty(4)) = \be = 433131113111111311111111113\ldots$
%
%
Here we have several evolutions of 2-blocks (when we speak about evolutions of 2-blocks, we should use the morphism $\varphi^4$, 
not just $\varphi$, to compute the descendants since when we defined descendants, we assumed that 
$\varphi$ is a strongly 1-periodic morphism with long images),
and it is clear that all left borders and all right borders of all these evolutions equal $3$ 
as abstract letters. Therefore, Case I holds for all evolutions of 2-blocks at the left.
The images under $\psi$ of all 2-blocks here consist of letters $1$ only, 
so each evolution $\mathcal E$ of 2-blocks here is continuously periodic 
for the index $\nker_2(\mathcal E)$. Therefore, we can apply Proposition \ref{finordercomplstop}
and say that the subword complexity is $\Theta(n^{3/2})$.

\subsection*{Acknowledgments}
The author thanks Yu. Pritykin, Yu. Ulyashkina, and N. Vereshchagin for a useful
discussion.

\end{document}